\documentclass[a4paper,11pt]{article}

\usepackage{amsmath}
\usepackage{amssymb}
\usepackage{amsthm}
\usepackage{graphicx}
\usepackage{psfrag}
\usepackage{verbatim}
\usepackage{color}
\usepackage{enumerate, enumitem}
\usepackage[%
ocgcolorlinks,%
linkcolor=blue,%
filecolor=blue,%
citecolor=blue,%
urlcolor=blue]{hyperref}

\usepackage[font={small,it}]{caption}

\usepackage[toc]{appendix}

\newcommand{\Ind}{1}

\newcommand{\bbH}{\mathbb{H}}
\newcommand{\bbL}{\mathbb{L}}

\newcommand{\bbE}{\mathbb{E}}
\newcommand{\bbV}{\mathbb{V}}

\newcommand{\Var}{\bbV{\rm ar}}
\newcommand{\Cov}{\bbC{\rm ov}}
\newcommand{\bbP}{\mathbb{P}}

\newcommand{\bbN}{\mathbb{N}}
\newcommand{\bbZ}{\mathbb{Z}}
\newcommand{\bbR}{\mathbb{R}}
\newcommand{\bbC}{\mathbb{C}}

\newcommand{\eps}{\epsilon}

\newcommand{\ol}{\overline}
\newcommand{\ubar}[1]{\underbar{$#1$}}

\newcommand{\wh}{\widehat}

\newcommand{\cX}{\mathcal X}

\newcommand{\cB}{\mathcal B}

\newcommand{\cW}{\mathcal W}
\newcommand{\cY}{\mathcal Y}
\newcommand{\cS}{\mathcal S}

\newcommand{\cN}{\mathcal N}

\newcommand{\cL}{\mathcal L}
\newcommand{\cR}{\mathcal R}

\newcommand{\cD}{\mathcal D}
\newcommand{\dH}{\rmd_{\bbH^\rmc}}
\newcommand{\Leb}{{\rm Leb}}

\newcommand{\frD}{\mathfrak D}

\newcommand{\frg}{\mathfrak g}
\newcommand{\fra}{\mathfrak a}

\newcommand{\frU}{\mathfrak U}
\newcommand{\frV}{\mathfrak V}

\newcommand{\rmc}{{\rm c}}

\newcommand{\rmd}{{\rm d}}
\newcommand{\rme}{{\rm e}}
\newcommand{\rmB}{{\rm B}}

\newcommand{\bfP}{{\mathbf P}}
\newcommand{\bfE}{{\mathbf E}}
\newcommand{\bfCov}{\mathbf{Cov}}
\newcommand{\bfVar}{\mathbf{Var}}

\newcommand{\rt}{T}
\newcommand{\rA}{A'}

\newcommand{\rr}{r}
\newcommand{\clR}{R}
\newcommand{\diam}{\mathop {\rm diam}}

\newcommand{\wt}{\widetilde}

\newcommand{\osc}{{\rm osc}}

\newcommand{\lto}{\longrightarrow}

\newcommand{\he}[2]{\overline{{#1}_{{#2}}}}
\newcommand{\heb}[2]{\overline{({#1})_{{#2}}}}
\newcommand{\hebb}[2]{\overline{\big({#1}\big)_{{#2}}}}

\newcommand{\OrenCO}[1]{}

\definecolor{purple}{rgb}{0.5,0,0.5}

\newtheorem{theorem}{Theorem}[section]
\newtheorem{cor}[theorem]{Corollary}
\newtheorem{lem}[theorem]{Lemma}
\newtheorem{prop}[theorem]{Proposition}
\newtheorem{thm}[theorem]{Theorem}
\newtheorem*{rem*}{Remark}
\newtheorem{rem}[theorem]{Remark}

\numberwithin{equation}{section}
\graphicspath{{pic/}}

\title
{Ballot Theorems for the
Two-Dimensional\\Discrete Gaussian Free Field}
\author
{Stephan Gufler\thanks{stephan.gufler@gmx.net} \\ Technion, Israel
\and Oren Louidor\thanks{oren.louidor@gmail.com}\\ Technion, Israel} 
\date{}

\begin{document}

\maketitle

\begin{abstract}
We provide uniform bounds and asymptotics for the probability that a two-dimensional discrete Gaussian free field on an annulus-like domain and with Dirichlet boundary conditions stays negative as the ratio of the radii of the inner and the outer boundary tends to infinity.
\end{abstract}

\setcounter{tocdepth}{2}
{\parskip=0pt \tableofcontents}

\section{Introduction and statement of the main results}

\subsection{Setup}
In this work we provide estimates for the probability that a two-dimensional discrete Gaussian free field (DGFF) on a domain which is a suitable generalization of an annulus, stays negative. We give both asymptotics and bounds for such probabilities in terms of the assumed Dirichlet boundary conditions, as the ratio of the radii of the inner and outer boundary tends to infinity. These results should be seen as the two-dimensional analogs of the classical ballot-type theorems, in which the probability that one-dimensional random walk (or bridge) with prescribed starting and ending positions stays negative for a given number of steps. We shall therefore keep referring to such probabilities and results as being of ballot-type, even in the context of the DGFF.

As in other approximately logarithmically correlated fields (such as branching Brownian motion or branching random walks), such ballot estimates prove crucial for the study of extreme value phenomena. Indeed, they have been used ad hoc in~\cite{BL-Full, Bi-LN} to study the microscopic structure of extreme values of the DGFF, and they are relied on in the analysis in many forthcoming works in the area~(\cite{Crit, Pinned, maximilian}). The purpose of this manuscript is therefore to derive estimates for such ballot probabilities in a
general and unified setting that could be used as a black box in a wide range of applications.

To describe the setup precisely, let us recall that the two-dimensional DGFF on a discrete finite  domain $\emptyset \subsetneq D \subset \bbZ^2$ with Dirichlet boundary conditions $w \in \bbR^{\partial D}$ is the Gaussian field $h^{D,w} = (h^{D,w}(x) :\: x \in \ol{D})$ with mean and covariances given 
for $x,y \in \ol{D}$ by 
\begin{equation}
\label{e:1.1a}
\bbE h^{D,w}(x) = \ol{w}(x)
\quad , \qquad
\Cov \big( h^{D,w}(x), h^{D,w}(y) \big) = G_D(x,y) \,.
\end{equation}
Above, $\ol{w}$ is a harmonic extension of $w$ to $\bbZ^2$ and $G_D$ is the discrete Green function associated with a planar simple random walk killed upon exit from $D$. Here and after, $\ol{D} := D \cup \partial D$ and $\partial D$ is the outer boundary of $D$, namely the set $\{x \in D^\rmc :\: |x-y| = 1 \text{ for some } y \in D\}$, with $|\cdot|$ always representing the Euclidean norm. Also, whenever $f$ is a real function from a finite non-empty subset of $\bbZ^2$, we denote by $\overline{f}$ its unique bounded harmonic extension to the whole discrete plane. We also set $h^D \equiv h^{D,0}$ (i.e. zero boundary conditions) and formally set $h^\emptyset \equiv 0$.

It follows from the Gibbs-Markov description of the law of $h^{D,w}$, that for any other non-empty domain $D'$ with $\ol{D'} \subset D$ and any $w' \in \bbR^{\partial D'}$, conditioned on $\{h^{D,w}(x) = w'(x) :\: x \in \partial D'\}$, the law of $h^{D,w}$ on $\ol{D'}$ is that of the DGFF $h^{D',w'}$. The definition of the DGFF extends naturally to an infinite domain which is not the whole discrete plane, but then the harmonic extension (and therefore the field itself) may not be unique. When $D = \bbZ^2$, the field does not exist, but for notational reasons we shall nevertheless denote
by $h$ a field that exists only under the formal conditional probability:
\begin{equation}
\label{e:1}
\bbP \big(h \in \cdot \,\big|\, h_{\partial D'} = w' \big)  \,,
\end{equation}
for $\emptyset \subsetneq D' \subsetneq \bbZ^2$, $w \in \bbR^{\partial D'}$, and in which case its law on $\ol{D'}$ is that of $h^{D',w'}$ as above.
Henceforth $f_A$ will stand for the restriction of a function $f$ to a subset of its domain $A$.

To get an asymptotic setup, we will consider discrete approximations of scaled versions of a given domain $W \subset \bbR^2$, which we define for $\ell \geq 0$ as
\begin{equation}
\label{e:discr}
W_\ell=\big\{x \in \bbZ^2 :\: \rmd(\rme^{-\ell}x, W^\rmc) > \rme^{-\ell}/2\big\} \,.
\end{equation}
Above $\rmd$ is the usual distance from a point to a set under the Euclidean Norm.
We shall sometimes loosely refer to such $W_\ell$ as a (discrete) set or domain of (exponential) ``scale'' $\ell$. We note that the scale $\ell$ need not be an integer.

We also define $W^+$, $W^-$ and $W^\pm$ as the interior of $W$, the interior of $W^\rmc$ and their union $W^+ \cup W^-$. For $\eta \geq 0$, let
\begin{equation}
W^{\eta} := \{x \in W :\: \rmd(x, W^\rmc) > \eta\} \,,
\end{equation}
Precedence is fixed so that $\partial U^{-,\eta}_{\ell} := \partial(((U^{-})^{\eta})_\ell)$.
We shall also loosely call $W^\eta$ or $W^\eta_\ell$ the ``bulk'' of $W$, resp.\ of $W_\ell$.
For $\eps>0$, we let $\frD_\epsilon$ denote the collection of all open subsets $W \subset \bbR^2$ whose topological boundary $\partial W$ consists of a finite number of connected components each of which has an Euclidean diameter at least $\epsilon$, and that satisfy
\begin{equation}
\label{e:1.5}
\rmB(0,\epsilon) \subset W \subset \rmB(0,\epsilon^{-1}) \,.
\end{equation}
Above and in the sequel $\rmB(x,r)$ is the open ball of radius $r$ about $x$ and $\rmB \equiv \rmB(0,1)$. We also set $\frD := \cup_{\epsilon > 0} \frD_\epsilon$.

We will study the DGFF on domains of the form $U_n \cap V^-_k$, where $U, V \in \frD$ with $U$, $V^-$ being connected and $n,k \geq 0$. Thanks to~\eqref{e:1.5} this can be thought of as a generalized version of the annulus $\rmB_n \cap \rmB^-_k$. It is well known~\cite{DZ14} that under zero boundary conditions, the maximum of the DGFF on domain of scale $\ell$, reaches height $m_\ell + O(1)$ with high probability, where
\begin{equation}
\label{e:ml}
m_\ell := 2\sqrt{g} \ell - \tfrac{3}{4}\sqrt{g} \log^+ \ell
\ , \quad g:=2/\pi \,.
\end{equation}
It is therefore natural to consider the boundary conditions $-m_n + u$ on $\partial U_n$ and $-m_k + v$ on $\partial V_k^-$, where $u \in \bbR^{\partial U_n}$ and $v \in \bbR^{\partial V_k^-}$. Indeed, when scales are exponential, the harmonic extension inside $U_n \cap V_k^-$ is roughly the linear interpolation between the values on the boundary. Therefore such boundary conditions will push down the mean of the field on the annulus at scale $\ell$ about the origin by precisely the height of the field maximum there on first order.

Ballot estimates and assumptions will often be expressed in terms of the (unique bounded) harmonic extensions $\ol{u}$, $\ol{v}$ of $u$, $v$ respectively. For bounds and conditions it will often be sufficient to reduce the boundary data to the value of $\ol{u}$ and $\ol{v}$ at test points, together with their oscillation in the bulk. We will typically take $0$ and $\infty$ as the test points for $\ol{u}$ and $\ol{v}$ respectively, and to this end define $f(\infty) := \lim_{|x| \to \infty} f(x)$.
We note that $f(\infty)$ always exists if $f:\bbZ^2\to\bbR$ is the bounded harmonic extension of a function on a finite domain (this is a consequence of Proposition~6.6.1 of~\cite{LaLi}).
The oscillations of $\ol{u}$ and $\ol{v}$ in the bulk are given by $\osc\,\ol{u}_\eta$ and $\osc\, \ol{v}_\zeta$, where $\ol{u}_\eta$ and $\ol{v}_\zeta$ stand for the restrictions of $\ol{u}$ and $\ol{v}$ to $U_{n}^{\eta}$ and $V_{k}^{-,\zeta}$ respectively, and $\osc f := \sup_{x,y\in D} |f(x)-f(y)|$, for $f: D \mapsto \bbR$. We shall also write $\osc_{D'} f$ for $\osc\,f_{D'}$. Functions which are used outside their domain of definition implicitly take their value there to be zero and indicators $1_A$ are assumed to be defined only on $A$.

Lastly, dependence of constants and asymptotic bounds on parameters are specified as subscripts. When present, such subscripts represent all the parameters on which the quantity depends. For example we shall write $c_\alpha$ for a constant that depends only on $\alpha$ and $a_n = o_\alpha(b_n)$, if there exists a sequence $c_n$ which depends only on on $\alpha$ such that $|a_n/b_n| \leq c_n \to 0$. As usual, constants change from one use to another and are always assumed to be positive and finite.

\subsection{Main results}
\label{s:results}

We start with uniform asymptotics for the ballot probability.
\begin{thm}
\label{t:1}
Fix $\eps\in(0,1)$, $\eta,\zeta \in[0,\eps^{-1}]$. Suppose that $U, V \in \frD_\epsilon$ such that $U$, $V^-$ are connected and $\rmB(0, \epsilon) \subset U^{\eta}$. For all $k,n \geq 0$, there exist
$\cL_n = \cL_{n, \eta, U} :\: \bbR^{\partial U_n} \to (0,\infty)$, 
$\cR_k = \cR_{k, \zeta, V} :\: \bbR^{\partial V^-_k} \to (0, \infty)$ such that as $n-k \to \infty$,
\begin{equation}
\label{e:2.4ub}
\bbP \Big( h_{U^{\eta}_{n} \cap V^{-,\zeta}_{k}} \leq 0 \, \Big|\,
h_{\partial U_n} = -m_n + u ,\,
h_{\partial V^-_k} = -m_k + v
\Big) 
= (2+o_{\epsilon}(1)) \frac{\cL_{n}(u) \cR_{k}(v)}{g(n-k)} \,,
\end{equation}
for all $u \in \bbR^{\partial U_n}$ and $v \in \bbR^{\partial V^-_k}$ satisfying
\begin{equation}
\label{e:uv-as}
\max \Big\{\ol{u}(0),\, \ol{v} (\infty),\, \osc\, \ol{u}_\eta,\, \osc\,\ol{v}_\zeta \Big\} \leq \epsilon^{-1} \ , \quad 
(1+\ol{u}(0)^-)(1+ \ol{v} (\infty)^-) \leq (n-k)^{1-\eps}  \,.
\end{equation}
\end{thm}

The reader should compare the statement in the theorem to that in the usual ballot theorem for a one dimensional random walk, say with standard Gaussian steps. Indeed, letting $(S_\ell)_{\ell \geq 0}$ denote such a walk, it can be shown, e.g., using the reflection principle for Brownian Motion, that as $n-k \to \infty$,
\begin{equation}
\bfP \big(S_{(k,n)} \leq 0 :\: S_k = -u,\, S_n = -v\big) = (2+o(1)) \frac{F(u)F(v)}{n-k} \,,
\end{equation}
uniformly in $u,v \leq \epsilon^{-1}$ satisfying $(1+u^-)(1+v^-) \leq (n-k)^{1-\epsilon}$, where $\bfP$ is the underlying probability measure and $\epsilon > 0$. Moreover $F(w)$ can be explicitly constructed via
\begin{equation}
\label{e:1.10}
F(w) := \lim_{r \to \infty} \bfE \big(S_r^- ;\; S_{(0,r]} \leq 0 \,\big|\, S_0 = w\big) 
\ ; \quad w \in \bbR \,,
\end{equation}
and obeys 
\begin{equation}
\label{e:1.11}
F(w) > 0 \text{ for all } w \leq \epsilon^{-1}
\quad \text{and} \quad F(w) = (1+o(1)) w^- \text{ as } w \to -\infty \,.
\end{equation}

In analog to~\eqref{e:1.10} we have the following explicit construction for the functionals $\cL_n$ and $\cR_k$ from Theorem~\ref{t:1}.
\begin{prop}
\label{p:1-2}
Fix $\eps\in(0,1)$, $\eta,\zeta \in[0,\eps^{-1}]$. There exists $(r_n)_{n \geq 0}$ satisfying $r_n \to \infty$ when $n \to \infty$ such that the following holds: For $U$, $V$ as in Theorem~\ref{t:1} and all $k,n \geq 0$, the functions $\cL_n$ and $\cR_k$ can be defined via 
\begin{align}
%\label{e:cL}
 \cL_{n}(u)
 &:= \bbE \Big( \big(\he{h}{\rmB_{n-r_n}}(0) + m_{n-r_n} \big)^- ;\; h_{U_{n}^{\eta} \cap \rmB^-_{n-r_n}} \leq 0 
\,\Big|\, h_{\partial U_n} = -m_n + u ,\,h_{\partial\rmB^-_0} = \ol{u}(0) \Big) \\
\cR_{k}(v)&:= \\
 \lim_{n \to \infty} &\bbE \Big( \big(\he{h}{\rmB^-_{k+r_{n-k}}}(\infty) 
+ m_{k+r_{n-k}}\big)^-;\:
h_{\rmB_{k+r_{n-k}} \cap V_{k}^{-,\zeta} } \leq 0 \,\Big|\, h_{\partial \rmB_n} = -m_n  ,\,
h_{\partial V^-_k} = -m_k + v
\Big)\notag
\end{align}
for $u \in \bbR^{\partial U_n}$, $v \in \bbR^{\partial V^-_k}$.
\end{prop}
From now on, we shall assume that $\cL_n$ and $\cR_k$ are defined as in Proposition~\ref{p:1-2}.
 
While explicit constructions of the functionals $\cL_n$ and $\cR_k$ are illustrative, they are of less use in computations involving the right hand side of~\eqref{e:2.4ub}. The purpose of the next several propositions is therefore to provide bounds and asymptotics for them.
For what follows,
we let $\bbH(W)$ be the space of bounded harmonic functions on an open non-empty domain $W \subset \bbR^2$, which we equip with the supremum norm.

The first proposition shows that, viewed as functionals of the harmonic extension of their argument, $\cL_n$ and $\cR_k$ admit an infinite version after proper scaling.

\begin{prop}
\label{p:LR-infty}
Fix $\eps\in(0,1)$, $\zeta,\eta \in[0,\eps^{-1}]$ and let $U,V$ be as in Theorem~\ref{t:1}. There exist $\cL = \cL_{\eta, U}: \bbH (U^{\eta}) \to (0,\infty)$ and
$\cR = \cR_{\zeta, V}: \bbH(V^{-,\zeta}) \to (0,\infty)$ such that,
\begin{equation}
\label{e:1.12}
\begin{split}
\cL_{n}(u_n)  \overset{n \to \infty} \longrightarrow \cL(\wh{u}_\infty)
& \quad \text{whenever} \quad \max_{x\in U^\eta_n}\big|\ol{u_n}(x)-\wh u_\infty(\rme^{-n}x)\big|\overset{n \to \infty} \longrightarrow 0 \,, \\
 \cR_{k}(v_k)  \overset{k \to \infty} \longrightarrow \cR(\wh{v}_\infty)
& \quad \text{whenever} \max_{x\in V^{-,\zeta}_k}\big|\ol{v_k}(x)-\wh v_\infty(\rme^{-k}x)\big|\overset{k \to \infty} \longrightarrow 0\,.
\end{split}
\end{equation}
Above $u_n \in \bbR^{\partial U_n}$, $v_k \in \bbR^{\partial V_k^-}$ for all $k, n \geq 0$ and $\wh{u}_\infty \in \bbH(U^{\eta})$, $\wh{v}_\infty \in \bbH(V^{-,\zeta})$.
\end{prop}

Next, in analog to~\eqref{e:1.11} we have the following two propositions:
\begin{prop}
\label{p:1-4}
Fix $\eps\in(0,1)$, $\eta,\zeta \in[0,\eps^{-1}]$ and let $U$, $V$ be as in Theorem~\ref{t:1}. There exists $c=c_\eps<\infty$ such that for all $k \geq 0$, $n\geq c$,
\begin{equation}
\begin{split}
\cL_n(u) & =\big(1+o_{\epsilon}(1)\big) \ol{u}(0)^- \quad \text{as } \ol{u}(0)^- \to \infty\text{ with } \osc\,\ol{u}_\eta \leq \eps^{-1}
\,, \\
\cR_k(v) & = \big(1+o_{\epsilon}(1)\big) \ol{v}(\infty)^- \quad \text{as } \ol{v}(\infty)^- \to \infty \text{ with } \osc\,\ol{v}_\zeta \leq \eps^{-1}\,,
\end{split}
\end{equation}
and also
\begin{equation}
\begin{split}
\cL(\wh{u}) & =\big(1+o_{\epsilon}(1)\big) \wh{u}(0)^- \quad \text{as } \wh{u}(0)^- \to \infty \text{ with } \osc\,\wh{u}_\eta \leq \eps^{-1}
\,, \\
\cR(\wh{v}) & = \big(1+o_{\epsilon}(1)\big) \wh{v}(\infty)^- \quad \text{as } \wh{v}(\infty)^- \to \infty \text{ with }\osc\, \wh{v}_\zeta\leq \eps^{-1} \,.
\end{split}
\end{equation}
\end{prop}
\begin{prop}
\label{p:1-5}
Fix $\eps\in(0,1)$, $\eta,\zeta \in[0,\eps^{-1}]$ and let $U$, $V$ be as in Theorem~\ref{t:1}. There exists $c = c_{\epsilon} > 0$ such that for all $k\geq 0$, $n\geq c$ and all $u \in \bbR^{\partial U_n}$ and $v \in \bbR^{\partial V_k^-}$ with $\max\{\ol{u}_\eta(0), \ol{v}_\zeta(\infty),  \osc\,\ol{u}_\eta, \osc\,\ol{v}_\zeta\} \leq \epsilon^{-1}$,
\begin{equation}
\label{e:1-5-nk}
 \cL_n(u) > c \  , \qquad \qquad \qquad
 \cR_k(v) > c \,.
\end{equation}
Also, for all $\wh{u} \in \bbH(U^{\eta})$ and $\wh{v} \in \bbH(V^{-,\zeta})$ with $\wh{u} \leq \epsilon^{-1}$ and $\wh{v} \leq \epsilon^{-1}$, we have
\begin{equation}
\label{e:1-5-infty}
 \cL(\wh{u}) > c \  , \qquad \qquad \qquad
 \cR(\wh{v}) > c \,.
\end{equation}
\end{prop}

Next, we address the continuity of $\cL_n$ and $\cR_k$ in the the domain, the values assigned to the field on the boundary, and $\eta$, $\zeta$. To measure distances between domains, if $W, W'$ are open non-empty and bounded subsets of $\bbR^2$, we let $\dH(W, W')$ be the Hausdorff distance of their complements inside any closed and bounded subset of $\bbR^2$ which includes both $W$ and $W'$. The reader can easily verify that the above definition is proper and that $\dH$ is indeed a metric on the space of open, non-empty and bounded subsets of $\bbR^2$ and in particular on $\frD$. 

\begin{prop}
\label{p:1-6}
Fix $\eps\in(0,1)$. There exist $c=c_\eps<\infty$, $\rho  = \rho_\epsilon:\bbR_+ \to \bbR_+$ with $\rho(t) \downarrow 0$ as $t \downarrow \infty$ such that the following holds:  For all $n> c$, all $U, U' \in \frD_\epsilon$ such that $U, U'$ are connected, all $\eta, \eta' \geq 0$ such that $\rmB(0,\epsilon) \subset U^{\eta} \cap U'^{\eta'}$, and all $u \in \bbR^{\partial U_n}$, $u' \in \bbR^{\partial U'_n}$ such that $\max\{\ol{u}_\eta(0), \ol{u}'_{\eta'}(0), \osc\,\ol{u}_\eta, \osc\,\ol{u}'_{\eta'}\} < \epsilon^{-1}$, we have 
\begin{equation}
\label{e:p:1-6-L}
\bigg| \frac{\cL_{n,\eta,U}(u)}{\cL_{n,\eta',U'}(u')} - 1 \bigg|
\leq \rho \Big(\dH\big(U, U'\big) + |\eta - \eta'|
+\|\ol{u}-\ol{u}'\|_{\bbL_\infty(U_{n}^{\eta} \cap U'^{\eta'}_{n})} + n^{-1}\Big) \,.
\end{equation}
Similarly, for all $k\geq 0$, all $V, V' \in \frD_\epsilon$ such that $V^-, V'^-$ are connected, all $\zeta, \zeta' \geq 0$ and all $v \in \bbR^{\partial V_k^-}$, $v' \in \bbR^{\partial V'^-_k}$ such that
$\max\{\ol{v}_\zeta(\infty), \ol{v}'_{\zeta'}(\infty), \osc\,\ol{v}_\zeta, \osc\,\ol{v}'_{\zeta'}\} < \epsilon^{-1}$, we have 
\begin{equation}
\label{e:p:1-6-R}
\bigg| \frac{\cR_{k,\zeta,V}(v)}{\cR_{k,\zeta',V'}(v')} - 1 \bigg|
\leq \rho \Big(\dH\big(V^-, V'^-\big)
+ |\zeta - \zeta'| +\|\ol{v}-\ol{v}'\|_{\bbL_\infty(V_{k}^{-,\zeta} \cap V'^{-,\zeta'}_{k})} + k^{-1} \Big) \,.
\end{equation}
The above holds also for $\cL$ and $\cR$ in place of $\cL_n$ and $\cR_k$, with $u, u', v, v'$, $\ol{u}, \ol{u}', \ol{v}, \ol{v}'$ and $\ol{u}_\eta, \ol{u}'_\eta, \ol{v}_\zeta, \ol{v}'_\zeta$ all replaced by $\wh{u}, \wh{u}', \wh{v}, \wh{v}'$ which are now harmonic functions on $U^{\eta}, U'^{\eta'}, V^{-,\zeta}, V'^{-,\zeta'}$ respectively, and with the $\bbL_\infty$ norm taken over $U^{\eta} \cap U'^{\eta'}$ and $V^{-,\zeta} \cap V'^{-, \zeta'}$ instead of $U_{n}^{\eta} \cap U'^{\eta'}_{n}$ and $V_{k}^{-,\zeta} \cap V'^{-,\zeta'}_{k}$.
\end{prop}

Lastly, we turn to bounds on the ballot probability. The following lower bound is a simple consequence of Theorem~\ref{t:1}, Propositions~\ref{p:1-4},~\ref{p:1-5} and monotonicity.
\begin{cor}
Under the conditions of Theorem~\ref{t:1}, there exist $c = c_{\epsilon} > 0$ and $C=C_\eps<\infty$, such that
\begin{equation}
\label{e:2.4lb}
\bbP \Big( h_{U_{n}^{\eta} \cap V^{-,\zeta}_{k}} \leq 0 \, \Big|\,
h_{\partial U_{n}} = -m_n + u ,\,
h_{\partial V^-_{k}} = -m_k + v
\Big) 
> c \frac{\big(\ol{u}(0)^{-} + 1\big) \big(\ol{v}(\infty)^- + 1\big)}{n-k}
\end{equation}
for all $n-k\geq C$.
\end{cor}

While a similar upper bound can be derived in the same way, a stronger result without any restrictions on $u$ and $v$ is given by the following theorem. For what follows, we abbreviate:
\begin{equation}
\label{e:1.22}
u_* \equiv \ol{u}(0)-2\osc\,\ol{u}_\eta
\quad, \qquad v_* \equiv \ol{v}(\infty)-2\osc\,\ol{v}_\zeta \,.
\end{equation}
\begin{thm}
\label{t:2.3}
Fix $\eps\in(0,1)$, $\eta,\zeta \in[0,\eps^{-1}]$ and let $U,V$ be as in Theorem~\ref{t:1}. Then there exists $C = C_{\epsilon} < \infty$ such that the following holds:
for all $0 \leq k < n$ and all $u \in \bbR^{\partial U_n}$ and $v \in \bbR^{\partial V^-_k}$,
\begin{equation}
\label{e:1.18}
\bbP \Big( h_{U_{n}^{\eta} \cap V^{-,\zeta}_{k}} \leq 0 \, \Big|\,
h_{\partial U_{n}} = -m_n + u ,\,
h_{\partial V^-_{k}} = -m_k + v
\Big) \leq 
C \frac{\big(u_*^- + 1\big)\big(v_*^- + 1\big)}{n-k} \,,
\end{equation}
If, moreover,
$u_*^- \leq (n-k)^{1-\eps}$, $v_*^{-} \leq (n-k)^{1-\eps}$, $n-k\geq C$,
then we have the stronger bound
\begin{equation}
\label{e:t:2.3-posUB}
\bbP \Big( h_{U_{n}^{\eta} \cap V^{-,\zeta}_{k}} \leq 0 \, \Big|\,
h_{\partial U_{n}} = -m_n + u ,\,
h_{\partial V^-_{k}} = -m_k + v
\Big)
\leq 
C \frac{\Big(u_*^- + 
\rme^{-( u_{*}^+ )^{2-\eps}}
\Big)
\Big(  v_*^- + 
\rme^{-( v_{*}^+ )^{2-\eps}}
\Big)}{n-k} \,.
\end{equation}
\end{thm}

{\bf Remarks.} 

\begin{enumerate}
\item \label{item:r:strongUB}{\em Improved upper bound for positive boundary conditions.} In Theorem~\ref{t:2.3}, the upper bound~\eqref{e:t:2.3-posUB} is stronger than~\eqref{e:1.18} when positive and large values are assumed on either the outer or the inner (or both) ``part'' of the boundary of the underlying domain. The almost-Gaussian decay is the cost of a shift down of the field close to the corresponding part of the boundary, to make room for its maximum there to still be negative, despite the positive boundary conditions. We take this cost only into account when the boundary conditions on the opposite scale are not too low, as otherwise the field may be sufficiently tilted down for this cost to be lower.

\item {\em Discretization.} 
One can of course choose different methods for discretizations than the one we used in~\eqref{e:discr}. In particular, one can replace $\rme^{-\ell}/2$ with any multiple, greater or equal than $1/2$, of $\rme^{-\ell}$ (as was done in~\cite{BL-ele, BL-Conf, BL-Full}) with all the statements still holding.
The choice of $1/2$ ensures that all connected components of the boundary of $W \in \cD$ appear (scaled) also in $W_\ell$ for large $\ell$ and at the same time allows, by choosing $\ell=0$, to have discrete domains in which only a few isolated points are excluded from the bulk, as for instance $(\rmB(0,\frac12))_0=\{0\}$.
More generally, as in~\cite{BL-Full} one can take as $W_\ell$ any set satisfying
\begin{equation}
\big\{x \in \bbZ^2 :\: \rmd(\rme^{-\ell}x, W^\rmc) > \lambda_\ell \} 
\subset 
W_\ell \subset \big\{x \in \bbZ^2 :\: \rmd(\rme^{-\ell}x, W^\rmc) > \rme^{-\ell}/2\big\} \,,
\end{equation}
for any fixed sequence $(\lambda_\ell)$ satisfying $\lim_{\ell \to \infty} \lambda_\ell = 0$ and $\lambda_\ell \geq \rme^{-\ell}/2$. All the statements in the present subsection still hold and uniformly with respect to the particular choice of the discretized versions of $U$ and $V^-$, but one must then restrict oneself to the case $\eta > 0$.
The continuity of the functionals $\cL_n$ and $\cR_k$ in the underlying domain, which is uniform in $n$ and $k$ respectively, can be also used to compare different discretization methods.

\item {\em Continuity in the underlying domain.} 
Continuity in the underlying domain is handled in Proposition~\ref{p:1-6}.
The $n^{-1}$ and $k^{-1}$ terms in~\eqref{e:p:1-6-L} and~\eqref{e:p:1-6-R} are the result of  the discretization. Indeed, two continuous domains can be arbitrarily close to each other, with their respective scaled-up discretizations differing by at least one point. These terms can an be left out if $U=U'$ or $V=V'$, respectively. 

The distance $\dH$, as defined above the proposition, can be reformulated as 
\begin{equation}
\dH(W,W') = \inf \big\{\eta > 0 :\: W^\eta \subset W' \text{ and } W'^\eta \subset W \big\} \,.
\end{equation}
The condition in the infimum guarantees that the Green Functions: $G_W$ and $G_{W'}$ are close to one another in the bulk of $W^\eta \cap W'^\eta \subset W \cap W' $, which in turn ensures that the corresponding DGFFs are stochastically not too far apart on this set. This shows that the topology induced by $\dH$ is a rather natural choice for the question of continuity in $U$ and $V^-$ of the functionals $\cL_n$ and $\cR_k$, resp. 

We remark, that using the Hausdorff distance for the set itself and not its complement would not have been a good choice. Even if we ignore the fact that the Hausdorff function is only a metric when one takes compact sets, using it to measure distances would have rendered the distance between a disk and a disk with a slit to be zero. This is clearly not desirable as the DGFFs on the discretized scale-ups of these two domains are very different from each other and therefore so are the corresponding $\cL_n$ or $\cR_k$ functionals.
\end{enumerate}

\subsection{Proof outline}
\label{ss:ProofOutline}
Let us explain the strategy behind the proof of the ballot asymptotics (Theorem~\ref{t:1}) and
ballot upper bound (first part of Theorem~\ref{t:2.3}), which constitute the main contribution of this manuscript. To study the ballot probability (and associated expectations such as those appearing in the definition of $\cL_n$ and $\cR_k$) we recast the event
\begin{equation}
\label{e:1.26}
\Big\{ h_{U_n^\eta \cap V_k^{-,\zeta}} \leq 0 \Big\}
\end{equation}
as 
\begin{equation}
\label{e:1.27}
\bigcap_{\ell=1}^T \big\{ \cS_\ell + \cD_\ell \leq 0 \big\}
\end{equation}
for some processes $(\cS_\ell)_{\ell=0}^T, (\cD_\ell)_{\ell=1}^T$ which are defined in terms of the field $h$ and some additional independent randomness. We show that under $\bbP(\cdot\mid h_{\partial U_n} = u, h_{\partial V_k^-} = v)$, for $u$,$v$ satisfying the conditions in Theorem~\ref{t:1}, these processes satisfy the following three assumptions:
\begin{enumerate}[label={A\arabic{*}}.,ref=A\arabic{*}]
\item\label{i:o1}  $(\cS_\ell)_{\ell=0}^T$ is a non-homogeneous random walk with Gaussian steps of $O(1)$ variance, conditioned to start from $\ol{u}(0)$ at time $0$ and end at $\ol{v}(\infty)$ at time $T$. Moreover $T = n-k +O(1)$.
\item\label{i:o2} For all $\ell$, given $\cS_\ell$, the random variables $(\cS_{\ell'}, \cD_{\ell'})_{\ell'\leq \ell}$ are conditionally independent of $(\cS_{\ell'})_{\ell' > \ell}$.
\item\label{i:o3} $(\cD_\ell)_{\ell=1}^T$ are stretched exponentially tight around their mean and satisfy $|\bbE \cD_\ell| = O(\min (\ell, T-\ell)^{1/2-\epsilon})$ for all $\ell$ and some $\epsilon > 0$.
\end{enumerate}
We shall refer to the $\cD_\ell$-s as {\em decorations} and to the pair of processes $((\cS_\ell)_{\ell=0}^T, (\cD_\ell)_{\ell=1}^T)$, as a {\em decorated (non-homogeneous) random walk} (DRW).

To carry out the reduction to the DRW, we use a new version of the {\em concentric decomposition}, which was originally introduced in~\cite{BL-Full}, that we call {\em inward concentric decomposition}. Assuming for illustrative purposes
that $U=V=\rmB$ and that $\eta=\zeta=0$, the idea is to condition on the values of the field at the boundaries of (discrete) balls of radii $\rme^{n-\ell}$ for $\ell = 1, 2, \dots, n-k-1$. Thanks to the Gibbs-Markov property of the DGFF, the field $h^{\rmB_n \cap \rmB_k^-}$ (with zero boundary conditions) can then be written as a sum of a {\em binding field} $\varphi$, namely the conditional expectation of $h^{\rmB_n \cap \rmB_k^-}$ given its values on $\cup_{\ell=1}^{n-k-1} \partial \rmB^\pm_{n-\ell}$, plus independent DGFFs $(h^{A_\ell})_{\ell=1}^{n-k}$ where $A_\ell$ is the annulus $\rmB_{n-\ell+1} \cap \rmB^-_{n-\ell}$.

Performing the conditioning in successive order from $\ell=1$ to $n-k-1$, we construct a process $(S'_\ell)_{\ell=1}^{n-k}$ as a function of $\varphi$ plus some additional independent randomness. This process is defined so that it has the law of a random walk with Gaussian steps, conditioned to start and end at $0$. More importantly, $S'_\ell$ approximates $\varphi$ in the bulk of $A_\ell$, up to an additive error which has a uniform stretched exponential tail. If instead of zero, one takes boundary conditions $-m_n+u$ on $\partial \rmB_n$ and $-m_k+v$ on $\partial \rmB^-_k$, then the added mean
$\ol{(-m_n1_{\partial \rmB_n} + u -m_k1_{\partial \rmB_k^-}+v)_{\partial \rmB_n\cup \partial \rmB^-_k}}$ is equal on $A_\ell$ to
\begin{multline}
\frac{(n-\ell)(-m_n + \ol{u}(0)) + (\ell-k)(-m_k + \ol{v}(\infty))}{n-k} + O(1) =  \\
-m_\ell + \frac{(n-\ell) \ol{u}(0) + (\ell-k) \ol{v}(\infty)}{n-k}  + O(\log^+((\ell-k)\wedge(n-\ell))) \,,
\end{multline}
whenever the conditions of Theorem~\ref{t:1} are assumed.
This follows from the fact that for a well-behaved harmonic function $f$ on $\rmB_n\cap \rmB^-_k$, $\log|x|\to f(x)$ is approximately linear.

Combining the above and using also the stretched exponential tightness of
the maximum of the DGFF around its mean,
it follows that under $\bbP(\cdot|h_{\partial \rmB_n} = u, h_{\partial \rmB_k^-} = v)$ the event in~\eqref{e:1.26} can be written as
\begin{multline}
\bigcap_{\ell=1}^{n-k} \Big\{ \max_{A_\ell} \big(\varphi + 
\ol{(-m_n1_{\partial \rmB_n} + u -m_k1_{\partial \rmB_k^-}+v)_{\partial \rmB_n\cup \partial \rmB^-_k}}
 + h_\ell\big) \leq 0 \Big\} = \\
\bigcap_{\ell=1}^{n-k} \Big\{ S'_\ell + \frac{(n-\ell) u(0) + (\ell-k) \ol{v}(\infty)}{n-k} + \cD_\ell \leq 0 \Big\}\,,
\end{multline}
where $(\cD_\ell)_{\ell=1}^T$ are stretched exponentially tight after centering by $O(\log^+((\ell-k)\wedge(n-\ell)))$. Setting $\cS_\ell :=  S'_\ell + \frac{(n-\ell) u(0) + (\ell-k) \ol{v}(\infty)}{n-k}$, we then obtain that the pair $((\cS_\ell)_{\ell=0}^T, (\cD_\ell)_{\ell=1}^T)$ satisfies Assumptions~\eqref{i:o1} --~\eqref{i:o3} above
and that the equality between events~\eqref{e:1.26} and~\eqref{e:1.27} holds.

This converts the original task to that of deriving asymptotics and an upper bound for the probability of the DRW ballot event~\eqref{e:1.27}. To this end, we can use the tightness of the decorations to reduce the problem to that in which the decoration process $(\cD_\ell)_{\ell=1}^T$ is replaced by a deterministic curve or {\em barrier}: $(\gamma_\ell)_{\ell=1}^T$, with $\gamma_\ell = O(\min (\ell, T-\ell)^{1/2-\epsilon})$ for some $\epsilon > 0$.
Upper and lower bounds for such barrier probabilities are then (essentially) readily available (for example in the work of Bramson~\cite{B_C} on barrier estimates for BM, from which such results can be derived, or~\cite{cortines2019decorated}). These in turn are sufficient to derived the desired estimates.

A difficulty is that because of the dependence between the walk $(\cS_\ell)_{\ell=0}^T$ and the decorations $(\cD_\ell)_{\ell=1}^T$, as captured by Assumption~\ref{i:o2}, the reduction to such a deterministic barrier problem can only be done when the endpoint of the walk $\cS_T$ is sufficiently low, that is if $\ol{v}(\infty) < -T^{\epsilon}$ for some $\epsilon > 0$. In other words, using the reduction we can only prove {\em weak versions} of Theorem~\ref{t:1} and (the first part of) Theorem~\ref{t:2.3}, where one imposes the additional constraint that $\ol{v}(\infty) < -(n-k)^{\epsilon}$.

To remove this constraint we proceed as follows. First, we  define a new DRW $(\cS^{\rm o}_\ell)_{\ell=0}^T, (\cD^{\rm o}_\ell)_{\ell=1}^T$ which (under the conditions of Theorem~\ref{t:1}) satisfies Assumptions~\ref{i:o1} --~\ref{i:o3} as before, only that the starting and end points are now reversed: $\cS^{\rm o}_0 = \ol{v}(\infty)$ and $\cS^{\rm o}_T = \ol{u}(0)$, and with the equivalence of events~\eqref{e:1.26} and~\eqref{e:1.27} (with $\cS^{\rm o}_\ell$, $\cD^{\rm o}_\ell$ in place of $\cS_\ell$ and $\cD_\ell$) still holding. This is done by employing an {\em outward concentric decomposition}, whereby the order of conditioning is reversed, such that $\cS_\ell + \cD_\ell$ now corresponds to the value of the field on the annulus
$\rmB_{k+\ell}\cap \rmB_{k+\ell-1}^-$. Proceeding as before, this gives weak version of Theorems~\ref{t:1} and Theorem~\ref{t:2.3} (first part), where now the additional assumption is that $\ol{u}(0) < -(n-k)^{\epsilon}$.

Now, given $U_n$ and $V^-_k$, we pick an intermediate scale, say, $\ell := (n+k)/2$ and condition on the values of $h$ on the $\partial \rmB_\ell^\pm$. Thanks to the Gibbs-Markov property, whenever $n-k$ is large, we can write $\bbP(h_{U_n^\eta \cap V_k^{-,\zeta}} \leq 0 | h_{\partial U_n} = -m_n+u, h_{\partial V_k^-} = -m_k+v)$ as 
\begin{multline}
\int_{w \leq m_\ell} 
\bbP \big(h_{U_n^\eta \cap \rmB_\ell^-} \leq 0 \,\big|\, h_{\partial U_n} = -m_n+u, h_{\partial \rmB_\ell^-} = -m_\ell+w) \\
\bbP \big(h_{\rmB_\ell \cap V_k^{-,\zeta}} \leq 0 \,\big|\, h_{\partial \rmB_\ell} = -m_\ell+w, h_{\partial V_k^{-,\zeta}} = -m_k+v) \\
\bbP \big( h_{\partial \rmB^\pm_\ell} + m_\ell \in \rmd w \,\big|\, h_{\partial U_n} = -m_n+u, h_{\partial V_k^-} = -m_k+v \big)\,.
\end{multline}

Observing that the first two terms in the integrand are precisely ballot probabilities of the form treated in this paper, we can use the weak ballot upper bound together with discrete harmonic analysis to show that for any $\delta > 0$, the integral can be restricted to $w \in \bbR^{\partial\rmB_\ell^\pm}$ satisfying $\ol{w}(0)^- \in [(n-\ell)^{-\epsilon}, (n-\ell)^{1-\epsilon}]$, $\ol{w}(\infty)^- \in [(\ell-k)^{\epsilon}, (\ell-k)^{1-\epsilon}]$ and $\osc_{\rmB_\ell^{\pm, \delta} \cup \{\infty\}} \ol{w} = O(1)$ at the cost of a negligible error. The fact that $\ol{h_{\partial \rmB^\pm_\ell}} \leq -(n-\ell)^{\epsilon} = -(\ell-k)^{\epsilon} = -((n+k)/2)^{-\epsilon}$ in the bulk with high probability conditional on~\eqref{e:1.26} is the DGFF analog of the entropic repulsion of a random walk bridge conditioned to stay negative.

We can now apply the weak version of Theorem~\ref{t:1} to estimate the first two terms in the integrand (first with $U^\eta_n\cap\rmB^{-,\delta}_l$ and $\rmB^\delta_l\cap V^{-,\zeta}_k$ in place of $U^\eta_n\cap \rmB^-_l$ and $\rmB_l\cap V^{-,\zeta}_k$, respectively, and then taking $\delta \to 0$). This together with known asymptotics for the third term in the integrand is sufficient to obtain the full version of Theorem~\ref{t:1}. The method, which we refer to as {\em stitching} (because we ``stitch'' the two ballot estimates together), can be used in a similar way to obtain the first part of Theorem~\ref{t:2.3} as well.

\subsection*{Paper outline}
The remainder of the paper is organized as follows. In Section~\ref{s:weak} we show how to reduce the DGFF ballot event to a corresponding ballot event involving the DRW. Two different DRWs are defined, using both the inward and the outward decompositions. This reduction is then used in Section~\ref{s:proof-DGFF}, in conjunction with ballot theory for the DRW, to derive a weak upper bound on the DGFF ballot probability, which is then used to prove Theorem~\ref{t:2.3} via stitching. Section~\ref{s:cont-as} includes the proof of Theorem~\ref{t:1}. As in the upper bound, the proof is based on weak statements for the asymptotics of the DGFF ballot probability, which are stated in this section and used, but not proved. The proofs of these weak statements together with the proofs of all the remaining main results are given in Section~\ref{s:LR-funct}. 

This manuscript has three appendices. Appendix~\ref{s:app-gen} includes statements concerning extreme value theory for the DGFF as well as general discrete harmonic analysis results. Appendix~\ref{s:tools} includes estimates for the harmonic extension of DGFF on a subset of its domain. Appendix~\ref{s:3} includes ballot theory for the DRW. While needed for the proofs in the paper, all of these results are derived using standard arguments and computations, and as such found their place in an appendix. We remark that the results in Appendix~\ref{s:tools} can be of independent use.

\section{Reduction to a decorated random walk ballot event}
\label{s:weak}
In this section we show how to express the ballot event on the left hand side of~\eqref{e:2.4ub} in terms of a corresponding ballot event for a suitable one-dimensional non-homogeneous random walk with decorations (DRW). The latter event can then be handled by generalized versions of the usual ballot estimates for random walks, which are given in Appendix~\ref{s:3} and used in the sections to follow to prove the main results in this manuscript.

The section begins with some general notation which will be used throughout the paper (Subsection~\ref{s:bno}). We then formally define the DRW process (Subsection~\ref{ss:DRW}) and introduce the inward concentric decomposition (Subsection~\ref{s:concdec}). The decomposition is then used to perform the reduction from the DGFF ballot event to the DRW ballot event (Subsection~\ref{s:def-DRW}). The reduction is summarized in Theorem~\ref{t:drw-i}. The proof of this theorem, in which the conditions of the DRW are verified, comes next (Subsection~\ref{s:concdec-ver}). Finally (Subsection~\ref{s:w-outw}), we introduce another reduction to a DRW, which is based on an outward concentric decomposition and results in a DRW of which the starting and end point are reversed. The reduction is summarized in Theorem~\ref{t:drw-o}. As the construction and proofs are analogous to the inward case, we only outline the required changes.

\subsection{General notation}
\label{s:bno}

In this subsection we introduce some notation and make some observations that will be used throughout the remaining part of the paper.

We define the collections of domains
\begin{equation}
\label{e:def-UV}
\begin{split}
\frU_\eps^\eta&=\{U\in\frD_\eps\,:\: U\text{ is connected, }\eps\rmB\subset U^\eta\}\\
\frV_\eps&=\{V\in\frD_\eps\,:\: V^-\text{ is connected}\}
\end{split}
\end{equation}
for $\eps\in(0,1)$, $\eta\geq 0$.
These definitions and~\eqref{e:1.5} imply the useful properties that
\begin{equation}
\label{e:UBBV}
\begin{split}
U^\eta_n&\supset \rmB_l
\quad\text{whenever }\, U\in\frU^\eta_\eps\,,\, n-l>-\log\eps\,,\\
V^{-,\zeta}_k &\supset \rmB^-_l
\quad\text{whenever }\, V\in\frV_\eps\,,\, l-k>\log(\zeta+\eps^{-1})\,.
\end{split}
\end{equation}
We also note that in all statements in the paper, the assumptions only become more general when $\eps$ is chosen smaller (at the expense of less sharp constants and rates of convergence).

We will use the shorthand notation
\begin{equation}
\bbP\Big(\cdot\,\Big|\, h_{\partial U_n}=u, h_{\partial V^-_k}=v\Big)
=\bbP_{V,k,v}^{U,n,u} = \bbP_{k,v}^{n,u}\,.
\end{equation}
The expectation, variance and covariance associated with this probability measure will also be denoted by $\bbE_{V,k,v}^{U,n,u}$ or $\bbE_{k,v}^{n,u}$, $\Var_k^n$ and $\Cov_k^n$, respectively.

The following property of our discretized domains will be useful:
\begin{rem}
\label{r:bd-pm}
We note that
\begin{equation}
\label{e:bd-pm}
\partial W^{\pm}_l=\partial W_l\cup \partial W^-_l
\quad\text{for all }l\geq 0\text{ and }W\in\frD\,.
\end{equation}
To show~\eqref{e:bd-pm}, it suffices by the definition of the outer boundary to show that $W_l\cap\partial W^-_l=\emptyset$ and $W^-_l\cap\partial W_l=\emptyset$.
The first statement follows as $x\in W_l$ implies by~\eqref{e:discr} that $\rmd(\rme^{-l}x,z)>\rme^{-l}/2$ for all $z\in\partial W$, hence $\rmd(\rme^{-l}y,\rme^{-l}x)>\rme^{-l}$ for all $y\in W^-_l$.
The second statement follows analogously.
As a consequence of~\eqref{e:bd-pm} and the definition of the outer boundary, $\partial W^{\pm}_l\setminus \partial W_l$ is not path-connected to $W_l$. Hence, specifying boundary conditions on $\partial W^{\pm}_l$ instead of $\partial W_l$ does not alter the DGFF on $W_l$:
\begin{equation}
\bbP\big(h\in\cdot\,\big|\, h_{\partial W_l}=w_{\partial W_l}\big)
=\bbP\big(h\in\cdot\,\big|\, h_{\partial W^\pm_l}=w\big)
\quad\text{on }W_l
\end{equation}
for any $w\in\bbR^{\partial W^\pm_l}$.
Analogously, specifying boundary conditions on $\partial W^{\pm}_l$ instead of $\partial W^-_l$ does not alter the DGFF on $U_n\cap W^-_l$.
\end{rem}

Let $(S_i)_{i=0}^\infty$ under a probability measure $P_x$ with associated expectation $E_x$ denote simple random walk on $\bbZ^2$ started in $x\in\bbZ^2$. For $A\subsetneq \bbZ^2$, we denote the first exit time of $S$ from $A$ by
$\tau^A=\inf\{i\in\bbN_0:\: S_i\notin A\}$, and the discrete Poisson kernel on $A$ by $\Pi_A(x,\cdot)=P_x(S_{\tau^A}\in\cdot)$. We recall that $G_A(x,y)=\sum_{i=0}^{\infty} P_x(S_i=y,i<\tau^A)$, $x,y\in\bbZ^2$ denotes the Green kernel of $S$ on $A$, and $\fra(x)=\sum_{i=0}^{\infty}\big(P_x(S_i=x)-P_0(S_i=x)\big)$ the potential kernel of $S$ on $\bbZ^2$.

By Theorem~4.4.4 of~\cite{LaLi}, there exists a constant $c_0>0$ such that
\begin{equation}
\label{e:444}
\fra(x)=g\log|x|+c_0+O\left(|x|^{-2}\right)\,,\qquad x\in\bbZ^2\,.
\end{equation}
We reserve the notation $c_0$ for this constant throughout the paper.
By Proposition~4.6.2 of~\cite{LaLi}, for finite $A$,
\begin{equation}
\label{e:462}
G_A(x,y)=\sum_{z\in\partial A}\Pi_A(x,y)\fra(y-z)-\fra(x-y)\,.
\end{equation}
We will often use the representation
\begin{equation}
\he{f}{A}(x)=\sum_{z\in \partial A^\rmc}f(z)P_x\big(S_{\tau^{A^\rmc}}=z\big)\,,\qquad x\in\bbZ^2
\end{equation}
for the bounded harmonic extension of a function $f:A\to\bbR$, for $A\subset \bbZ^2$ finite and non-empty.

\subsection{The decorated non-homogeneous random walk process} 
\label{ss:DRW}
We now formally define the DRW process. Let $\rt \in \bbN \cup \{\infty\}$, $a,b \in \bbR$ and $(\sigma_k^2 :\: k=1,\dots, \rt)$ be a sequence of positive real numbers. Set $s_k := \sum_{\ell=1}^k \sigma_{\ell}^2$ and consider the collections of random variables $(\cS_k)_{k=0}^\rt$ and $(\cD_k)_{k=1}^\rt$ which we assume to be defined on the same probability space, equipped with a probability measure to be denoted by $\bfP$ (with expectation, covariance and variance given by $\bfE$, $\bfCov$ and $\bfVar$).

We shall impose the following three assumptions which depend on a parameter $\delta > 0$ (this parameter will usually determine the value of constants and rates of convergence in the theorems we prove for this walk). 
\begin{enumerate}[label=({\bf A\arabic{*}}),ref=(A\arabic{*})]
\item 
\label{i.a1}
The process $(\cS_k)_{k=0}^\rt$ is Gaussian with means and covariances given by
\begin{equation}
\bfE \cS_k = \frac{b s_k + a(s_\rt-s_k)}{s_\rt} \ , \quad
\bfCov (\cS_k, \cS_m) = \frac{s_{k}(s_\rt - s_m)}{s_\rt} 
\quad : \quad 0 \leq k \leq m \leq \rt \,.
\end{equation}
Moreover, $\sigma_k^2 \in (\delta, \delta^{-1})$ for all $k=1,\ldots,\rt$.
\item 
\label{i.a2}
For all $m=1, \dots, \rt -1$, given $\cS_m$ the two collections below are conditionally independent of each other:
\begin{equation}
\label{e:1.1}
(\cS_k, \cD_k :\: k = 1, \dots, m) \quad, \qquad (\cS_\ell :\: \ell = m+1, \dots, \rt) \,.
\end{equation}
\item 
\label{i.a3}
For all $k=1, \dots, \rt$, $t > 0$,
\begin{equation}
\label{e:5.1}	
\bfP \big(|\cD_k| > \delta^{-1} \wedge^{1/2-\delta}_{\rt,k} + t
\big) \leq \delta^{-1} \rme^{-t^\delta}\,.
\end{equation}
\end{enumerate}
If $\rt = \infty$, we take $a$ and $s_k$ to be the meaning of the respective right hand sides in Assumption~\ref{i.a1}. The reader will recognize that the law of $(\cS_k)_{k=0}^\rt$ under $\bfP$ is that of a random walk whose $k$-th step is $\cN(0, \sigma_k^2)$, starting from $a$ at time $0$ and conditioned to be at $b$ at time $\rt$, if $\rt < \infty$, or otherwise unrestricted (in which case $b$ is irrelevant). We shall sometimes explicitly indicate the boundary conditions of such walk by writing $\bfP_{0,a}^{\rt,b}$ in place of $\bfP$ (and using a similar convention for variances and covariances). Notice that the $\cD_k$-s, which we refer to as ``decorations'', may indeed depend on the walk, as long as they satisfy the Markovian-type structure given in Assumption~\ref{i.a2}. We shall refer to the pair $\big((\cS_k)_{k=0}^\rt, (\cD_k)_{k=1}^\rt\big)$ as a DRW process.

\subsection{Inward concentric decomposition}
\label{s:concdec}
Next, we introduce the inward concentric decomposition.
For $\eps\in(0,1)$, $\eta,\zeta\geq 0$, $0\leq k<n$, and $U\in\frU_\eps^\eta$, $V\in\frV_\eps$,
we now define two collections of concentric discretized ball-like sets and their complements, using which we will construct collections of annulus-like sets that we then pack into the domain $U_n\cap V^-_k$.
Throughout the paper, we always define $\rt_{n-k}$ as a function of $n-k$ by
\begin{equation}
\label{e:def-r}
\rt_{n-k} =\lfloor n-k\rfloor + \lfloor\log\eps\rfloor - \lceil\log(\zeta +\eps^{-1})\rceil\,,
\end{equation}
and we often abbreviate $\rt=\rt_{n-k}$.
We assume that $\rt_{n-k}\geq 1$, then~\eqref{e:UBBV} ensures that the concentric ball-like sets
\begin{multline}
\Delta_{-1}=U_n\,,\quad \Delta_0= U^\eta_n\,,\quad
\Delta_p = \rmB_{n+\lfloor \log\eps\rfloor-p}\quad\text{for }p=1,\ldots,\rt -1\,,\quad
\Delta_\rt = (V^{-,\zeta}_k)^\rmc\,,\\
\Delta_{\rt+1}=(V^-_k)^\rmc
\end{multline}
are nested, $\Delta_{-1}\supset \ldots \supset \Delta_{\rt+1}$.
Furthermore, we define the sets
\begin{equation}
\label{e:Delta'}
\Delta'_{-1}=(U_n)^\rmc\,,\quad \Delta'_0= U^{\eta,-}_n\,,\quad
\Delta'_p = \rmB^-_{n+\lfloor \log\eps\rfloor-p}\quad\text{for }p=1,\ldots,\rt -1\,,\quad
\Delta'_\rt = V^{-,\zeta}_k
\end{equation}
For $p=0,\ldots,\rt$, let $h_p$ be a DGFF on $\rA_p:=\Delta'_p\cap\Delta_{p-1}$ with boundary values zero.
Moreover, we define $J_p:=\rA_p\cup(\Delta_p\cap V^-_k)$ for $p=0,\ldots,\rt-1$, $J_{\rt}=\rA_\rt$, and let $\varphi_{p}$ be distributed as $\heb{h'_p}{\bbZ^2 \setminus J_p}$ where $h'_p$ is a DGFF on $\Delta_{p-1}\cap V^-_k$ with boundary values $0$ (and, according to our notational conventions, equal to zero on $\bbZ^2\setminus(\Delta_{p-1}\cap V^-_k)$).
We note that $\heb{h'_p}{\bbZ^2\setminus J_p}$ is the binding field from $\Delta_{p-1}\cap V^-_k$ to $J_p$, as defined in the Gibbs-Markov decomposition (Lemma~\ref{l:GM}).
We assume that the random fields $\varphi_0,\ldots,\varphi_{\rt},h_0,\ldots, h_{\rt}$ are independent.

We will also use the notation $\varphi_{p,q}:=\sum_{j=p}^q\varphi_j$ where $0\leq p\leq q\leq \rt$. For $p<q$, we set $\varphi_{p,q}=0$.
Note that for $q\leq \rt-1$, $\varphi_{p,q}$ is distributed as
$\varphi^{\Delta_{p-1}\cap V^-_k,\rA_p\cup\ldots\rA_q\cup \Delta_q\cap V^-_k}$,
and that $\varphi_{p,\rt}$ is distributed as
$\varphi^{\Delta_{p-1}\cap V^-_k,\rA_p\cup\ldots\rA_\rt\cap V^-_k}$,
which can be seen by applying the Gibbs-Markov property (Lemma~\ref{l:GM}) successively to the subsets $\rA_p,\ldots,\rA_q$.

In the next proposition, we decompose the DGFF on $U_n\cap V^-_k$ on each of the annuli $A_p:=\Delta_{p-1}\setminus\Delta_{p}$ with $p=0,\ldots,\rt+1$ in terms of the fields $\varphi_0,\ldots,\varphi_{p\wedge\rt}$ and $h_p1_{q\leq\rt}$. We note that $A_p\supset \rA_p$ for $p=0,\ldots,\rt$ and that $(A_p)_{p=0}^{\rt +1}$ forms a disjoint covering of $U_n\cap V^-_k$.
\begin{prop}[Inward concentric decomposition]
\label{p:concdec-in}
Assume that $\rt=\rt_{n-k}\geq 1$. There exists a coupling of $h^{U_n\cap V^-_k}$ and $(\varphi_p,h_p)_{p=0}^{\rt}$ such that
\begin{equation}
\label{e:concdec-in}
h^{U_n\cap V^-_k}(x)=\sum_{p=0}^{\rt}\big(\varphi_p(x)+h_p(x)\big)=\sum_{p=0}^{q\wedge\rt}\varphi_p(x)+h_q(x)1_{q\leq\rt}
\end{equation}
for $x\in A_{q}$, $q=0,\ldots,\rt +1$.
\end{prop}
Note that the binding field $\varphi_p$ is zero outside $\Delta_{p-1}\cap V^-_k$, and $h_p$ is zero outside $A_p$.
Also, if $\eta=0$, then $A_0=\emptyset$ and $\varphi_0=0$.
\begin{proof}[Proof of Proposition~\ref{p:concdec-in}]
By the Gibbs-Markov property (Lemma~\ref{l:GM}), $h^{U_n\cap V^-_k}$ is distributed as $\varphi_0+h'_0$ where $h'_0$ is a DGFF on $\rA_0$ with boundary values zero that is independent of $\varphi_0$. The assertion follows by iterating this step, next applying the Gibbs-Markov property to $h'_0$ and $\rA_1$ and so on.
\end{proof}

\subsection{The reduction}
\label{s:def-DRW}
We now provide the correspondence between the DGFF and the decorated random walk bridge, based on the inward concentric decomposition. An analogous correspondence based on the outward concentric decomposition is given in Subsection~\ref{s:w-outw}.

In Subsections~\ref{s:def-DRW} and~\ref{s:concdec-ver}, we assume $n-k$ to be sufficiently large such that $\rt=\rt_{n-k}\geq 1$, and that $h$ under $\bbP_{k,v}^{n,u}$ is coupled to $(\varphi_p,h_p)_{p=1}^{\rt}$ from Proposition~\ref{p:concdec-in} such that
\begin{equation}
\label{e:hphih}
h=\sum_{p=0}^{\rt}\big(\varphi_p+h_p\big)+
\he{(-m_n\Ind_{\partial U_n}+u -m_k\Ind_{\partial V^-_k} + v)}{\partial U_n\cup\partial V^-_k}\,.
\end{equation}

In each of the sets $A_p$ for $p=1,\ldots, \rt$, we want to approximate $h$ by the position of a random walk bridge. To this aim, we define the harmonic average
\begin{equation}
\label{e:Xm}
\cX_p=\sum_{z\in\partial \Delta_p}\varphi_p(z)\Pi_{\Delta_p}(0,z)
=\heb{\varphi_p}{\partial \Delta_p}(0)
=\heb{h'_p}{\partial \Delta_p}(0)
\end{equation}
and denote its variance by
$\sigma_p^2:=\Var \cX_p$. We write $s_{p,q}=\sum_{i=p}^q \sigma_i^2$. Now we couple $(\cX_1,\ldots,\cX_\rt)$ with the process
\begin{equation}
\label{e:deftS}
\wt{\cS}_t:=(s_{1,\rt}-t)\int_0^t\frac{\rmd \cW_\tau}{s_{1,\rt}-\tau}\,,\quad t\in[0,s_{1,\rt}]
\end{equation}
by choosing the Brownian motion $(\cW_\tau,\tau\in [0,s_{1,\rt}])$ as follows:
we define $\cW_0=0$, $\cW_{s_{1,p}}= \cX_1 +\ldots+ \cX_p$ for $p=1,\ldots,\rt$, and
for $\tau\in[s_{1,p-1},s_{1,p}]$, we let
\begin{equation}
\label{e:defW}
\cW_\tau=\cX_{p-1}+\tfrac{\tau-s_{1,p-1}}{\sigma_p^2}(\cX_p-\cX_{p-1})
+\sqrt{\sigma_p^2}\cB^{(p)}_{(\tau-s_{1,p-1})/\sigma_p^2}\,,
\end{equation}
where $\cB^{(p)}$ is a standard Brownian bridge from $0$ to $0$ of length $1$, independent of everything else.
As
\begin{equation}
\label{e:CovS}
\Cov(\wt{\cS}_t,\wt{\cS}_{t'})=
(s_{1,\rt}-t)(s_{1,\rt}-t')\int_0^t\frac{\rmd \tau}{(s_{1,\rt}-\tau)^2}
=\frac{t(s_{1,\rt}-t')}{s_{1,\rt}}
\end{equation}
for $0\leq t<t'\leq s_{1,\rt}$, the centered Gaussian process $(\wt{\cS}_t,t\in[0,s_{1,\rt}])$ is a Brownian bridge of length $s_{1,\rt}$ from $0$ to $0$. Hence,
\begin{equation}
\label{e:defS}
\cS'_p:=\wt{\cS}_{s_{1,p}}\,,\quad p=1,\ldots, \rt\,,\quad \cS'_0=0
\end{equation}
is distributed as a random walk with centered Gaussian steps having variances $\sigma_p^2$, $p=1,\ldots, \rt$, starting at $0$ and conditioned on hitting $0$ after $\rt$ steps.

We approximate the harmonic extension of the boundary values $u\in\bbR^{\partial U_n}$, $v\in\bbR^{\partial V^-_k}$ by
\begin{equation}
\label{e:bell}
\beta_p=\frac{s_{1+p,\rt}}{s_{1,\rt}}\ol{u}(0)+\frac{s_{1,p}}{s_{1,\rt}}\ol{v}(\infty)\,,\quad p=0,\ldots,\rt\,.
\end{equation}

Also recall the independent DGFF $h_p$. By the Gibbs-Markov property,
\begin{equation}
\label{e:conc-h}
h_p=h-\varphi_{0,p}
- \he{(-m_n\Ind_{\partial U_n}+u -m_k\Ind_{\partial V^-_k} + v)}{\partial U_n\cup\partial V^-_k}
\quad\text{on }A_{p}\,.
\end{equation}
Moreover, we define
\begin{equation}
\label{e:corr-bd}
\gamma(y)=\he{(-m_{k}\Ind_{\partial V^-_k}+v)}{\partial U_n\cup \partial V^-_k}(y)+\he{(-m_n\Ind_{\partial U_n}+u)}{\partial U_n\cup \partial V^-_k}(y)
+m_{n-p}-\beta_p\quad
\text{for }y\in A_{p}\,,
\end{equation}
Finally, we set
\begin{equation}
\label{e:Dell}
\cD_p=\max_{y\in A_{p}}\big\{h_p(y)-m_{n-p}+\varphi_{0,p}(y)-\cS'_p
+\gamma(y)\big\}\,,
\end{equation}
and
\begin{equation}
\label{e:cS}
\cS_p = \cS'_p + \beta_p\,.
\end{equation}

With the definitions above, we can now state the following theorem, which shows that the processes just defined form a suitable DRW and relates the DGFF ballot event with that of the DRW. The proof of this theorem is the subject of the next subsection. Notice that~\eqref{e:h-S} follows directly and immediately from the definitions above.
\begin{thm}
\label{t:drw-i}
Let $\eps\in(0,1)$.
Then there exists $\delta\in(0,1/3)$ such that for all
$\eta,\zeta \in[0,\eps^{-1}]$,
$U\in\frU^\eta_\eps$, $V\in\frV_\eps$, $0\leq n<k$ with $T_{n-k}\geq 1$,
$u\in\bbR^{\partial U_n}$, $v\in\bbR^{\partial V^-_k}$ that satisfy
\begin{equation}
\label{e:uv-diff}
\max\big\{\ol{u}(0),\ol{v}(\infty),\osc\,\ol{u}_\eta,\osc\,\ol{v}_\zeta \big\}\leq \eps^{-1}\,,\quad
\big| \ol{u}(0) - \ol{v}(\infty) \big| \leq (n-k)^{1-\eps}\,,
\end{equation}
we have that $(\cS_i)_{i=0}^\rt$,
$(\cD_i)_{i=1}^\rt$
satisfies Assumptions~\ref{i.a1} --~\ref{i.a3} with $\delta$, $a=\ol{u}(0)$ and $b=\ol{v}(\infty)$ under the identification 
$\bfP = \bfP_{0,\ol{u}(0)}^{\rt,\ol{v}(\infty)} \equiv \bbP_{k,v}^{n,u}$.
Moreover,
\begin{equation}
\label{e:h-S}
\{\cS_{p}+\cD_p\leq 0\} =
\big\{ h_{A_p}\leq 0\big\}\text{ for } p=1,\ldots,\rt\,,\quad
\text{and}\quad
\bigcap_{p=1}^{\rt}\{\cS_{p}+\cD_p\leq 0\}
=
\big\{h_{U^\eta_n\cap V^{-,\zeta}_k}\leq 0\big\}
\,.
\end{equation}
\end{thm}

\subsection{Verification of the DRW conditions - Proof of Theorem~\ref{t:drw-i}}
\label{s:concdec-ver}
In this subsection we prove Theorem~\ref{t:drw-i}. We begin, by verifying the assumption on the dependency structure of the decorated random walk.
\begin{lem}
\label{l:dep}
The random variables $(\cS_p)_{p=0}^\rt$, $(\cD_p)_{p=1}^\rt$ satisfy~\ref{i.a2} under the identification 
$\bfP = \bfP_{0,\ol{u}(0)}^{\rt,\ol{v}(\infty)} \equiv \bbP_{k,v}^{n,u}$.
\end{lem}
\begin{proof}
As $\cS_p-\cS'_p=\beta_p$ is deterministic, it suffices to show the assertion for $(\cS'_p)_{p=0}^\rt$, $(\cD_p)_{p=1}^\rt$.
For $p=1,\ldots,\rt -1$, let $\mathcal F_p$ be the sigma algebra generated by $(\cW_t,t\leq s_{1,p})$, $\varphi_1,\ldots, \varphi_p$ and $h_1,\ldots, h_p$.
By the Gibbs-Markov property and the definition of $\cW$, the process $(\cW_{s_{1,p}+t}-\cW_{s_{1,p}}, t\in[0,s_{1,\rt}-s_{1,p}])$ is independent of $\mathcal F_p$.
While $\cD_1,\ldots, \cD_p$ and $\cS'_1,\ldots,\cS'_p$ are $\mathcal F_p$-measurable, the following representation shows that
$(\cS'_j-\cS'_p)_{j=p+1}^\rt$ is measurable with respect to $(\cW_{s_{1,p}+t}-\cW_{s_{1,p}}, t\in[0,s_{1,\rt}-s_{1,p}])$ and $\cS'_p$ and hence conditionally independent of $\mathcal F_p$ given $\cS'_p$:
\begin{multline}
\cS'_j - \cS'_p
= (s_{1,\rt} - s_{1,j})\int_{s_{1,p}}^{s_{1,\rt}}
\frac{\rmd \cW_\tau}{s_{1,\rt}-\tau}
-(s_{1,j}-s_{1,p})\int_{0}^{s_{1,p}}
\frac{\rmd \cW_\tau}{s_{1,\rt}-\tau}\\
= (s_{1,\rt} - s_{1,j})\int_{s_{1,p}}^{s_{1,\rt}}
\frac{\rmd \cW_\tau}{s_{1,\rt}-\tau}
-\frac{s_{1,j}-s_{1,p}}{s_{1,\rt}-s_{1,p}}\cS'_p\,,
\end{multline}
where we used the definition of $(\cS'_p)$ which is given by~\eqref{e:defS} and~\eqref{e:deftS}.
\end{proof}

As the binding field is harmonic, we have
\begin{equation}
\label{e:phi-rw}
\varphi_p(y)
=\sum_{z\in\partial A_p\cup \partial \Delta_p}\varphi_p(z)\Pi_{A_p\cup \Delta_p\cap V^-_k}(y,z)\,,\quad y\in\bbZ^2
\end{equation}
for $p=0,\ldots,\rt$.

Next we show bounds for the variance $\sigma_p=\Var_k^n \cX_p$.
\begin{lem}
\label{l:varXm}
Let $\eps\in(0,1)$.
There exists $C=C_{\eps}<\infty$ such that $C^{-1}<\Var_k^n \cX_p< C$ and
\begin{equation}
\label{e:varsil}
|s_{p,q}-g(q-p+1)|\leq C\big(1+\log (1+\rt -p) -\log (1+\rt -q)\big)
\end{equation}
for all $\eta,\zeta \in[0,\eps^{-1}]$, $0\leq k<n$ with $T_{n-k}\geq 1$, $p,q\in\bbN$ with $p\leq q \leq \rt$, and all $U\in\frU^\eta_\eps$, $V\in\frV_\eps$.
Moreover, for $p\in\bbN$ fixed, $\Var_k^n \cX_p$ converges as $n-k\to\infty$ uniformly in $U\in\frU^\eta_\eps$, $V\in\frV_\eps$. We also have $\Var_k^n \cX_0<C$.
\end{lem}
\begin{proof}
For $p\geq 1$, the bound for $\Var_k^n \cX_p$ and the convergence of $\Var_k^n \cX_p$ follow from Lemma~\ref{l:Cov-as} where we choose $j=1$, and we choose $U$, $W$ and $n$ such that $U_{n}$ becomes our $\Delta_{p-1}$, and $W_{n-1}=\Delta_p$.
Also $\Var_k^n\cX_0$ is bounded by a constant as
\begin{equation}
\Var_k^n\Big(\he{h}{\partial U^\eta_n}(0)\,\Big|\,
h_{(2\eps^{-1}\rmB)_n}=0,h_{\partial V^-_k}=0\Big)
=\Var_k^n\cX_0 + \Var\Big(\he{\varphi^{(2\eps^{-1}\rmB)_n\cap V^-_k,U_n\cap V^-_k}}{\partial U^\eta_n}(0)\Big)\,,
\end{equation}
by the Gibbs-Markov property, and the the left-hand side is bounded by a constant by the same argument as for $\cX_p$ with $p\geq 1$.

For assertion~\eqref{e:varsil}, let $j\in\{p,\ldots,q\}$.
By the Gibbs-Markov property (Lemma~\ref{l:GM}) applied at $\partial V^-_k$,
\begin{equation}
\label{e:VarXj}
\Var_k^n \cX_j=\Var\he{{h^{\Delta_{j-1}}}}{\partial \Delta_j}(0)
-\Var\hebb{\heb{h^{\Delta_{j-1}}}{\partial \Delta_{j-1}\cup\partial V^-_k}}{\partial \Delta_{j}}(0)\,.
\end{equation}
The second variance on the right-hand side is of order $(\rt -j+2)^{-1}$ by Lemma~\ref{l:Cov-bdg}.

Moreover, by successively applying the Gibbs-Markov property at $\partial \Delta_p,\ldots,\partial \Delta_q$,
\begin{equation}
\sum_{j=p}^q\Var\he{{h^{\Delta_{j-1}}}}{\partial \Delta_j}(0)
=\Var\he{{h^{\Delta_{p-1}}}}{\partial \Delta_q}(0)\,.
\end{equation}
By Lemma~\ref{l:Var-inw-1}, $\big|\Var\he{{h^{\Delta_{p-1}}}}{\partial \Delta_q}(0)-g(q-p)\big|$ is uniformly bounded by a constant.
The assertion now follows by also summing the second term on the right-hand side of~\eqref{e:VarXj} over $j=p,\ldots,q$.
\end{proof}

To compare $\cS_p$ with $\varphi_{0,p}$ on $A_{p}$, we also define
\begin{equation}
\label{e:R}
\cY_p=\sum_{j=1}^p\frac{s_{p,\rt}}{s_{j,\rt}}\cX_j
\end{equation}
for $p=1,\ldots, \rt$. The centered Gaussian random variables $\cY_p$ and $\cS'_p$ can be compared as follows.
\begin{lem}
\label{l:dec-coupl}
Let $\eps\in(0,1)$.
Then there exists $C = C_{\eps} <\infty$ such that
\begin{equation}
\label{e:dec-coupl-Var}
\Var_k^n(\cS'_p-\cY_p)\leq C(n-k-p+1)^{-1}
\end{equation}
for all $\eta,\zeta \in[0,\eps^{-1}]$, $0\leq n<k$ with $T_{n-k}\geq 1$, $U\in\frU^\eta_\eps$, $V\in\frV_\eps$, $p=1,\ldots,\rt$.
\end{lem}
\begin{proof}
If $p=\rt$, then we use that $\cS'_\rt=0$ a.\,s.\ and we have
\begin{equation}
\Var_k^n(\cS'_\rt-\cY_\rt)=\sum_{j=1}^\rt \frac{\sigma_\rt^2}{s_{j,\rt}^2}\leq C_\eps
\end{equation}
by Lemma~\ref{l:varXm}. In the following, we can thus assume $p\leq \rt -1$.
From the definition of $\cS'_p$ and $\cY_p$, we obtain
\begin{equation}
\Var_k^n\big(\cS'_p-\cY_p\big)
=\sum_{j=1}^p\int_{s_{1,j-1}}^{s_{1,j}}\Big(
\frac{s_{p+1,\rt}}{s_{1,\rt}-t}-\frac{s_{p,\rt}}{s_{j,\rt}}\Big)^2\rmd t\,.
\end{equation}
The expression in the brackets on the right-hand side can be bounded as follows:
\begin{equation}
-\frac{\sigma^2_p}{s_{j,\rt}}\leq \frac{s_{p+1,\rt}}{s_{1,\rt}-t}-\frac{s_{p,\rt}}{s_{j,\rt}}
\leq \frac{s_{p+1,\rt}}{s_{j+1,\rt}}-\frac{s_{p+1,\rt}}{s_{j+1,\rt}+\sigma^2_j}\leq
-\frac{\sigma^2_p}{s_{j+1,\rt}}+s_{p+1,\rt}\frac{\sigma_j^2}{s_{j,\rt}s_{j+1,\rt}}
\end{equation}
for $t\in[s_{1,j-1},s_{1,j}]$.
Using also Lemma~\ref{l:varXm},
we obtain that
\begin{equation}
\Var_k^n\big(\cS'_p-\cY_p\big)\leq C_{\eps}\sum_{j=1}^p(\rt -j+2)^{-2}\leq C_{\eps}(\rt -p+1)^{-1}\,,
\end{equation}
and the assertion follows.
\end{proof}

We now come to the tails of the decorations. First we consider the part of $\cD_p$ which does not depend on the boundary values $u$ and $v$, namely $\wt \cD_p:=\max_{y\in A_{p}}\big\{\varphi_{0,p}(y)-\cS'_p+h_p(y)-m_{n-p}\big\}$.
\begin{lem}
\label{l:dec-tail}
For each $\eps\in(0,1)$, there exists $\delta\in(0,\tfrac13)$ such that the bounds in~\ref{i.a3} hold with $(\cS'_p)_{p=0}^\rt$, $(\wt \cD_p)_{p=1}^\rt$ in place of $(\cS_p)_{p=0}^\rt$, $(\cD_p)_{p=1}^\rt$ there,
with $a=\ol{u}(0)$, $b=\ol{v}(\infty)$, for all $\eta,\zeta \in[0,\eps^{-1}]$, $U\in\frU^\eta_\eps$, $V\in\frV_\eps$, $0\leq k<n$ with $T_{n-k}\geq 1$ and $u\in\bbR^{\partial U_n}$, $v\in\bbR^{\partial V^-_k}$, under the identification 
$\bfP = \bfP_{0,\ol{u}(0)}^{\rt,\ol{v}(\infty)} \equiv \bbP_{k,v}^{n,u}$.
\end{lem}
\begin{proof}
To avoid the roughness of $\varphi_{0,p}$ near $\partial \rA_p$, we lump the binding fields $\varphi_p$ and $\varphi_{p-1}$ together with $h_p$ by defining
\begin{equation}
\label{e:wth}
\wt h_p(y) =
h_p(y) + \varphi_p(y) + \varphi_{p-1}(y) - m_{n-p}
\,,\quad y\in A_{p}
\end{equation}
and
\begin{equation}
\label{e:tDelta}
\cD'_p(y) =
\varphi_{0,p-2}(y)-\cY_p\,,
\quad y\in A_{p}\,,
\end{equation}
for $p=1,\ldots,\rt$ (we recall that $\varphi_{0,-1}=0$ by definition).
Then, for $t>0$ and $p=1,\ldots,\rt$, a union bound gives
\begin{equation}
\label{e:dec-tail-r}
\bbP_{k,v}^{n,u}\big(\wt \cD_p>t\big)
\leq \bbP_{k,v}^{n,u}\Big(\max_{A_{p}} \wt h_p  >t/3\Big)
+\bbP_{k,v}^{n,u}\Big(\max_{A_{p}}\cD'_p>t/3\Big)
+\bbP_{k,v}^{n,u}\big(|\cY_p-\cS'_p|>t/3\big)
\end{equation}
and
\begin{equation}
\label{e:dec-tail-l}
\bbP_{k,v}^{n,u}\big(-\wt \cD_p>t\big)
\leq \bbP_{k,v}^{n,u}\Big(\max_{A_{p}} \wt h_p  <-t/3\Big)
+\bbP_{k,v}^{n,u}\Big(\min_{A_{p}}\cD'_p<-t/3\Big)
+\bbP_{k,v}^{n,u}\big(|\cY_p-\cS'_p|>t/3\big)\,.
\end{equation}

By the Gibbs-Markov property (Lemma~\ref{l:GM}) and by~\eqref{e:wth}, the restriction of $\wt h_p + m_{n-p}$ to $A_{p}$ is distributed as $h^{\Delta_{p-2}\cap V^-_k}$. Hence, again by the Gibbs-Markov property,
we can represent the restriction of the DGFF $h^{\Delta_{p-2}}$ to $A_{p}$ as the sum of the independent fields $\wt h_p + m_{n-p}$ and the binding field
$\varphi^{\Delta_{p-2},\Delta_{p-2}\cap V^-_k}$.
As $\varphi^{\Delta_{p-2},\Delta_{p-2}\cap V^-_k}$ is with probability $1/2$ positive at the maximizer of $\wt h_p$, it follows that
\begin{equation}
\tfrac12\bbP_{k,v}^{n,u}\Big(\max_{A_{p}} \wt h_p  >t/3\Big)
\leq\bbP\Big(\max_{A_{p}} h^{\Delta_{p-2}}-m_{n-p}>t/3\Big)\,.
\end{equation}
The right-hand side is bounded from above by a Gaussian tail which follows from extreme value theory for the DGFF (Lemma~\ref{l:DGFF-ut}).
By Lemma~\ref{l:Ding}, the first probability on the right-hand side of~\eqref{e:dec-tail-l} is bounded by a constant times $\rme^{-t^{2-\eps}}$.
The third summand on the right-hand side of~\eqref{e:dec-tail-r} and~\eqref{e:dec-tail-l} has uniformly Gaussian tails by Lemma~\ref{l:dec-coupl}.

To bound the second summand on the right-hand side of~\eqref{e:dec-tail-r} and of~\eqref{e:dec-tail-l}, we plug~\eqref{e:R} into~\eqref{e:tDelta}, so as to obtain for $y\in A_{p}$ that
\begin{equation}
\label{e:dec-tail-23}
\Var_k^n \cD'_p(y)=\sum_{j=1}^{p-2}
\Var_k^n\Big(\varphi_j(y)-\frac{s_{p,\rt}}{s_{j,\rt}}\cX_j\Big)
+\Ind_{\{p\geq 2\}}\Var_k^n\varphi_0(y)
+\Var_k^n\Big(\frac{s_{p,\rt}}{s_{p-1,\rt}}\cX_{p-1}\Big)
+\Var_k^n\cX_{p}\,.
\end{equation}
By~\eqref{e:Xm} and~\eqref{e:phi-rw}, the first sum on the right-hand side is bounded by
\begin{multline}
\label{e:dec-tail-Pois}
\sum_{j=1}^{p-2}\sum_{w,w'\in\partial \Delta_j}
\Big|\Pi_{\Delta_{j}\cap V^-_k}(y,w) -\frac{s_{p,\rt}}{s_{j,\rt}}\Pi_{\Delta_j}(0,w)\Big|\\
\times\Big|\Pi_{\Delta_j\cap V^-_k}(y,w')-\frac{s_{p,\rt}}{s_{j,\rt}}\Pi_{\Delta_j}(0,w')\Big|
\Cov_k^n\big(\varphi_j(w),\varphi_j(w')\big)\,,
\end{multline}
and the other terms on the right-hand side of~\eqref{e:dec-tail-23} are bounded by a constant by Lemma~\ref{l:varXm}.
Again by Lemma~\ref{l:varXm},
\begin{equation}
\label{e:p-dec-tail-jlogj}
\frac{s_{p,\rt}}{s_{j,\rt}}\leq \frac{\rt+1-p +C_{\eps}(1+\log(\rt +1-p))}{\rt+1-j-C_{\eps}(1+\log(\rt +1-j))}
\leq \frac{\rt+1-p}{\rt+1-j}+C_{\eps}\frac{1+\log(\rt +1-j)}{\rt+1-j}\,.
\end{equation}
Using~\eqref{e:p-dec-tail-jlogj}, the analogous lower bound, and Lemma~\ref{l:Poisson},
we bound the absolute differences in~\eqref{e:dec-tail-Pois} by the left-hand side of
\begin{equation}
\label{e:p-dectail-lPoisson}
\bigg[\frac{1}{p-j} + \bigg|
\frac{s_{p,\rt}}{s_{j,\rt}} - \frac{\rt+1-p}{\rt+1-j} \bigg|\bigg]
\Pi_{\Delta_j}(0,w)
\leq C_{\eps}\,\frac{1+\log(p-j)}{p-j}
\Pi_{\Delta_j}(0,w)
\end{equation}
so as to obtain that the triple sum in~\eqref{e:dec-tail-Pois} is bounded by
a constant times $\sum_{j=1}^{p-2}{(p-j)^{-2+1/10}}\Var_k^n \cX_j$.
By Lemma~\ref{l:varXm}, it follows that $\Var_k^n\cD'_p(y)$ bounded by a constant.
Hence, the Borell-TIS inequality (see e.\,g.\ Theorem~2.1.1 in~\cite{AdlerTaylor}) shows that
\begin{equation}
\bbP_{k,v}^{n,u}\Big(\Big|\max_{A_{p}}\cD'_p -
\bbE_{k,v}^{n,u}\big(\max_{A_{p}}\cD'_p\big)
\Big|> t \Big)
\leq C_{\eps}\rme^{-c_{\eps}t^2}
\end{equation}
for all $t>0$, $U\in\frU^\eta_\eps$, $V\in\frV_\eps$, $0\leq k<n$, $p=1,\ldots,\rt$.

Now we show that $\bbE_{k,v}^{n,u}\max_{A_{p}}\cD'_p$ is bounded by a constant.
From Definitions~\eqref{e:R} and~\eqref{e:tDelta}, we have $\max_{y\in A_{p}}\cD'_1(y)=\cY_1=\cX_1$ which is a centered Gaussian whose variance is bounded by a constant by Lemma~\ref{l:varXm}. For $p\geq 2$, we have the following bound for the (squared) intrinsic metric:
For $x,y\in A_{p}$,
\begin{multline}
\bbE_{k,v}^{n,u}\big[\big(\cD'_p(x)-\cD'_p(y)\big)^2\big]
= \sum_{j=0}^{p-2}
\bbE\big[\big( \varphi^{\Delta_{j-1}\cap V^-_k, \Delta_{j}\cap V^-_k}(x)
-\varphi^{\Delta_{j-1}\cap V^-_k, \Delta_{j}\cap V^-_k}(y)\big)^2\big]\\
= \sum_{j=0}^{p-2}\sum_{z,z'\in\partial \Delta_{j}}\big[\Pi_{\Delta_{j}\cap  V^-_k}(x,z)-\Pi_{\Delta_{j}\cap V^-_k}(y,z)\big]
\big[\Pi_{\Delta_{j}\cap V^-_k}(x,z')-\Pi_{\Delta_{j}\cap V^-_k}(y,z')\big]\\
\times \bbE\big[ h^{\Delta_{j-1}\cap V^-_k}(z) h^{\Delta_{j-1}\cap V^-_k}(z')\big]
\label{e:dectailfin}
\end{multline}
where we used the definition of $\varphi_{0,p-2}$, the Gibbs-Markov property, and \eqref{e:phi-rw}.
Applying Theorem~6.3.8 of~\cite{LaLi} as in the proof of Lemma~\ref{l:metr} in Section~\ref{s:p-compl},
we see that the differences of the Poisson kernels in the last display are bounded (in absolute value) by a constant times
$\rme^{-(n-j)}|x-y|\Pi_{\Delta_{j}\cap V^-_k}(x,z)$
and
$\rme^{-(n-j)}|x-y|\Pi_{\Delta_{j}\cap V^-_k}(x,z')$, respectively.
Hence,~\eqref{e:dectailfin} is further bounded from above by a constant times
\begin{equation}
\sum_{j=0}^{p-2}\frac{|x-y|^2}{\rme^{2(n-j)}}
\sum_{z,z'\in\partial \Delta_{j}}
\Pi_{\Delta_{j}}(x,z)\Pi_{\Delta_{j}}(y,z)\bbE\big[ h^{\Delta_{j-1}\cap V^-_k}(z) h^{\Delta_{j-1}\cap V^-_k}(z')\big]
\leq C_{\eps}\frac{|x-y|^2}{\rme^{2(n-p)}}\,,
\end{equation}
where we used that on the left-hand side, the covariance is nonnegative, and the sum over $z,z'$ is bounded by a constant as it is equal to $\sigma(x,y)$ in Proposition~\ref{p:Cov-bd}.

By Fernique majorization (see e.\,g.\ Theorem~4.1 in \cite{Adler}) with the uniform probability measure on $A_{p}$ as the majorizing measure, it follows that the second terms on the right-hand sides of~\eqref{e:dec-tail-r} and~\eqref{e:dec-tail-l} are uniformly upper bounded by a Gaussian tail, which completes the verification of~\ref{i.a3}.
\end{proof}

Next we treat the full decoration process $(\cD_p)_{p=1}^{\rt}$.
\begin{lem}
\label{l:gamma}
For each $\eps\in(0,1)$, there exists $\delta\in(0,\tfrac13)$ such that $(\cS_p)_{p=0}^\rt$, $(\cD_p)_{p=1}^\rt$ satisfies the bounds in~\ref{i.a3} with $a=\ol{u}(0)$, $b=\ol{v}(\infty)$ for all $\eta,\zeta \in[0,\eps^{-1}]$, $U\in\frU_\eps^\eta$, $V\in\frV_\eps$, $0\leq k<n$ with $T_{n-k}\geq 1$ and $u\in\bbR^{\partial U_n}$, $v\in\bbR^{\partial V^-_k}$ that satisfy~\eqref{e:uv-diff}, under the identification 
$\bfP = \bfP_{0,\ol{u}(0)}^{\rt,\ol{v}(\infty)} \equiv \bbP_{k,v}^{n,u}$.
\end{lem}
\begin{proof}
Thanks to Lemma~\ref{l:dec-tail}, it will turn out sufficient to bound $\gamma$ on $A_p$.
By~\eqref{e:corr-bd}, we have $\gamma(y)=\bbE_{k,v}^{n,u}h(y)+m_{n-p}-\beta_p$ for $y\in A_p$.
We now use the Definition~\eqref{e:bell} of $\beta_p$ and linearity in the following estimate:
\begin{equation}
\label{e:gamma2}
\bigg|\beta_p-
\frac{(\rt -p)\ol{u}(0)  + p\ol{v}(\infty) }{\rt}\bigg|
= \bigg|\frac{s_{1,p}}{s_{1,\rt}}-\frac{p}{\rt}\bigg|
\big|\ol{v}(\infty)-\ol{u}(0)\big| 
\leq C_{\eps} \big|\ol{v}(\infty)-\ol{u}(0)\big|\frac{\log \rt}{\rt}
\end{equation}
where we also used that
\begin{equation}
\label{e:s-int}
\frac{s_{1,p}}{s_{1,\rt}}\leq\frac{p+C_{\eps}(1+\log(1+ \rt -p))}{\rt-C_{\eps}}\leq \frac{p}{\rt}+C_{\eps}\frac{\log \rt}{\rt}\,,\qquad
\frac{s_{1,p}}{s_{1,\rt}}\geq \frac{p}{\rt}-C_{\eps}\frac{\log \rt}{\rt}
\end{equation}
by Lemma~\ref{l:varXm}.
Furthermore, by combining Lemmas~\ref{l:UV-U} and~\ref{l:ruin}\ref{i:ruin-C} (as in~\eqref{e:E0} in the proof of Proposition~\ref{l:E}) we obtain
\begin{equation}
\bigg|\bbE_{k,v}^{n,u} h(y) - \frac{ T-p  }{T}\ol{u}(0) - \frac{p}{T}\ol{v}(\infty) -\wh{m}_{n-p}\bigg|
\leq 2\osc\,\ol{u}_\eta+2\osc\,\ol{v}_\zeta+C_\eps+C_\eps\frac{|\ol{u}(0)|+|\ol{v}(\infty)|}{n-k}
\end{equation}
as $n-k-\rt$ and $\big|n-p-\log|y|\big|$ are bounded by a constant.
Combining the above estimates and bounding $|m_{n-p}-\wh m_{n-p}|\leq \log \wedge_{T,p}$ by Lemma~\ref{l:blml}, we obtain that
\begin{equation}
\label{e:bd-gamma}
|\gamma(y)|\leq 2\osc_{U^\eta_n}\ol{u}+2\osc_{V^{-,\zeta}_k}\ol{v}+\delta^{-1}+\delta^{-1}\wedge_{\rt,p}^\delta+\delta^{-1}\big|\ol{v}(\infty)-\ol{u}(0)\big|\frac{\log \rt}{\rt}
\end{equation}
for sufficiently small $\delta$. The assertion now follows from Lemma~\ref{l:dec-tail} as
\begin{equation}
\bbP_{k,v}^{n,u}\big(|\cD_p|>t+\delta^{-1}+\delta^{-1}\wedge_{\rt,p}^{1/2-\delta}\big)\leq
\bbP_{k,v}^{n,u}\big(|\wt \cD_p|>t\big)
\end{equation}
for sufficiently small $\delta>0$.
\end{proof}

Finally, we are ready for
\begin{proof}[Proof of Theorem~\ref{t:drw-i}]
The correspondence~\eqref{e:h-S} follows from the definitions in Section~\ref{s:def-DRW}.
Assumption~\ref{i.a1} is verified by~\eqref{e:CovS}, \eqref{e:defS}, \eqref{e:bell} and~\eqref{e:cS}.
For $\delta>0$ sufficiently small, depending only on $\eps$,
Lemma~\ref{l:varXm} shows that $\sigma_k\in(\delta,\delta^{-1})$ for $k=1,\ldots,\rt$. Assumptions~\ref{i.a2} and~\ref{i.a3} are verified by Lemmas~\ref{l:dep} and~\ref{l:gamma}.
\end{proof}

\subsection{Outward concentric decomposition and corresponding DRW}
\label{s:w-outw}
In this subsection we present another reduction from the DGFF ballot event to that involving a DRW, only that the DRW is defined using an outward concentric decomposition instead of the inward concentric decomposition which was used before. Here the conditioning is performed in the opposite order, so that scales are increasing. This results in a DRW $((\cS^{\rm o}_i)_{i=0}^\rt, (\cD^{\rm o}_i)_{i=1}^\rt)$ in which the starting and ending point are reversed: $\cS^{\rm o}_0 = \ol{v}(\infty)$, $\cS^{\rm o}_T = \ol{u}(0)$. The reduction is summarized in Theorem~\ref{t:drw-o} which is the analog of Theorem~\ref{t:drw-i}.

In analog to Subsection~\ref{s:concdec}, we assume that $T_{n-k}\geq 1$ and define
the concentric sets
\begin{multline}
\Delta^{\rm o}_{-1}=V^-_k\,,\quad \Delta^{\rm o}_0= V^{-,\zeta}_k\,,\quad
\Delta_p = \rmB^-_{k+\lceil \log(\eps^{-1}+\zeta)\rceil+p}\quad\text{for }p=1,\ldots,\rt -1\,,\quad
\Delta^{\rm o}_\rt = (U^{\eta}_n)^\rmc\,,\\
\Delta^{\rm o}_{\rt+1}=(U_n)^\rmc
\end{multline}
which are nested, $\Delta^{\rm o}_{-1}\supset \ldots \supset \Delta^{\rm o}_{\rt+1}$.
Furthermore, we define the sets
\begin{equation}
\Delta'^{\rm o}_{-1}=(V^-_k)^\rmc\,,\quad \Delta'^{\rm o}_0= (V^{\zeta,-}_k)^\rmc\,,\quad
\Delta'^{\rm o}_p = \rmB_{k+\lceil \log(\eps^{-1}+\zeta)\rceil + p}\quad\text{for }p=1,\ldots,\rt -1\,,\quad
\Delta'^{\rm o}_\rt = U^{\eta}_n
\end{equation}
For $p=0,\ldots,\rt$, let $h^{\rm o}_p$ be a DGFF on $\rA^{\rm o}_p:=\Delta'^{\rm o}_p\cap\Delta^{\rm o}_{p-1}$ with boundary values zero.
Moreover, we define $J^{\rm o}_p:=\rA^{\rm o}_p\cup(\Delta^{\rm o}_p\cap U_n)$ for $p=0,\ldots,\rt-1$, $J^{\rm o}_{\rt}=\rA^{\rm o}_\rt$, and let $\varphi^{\rm o}_{p}$ be distributed as $\heb{h''_p}{\bbZ^2 \setminus J^{\rm o}_p}$ where $h''_p$ is a DGFF on $\Delta^{\rm o}_{p-1}\cap U_n$ with boundary values $0$ (and, according to our notational conventions, equal to zero on $\bbZ^2\setminus(\Delta^{\rm o}_{p-1}\cap U_n)$).
We note that $\heb{h''_p}{\bbZ^2\setminus J^{\rm o}_p}$ is the binding field from $\Delta^{\rm o}_{p-1}\cap U_n$ to $J^{\rm o}_p$, as defined in the Gibbs-Markov decomposition (Lemma~\ref{l:GM}).
We assume that the random fields $\varphi^{\rm o}_0,\ldots,\varphi^{\rm o}_{\rt},h^{\rm o}_0,\ldots, h^{\rm o}_{\rt}$ are independent.

With the annuli $A^{\rm o}_p:=\Delta^{\rm o}_{p-1}\setminus\Delta^{\rm o}_{p}$, $p=0,\ldots,\rt+1$, which form a disjoint covering of $U_n\cap V^-_k$ and
satisfy $A^{\rm o}_p\supset \rA^{\rm o}_p$ for $p\leq T$, we have the following analog of Proposition~\ref{p:concdec-in}:
\begin{prop}[Outward concentric decomposition]
\label{p:concdec-out}
Assume that $\rt=\rt_{n-k}\geq 1$. There exists a coupling of $h^{U_n\cap V^-_k}$ and $(\varphi^{\rm o}_p,h^{\rm o}_p)_{p=0}^{\rt}$ such that
\begin{equation}
\label{e:concdec-out}
h^{U_n\cap V^-_k}(x)=\sum_{p=0}^{\rt}\big(\varphi^{\rm o}_p(x)+h^{\rm o}_p(x)\big)=\sum_{p=0}^{q\wedge\rt}\varphi^{\rm o}_p(x)+h^{\rm o}_q(x)1_{q\leq\rt}
\end{equation}
for $x\in A^{\rm o}_{q}$, $q=0,\ldots,\rt +1$.
\end{prop}
\begin{proof}
We proceed as in the proof of Proposition~\ref{p:concdec-in} by applying the Gibbs-Markov property first to $h^{\rm o}_0$ and $\rA^{\rm o}_0$.
\end{proof}

We represent the DGFF on $U_n\cap V^-_k$ analogously to Section~\ref{s:def-DRW}, now building on the outward concentric decomposition.
We assume that $h$ under $\bbP_{k,v}^{n,u}$ is coupled to $(\varphi^{\rm o}_p,h^{\rm o}_p)_{p=1}^{\rt}$ from Proposition~\ref{p:concdec-out} such that
\begin{equation}
\label{e:hphih-o}
h=\sum_{p=0}^{\rt+1}\big(\varphi_p^{\rm o}+h_p^{\rm o}\big)+
\he{(-m_n\Ind_{\partial U_n}+u -m_k\Ind_{\partial V^-_k} + v)}{\partial U_n\cup\partial V^-_k}\,.
\end{equation}

We define the harmonic average
\begin{equation}
\label{e:Xm-o}
\cX^{\rm o}_p=\sum_{z\in\partial A^{\rm o}_p \cup \partial \Delta_p^{\rm o}}\varphi_p^{\rm o}(z)\Pi_{\Delta_p^{\rm o}}(\infty,z)
=\heb{\varphi_p^{\rm o}}{\partial \Delta_p^{\rm o}}(\infty)
\end{equation}
and denote its variance by
$\sigma_p^{\rm o, 2}:=\Var \cX_p^{\rm o}$. Writing $s_{p,q}^{\rm o}=\sum_{i=p}^q \sigma_i^{\rm o,2}$, we set
\begin{equation}
\beta_p^{\rm o}=\frac{s_{1+p,\rt}^{\rm o}}{s_{1,\rt}^{\rm o}}\ol{v}(\infty)+\frac{s_{1,p}^{\rm o}}{s_{1,\rt}^{\rm o}}\ol{u}(0)\,,
\end{equation}
and we define $(\cS'^{\rm o}_p)_{p=0}^\rt$ from $(\cX_p^{\rm o})_{p=1}^\rt$ and independent Brownian bridges in the same way as $(\cS'_p)_{p=0}^\rt$ was defined from $(\cX_p)_{p=1}^\rt$ in~\eqref{e:deftS} and~\eqref{e:defS}.
Moreover, we define
\begin{equation}
\label{e:corr-bd-o}
\gamma^{\rm o}(y)=\he{(-m_{k}\Ind_{\partial V^-_k}+v)}{\partial U_n\cup \partial V^-_k}(y)+\he{(-m_n\Ind_{\partial U_n}+u)}{\partial U_n\cup \partial V^-_k}(y)
+m_{k+p}-\beta_p^{\rm o}\quad
\text{for }y\in A_{p}^{\rm o}\,.
\end{equation}	
We set
\begin{equation}
\label{e:Dell-o}
\cD_p^{\rm o}=\max_{y\in A_{p}^{\rm o}}\big\{h_p^{\rm o}(y)-m_{k+p}+\varphi_{0,p}^{\rm o}(y)-\cS'^{\rm o}_p
+\gamma^{\rm o}(y)\big\}
\end{equation}
and $\cS^{\rm o}_p=\cS'^{\rm o}_p +  \beta^{\rm o}_p$.
Then, in analogy to Theorem~\ref{t:drw-i}, we have:
\begin{thm}
\label{t:drw-o}
Let $\eps\in(0,1)$.
Then there exists $\delta\in(0,1/3)$ such that for all
$\eta,\zeta \in[0,\eps^{-1}]$, $U\in\frU^\eta_\eps$, $V\in\frV_\eps$, $0\leq n<k$ with $T_{n-k}\geq 1$,
$u\in\bbR^{\partial U_n}$, $v\in\bbR^{\partial V^-_k}$ that satisfy~\eqref{e:uv-diff}, we have that $(\cS^{\rm o}_i)_{i=0}^\rt$,
$(\cD^{\rm o}_i)_{i=1}^\rt$
satisfies Assumptions~\ref{i.a1} --~\ref{i.a3} with $\delta$, $a=\ol{v}(\infty)$ and $b=\ol{u}(0)$ under the identification 
$\bfP = \bfP_{0,\ol{v}(\infty)}^{\rt,\ol{u}(0)} \equiv \bbP_{k,v}^{n,u}$. Moreover,
\begin{equation}
\label{e:h-S-o}
\{\cS^{\rm o}_{p} + \cD_p^{\rm o}\leq 0\}
=\{h_{A^{\rm o}_p}\leq 0\big\}
\text{ for } p=1,\ldots,\rt\,,\quad\text{and}\quad
\bigcap_{p=1}^{\rt}\{\cS^{\rm o}_{p} + \cD_p^{\rm o}\leq 0\}
=
\big\{h_{U^\eta_n\cap V^{-,\zeta}_k}\leq 0\big\}\,.
\end{equation}
\end{thm}
When a DGFF $h^D$ on an infinite domain $D\subsetneq\bbZ^2$ with boundary values $u\in\bbR^{\partial D}$ arises in the setting of the outward concentric decomposition, we define it as the Gaussian field on $\ol D$ with mean and covariance given by \eqref{e:1.1a}, where $\ol{w}$ denotes the bounded harmonic extension of $w$ to $\bbZ^2$.
\begin{proof}[Proof of Theorem~\ref{t:drw-o}]
The correspondence~\eqref{e:h-S-o} again follows from the definitions above.
Assumptions~\ref{i.a1} --~\ref{i.a3} are verified along the lines of Section~\ref{s:concdec} in the same way as for Theorem~\ref{t:drw-i}.
We always use the objects defined in this subsection, e.\,g.\  $A_{p}$ in place of $A^{\rm o}_{p}$.
Equation~\eqref{e:VarXj} now reads
\begin{equation}
\Var_k^n \cX^{\rm o}_j=\Var\he{{h^{\Delta^{\rm o}_{j-1}}}}{\partial \Delta^{\rm o}_j}(\infty)
-\Var\hebb{\heb{h^{\Delta^{\rm o}_{j-1}}}{\partial \Delta^{\rm o}_{j-1}\cap U_n}}{\partial A^{\rm o}_{j}}(\infty)\,,
\end{equation}
where the second variance is bounded by~\eqref{e:Cov-bdg-U} by a constant times $(\rt -j+2)^{-1}$. We also interchange the use of Lemmas~\ref{l:Poisson} and~\ref{l:Poisson-out}, and we now use~\eqref{e:A1-o} and~\eqref{e:DingV-o} in Lemmas~\ref{l:DGFF-ut} and~\ref{l:Ding}, respectively. In the analog of Lemma~\ref{l:varXm}, it suffices to assume that $n-k\to\infty$ and $k\to\infty$ for the convergence of $\Var \cX_{p}^{\rm o}$ as this will only be used in Lemma~\ref{l:SDconv-o} below.
\end{proof}

\section{The ballot upper bound}
\label{s:proof-DGFF}
In this section we prove Theorem~\ref{t:2.3}. Following the strategy which was outlined in Subsection~\ref{ss:ProofOutline}, we first prove a weak version of the upper bound (Subsection~\ref{ss:weak}), which is a direct consequence of a similar bound for the DRW ballot probability. We then show (Subsection~\ref{s:ub-dec}) that, conditional on the ballot event, the harmonic extension of the values of the field at an intermediate scale is entropically repelled in the bulk with high probability and otherwise well behaved. This is then used (Subsection~\ref{ss:3.3}) to prove the full version of the theorem via ``stiching'' as explained in the proof outline. 
 
\subsection{A weak upper bound}
\label{ss:weak}
As a first consequence of the reduction to the DRW (Theorems~\ref{t:drw-i} and~\ref{t:drw-o}), one obtains the following weaker version of Theorem~\ref{t:2.3}, where the boundary values on either the inner or outer domains are sufficiently low. This follows from a corresponding weak upper bound on the ballot event for the DRW (Theorem~\ref{t:3.2}). We recall that $u_*$ and $v_*$ were defined in~\eqref{e:1.22}.

\begin{prop}
\label{p:1.3}
Let $\eps\in(0,1)$. There exist $C = C_{\eps} < \infty$ and $\theta=\theta_\eps\in(0,\eps/3)$, such that for all $\eta,\zeta \in[0,\eps^{-1}]$, $0\leq k<n$, $U\in\frU^\eta_\eps$, $V\in\frV_\eps$, and all $u\in\bbR^{\partial U_n}$, $v\in\bbR^{\partial V^-_k}$ with
$\max\{u_*^-,v_*^-\} \geq (n-k)^\eps$,
we have
\begin{equation}
\label{e:2.4-w}
\bbP_{k,v}^{n,u} \Big( h_{(\theta\rmB)_n\cap(\theta^{-1}\rmB^-)_k} \leq 0 \Big) \leq 
C \frac{\big(u_*^- + 1  \big)
\big(v_*^- + 1 \big)}{n-k}\,.
\end{equation}
\end{prop}
\begin{rem}
\label{r:p:1.3alt}
We note that $U\in\frU^\eta_\eps$, $V\in\frV_\eps$, $\zeta\in[0,\eps^{-1}]$ and $\theta<\eps/3$ imply that $(\theta\rmB)_n\cap(\theta^{-1}\rmB^-)_k \subset U^\eta_n\cap V^{-,\zeta}_k$. Therefore, the right hand side in~\eqref{e:2.4-w} is also an upper bound for $\bbP_{k,v}^{n,u} (h_{U_{n}^{\eta} \cap V^{-,\zeta}_{k}} \leq 0)$ under the conditions of the proposition. Also, as we show in the proof below, the following alternate formulations of Proposition~\ref{p:1.3} hold under the same assumptions but without the restriction on $ \max\{u_*^-,v_*^-\}$:
\begin{equation}
\label{e:2.4-w-av}
\bbP_{k,v}^{n,u} \Big(   h_{(\theta\rmB)_n\cap(\theta^{-1}\rmB^-)_k}  \leq 0 \Big) \leq 
C \frac{\big(u_*^- + 1  \big)
\big(v_*^- + (n-k)^\eps \big)}{n-k}\,,
\end{equation}
\begin{equation}
\label{e:2.4-w-au}
\bbP_{k,v}^{n,u} \Big(  h_{(\theta\rmB)_n\cap(\theta^{-1}\rmB^-)_k} \leq 0 \Big) \leq 
C \frac{\big(u_*^- + (n-k)^\eps  \big)
\big(v_*^- + 1 \big)}{n-k}\,.
\end{equation}
\end{rem}

The following lemma conveniently reduces the general case to that involving constant boundary conditions. It will be used in various places in the sequel, including in the proof of Proposition~\ref{p:1.3}. 
\begin{lem}
\label{l:UVmin}
Let $\eps\in(0,1)$. There exist $C=C_\eps<\infty$ and $\theta=\theta_\eps\in(0,\eps/3)$ such that for all $\eta,\zeta\in[0,\eps^{-1}]$, $0\leq k \leq n-C$, $U\in\frU^\eta_\eps$, $V\in\frV_\eps$, $u\in\bbR^{\partial U_n}$, $v\in\bbR^{\partial V^-_k}$, we have
\begin{equation}
\bbP_{V,k,v}^{U,n,u}\big(h_{ (\theta\rmB)_n\cap(\theta^{-1}\rmB^-)_k}\leq 0\big)\leq
\bbP_{V,k,v_{*}}^{U,n,u_{*}}\big(h_{ (\theta\rmB)_n\cap(\theta^{-1}\rmB^-)_k}\leq 0\big)\,.
\end{equation}
\end{lem}
\begin{proof}
This follows from Lemmas~\ref{l:UV-U} and~\ref{l:osc-far} as $h$ under $\bbP_{k,v}^{n,u}$ is distributed as $h+\he{(u+v)}{\partial U_n\cup\partial V^-_k}$ under $\bbP_{k,0}^{n,0}$.
\end{proof}

\begin{proof}[Proof of Proposition~\ref{p:1.3}]
By Lemma~\ref{l:UVmin}, there exists $\theta<\eps/3$ with
\begin{equation}
\label{e:ub-mon}
\bbP_{k,v}^{n,u} \Big( h_{U^\eta_n \cap V^{-,\zeta}_k} \leq 0 \Big)\leq
\bbP_{k,v_{*}}^{n,u_{*}} \Big( h_{ (\theta\rmB)_n\cap(\theta^{-1}\rmB^-)_k} \leq 0 \Big)\,,
\end{equation}
whenever $n-k$ is sufficiently large, which we can assume.
We can also assume w.\,l.\,o.\,g.\ that $\max\{u_{*},v_{*}\}\leq 0$.

In case $|u_{*} -  v_{*}| \leq (n-k)^{1-\eps/2}$ and $-v_{*}>(n-k)^\eps$, Theorem~\ref{t:drw-i} is applicable with $\eps/2$ in place of $\eps$, $\delta = \delta(\epsilon/2) \in (0,1/3)$, $a=u_{*}$, $b=v_{*}$, so that~\eqref{e:2.4-w} and~\eqref{e:2.4-w-av} follow from the upper bound in Theorem~\ref{t:3.2}
applied to
\begin{equation}
\bfP_{0,u_{*}}^{\rt, v_{*}}\Big(\bigcap_{p=p_*}^{\rt-p_*}\{\cS_p+\cD_p\leq 0\}\Big)\,.
\end{equation}
Here we choose $p_*$ such that $\Delta_{p_*-1}\setminus\Delta_{T-p_*}\subset  (\theta\rmB)_n\cap(\theta^{-1}\rmB^-)_k$ and assume that $n-k$ is sufficiently large such that $T\geq 2p_*+1$.
In case $|u_{*}-v_{*}|>(n-k)^{1-\eps/2}$ and $- v_{*} >(n-k)^\eps$, the right-hand sides in~\eqref{e:2.4-w} and~\eqref{e:2.4-w-av} are bounded from below by $C(n-k)^{\eps/2}$, which is further bounded from below by $1$ for $n-k\geq 1$.

The corresponding cases with $-u_{*}>(n-k)^\eps$, as well as~\eqref{e:2.4-w-au}, are handled in the same way by using Theorem~\ref{t:drw-o} in place of Theorem~\ref{t:drw-i}.
\end{proof}

\subsection{Entropic repulsion at an intermediate scale}
\label{s:ub-dec}

In order to derive Theorem~\ref{t:2.3} we will condition on the values of the DGFF $h$ at the boundary of a centered ball at an intermediate scale and use the weak ballot estimate from the previous subsection on each of the resulting inner and outer domains. For the latter to be in effect, we need to control the harmonic extension of the values of the boundary of the intermediate ball. In particular, we need to show that it is typically (entropically) repelled below 0 and that its oscillations are not too large. 

To this end, we set for $l,M\geq 0$, $\eps>0$,
\begin{equation}
\label{e:4.2}
E_{ l,M,\eps}:=\Big\{w\in\bbR^{\partial \rmB^\pm_l}:\:
- \wedge_{n,l,k}^{1-\eps} \leq \ol{w}(0) \leq - \wedge_{n,l,k}^{\eps},\,
\osc_{\rmB^{\pm,\eps}_l}\ol{w} \leq M\Big\}\,,
\end{equation}
where we set $\wedge_{n,l,k}:=(n-l)\wedge(l-k)$. We then show,
\begin{prop}
\label{l:stitch-restr}
Let $\eps\in(0,\tfrac{1}{10})$.
Then there exist $C=C_{\eps}<\infty$ and $c=c_{\eps}>0$ such that for all $M\geq 0$, $\eta,\zeta \in[0,\eps^{-1}]$, $U\in\frU^\eta_\eps$, $V\in\frV_\eps$, all $0\leq k < l < n$ with
$l\in k+ (n-k)[\eps,1-\eps]$,
and all
$u\in\bbR^{\partial U_n}$, $v\in\bbR^{\partial V^-_k}$ with $\max\{|\ol{u}(0)|,|\ol{v}(\infty)|,\osc_{U^\eta_n}\ol{u},\osc_{V^{-,\zeta}_k}\ol{v}\}\leq \wedge_{n,l,k}^{\eps}$,
and $\theta=\theta_\eps\in(0,\eps/3)$ as in Lemma~\ref{l:UVmin},
we have
\begin{equation}
\label{e:stitch-restr}
\bbP_{V,k,v}^{U,n,u}\big(h_{ (\theta\rmB)_n\cap(\theta^{-1}\rmB^-)_k}\leq 0,
h_{\partial \rmB^\pm_l}+m_l\notin E_{l,M,\eps}\big)
\leq C\frac{\big(u_{*}
+1\big)
\big(v_{*}^- +1\big)}{n-k}
\Big(\wedge_{n,l,k}^{-3/2+4\eps}+\rme^{-cM^2}\Big)\,.
\end{equation}
\end{prop}

In order to prove the proposition, we decompose the complement of the set $E_{l,M,\eps}$ as
$E_{ l,M,\eps}^\rmc=E_1\cup E_2\cup E_3 \cup E_4$, where:
\begin{multline}
\label{e:E1234}
E_1=\big\{w\in\bbR^{\partial \rmB^\pm_l}:\: |\ol{w}(0)|> \wedge_{n,l,k}^{1-\eps} \big\}\,,\quad
E_2=\big\{w\in\bbR^{\partial \rmB^{\pm}_l}:\:-\wedge_{n, l,k}^{\eps}
< \ol{w}(0)\leq 
\wedge_{n, l,k}^{2\eps}\big\}\,,\\
E_3=\big\{w\in\bbR^{\partial \rmB^{\pm}_l}:\:\osc_{\rmB^{\pm,\eps}_l}\ol{w}>M\big\}\cap E_1^\rmc\cap E_2^\rmc\,,\quad
E_4=\big\{w\in\bbR^{\partial \rmB^{\pm}_l}:\:
\wedge_{n,l,k}^{2\eps}<\ol{w}(0)\leq \wedge_{n,l,k}^{1-\eps}\big\}
\cap E_3^\rmc\,,
\end{multline}
and treat each of the $E_i$-s separately. $E_1$ is handled by,
\begin{lem}
\label{l:sr-}
Let $\eps\in\big(0,\tfrac{1}{10}\big)$ and $\eta,\zeta\in[0,\eps^{-1}]$.
Then there exist $C=C_{\eps}<\infty$ and $c=c_\eps>0$ such that
\begin{equation}
\label{e:sr-}
\bbP_{k,v}^{n,u}\Big(
\he{h}{\partial \rmB^{\pm}_l}+m_l\in E_1
\Big)
\leq C \rme^{- c\wedge_{n,l,k}^{1-2\eps}}
\end{equation}
and
\begin{equation}
\label{e:sr-E}
\bbE_{k,v}^{n,u}\Big(\big(\he{h}{\partial \rmB^{\pm}_l}(0)+m_l\big)^-
+\osc_{\rmB^{\pm,\eps}_l}\he{h}{\partial \rmB^{\pm}_l}\,;
\he{h}{\partial \rmB^{\pm}_l}+m_l\in E_1
\Big)
\leq C \rme^{- c\wedge_{n,l,k}^{1-2\eps}}
\end{equation}
for all $U\in\frU^\eta_\eps$, $V\in\frV_\eps$, all $k,l,n\geq 0$ with
$\partial\rmB_l\subset U^{\eta\vee\eps}_n$ and $\partial\rmB^{-}_l\subset V^{-,\zeta\vee\eps}_k$, and all
$u\in\bbR^{\partial U_n}$, $v\in\bbR^{\partial V^-_k}$ with $\max\{|\ol{u}(0)|,|\ol{v}(\infty)|,\osc_{U^\eta_n}\ol{u},\osc_{V^{-,\zeta}_k}\ol{v}\}\leq \wedge_{n,l,k}^{1-2\eps}$.
\end{lem}
\begin{proof}
The random variable $\he{h}{\partial \rmB^{\pm}_l}(0)+m_l$ under $\bbP_{k,v}^{n,u}$ is Gaussian. By Lemma~\ref{l:blml}, and Propositions~\ref{l:E} and~\ref{p:Cov-bd}, its mean is bounded (in absolute value) by $C_{\eps}+C_{\eps}\wedge_{n,l,k}^{1-2\eps}$, and its variance is bounded from above by $g\wedge_{n,l,k}+C_\eps$.
Hence, the left-hand side of~\eqref{e:sr-} is bounded by a constant times
\begin{equation}
\wedge_{n,l,k}^{-1/2}\int_{|r|\geq \wedge_{n,l,k}^{1-\eps}}
\rme^{-c_{\eps}\frac{r^2}{\wedge_{n,l,k}}}\rmd r
= C_{\eps}\int_{|r|\geq \wedge_{n,l,k}^{1/2-\eps}}
\rme^{-c_\eps r^2}\rmd r
\leq C_{\eps}\rme^{-c_\eps\wedge_{n,l,k}^{1-2\eps}}
\end{equation}
where the equality follows by changing the variable $r\wedge_{n,l,k}^{-1/2}$ to $r$.
The left-hand side of~\eqref{e:sr-E} is bounded by a constant times
\begin{multline}
\wedge_{n,l,k}^{-1/2}\int_{|r|\geq \wedge_{n,l,k}^{1-\eps}}
\rme^{-c_{\eps}\frac{r^2}{\wedge_{n,l,k}}}
\sum_{a=1}^\infty \bbP_{k,v}^{n,u}\Big(\osc_{\rmB^{\pm,\eps}_l}
\he{h}{\partial\rmB^\pm_l}\geq a-1\,\Big|\,
\he{h}{\partial\rmB^\pm_l}(0)+m_l=r\Big) (r^- + a)\,\rmd r\\
\leq C_\eps\wedge_{n,l,k}^{-1/2}\int_{|r|\geq \wedge_{n,l,k}^{1-\eps}}
\rme^{-c_{\eps}\frac{r^2}{\wedge_{n,l,k}}}
\sum_{a=1}^\infty
\rme^{-c_\eps\big[\big( a - \frac{|r|}{\wedge_{n,l,k}}\big)^+\big]^2} (r^- + a)\,\rmd r\\
\leq C_{\eps}\int_{|r|\geq \wedge_{n,l,k}^{1/2-\eps}}
\rme^{-c_\eps r^2}\rme^{-a+|r|}(\sqrt{\wedge_{n,l,k}}r^- + a)\rmd r
\leq C_{\eps}\rme^{-c_\eps\wedge_{n,l,k}^{1-2\eps}}
\end{multline}

\end{proof}

Next we treat $E_2$.
\begin{lem}
\label{l:sr+}
Let $\eps\in\big(0,\tfrac{1}{10}\big)$ and $\eta,\zeta\in [0,\eps^{-1}]$.
Then there exists $C=C_{\eps}<\infty$ such that
\begin{equation}
\label{e:sr+}
\bbE_{k,v}^{n,u}\Big[\Big(\big(\he{h}{\partial \rmB^{\pm}_l}(0)+m_l\big)^- +\osc_{\rmB^{\pm,\eps}_l}\he{h}{\partial \rmB^{\pm}_l}
+\wedge_{n, l,k}^\eps\Big)^2\,;
h_{\partial \rmB^{\pm}_l}+m_l\in E_2
\Big]
\leq C \wedge_{n, l,k}^{-1/2+4\eps}
\end{equation}
for all $U\in\frU^\eta_\eps$, $V\in\frV_\eps$, all $0\leq k<l<n$ with
$\partial \rmB_l\subset U^{\eta\vee\eps}_n$ and $\partial
\rmB^{-}_l\subset V^{-,\zeta\vee\eps}_k$, and all
$u\in\bbR^{\partial U_n}$, $v\in\bbR^{\partial V^-_k}$ with $\max\{|\ol{u}(0)|,|\ol{v}(\infty)|,\osc_{U^\eta_n}\ol{u},\osc_{V^{-,\zeta}_k}\ol{v}\}\leq \eps^{-1}\wedge_{n,l,k}$.
\end{lem}
\begin{proof}
The left-hand side of~\eqref{e:sr+} is bounded from above by
\begin{multline}
\int_{r\in\big[-\wedge_{n, l,k}^{\eps},\wedge_{n, l,k}^{2\eps}\big]}\bbP_{k,v}^{n,u}\big(\he{h}{\partial \rmB^{\pm}_l}(0)+m_l\in\rmd r\big)\\
\times\sum_{{a}=0}^\infty\bbP_{k,v}^{n,u}\big(
\osc_{\rmB^{\pm,\eps}_l}\he{h}{\partial \rmB^{\pm}_l}\geq {a}
\,\big|\, \he{h}{\partial \rmB^{\pm}_l}(0)+m_l=r\big)\big[\wedge_{n,l,k}^\eps+r^- +a\big]^2\,.
\end{multline}
By Proposition~\ref{p:Cov-bd}, we bound the Gaussian density by a constant times $\wedge_{n, l,k}^{-1/2}$, and we plug in the conditional tail probability from Proposition~\ref{p:2.6}. Then the previous expression is bounded from above by a constant times
\begin{equation}
\label{e:sr-1}
\int_{r\in\big[-\wedge_{n, l,k}^{\eps},\wedge_{n, l,k}^{2\eps}\big]} 
\sum_{{a}=0}^\infty\rme^{-c_\eps{a}^2}
\big[\wedge_{n,l,k}^\eps+r^- +{a}\big]^2
\wedge_{n, l,k}^{-1/2}\,\rmd r\,.
\end{equation}
Bounding $r^-$ by $\wedge_{n,l,k}^\eps$ yields the assertion.
\end{proof}

For $E_3$ we have,
\begin{lem}
\label{l:sr-osc}
Let $\eps\in\big(0,\tfrac{1}{10}\big)$ and $\eta,\zeta\in[0,\eps^{-1}]$.
Then there exist $C=C_{\eps}<\infty$, $c=c_{\eps}>0$ such that
\begin{multline}
\label{e:sr-osc}
\bbE_{k,v}^{n,u}\Big[\Big(\big(\he{h}{\partial \rmB^{\pm}_l}(0)+m_l\big)^- +\osc_{\rmB^{\pm,\eps}_l}\he{h}{\partial \rmB^{\pm}_l}
+ \wedge_{n,l,k}^\eps \Big)^2
;\,h_{\partial \rmB^{\pm}_l}+m_l\in E_3 \Big]\\
\leq C\rme^{-cM^2}\bigg(1+\frac{(n- l)( l-k)}{n-k}+
\big(\|\ol{u}\|_{\bbL_\infty(U^\eta_n)}+
\|\ol{v}\|_{\bbL_\infty(V^{-,\zeta}_k)}\big)^2\bigg)
\end{multline}
and
\begin{multline}
\label{e:sr-osc2}
\bbE_{k,v}^{n,u}\Big[\big(\he{h}{\partial \rmB^{\pm}_l}(0)+m_l\big)^- +\osc_{\rmB^{\pm,\eps}_l}\he{h}{\partial \rmB^{\pm}_l}
+ 1
\,;\,\big| h_{\partial \rmB^{\pm}_l}(0)+m_l\big|\leq \wedge_{n,l,k},\,
\osc_{\rmB^{\pm,\eps}_l} h_{\partial \rmB^{\pm}_l} \geq M\Big]\\
\leq C\rme^{-cM^2}\bigg(1+\frac{(n- l)( l-k)}{n-k}+
\Big(\frac{l-k}{n-k}\ol{u}(0)+\frac{n-l}{n-k}\ol{v}(\infty)
+\osc_{U^\eta_n}\ol{u}+\osc_{V^{-,\zeta}_k}\ol{v}\Big)^2\bigg)
\end{multline}
for all $M\geq 0$, $U\in\frU^\eta_\eps$, $V\in\frV_\eps$ all $0 \leq k<l<n$ with
$\partial\rmB_l\subset U^{\eta\vee\eps}_n$ and $\partial\rmB^{-}_l\subset V^{-,\zeta\vee\eps}_k$, and all
$u\in\bbR^{\partial U_n}$, $v\in\bbR^{\partial V^-_k}$ with $\max\{|\ol{u}(0)|,|\ol{v}(\infty)|,\osc_{U^\eta_n}\ol{u},\osc_{V^{-,\zeta}_k}\ol{v}\}\leq\eps^{-1}\wedge_{n,l,k}$.
\end{lem}
\begin{proof}
The left-hand side in assertion~\eqref{e:sr-osc} is bounded from above by a constant times
\begin{multline}
\label{e:sr-osc-p1}
\sum_{{a}=\lfloor M\rfloor }^\infty
\int_{\wedge_{n,l,k}^\eps \leq |r| \leq \wedge_{n,l,k}^{1-\eps}}
\bbP_{k,v}^{n,u}\big(\he{h}{\partial \rmB^{\pm}_l}(0)+m_l\in \rmd r\big)\\
\times \bbP_{k,v}^{n,u}\big(\osc_{\rmB^{\pm,\eps}_l}\he{h}{\partial \rmB^\pm_l}>a\,\big|\, \he{h}{\partial \rmB^{\pm}_l}(0)+m_l = r\big)(1+r^2)(1+a^2)\,.
\end{multline}
The second probability in the integrand is bounded from above by $\rme^{-2c_\eps a^2}$ for a constant $c_\eps>0$ by Proposition~\ref{p:2.6}, so that~\eqref{e:sr-osc-p1} is further bounded from above by a constant times
\begin{equation}
\rme^{-cM^2}\Big(1+\bbE_{k,v}^{n,u}\Big(\big(\he{h}{\partial \rmB^{\pm}_l}(0)+m_l\big)^2\Big)\,.
\end{equation}
Bounding the second moment by Lemma~\ref{l:blml} and Propositions~\ref{l:E} and~\ref{p:Cov-bd} yields~\eqref{e:sr-osc}.
The proof of~\eqref{e:sr-osc2} is analogous.
\end{proof}

Finally for $E_4$, we have
\begin{lem}
\label{l:srpos}
Let $\eps\in\big(0,\tfrac{1}{10}\big)$, and $\eta,\zeta\in[0,\eps^{-1}]$.
Then there exists $C=C_{\eps}<\infty$ such that
\begin{equation}
\label{e:srpos}
\bbE_{k,v}^{n,u}\Big[\exp\Big\{- \Big(
\big(\he{h}{\partial \rmB^{\pm}_l}(0)+m_l-2\osc_{\rmB^{\pm,\eps}_l}\he{h}{\partial \rmB^{\pm}_l}\big)^+\Big)^{2-\eps}\Big\}\,;
h_{\partial \rmB^{\pm}_l}+m_l\in E_4 \Big]
\leq C \rme^{-\wedge_{n, l,k}^\eps}
\end{equation}
for all $M,U,V,k,l,n,u,v$ as in Lemma~\ref{l:sr-osc}.
\end{lem}
\begin{proof}
The left-hand side of~\eqref{e:srpos} is bounded from above by a constant times
\begin{multline}
\label{e:st-mu}
\int_{r\in\big[\wedge_{n, l,k}^{2\eps},\wedge_{n,l,k}^{1-\eps}\big]}
\bbP_{k,v}^{n,u}\big(\he{h}{\partial \rmB^{\pm}_l}(0)+m_l\in\rmd r\big)
\sum_{{a}=0}^\infty\bbP_{k,v}^{n,u}\big(
\osc_{\rmB^{\pm,\eps}_l}\he{h}{\partial \rmB^{\pm}_l}\geq{a}
\,\big|\, \he{h}{\partial \rmB^{\pm}_l}(0)=-m_l+r\big)\\
\times\rme^{-[(r - 2  a)^+]^{2-\eps}}\,.
\end{multline}
By Proposition~\ref{p:Cov-bd}, we bound the Gaussian density by a constant, and we plug in the bound from Proposition~\ref{p:2.6} for the second probability in the integrand so that the previous expression is bounded by a constant times
\begin{equation}
\int_{r\in\big[\wedge_{n, l,k}^{2\eps},\wedge_{n,l,k}^{1-\eps}\big]}
 \sum_{a=0}^\infty\rme^{-c_\eps a^2}
\rme^{-r + 2a}\rmd r\,.
\end{equation}
\end{proof}

As a last ingredient in the proof of Proposition~\ref{l:stitch-restr}, we need an upper bound on the probability that the DGFF stays negative in the bulk despite positive boundary values. This will be used also later for the stronger statement in Theorem~\ref{t:2.3}.
\begin{lem}
\label{l:bdr-dec}
Let $\eps\in(0,\tfrac{1}{10})$ and $\eta,\zeta\in [0,\eps^{-1}]$. Then there exists $C=C_{\eps}<\infty$ such that, for $\theta$ as in Lemma~\ref{l:UVmin}, we have
\begin{equation}
\label{e:bdr-dec-u}
\bbP_{k,v}^{n,u}\big(h_{ (\theta\rmB)_n\cap(\theta^{-1}\rmB^-)_k}\leq 0\big)
\leq C\exp\Big\{-\big(u_{*}^+\big)^{2-\eps}1_{v_{*}\geq -(n-k)} \Big\}
\end{equation}
and
\begin{equation}
\label{e:bdr-dec-v}
\bbP_{k,v}^{n,u}\big(h_{ (\theta\rmB)_n\cap(\theta^{-1}\rmB^-)_k}\leq 0\big)\leq
C\exp\Big\{ -\big(v_{*}^+\big)^{2-\eps}1_{u_{*}\geq -(n-k)} \Big\}
\end{equation}
for all $U\in\frU^\eta_\eps$, $V\in\frV_\eps$, $0\leq k<n-C$,
and all
$u\in\bbR^{\partial U_n}$, $v\in\bbR^{\partial V^-_k}$.
\end{lem}
\begin{proof}
We show~\eqref{e:bdr-dec-u}, the proof of~\eqref{e:bdr-dec-v} is analogous.
By Lemma~\ref{l:UVmin},
$\bbP_{k,v}^{n,u}\big(h_{ (\theta\rmB)_n\cap(\theta^{-1}\rmB^-)_k}\leq 0\big)
\leq\bbP_{k,v_{*}\wedge 0}^{n,u_{*}}\big(h_{ (\theta\rmB)_n\cap(\theta^{-1}\rmB^-)_k}\leq 0\big)$ for all sufficiently large $n-k$. Moreover, $h$ under $\bbP_{k,v_{*}\wedge 0}^{n,u_{*}}$ is distributed as the field
\begin{equation}
	h^{U_n\cap V^-_k}(x)+(-m_n+u_{*})P_x\big(\tau^{U_n}\leq\tau^{V^-_k}\big)
	+(-m_k+v_{*} \wedge 0)P_x\big(\tau^{U_n}>\tau^{V^-_k}\big)\,,\quad x\in U_n\cap V^-_k\,.
\end{equation}
W.\,l.\,o.\,g.\ we assume that $u_{*} \geq 0$ and
$v_{*}\leq 0$.
For $x\in A:= (\theta\rmB)_n\cap(\tfrac{\theta}{2}\rmB^-)_n$, we estimate
\begin{multline}
	(-m_n+u_{*})P_x\big(\tau^{U_n}\leq\tau^{V^-_k}\big)
	+(-m_k+v_{*})P_x\big(\tau^{U_n}>\tau^{V^-_k}\big)\\
	=-m_n+ \big(1-P_x\big(\tau^{U_n}>\tau^{V^-_k}\big)\big)u_{*} + P_x\big(\tau^{U_n}>\tau^{V^-_k}\big)
	\big(m_n - m_k  + v_{*}\big)\\
	\geq - m_n + c_\eps u_{*} - C_\eps + C_\eps(n-k)^{-1}v_{*}\,,
\end{multline}
for some $c_\eps>0$ and $C_\eps<\infty$, where we used for the inequality the definition~\eqref{e:ml} of $m_n, m_k$, that
$P_x\big(\tau^{U_n}>\tau^{V^-_k}\big)\leq C_\eps(n-k)^{-1}$
by Lemma~\ref{l:ruin}, in particular that
$1-P_x\big(\tau^{U_n}>\tau^{V^-_k}\big)\geq c_\eps$ for sufficiently large $n-k$.
Then,
\begin{equation}
\bbP_{k,v}^{n,u}\big(h_{A}\leq 0\big)
\leq \bbP\big( h^{U_n\cap V^-_k}_{A} - m_n + c_\eps u_{*} - C_\eps + C_\eps(n-k)^{-1}v_{*} \leq 0\big)\,.
\end{equation}
The right-hand side is bounded by Lemma~\ref{l:Ding} for sufficiently large $n-k$, and the left-hand side bounds $\bbP_{k,v}^{n,u}\big(h_{ (\theta\rmB)_n\cap(\theta^{-1}\rmB^-)_k}\leq 0\big)$, so that~\eqref{e:bdr-dec-u} follows.
\end{proof}

We can now prove Proposition~\ref{l:stitch-restr}.
\begin{proof}[Proof of Proposition~\ref{l:stitch-restr}]
We assume w.\,l.\,o.\,g.\ that $n-k$ is sufficiently large (depending only on $\eps$) such that the assumptions of Lemmas~\ref{l:sr-} --~\ref{l:bdr-dec} are satisfied in their applications below. We can choose $C$ sufficiently large such that the right-hand side of~\eqref{e:stitch-restr} is larger than $1$ when $n-k$ is not as large.

By a union bound,
\begin{multline}
\label{e:st-restr-union}
\bbP_{k,v}^{n,u}\big(h_{  (\theta\rmB)_n\cap(\theta^{-1}\rmB^-)_k  }\leq 0,
h_{\partial \rmB^\pm_l}+m_l\notin E_{l,M,\eps}\big)\\
\leq \bbP_{k,v}^{n,u}(h_{\partial \rmB^\pm_l}+m_l\in E_1)
+\sum_{i=2}^4\bbP_{k,v}^{n,u}(h_{  (\theta\rmB)_n\cap(\theta^{-1}\rmB^-)_k }\leq 0,h_{\partial \rmB^\pm_l}+m_l \in E_i)\,.
\end{multline}
The first term on the right-hand side is, by Lemma~\ref{l:sr-} and by the assumptions on $u, v, l$, bounded to be of smaller order than the right-hand side of~\eqref{e:stitch-restr}.
Using the Gibbs-Markov property (Lemma~\ref{l:GM}) and Remark~\ref{r:bd-pm}, 
we bound each of the summands with $i=2,3,4$ in~\eqref{e:st-restr-union} by
\begin{multline}
\label{e:st-disint}
\bbP_{V,k,v}^{U,n,u}(h_{(\theta\rmB)_n \cap \rmB^{-,\eps}_l}\leq 0,
h_{\rmB^\eps_l\cap (\theta^{-1}\rmB^-)_k}\leq 0,
h_{\partial \rmB^\pm_l}+m_l \in E_i)
\\
=
\int_{E_i} \bbP_{\rmB, l,w}^{U,n,u} \big( h_{\rmB_{(\theta\rmB)_n \cap \rmB_l^{-, \eps}}} \leq 0 \big)
\bbP_{V,k,v}^{\rmB,  l,w} \big(h_{\rmB^\eps_l \cap (\theta^{-1}\rmB^-)_k} \leq 0 \big) 
\bbP_{V,k,v}^{U,n,u} \big( h_{\partial \rmB^{\pm}_l} + m_l \in \rmd w \big).
\end{multline}
We bound the factors in the integrand on the right-hand side by Proposition~\ref{p:1.3} and Remark~\ref{r:p:1.3alt} so that the integral is bounded by a constant times
\begin{multline}
\label{e:st-disint2}
\frac{\big(u_*^- +1\big)
\big(v_*^- +1\big)}{(n- l)( l-k)}
\int_{E_i} \bigg(\ol{w}(0)^- +\osc_{\rmB^{\pm,\eps}_l}\ol w +(l-k)^{2\eps}\bigg)\\
\times
\bigg(\ol{w}(0)^- + \osc_{\rmB^{\pm,\eps}_l}\ol w +(n-l)^{2\eps}\bigg)
\bbP_{V,k,v}^{U,n,u} \big( h_{\partial \rmB^{\pm}_l} + m_l \in \rmd w \big)\,.
\end{multline}
For $i=2$, the integral in the last display is bounded by
$C_{\eps}\wedge_{n,l,k}^{-1/2+4\eps}$
by the assumption on $l$ and Lemma~\ref{l:sr+}. For $i=3$, the integral in~\eqref{e:st-disint2} is bounded by a constant times the right-hand side of~\eqref{e:sr-osc} by the first part of Lemma~\ref{l:sr-osc}. To bound the summand with $i=4$ on the right-hand side of~\eqref{e:st-restr-union}, we bound the first factor in the integrand in~\eqref{e:st-disint} by $1$ and the second factor
$\bbP_{V,k,v}^{\rmB,l,w}(h_{\rmB^\eps_l\cap(\theta^{-1}\rmB^-)_k}\leq 0)$
by~\eqref{e:bdr-dec-u}, so as to obtain
\begin{multline}
\label{e:st-disint-p99}
\bbP_{k,v}^{n,u}\big(h_{(\theta\rmB)_n\cap(\theta^{-1}\rmB^-)_k}\leq 0,
h_{\partial \rmB^\pm_l}+m_l\in E_4\big)\\
\leq 
C_\eps \int_{E_4}
\rme^{-\big( \big(\ol{w}(0)-2\osc_{\rmB^{\pm,\eps}_l}\ol w\big)^+  \big)^{2-\eps}}
\bbP_{V,k,v}^{U,n,u} \big( h_{\partial \rmB^{\pm}_l} + m_l \in \rmd w \big)
\end{multline}
which is bounded by a constant times $\rme^{-\wedge_{n,l,k}^\eps}$ by Lemma~\ref{l:srpos}.
Hence, the right-hand side of~\eqref{e:st-restr-union} is bounded by the right-hand side of~\eqref{e:stitch-restr}.
\end{proof}

\subsection{Stitching - Proof of Theorem~\ref{t:2.3}}
\label{ss:3.3}
From the previous subsection, we have an upper bound on the probability that the DGFF takes unlikely values at the boundary of an intermediately scaled ball, so that we can apply the weak upper bound from Proposition~\ref{p:1.3} at the two subdomains that are separated by this boundary.
We are therefore ready for,
\begin{proof}[Proof of Theorem~\ref{t:2.3}]
As noted in Section~\ref{s:bno}, we can assume w.\,l.\,o.\,g.\ that $\eps\in(0,1/10)$.
We first show the weaker statement, namely~\eqref{e:1.18}.
We set $l=k+\tfrac12(n-k)$ and take $n-k$ sufficiently large such that Lemma~\ref{l:UVmin} is applicable, and such that
$\partial \rmB_l\subset(\theta\rmB)_n$ and
$\partial \rmB^-_l\subset (\theta^{-1}\rmB^-)_k$, where we take $\theta\in(0,\eps/3)$ from Lemma~\ref{l:UVmin}.
(We can always choose $C$ sufficiently large such that the right-hand side of~\eqref{e:1.18} is larger than $1$ if $n-k$ is not as large.)
By Lemma~\ref{l:UVmin}, it suffices to bound $\bbP_{k,v_{*}}^{n,u_{*}}(h_{ (\theta\rmB)_n\cap(\theta^{-1}\rmB^-)_k}\leq 0)$ by the right-hand side of~\eqref{e:1.18}.  W.\,l.\,o.\,g.\ we assume that $\max\{u_{*}, v_{*}\}\leq 0$.
We can also suppose that $\min\{u_{*},v_{*}\}\geq -(n-k)^\eps$, as in the converse case, the assertion follows directly from Proposition~\ref{p:1.3}. 

Clearly,
\begin{equation}
\label{e:DGFF-ub-int-viol}
\bbP_{V,k,v_{*}}^{U,n,u_{*}}\big(h_{ (\theta\rmB)_n\cap(\theta^{-1}\rmB^-)_k}\leq 0\big)
\leq \bbP_{V,k,v_{*}}^{U,n,u_{*}}\big(h_{(\theta\rmB)_n\cap \rmB^{-,\eps}_l}\leq 0,
h_{\rmB^\eps_l\cap (\theta^{-1}\rmB^-)_k}\leq 0\big)\,.
\end{equation}
Then, by conditioning on the value of $h$ on $\partial\rmB^\pm_l$ and using Remark~\ref{r:bd-pm}, we write the probability on the right-hand side of~\eqref{e:DGFF-ub-int-viol} as
\begin{equation}
\label{e:st-disint-UB}
\int\bbP_{\rmB, l,w}^{U,n,u_{*}} \big(h_{(\theta\rmB)_n \cap \rmB_l^{-, \eps}} \leq 0 \big)
\bbP_{V,k,v_{*}}^{\rmB,  l,w} \big(h_{\rmB^\eps_l \cap (\theta^{-1}\rmB^-)_k} \leq 0 \big)
\bbP_{V,k,v_{*}}^{U,n,u_{*}} \big( h_{\partial \rmB^\pm_l} + m_l \in \rmd w \big)\,.
\end{equation}

As the constant boundary values $u_{*}$ on $\partial U_n$ and $v_{*}$ on $\partial V^-_k$ also have zero oscillation, we can apply Proposition~\ref{l:stitch-restr} to bound
$\bbP_{V,k,v_{*}}^{U,n,u_{*}}\big(h_{ (\theta\rmB)_n\cap(\theta^{-1}\rmB^-)_k}\leq 0,\,
h_{\partial\rmB^\pm_l}+m_l\notin E_{l,1,\eps}\big)$
by the right-hand side of~\eqref{e:1.18}.

It therefore suffices to bound the integral in~\eqref{e:st-disint-UB} restricted to $w\in E_{ l,1,\eps}$. We bound each factor in the integrand there by Proposition~\ref{p:1.3} and obtain as a bound for the integral a constant times 
\begin{multline}
\label{e:p2.3-disint}
\frac{\big(u_{*}+1\big)
\big(v_{*}+1\big)}{(n- l)( l-k)}\bbE_{k,v_{*}}^{n,u_{*}}\Big[\big((\he{h}{\partial \rmB^\pm_l}(0)+m_l)^-+\osc_{\rmB^\eps_l}\he{h}{\partial\rmB^\pm_l}+1\big)\\
\times \big((\he{h}{\partial\rmB^\pm_l}(0)+m_l)^-+\osc_{\rmB^{ \pm,\eps}_l}\he{h}{\partial\rmB^\pm_l}+1\big);
h_{\partial\rmB^\pm_l}+m_l\in E_{ l,1,\eps}\Big]\,.
\end{multline}
By definition of $E_{ l,1,\eps}$, both oscillation terms in the last display are bounded by $1$. Hence, the expectation there is bounded by $2\bbE_{k,v_{*}}^{n,u_{*}}
\Big(\big((\he{h}{\partial\rmB^\pm_l}(0)+m_l)^-\big)^2\Big) + 8$, where the second moment in turn is bounded by
$\Var_{k}^{n}
\Big(\he{h}{\partial\rmB^\pm_l}(0)\Big)+
\Big(\bbE_{k,v_{*}}^{n,u_{*}}\he{h}{\partial\rmB^\pm_l}(0) + m_l \Big)^2$. From Lemma~\ref{l:blml} and Propositions~\ref{l:E} and~\ref{p:Cov-bd}, it now follows that
\begin{equation}
\label{e:DGFF-UB-1}
\bbP_{k,v_{*}}^{n,u_{*}}\big(h_{ (\theta\rmB)_n\cap(\theta^{-1}\rmB^-)_k}\leq 0\big)\leq C_{\eps}\frac{(u_{*}^-+1)(v_{*}^-+1)}{n-k}\,,
\end{equation}
which yields assertion~\eqref{e:1.18}.

Next, we show the stronger statement in the theorem, namely~\eqref{e:t:2.3-posUB}, and suppose that $u_{*}>0$ or $v_{*}>0$ (as otherwise~\eqref{e:t:2.3-posUB} identifies with~\eqref{e:1.18}). We first consider the case that $0<u_{*}\leq (n-k)^{\eps}$ and $-(n-k)^{1-\eps}\leq v_{*}\leq 0$.
We now set $l=n-u_{*}^{2+\eps}$. As $\eps\in(0,\tfrac{1}{10})$ by assumption, we have that $\wedge_{n,l,k}^{3/2}\leq (n-k)^{1-\eps}$.
We also assume that $n-k$ is sufficiently large such that $n-l\leq l-k$, $\partial \rmB_l\subset(\theta\rmB)_n$,
$\partial \rmB^-_l\subset (\theta^{-1}\rmB^-)_k$, the assumptions of Lemmas~\ref{l:UVmin} and~\ref{l:bdr-dec} are satisfied, and we take $\theta$ again from Lemma~\ref{l:UVmin}.
If $n-k$ is not as large, then $u_{*}$ is bounded by a constant that depends only on $\eps$, in which case we can make the right-hand side of~\eqref{e:t:2.3-posUB} larger than $1$ by choosing $C$ sufficiently large.

We again apply Lemma~\ref{l:UVmin} to bound the left-hand side of~\eqref{e:t:2.3-posUB} by~\eqref{e:st-disint-UB}.
In~\eqref{e:st-disint-UB}, we now bound the second factor in the integrand by~\eqref{e:DGFF-UB-1} (with $l$ in place of $n$, etc). We bound the first factor as follows: if  $|\ol{w}(0)|\leq \tfrac12 \wedge_{n,l,k}$ and $\osc_{\rmB^{\pm,\eps}_l}\ol{w}\leq \tfrac12 \wedge_{n,l,k}$, we bound it by
\begin{equation}
\bbP_{\rmB,l,w}^{U,  n,u_{*}} \big(h_{(\tfrac{\eps}{2}\rmB)_n\cap\rmB^{-,\eps}_l} \leq 0 \big)
\leq C_\eps \rme^{-u_{*}^{2-\eps}}
\end{equation}
by~\eqref{e:bdr-dec-u}, and otherwise we bound it by $1$.
This gives
\begin{multline}
\label{e:ub-pos-E}
\bbP_{k,v_{*}}^{n,u_{*}}\big(h_{ (\theta\rmB)_n\cap(\theta^{-1}\rmB^-)_k}\leq 0\big)
\leq C_{\eps}\frac{v_{*}^- +1}{l-k}\\
\times\bigg\{
\bbE_{k,v_{*}}^{n,u_{*}} \Big(
 h_{*} +1;\:
\big|\he{h}{\partial\rmB^\pm_l}(0)+m_l\big|\leq \tfrac12\wedge_{n,l,k},
\osc_{\rmB^{\pm,\eps}_l}\he{h}{\partial\rmB^\pm_l}\leq \tfrac12\wedge_{n,l,k}
\Big)\rme^{-u_{*}^{2-\eps}}\\
+\bbE_{k,v_{*}}^{n,u_{*}} \Big( h_{*} +1;\:
\big|\he{h}{\partial\rmB^\pm_l}(0)+m_l\big|\geq \tfrac12\wedge_{n,l,k}\Big)\\
+\bbE_{k,v_{*}}^{n,u_{*}} \Big(  h_{*} +1
;\:\big|\he{h}{\partial\rmB^\pm_l}(0)+m_l\big|\leq \tfrac12\wedge_{n,l,k},
\osc_{\rmB^{\pm,\eps}_l}\he{h}{\partial\rmB^\pm_l}\geq \tfrac12\wedge_{n,l,k}\Big)
\bigg\}\,,
\end{multline}
where we write
\begin{equation}
h_{*}:=\big(\he{h}{\partial\rmB^\pm_l}(0)+m_l\big)^-
+\osc_{\rmB^{\pm,\eps}_l}\,\he{h}{\partial\rmB^\pm_l}\,.
\end{equation}
The first expectation in the curly brackets on the right-hand side of~\eqref{e:ub-pos-E} is bounded by $n-l+1\leq u_{*}^{2+\eps}+1$.
The second expectation there is bounded by a constant times
$\rme^{-u_{*}^{2-5\eps}}$
by Lemma~\ref{l:sr-}.
The expectation in the fourth line in~\eqref{e:ub-pos-E} is bounded by Lemma~\ref{l:sr-osc} by a constant times $(1+u_{*}^{2+\eps})\rme^{-u_{*}^{4+2\eps}}$. Plugging these bounds into~\eqref{e:ub-pos-E} yields the assertion
when we replace $\eps$ with $\eps/5$.

For the case that $u_{*}>(n-k)^\eps$ and $-(n-k)^{1-\eps}\leq v_{*}\leq 0$, we estimate
\begin{equation}
\bbP_{k,v_{*}}^{n,u_{*}}\big(h_{ (\theta\rmB)_n\cap(\theta^{-1}\rmB^-)_k}\leq 0\big)
\leq C_{\eps}\rme^{-(u_{*}^+)^{2-\eps/2}}
\end{equation}
by Lemma~\ref{l:bdr-dec} (with $\eps/2$ in place of $\eps$), which is less than the right-hand side of~\eqref{e:t:2.3-posUB} in this case.

The case $v_{*}>0$, $u_{*}\leq 0$ is analogous.
If $u_{*}\geq v_{*} >0$, we use monotonicity in the boundary conditions and the previously handled case ($u_{*}>0,v_{*}\leq 0$) in the estimate
\begin{equation}
\label{e:strongUB-p-uv}
\bbP_{k,v_{*}}^{n,u_{*}}(h_{ (\theta\rmB)_n\cap(\theta^{-1}\rmB^-)_k}\leq 0)
\leq \bbP_{k,0}^{n,u_{*}}(h_{ (\theta\rmB)_n\cap(\theta^{-1}\rmB^-)_k}\leq 0)
\leq C_{\eps}\frac{\rme^{-u_{*}^{2-\eps}}}{n-k}
\leq  C_{\eps}\frac{\rme^{-u_{*}^{2-\eps}/2} \rme^{-v_{*}^{2-\eps}/2}}{n-k}
\end{equation}
which is further bounded by the right-hand side in the assertion when we replace there $\eps$ with $2\eps$, and assume in particular $C\geq 1$ so that we also have $u_*>0$ and $v_*>0$.
If $v_{*}> u_{*} >0$, we argue as in~\eqref{e:strongUB-p-uv} but instead of $v_{*}$ we now replace $u_{*}$ with $0$.
\end{proof}

\section{The ballot asymptotics}
\label{s:cont-as}

As in the upper bound, the reduction to the DRW (Theorems~\ref{t:drw-i} and~\ref{t:drw-o}), permits us to derive ballot asymptotics for the DGFF from the corresponding asymptotics of the DRW. For the latter, asymptotics are given by Theorem~\ref{t:3.4}, albeit in a ``weak'' form, namely when one end point of the walk is required to be sufficiently negative. In addition, work is required to convert the asymptotic statement in Theorem~\ref{t:3.4} to a form which is similar to that in Theorem~\ref{t:1}. Accordingly, the desired weak versions of Theorem~\ref{t:1} are first stated 
(Subsection~\ref{ss:WeakAsymptotics}) with their proofs relegated to the next section. We then use these weak statements to prove Theorem~\ref{t:1} in its full generality (Subsection~\ref{ss:ProofOfFullAsymp}).

\subsection{Weak ballot asymptotics}
\label{ss:WeakAsymptotics}
Let $\cL_n=\cL_{n,\eta,U}$ and $\cR_k=\cR_{k,\zeta,V}$ be defined as in Proposition~\ref{p:1-2}.
\begin{prop}
\label{c:1.1}
Let $\eps\in(0,1)$ and $\eta,\zeta\in[0,\eps^{-1}]$.
Then
\begin{equation}
\label{e:c:1.1}
\bbP_{k,v}^{n,u} \Big( h_{U^\eta_n \cap V^{-,\zeta}_k} \leq 0 \Big)
=\big(2+o_{\eps}(1)\big)
\frac{\cL_{n} (u) \ol{v}(\infty)^-
}{g(n-k)}
\end{equation}
as $n-k\to\infty$, for all $U\in\frU^\eta_\eps$, $V\in\frV_\eps$, $0\leq k< n$, $u\in\bbR^{\partial U_n}$, $v\in\bbR^{\partial V^-_k}$ satisfying~\eqref{e:uv-as} and $\ol{v}(\infty)^- \geq (n-k)^\eps$.
\end{prop}

\begin{prop}
\label{c:1.2}
Let $\eps\in(0,1)$ and $\eta,\zeta\in[0,\eps^{-1}]$. Then,
\begin{equation}
\label{e:c:1.2}
\bbP_{k,v}^{n,u} \Big( h_{U^\eta_n \cap V^{-,\zeta}_k} \leq 0 \Big)
=\big(2+o_{\eps}(1)\big)
\frac{\cR_{k} (v) \ol{u}(0)^-
}{g(n-k)}
\end{equation}
as $n-k\to\infty$, for all $U\in\frU^\eta_\eps$, $V\in\frV_\eps$, $0\leq k< n$, $u\in\bbR^{\partial U_n}$, $v\in\bbR^{\partial V^-_k}$ satisfying~\eqref{e:uv-as} and $\ol{u}(0)^- \geq (n-k)^\eps$.
\end{prop}

\subsection{Stitching - Proof of Theorem~\ref{t:1}}
\label{ss:ProofOfFullAsymp}
As in the proof of Theorem~\ref{t:2.3}, the proof of Theorem~\ref{t:1} in its full generality goes by conditioning on the value of the DGFF on the boundary of a centered ball at an intermediate scale.
Using the entropic repulsion statement, Proposition~\ref{l:stitch-restr}, one can restrict attention to the case when the corresponding harmonic extension in the bulk is sufficiently repelled below $0$ and its oscillation therein is not too large. This then permits us to use the weak ballot asymptotic statements in the previous sub-section, as long as the restriction to stay negative is imposed away from the boundary of the ball. To make sure that the asymptotics for the ballot probability does not increase by leaving out the region close to this boundary, we shall also need the following proposition.
We recall that the set $E_{l,M,\eps}$ was defined in~\eqref{e:4.2}. 
\begin{prop}
\label{p:stitch-lb}
Let $\eps\in\big(0,\tfrac{1}{10}\big)$ and $\eta,\zeta\in[0,\eps^{-1}]$.
Then there exist $C=C_{\eps}<\infty$ and $c=c_{\eps}>0$ such that
\begin{multline}
\label{e:st-lb}
\bbP_{k,v}^{n,u}
\Big(\max_{\rmB_{ l+1}\cap \rmB^{-}_{ l}} h>0,
h_{\partial \rmB^\pm_{ l-1}}+m_{ l-1}\in E_{ l-1,M,\eps},
h_{\partial \rmB^\pm_{ l+2}}+m_{ l+2}\in E_{ l+2,M,\eps}\Big)\\
\leq C\rme^{-c \wedge_{n,l,k}^\eps+CM}
\end{multline}
for all $M\geq 0$, $U\in\frU^\eta_\eps$, $V\in\frV_\eps$ all $k,l,n\geq 0$ with
$\partial\rmB_{l+2}\subset U^\eta_n$ and $\partial\rmB^-_{l-1}\subset V^{-,\zeta}_k$, and all
$u\in\bbR^{\partial U_n}$, $v\in\bbR^{\partial V^-_k}$.
\end{prop}
\begin{proof}
By Lemma~\ref{l:UV-U}, under $\bbP_{k,v}^{n,u}$, for $x\in \rmB_{ l+1}\cap \rmB^{-}_{ l}$,
\begin{multline}
\he{h}{\partial \rmB^-_{ l-1}\cup\partial \rmB_{ l+2}}(x)
=\heb{h_{\partial \rmB^-_{l-1}}}{\partial \rmB^-_{ l-1}\cup\partial \rmB_{ l+2}}(x)
+\heb{h_{\partial \rmB^-_{l+2}}}{\partial \rmB^-_{ l-1}\cup\partial \rmB_{ l+2}}(x)\\
\leq P_x\big(\tau^{\rmB_{l+2}} \leq \tau^{\rmB^-_{l-1}}\big)
\he{h}{\partial \rmB_{ l+2}}(x)
+P_x\big(\tau^{\rmB^-_{l-1}}<\tau^{\rmB_{l+2}}\big)
\osc_{\rmB_{l+1}}\he{h}{\partial \rmB_{ l+2}}\\
+P_x\big(\tau^{\rmB^-_{l-1}}<\tau^{\rmB_{l+2}}\big)
\he{h}{\partial \rmB^-_{ l-1}}(x)
+P_x\big(\tau^{\rmB_{l+2}} \leq \tau^{\rmB^-_{l-1}}\big)
\osc_{\rmB^-_{l}}\he{h}{\partial \rmB^-_{ l-1}}\,.
\end{multline}
Hence, on the event $E$ that $h_{\partial \rmB^\pm_{ l-1}}+m_{ l-1}\in E_{ l-1,M,\eps}$ and $h_{\partial \rmB^\pm_{ l+2}}+m_{ l+2}\in E_{ l+2,M,\eps}$, we have
\begin{equation}
\he{h}{\partial \rmB^-_{ l-1}\cup\partial \rmB_{ l+2}}(x)
+m_{l-1}
\leq -\min\big\{\wedge_{n,l-1,k}^\eps, \wedge_{n,l+2,k}^\eps\big\} + 2M
\end{equation}
under $\bbP_{k,v}^{n,u}$. Now the Gibbs-Markov property (Lemma~\ref{l:GM}) yields
\begin{multline}
\bbP_{k,v}^{n,u}\Big(\max_{\rmB^{-}_{ l}\cap \rmB_{ l+1}} h>0, E\Big)
\leq 
\bbP\Big(\max_{\rmB^{-}_{ l}\cap \rmB_{ l+1}}h^{\rmB^-_{ l-1}\cap \rmB_{ l+2}}
> m_{ l-1}+ \min\big\{\wedge_{n,l-1,k}^\eps, \wedge_{n,l+2,k}^\eps\big\} -2M\Big)\\
\leq C\rme^{-c \wedge_{n,l,k}^\eps+CM}
\end{multline}
for constants $C=C_\eps<\infty,c=c_\eps>0$ from extreme value theory of the DGFF (Lemma~\ref{l:DGFF-ut}).
\end{proof}
In the next proof, we also rely on Propositions~\ref{p:1-4} and~\ref{p:1-5} which we prove independently in Section~\ref{s:LR-funct}.

\begin{proof}[Proof of Theorem~\ref{t:1} and Proposition~\ref{p:1-2}]
We assume w.\,l.\,o.\,g.\ that $\eps\in(0,1/15)$.
It suffices to consider the case that $\max\{\ol{u}(0)^-,\ol{v}(\infty)^-\}\leq (n-k)^\eps$ as the converse case is already dealt with in Propositions~\ref{c:1.1} and~\ref{c:1.2}.

For $0\leq k\leq n$ and $l=k+\tfrac12(n-k)$ such that $\partial \rmB_l\subset U^\eta_n$ and $\partial\rmB^-_l\subset V^{-,\zeta}_k$, we consider, for sufficiently large $n-k$ as in~\eqref{e:st-disint-UB},	
\begin{multline}
\label{e:st-disint-8}
\bbP_{V,k,v}^{U,n,u}\big(h_{U^\eta_n\cap V^{-,\zeta}_k}\leq 0\big)
\leq \bbP_{V,k,v}^{U,n,u}\big(h_{U^\eta_n\cap \rmB^{-,\eps}_l}\leq 0,
h_{\rmB^\eps_l\cap V^{-,\zeta}_k}\leq 0\big)\\
=\int\bbP_{\rmB, l,w}^{U,n,u} \big(h_{U^\eta_n \cap \rmB_l^{-, \eps}} \leq 0 \big)
\bbP_{V,k,v}^{\rmB,  l,w} \big(h_{\rmB^\eps_l \cap V^{-,\zeta}_k} \leq 0 \big)
\bbP_{V,k,v}^{U,n,u} \big( h_{\partial\rmB^\pm_l} + m_l \in \rmd w \big)\,.
\end{multline}

Let $M\in(1,\infty)$.
We first evaluate the integral on the right-hand side of~\eqref{e:st-disint-8} restricted to $w\in E_{l, M,3\eps}$.
In this range of integration, for the first factor in the integrand, we plug in the uniform asymptotics from Proposition~\ref{c:1.1}, and for the second one, we use Proposition~\ref{c:1.2}.
Then, as $n-k\to\infty$, the restricted integral is asymptotically equivalent to
\begin{equation}
\label{e:restr-as}
4\frac{\cL_{n}(u) \cR_{k} (v)
}{g^2(n- l)( l-k)}
\bbE_{k,v}^{n,u}\Big((\he{h}{\partial\rmB^\pm_l}(0)+m_l)^-(\he{h}{\partial\rmB^\pm_l}(\infty)+m_l)^-;h_{\partial\rmB^\pm_l}+m_l\in E_{ l,M,3\eps}\Big)\,.
\end{equation}
On $E_{k,M,3\eps}$, we have $\osc_{\rmB^{\pm,\eps}_l}\he{h}{\partial \rmB^\pm_l}\leq M$ and
$(\he{h}{\partial\rmB^\pm_l}(0)+m_l)^-\geq l^{3\eps}$, and hence
\begin{equation}
\label{e:t1-pO1}
(\he{h}{\partial\rmB^\pm_l}(0)+m_l)^-(\he{h}{\partial\rmB^\pm_l}(\infty)+m_l)^-
=\big[(\he{h}{\partial\rmB^\pm_l}(0)+m_l)^-\big]^2\big(1+O(Ml^{-3\eps})\big)\,.
\end{equation}
Setting 
\begin{equation}
\mu(0;u,v):=\bbE_{k,v}^{n,u}\,\he{h}{\partial\rmB^\pm_l}(0)\,,
\end{equation}
we also have
\begin{equation}
\label{e:p-as-blml}
\big| \mu(0;u,v) + m_l \big|
\leq C_\eps\big(1+l^\eps\big)
\end{equation}
by Lemma~\ref{l:blml}, Proposition~\ref{l:E} and the assumptions on $u$ and $v$. Hence, \eqref{e:t1-pO1} is equal to
$\big[\big(\he{h}{\partial\rmB^\pm_l}(0)-\mu(0;u,v)\big)^-\big]^2 \big(1+O(l^{-\eps}+Ml^{-3\eps})\big)$ so that
the expectation in~\eqref{e:restr-as} equals
\begin{equation}
\label{e:pt1.1-p2}
\bbE_{k,v}^{n,u}\Big(\big((\he{h}{\partial\rmB^\pm_l}(0)-\mu(0;u,v))^-\big)^2;h_{\partial\rmB^\pm_l}+m_l\in E_{ l,M,3\eps}\Big)\big(1+O(Ml^{-\eps})\big)\,.
\end{equation}
We claim that
\begin{equation}
\label{e:claimEc-pt-1.1}
\bbE_{k,v}^{n,u}\Big(\big((\he{h}{\partial\rmB^\pm_l}(0)-\mu(0;u,v))^-\big)^2;h_{\partial\rmB^\pm_l}+m_l\notin E_{ l,M,3\eps}\Big)=O(l^{6\eps}+l\rme^{-cM^2})\,.
\end{equation}
Using this claim, \eqref{e:pt1.1-p2} and Proposition~\ref{p:Cov-bd}, we obtain that the expectation in~\eqref{e:restr-as} equals
\begin{equation}
\tfrac12\Var_{k}^{n}\Big(\he{h}{\partial\rmB^\pm_l}(0)\Big)\big(1+O(l^{6\eps-1}+\rme^{-cM^2})\big)
=\tfrac18g(n-k)\big(1+O(Ml^{-\eps}+l^{6\eps-1}+\rme^{-cM^2})\big)\,.
\end{equation}
Moreover, the restriction of the integral on the right-hand side of~\eqref{e:st-disint-8} to $w\notin E_{l, M,3\eps}$ is, by Proposition~\ref{l:stitch-restr}, bounded from above by a constant times $(\rme^{-cM^2} + (n-k)^{-1})$ times the right-hand side of~\eqref{e:2.4ub}, as $\cL_{n}(u)\geq c_{\eps}\big(1+\ol{u}(0)^-\big)$ and $\cR_{k}(v)\geq c_{\eps}\big(1+\ol{v}(\infty)^-\big)$ by Propositions~\ref{p:1-4} and~\ref{p:1-5} (which are proved independently in Section~\ref{s:LR-funct}). Altogether it then follows that the right-hand side of~\eqref{e:st-disint-8} is asymptotically equivalent to the right-hand side of~\eqref{e:2.4ub} as $n-k\to\infty$ followed by $M\to\infty$.

We now show the claim~\eqref{e:claimEc-pt-1.1}. Using the decomposition~\eqref{e:E1234} (with $3\eps$ in place of $\eps$), we have to bound
\begin{equation}
\label{e:claim-p1.1-1234}
\bbE_{k,v}^{n,u}\Big(\big((\he{h}{\partial\rmB^\pm_l}(0)-\mu(0;u,v))^-\big)^2;h_{\partial\rmB^\pm_l}+m_l\in E_i\Big)
\end{equation} 
for $i=1,\ldots,4$. For $i=1$, this expectation is bounded by
\begin{equation}
\bbE_{k,v}^{n,u}\Big(\big(\he{h}{\partial\rmB^\pm_l}(0)-\mu(0;u,v)\big)^4\Big)^{1/2}\bbP_{k,v}^{n,u}\Big(\he{h}{\partial\rmB^\pm_l}(0)+m_l\in E_1\Big)^{1/2}
\end{equation}
by the Cauchy-Schwarz inequality.  The centered fourth moment of the Gaussian random variable $\he{h}{\partial\rmB^\pm_l}(0)$, whose variance is given by Proposition~\ref{p:Cov-bd}, is bounded by a constant times $l^2$, and by Lemma~\ref{l:sr-}, the probability is bounded by a constant times $\rme^{-l^{1-6\eps}}$.
For $i=2$, \eqref{e:claim-p1.1-1234} is bounded by a constant times
$l^{6\eps}+l^{2\eps}$ by definition of $E_2$ and~\eqref{e:p-as-blml}.
For $i=3$, \eqref{e:claim-p1.1-1234} is bounded by a constant times
$\rme^{-cM^2}(1+l+l^{2\eps})$ by Lemma~\ref{l:sr-osc}, \eqref{e:p-as-blml} and the assumptions on $u,v$.
On $E_4$, we have $\he{h}{\partial\rmB^\pm_l}(0)+m_l\geq l^{6\eps}$. Hence, using also~\eqref{e:p-as-blml}, the expression in~\eqref{e:claim-p1.1-1234} for $i=4$ is equal to zero whenever $l^{6\eps}-l^{2\eps}\geq 0$. This shows the claim~\eqref{e:claimEc-pt-1.1}.

It remains to bound the difference between the right and left sides of~\eqref{e:st-disint-8} which is equal to
\begin{equation}
\label{e:p-as-diff-lb}
\bbP_{k,v}^{n,u}\big(h_{U^\eta_n\cap \rmB^{-,\eps}_l}\leq 0,
h_{\rmB^\eps_l\cap V^{-,\zeta}_k}\leq 0,
\max_{\bbZ^2\setminus \rmB^{\pm,\eps}_l} h>0\big)\,.
\end{equation}
This probability is bounded from above by
\begin{multline}
\bbP_{k,v}^{n,u}\big(h_{U^\eta_n\cap V^{-,\zeta}_k}\leq 0, h_{\partial \rmB_{l-1}}\notin E_{l-1,M,\eps}\big)
+ \bbP_{k,v}^{n,u}\big(h_{U^\eta_n\cap V^{-,\zeta}_k}\leq 0, h_{\partial \rmB_{l+2}}\notin E_{l+2,M,\eps}\big)\\
+\bbP_{k,v}^{n,u}\big(\max_{\bbZ^2\setminus \rmB^{\pm,\eps}_l} h>0, h_{\partial \rmB^\pm_{l-1}}\in E_{l-1,M,\eps}, h_{\partial \rmB^\pm_{l+2}}\in E_{l+2,M,\eps}\big)
\end{multline}
by a union bound. Applying Proposition~\ref{l:stitch-restr} to the first two probabilities and Proposition~\ref{p:stitch-lb} to the third one shows that~\eqref{e:p-as-diff-lb} is bounded by a constant times
\begin{equation}
\label{e:p:1.1-1-5-lb}
\frac{\big(\ol{u}(0)^-
+1\big)
\big(\ol{v}(\infty)^- +1\big)}{n-k}
\Big(\wedge_{n,l,k}^{-3/2+4\eps}+\rme^{-cM^2}
+\rme^{-c_\eps (n-k)^\eps+C_\eps M}\Big)\,,
\end{equation}
where we also used~\eqref{e:uv-as}. The expression in~\eqref{e:p:1.1-1-5-lb} is further bounded by a constant times
\begin{equation}
\frac{\cL_n(u)\cR_k(v)}{n-k}
\Big(\wedge_{n,l,k}^{-3/2+4\eps}+\rme^{-cM^2}
+\rme^{-c_\eps (n-k)^\eps+C_\eps M}\Big)
\end{equation}
by Propositions~\ref{p:1-4} and~\ref{p:1-5} and hence of smaller order than the right-hand side of~\eqref{e:2.4ub} as $n-k\to\infty$ followed by $M\to\infty$. This completes the proof.
\end{proof}

\section{The ballot functionals}
\label{s:LR-funct}

Comparison statements for the ballot functionals for the DRW and the DGFF (from Proposition~\ref{p:1-2}) are stated (Subsection~\ref{s:comp-ballot}) and proved  (Subsection~\ref{s:pf-aux}). These are then used (Subsection~\ref{s:DGFF-as}) to prove Propositions~\ref{c:1.1}, \ref{c:1.2}, and~\ref{p:LR-infty} --~\ref{p:1-6}. The presentation focuses on the functional $\cL_n$. The functional $\cR_k$ is handled analogously (Subsection~\ref{s:R-funct}).

\subsection{Auxiliary results}
\label{s:comp-ballot}
In the setting of the inward concentric decomposition (Section~\ref{s:def-DRW}) and using the correspondence from Theorem~\ref{t:drw-i}, we define
\begin{equation}
\label{e:def-l}
\ell_{\rt,r}\big(\ol{u}(0),\ol{v}(\infty)\big)=
\bbE_{V,k,v}^{U,n,u}\Big(\cS_r^-;\, \max_{i=1}^r\big( \cS_i + \cD_i\big)\leq 0\Big)
\end{equation}
in accordance with~\eqref{e:5.4} in Appendix~\ref{s:3}, and recalling from~\eqref{e:def-r} that $\rt:=\rt_{n-k} :=\lfloor n - k \rfloor + \lfloor\log\eps\rfloor - \lceil\log(\eps^{-1}+\zeta)\rceil$.
As an intermediate step between~\eqref{e:def-l} and the functional $\cL_n$ from Proposition~\ref{p:1-2}, we define the functional
\begin{equation}
\label{e:def-cLm}
\cL^r_{n,k}(u,v):=\cL_{n,\eta,U,k,V}^r(u,v)
:=\bbE_{V,k,v}^{U,n,u}\big(
(\he{h}{\partial\rmB^-_{n-r}} (0)+ m_{n-r})^-;
h_{U^\eta_n\cap \rmB^-_{n-r}} \leq 0\big)\,.
\end{equation}
These functionals for the DGFF and the DRW can be compared as follows:
\begin{lem}
\label{l:LUV}
Let $\eps\in(0,1)$.
Then,
\begin{equation}
\label{e:LUV-add}
\lim_{\substack{r\to\infty\\r\in\bbN}}
\lim_{n-k\to\infty}
\big(1+\ol{u}(0)^-\big)^{-\eps}
\Big|\ell_{\rt_{n-k},r+\lfloor\log\eps\rfloor}\big(\ol{u}(0),\ol{v}(\infty)\big)
-\cL^{r}_{n,k}(u,v)\Big|=0
\end{equation}
uniformly in $\eta,\zeta \in[0,\eps^{-1}]$, $U\in \frU^\eta_\eps$, $V\in\frV_\eps$, $u\in\bbR^{\partial U_n}$, $v\in\bbR^{\partial V^-_k}$ satisfying~\eqref{e:uv-diff}.
\end{lem}

The following lemma allows to vary the parameter $\eta$ in $\cL^r_{n,k}$. For a subset $W \subset U$, we define
\begin{equation}
\label{e:def-Lmuv}
\cL^r_{n,W,U,k,V}(u,v)
=\bbE_{V,k,v}^{U,n,u}\Big(\big(\he{h}{\partial\rmB_{n-r}}(0)+m_{n-r}\big)^-;\;
h_{W_n\cap \rmB^-_{n-r}}\leq 0\Big)\,,
\end{equation}
so that in particular
$\cL^r_{n,\eta,U,k,V}(u,v)=\cL^r_{n,U^\eta,U,k,V}(u,v)$.
\begin{lem}
\label{l:delta0}
Let $\eps\in(0,1)$ and $\eta,\zeta\in[0,\eps^{-1}]$.
There exists $C=C_\eps<\infty$ such that
\begin{equation}
\bigg|\frac{\cL^r_{n,W,U,k,V}(u,v)}{\cL^r_{n,W',U,k,V}(u,v)}-1\bigg|
\leq C\Leb(W \triangle W')
\end{equation}
for all $U\in\frU^\eta_\eps$, $V\in\frV_\eps$, Borel measurable $W,W'$ with $\eps^2\rmB \subset W,W'\subset U^\eta$, $n,k,r\geq 0$ with
$\partial \rmB_{n-r}\subset (\eps^2\rmB)_n\cap V^{-,\zeta}_k$,
and $u\in\bbR^{\partial U_n}$, $v\in\bbR^{\partial V^-_k}$ that satisfy
$\max\{\ol{u}(0),\osc\,\ol{u}_\eta, \ol{u'}(0), \osc\,\ol{u}'_{\eta'}\} < \epsilon^{-1}$.
\end{lem}

To further approach the functional $\cL_n$ from Proposition~\ref{p:1-2}, we also set
$\cL^r_{n}(u)=\cL^r_{n,\eta,U,0,\rmB}(u,\ol{u}(0))$.
To compare $\cL^r_n(u)$  with
$\cL^r_{n,k}(u,v)$, we need the following comparison statement:
\begin{lem}
\label{l:cont-V}
Let $r\geq 0$ and $\eps\in(0,\tfrac{1}{10})$. Then there exists $C=C_{r,\eps}<\infty$ such that for all $\eta\in[\eps,\eps^{-1}]$, $\zeta\in[0,\eps^{-1}]$, $U\in\frU^\eta_\eps$, $V\in\frV_\eps$, $0<k<n$ with $\partial \rmB_{n-r}\subset U^{\eta}_n$ and $\partial \rmB^-_{n-r-1}\subset V^{-,\zeta\vee\eps}_k$,
and all $u\in\bbR^{\partial U_n}$, $v\in\bbR^{\partial V^-_k}$ satisfying~\eqref{e:uv-diff}, 
we have
\begin{equation}
\label{e:cont-V-assertion}
\big|\cL^r_{n,\eta,U,0,\rmB}(u,\ol{u}(0))-\cL^r_{n,\eta,U,k,V}(u,v) \big|
\leq C(n-k)^{-\eps/ 4}\,.
\end{equation}
\end{lem}
With the next lemma, in which the functional $\cL^r_{n,k}(u,v)$ is considered for fixed $r$, we prepare the proof of Proposition~\ref{p:1-6}.
\begin{lem}
\label{l:1-6-m}
Let $\eps\in(0,1)$, $r\geq 0$.
Then there exists $\rho=\rho_{r,\eps}:\bbR_+\to\bbR_+$ with $\rho(t)\downarrow 0$ as $t\downarrow 0$ such that the following holds: 
For all $n>0$, all $\eta, \eta' \geq 0$,
$U, U'\in\frU^{\eta\vee\eta'}_\eps$, and all $u \in \bbR^{\partial U_n}$, $u' \in \bbR^{\partial U'_n}$ such that $\max\{\ol{u}(0),\osc\,\ol{u}_\eta, \ol{u'}(0), \osc\,\ol{u}'_{\eta'}\} < \epsilon^{-1}$,
we have
\begin{multline}
\label{e:p:1-6-L-m}
\big| \cL^r_{n,\eta',U',0,\rmB}(u',\ol{u'}(0)) -
\cL^r_{n,\eta,U,0,\rmB}(u,\ol{u}(0)) \big|\\
\leq\Big(\cL^r_{n,\eta,U,0,\rmB}(u,\ol{u}(0))
+\cL^r_{n,\eta',U',0,\rmB}(u',\ol{u'}(0))+1\Big)
\rho \Big(\dH\big(U, U'\big) + |\eta - \eta'| + \|\ol{u}-\ol{u'}\|_{\bbL_\infty(U_{n}^{\eta}
\cap U'^{\eta'}_{n})}  + n^{-1} \Big) \,.
\end{multline}
\end{lem}

For the proof of Proposition~\ref{p:LR-infty}, we also need the following lemma which ensures the convergence of the DRW bridge and decorations.
\begin{lem}
\label{l:SDconv}
Let $\eps\in(0,1)$ and $\eta,\zeta \in[0,\eps^{-1}]$, $q\in\bbN$.
For $n,k\geq 0$, assume that $u\in\bbR^{\partial U_n}$, $v\in\bbR^{\partial V^-_k}$ satisfy~\eqref{e:uv-as}, and that
\begin{equation}
\label{e:SDconv-u}
\lim_{n\to\infty}
\max_{x\in U^\eta_n}\big| \ol{u}(x) - \wh u_\infty(\rme^{-n}x)\big|
=0
\end{equation}
for some $\wh u_\infty\in \bbH_\infty(U^\eta)$.
Then $(\cS_p,\cS_p+\cD_p)_{p=1}^q$ converges in distribution as $n-k\to\infty$ to some limit $(\cS^{(\infty)}_p,\cS^{(\infty)}_p+\cD^{(\infty)}_p)_{p=1}^q$ whose distribution is absolutely continuous with respect to Lebesgue measure and depends only on $\wh u_\infty$ and $\eta$.
\end{lem}

\subsection{Proofs of weak ballot asymptotics and properties of $\cL_n$}
\label{s:DGFF-as}

As in Proposition~\ref{p:1-2}, we define
\begin{equation}
\label{e:cLmr}
\cL_n(u)=\cL_{n,\eta,U}(u)=\cL^{r_n}_{n,\eta,U,0,\rmB}(u,\ol{u}(0))\,,
\qquad u\in\bbR^{\partial U_n}\,.
\end{equation}
Throughout the paper, in particular in Proposition~\ref{p:1-2}, we choose $(r_n)_{n\geq 0}$ as a collection of integers such that $r_n\in[ 0, n^{\eps/2}]$ and $r_n\to\infty$ as $n\to\infty$ slowly enough, depending only on given $\eps\in(0,1)$, such that the following properties~\ref{i:mr-LUV} -- \ref{i:mr-VV-fluct} hold. Here, we take $\delta$ as half the supremum of the $\delta$ that we can obtain from both Theorems~\ref{t:drw-i} and~\ref{t:drw-o} for our given $\eps$. The conditions involving~\eqref{e:LUV-add-R}, \eqref{e:cont-V-assertion-R} and~\eqref{e:LUVt-R} can be ignored for the moment and will only be needed in Section~\ref{s:R-funct}.
\begin{enumerate}[label=(\roman{*}),ref=(\roman{*})]
\item\label{i:mr-LUV}\eqref{e:LUV-add} and~\eqref{e:LUV-add-R} hold when we replace $r$ with any element of $[r,r_{n-k}]$.
\item \eqref{e:cont-V-assertion}, \eqref{e:cont-V-assertion-R} and~\eqref{e:LUVt-R} hold with $C=C_{\eps}$ when we replace $r$ with any element of $[r,r_{n-k}]$.
\item\label{i:mr-leqrdelta} We have $r_n\leq n^{\delta/4}$.
\item\label{i:mr-as} \eqref{e:20.39'} holds when we replace $r$ with any element of $[r,r_{\rt'}]$.
\item\label{i:mr-1.5} With $T_{(r)}$ from Proposition~\ref{p:1.5}, we have $T_{(r_n)}\leq T_n$.
\item\label{i:mr-VV-fluct} The constants $C$ and $c$ in \eqref{e:VV-fluct-o} and~\eqref{e:VV-fluct-o-P} can be chosen as depending only on $\eps$ (not on $r$), under the restriction that $t\geq 1$ and $r\leq r_{n-k}$.
\end{enumerate}

We now prove Proposition~\ref{c:1.1} and the assertions on $\cL_n$ and $\cL$ in Propositions~\ref{p:LR-infty} -- \ref{p:1-6}.
The proofs of Proposition~\ref{c:1.2} and of the assertions on $\cR_k$ and $\cR$ are outlined in Section~\ref{s:R-funct}.

We need a lower bound on the ballot probability:
\begin{lem}
\label{l:DGFF-lb}
Let $\eps\in(0,1)$ and $\eta,\zeta\in[0,\eps^{-1}]$. Then there exists $c=c_\eps>0$ such that
\begin{equation}
\label{e:DGFF-lb}
\bbP_{k,v}^{n,u}\big(h_{U^\eta_n\cap V^{-,\zeta}_k}\leq 0\big)
\geq c\frac{(1+\ol{u}(0)^-)\ol{v}(\infty)^-}{n-k}
\end{equation}
for all $U\in\frU^\eta_\eps$, $V\in\frV_\eps$, $0\leq k< n$ with $T_{n-k}\geq 1$, $u\in\bbR^{\partial U_n}$, $v\in\bbR^{\partial V^-_k}$ satisfying~\eqref{e:uv-as} and $\ol{v}(\infty)^- \geq (n-k)^\eps$.
\end{lem}
\begin{proof}
By Theorems~\ref{t:drw-i} and~\ref{t:lb}, there exists $a_0\in(-\infty,0)$, depending only on $\eps$, such that the assertion holds if $\ol{u}(0)< a_0$.
To infer the general case, we use FKG (see e.\,g.\ Section~5.4 of~\cite{Bi-LN}) in
\begin{equation}
\bbP_{k,v}^{n,u}\big(h_{U^\eta_n\cap V^{-,\zeta}_k}\leq 0\big)\geq
\bbP_{k,v}^{n,u}\big(h_{U^{\eta}_n\setminus U^{\eta+2t}_n}\leq 0\big)
\bbP_{k,v}^{n,u}\big(h_{U^{\eta+2t}_n\cap V^{-,\zeta}_k}\leq 0\big)\,,
\end{equation}
where we can choose $t>0$, depending only on $\eps$, such that the first factor is at least $1/2$ by Lemma~\ref{p:viol}. To bound the second factor, let $\frg:U_n\cap V^-_k\to\bbR$ be harmonic with boundary values $u$ on $\partial U_n$ and $v$ on $\partial V^-_k$. By Lemma~\ref{l:UV-U} and Assumption~\eqref{e:uv-as}, $\frg\leq 2\eps^{-1}$ on $U^\eta_n\cap V^{-,\zeta}_k$. Hence, for $\lambda>0$,
\begin{multline}
\label{e:p-DGFF-lb}
\bbP_{k,v}^{n,u}\big(h_{U^{\eta+2t}_n\cap V^{-,\zeta}_k}\leq 0\big)
=\bbP_{k,0}^{n,0}\big(h_{U^{\eta+2t}_n\cap V^{-,\zeta}_k}+\frg\leq 0\big)
\geq \bbP_{k,2\eps^{-1}}^{n,2\eps^{-1}}\big(h_{U^{\eta+2t}_n\cap V^{-,\zeta}_k}\leq 0\big)\\
\geq \bbP_{k,2\eps^{-1}}^{n,2\eps^{-1}}\big(h_{U^{\eta+2t}_n\cap V^{-,\zeta}_k}\leq 0\,,
\he{h}{\partial U^{\eta+t}_n}(0)+m_n\in(a_0-1,a_0),
\osc_{U^{\eta+2t}_n}\he{h}{\partial U^{\eta+t}_n}\leq\lambda\big)\\
\geq
\bbP_{k,2\eps^{-1}}^{n,2\eps^{-1}}\big(\he{h}{\partial U^{\eta+t}_n}(0) +m_n \in (a_0-1,a_0)\,,
\osc_{U^{\eta+2t}_n}\he{h}{\partial U^{\eta+t}_n}\leq\lambda\big)\\
\times\inf\Big\{
\bbP^{U^{\eta+t},n,w}_{V,k,2\eps^{-1}}\big(h_{U^{\eta+2t}_n\cap V^{-,\zeta}_k} \leq 0\big)\::
w\in\bbR^{\partial U^{\eta+t}_n}\,,
\ol{w}(0)\in(a_0-1,a_0)\,,\osc_{U^{\eta+2t}_n}\ol{w}\leq \lambda\Big\}\,,
\end{multline}
where we used the Gibbs-Markov property in the last step.
By Lemma~\ref{l:blml} and Propositions~\ref{l:E} and~\ref{p:Cov-bd},
\begin{equation}
\bbP_{k,2\eps^{-1}}^{n,2\eps^{-1}}\big(\he{h}{\partial U^{\eta+t}_n}(0)+m_n\in(a_0-1,a_0)\big)
\geq c_\eps\,.
\end{equation}
Moreover, by Proposition~\ref{p:2.6}, and for sufficiently large $\lambda$, depending only on $\eps$,
\begin{equation}
\bbP_{k,2\eps^{-1}}^{n,2\eps^{-1}}\big(
\osc_{U^{\eta+2t}_n}\he{h}{\partial U^{\eta+t}_n}\leq\lambda\,\big|\,
\he{h}{\partial U^{\eta+t}_n}(0)+m_n=r\big)
\geq 1 - C_\eps\rme^{-c_\eps\lambda} \geq \tfrac12
\end{equation}
for all $r\in(a_0-1,a_0)$. Hence, the probability in the third line of~\eqref{e:p-DGFF-lb} is bounded from below by a constant.
As noted in the beginning of the proof, the infimum in the last line of~\eqref{e:p-DGFF-lb}  is bounded from below as in the right-hand side of~\eqref{e:DGFF-lb}. Combining the above bounds yields the assertion.
\end{proof}

We are now ready for,
\begin{proof}[Proof of Proposition~\ref{c:1.1}]
We first consider the case that $\eta\in[\eps,\eps^{-1}]$.
From Theorem~\ref{t:3.4}, condition~\ref{i:mr-leqrdelta} in the choice of $r_{n}$, and Theorem \ref{t:drw-i}, we have
\begin{multline}
\label{e:DGFF-as-add}
\bbP_{k,v}^{n,u}\big(h_{U^\eta_n\cap V^{-,\zeta}_k}\leq 0\big)
=\bbP_{k,v}^{n,u}\big(\max_{i=1}^\rt (\cS_i+\cD_i )\leq 0\big)\\
= 2
\frac{\ell_{\rt_{n-k},r_{n-k}+\lfloor\log\eps\rfloor}(\ol{u}(0),\ol{v}(\infty))\ol{v}(\infty)^-}{g(n-k)}
+o_{\eps}(1)\frac{\big(1+\ol{u}(0)^-\big)\ol{v}(\infty)^-}{n-k}
\end{multline}
as $n-k\to\infty$, uniformly as in the assertion. We now convert the $\ell$ functional on the right-hand side to the $\cL$ functional in the statement of the proposition.
To this aim, we write
\begin{multline}
\label{e:pf-weak-as}
\bigg|\frac{\ell_{\rt_{n-k},r_{n-k}+\lfloor\log\eps\rfloor}(\ol{u}(0),\ol{v}(\infty))\ol{v}(\infty)^-}{g(n-k)}-
\frac{\cL_{n}(u)\ol{v}(\infty)^-}{g(n-k)}
\bigg|\\
\leq
\frac{\big|\ell_{\rt_{n-k},r_{n-k}+\lfloor\log\eps\rfloor}(\ol{u}(0),\ol{v}(\infty))-\cL_{n,k}^{r_{n-k}}(u,v)\big|
+\big|\cL_{n,k}^{r_{n-k}}(u,v)-\cL_{n}(u)\big|}{1+\ol{u}(0)^-}\frac{\big(1+\ol{u}(0)^-\big)\ol{v}(\infty)^-}
{g(n-k)}\,,
\end{multline}
The first ratio on the right-hand of~\eqref{e:pf-weak-as} is $o_{\eps}(1)$ as $n-k\to\infty$ by Lemmas~\ref{l:LUV} and~\ref{l:cont-V} and the choice of $r_{n}$.
Hence,
\begin{equation}
\label{e:DGFF-as-L-add}
\bbP_{k,v}^{n,u}\big(h_{U^\eta_n\cap V^{-,\zeta}_k}\leq 0\big)
= 2
\frac{\cL_n(u)\ol{v}(\infty)^-}{g(n-k)}
+o_{\eps}(1)\frac{\big(1+\ol{u}(0)^-\big)\ol{v}(\infty)^-}{n-k}\\
=\big(2+o_\eps(1)\big)
\frac{\cL_n(u)\ol{v}(\infty)^-}{g(n-k)}\,,
\end{equation}
where the first equality follows from~\eqref{e:DGFF-as-add} and~\eqref{e:pf-weak-as}, and the second equality from Lemma~\ref{l:DGFF-lb}.

For general $\eta\in[0,\eps^{-1}]$, we observe that for any $\wt \eps\in(0,\eps^{-1})$, by FKG and Lemma~\ref{p:viol},
\begin{multline}
0\leq \bbP_{k,v}^{n,u}\Big( h_{U^{\eta\vee\wt\eps}_{n}\cap V^{-,\zeta}_k} \leq 0 \Big)
-\bbP_{k,v}^{n,u}\Big(h_{U^\eta_n\cap V^{-,\zeta}_k}\leq 0\Big)\\
\leq \bbP_{k,v}^{n,u}\Big(
\max_{x\in (U^\eta_n\setminus U^{\eta\vee\wt\eps}_{n})} h(x)>0\,,
h_{U^{\eta\vee\wt\eps}_{n}\cap V^{-,\zeta}_{k}} \leq 0 \Big)
\leq O_\eps\big(|\eta\vee\wt\eps-\eta|\big)
\bbP_{k,v}^{n,u}\Big(h_{U^{\eta\vee\wt\eps}_{n}\cap V^{-,\zeta}_{k}} \leq 0 \Big)\,,
\end{multline}
as $\wt\eps\to 0$, where we also used~\eqref{e:uv-as}.
Similarly, by Lemma~\ref{l:delta0},
\begin{equation}
\big|\cL_{n,\eta\vee\wt\eps,U}(u)-\cL_{n,\eta,U}(u)\big|
=O_\eps\big(|\eta\vee\wt\eps-\eta|\big)\cL_{n,\eta\vee\wt\eps,U}(u)\,.
\end{equation}
Replacing $\eps$ by $\wt\eps$, we have already shown that
\begin{equation}
\bbP_{k,v}^{n,u}\Big( h_{U^{\eta\vee\wt\eps}_{n}\cap V^{-,\zeta}_k} \leq 0 \Big)=
\big(2+o_{\wt\eps}(1)\big)\frac{\cL_{n,\eta\vee\wt\eps,U}(u)\ol{v}(\infty)^-}{g(n-k)}
\end{equation}
as $n-k\to\infty$. Combining the above, and taking first $n-k\to\infty$ and then $\wt\eps\to 0$ shows~\eqref{e:c:1.1} for all $\eta\in[0,\eps^{-1}]$ with $o_\eps(1)$ as desired.
\end{proof}

\begin{proof}[Proof of Proposition~\ref{p:LR-infty} (for $\cL_n$)]
With
$\ell_{\rt_n,r_n}\big(\ol{u_n}(0),\ol{u_n}(0)\big)$
defined by~\eqref{e:def-l} for $V=\rmB$, $k=0$, $\zeta=0$, $v=\ol{u}(0)$, we write
\begin{multline}
\label{e:p-LR-infty-DRWDGFF}
\cL_n(u_n)=\cL^{r_{n}}_{n,0}(u_n,\ol{u_n}(0))\\
=\Big(\cL^{r_{n}}_{n,0}(u_n,\ol{u_n}(0))
-\ell_{\rt_n,r_{n}+\lfloor\log\eps\rfloor}\big(\ol{u_n}(0),\ol{u_n}(0)\big)\Big)
+\ell_{\rt_n,r_{n}+\lfloor\log\eps\rfloor}\big(\ol{u_n}(0),\ol{u_n}(0)\big)\,.
\end{multline}
The expression in brackets on the right-hand side converges to zero as $n\to\infty$ by Lemma~\ref{l:LUV} and as $\ol{u}(0)^-$ is bounded by the assumption on the convergence of $u_n$.
The last term on the right-hand side of~\eqref{e:p-LR-infty-DRWDGFF} converges to a limit $\cL(\wh u_\infty)$ by Proposition~\ref{t:ta} which we apply as follows: for any sequence
$(n^{(i)})_{i=1}^\infty$ with $n^{(i)}\to\infty$, we set
$a^{(i)}=\ol{u_{n^{(i)}}}(0)$, $b^{(i)}=\ol{u_{n^{(i)}}}(0)$,
$\rt^{(i)}=\rt_{n^{(i)}}$,
$r^{(i)}=r_{\rt^{(i)}}+\lfloor\log\eps\rfloor$,
$\rt^{(\infty)}=r_{\infty}=\infty$, $a^{(\infty)}=b^{(\infty)}=\wh u_\infty(0)$,
and we define $\big(\cS^{(i)}_j\big)$, $\big(\cD^{(i)}_j\big)$ as in~\eqref{e:deftS}, \eqref{e:defS} and~\eqref{e:Dell}
using $U$, $\eta$, $n=n^{(i)}$, $u=u_{n^{(i)}}$, $V=\rmB$, $k=0$, $v=\ol{u_{n^{(i)}}}(0)$ and $\zeta=0$.
Then, Lemma~\ref{l:SDconv} yields the limit in distribution $(\cS^{(\infty)}_j, \cD^{(\infty)}_j)$ of $\big(\cS^{(i)}_j, \cD^{(i)}_j\big)$,
verifies Assumption~\eqref{e:5.7} and the assumption on stochastic absolute continuity, and shows that
\begin{equation}
\cL(\wh u):=\lim_{r\to\infty}
\bbE\Big(\big(\cS^{(\infty)}_r\big)^-;\, \max_{j=1}^r \cS^{(\infty)}_j + \cD^{(\infty)}_j\leq 0\Big)
\end{equation}
depends only on $\wh u_\infty$ and $\eta$. Assumptions~\ref{i.a1} --~\ref{i.a3} are again verified by Theorem~\ref{t:drw-i}.
\end{proof}

\begin{proof}[Proof of Proposition~\ref{p:1-4} (for $\cL_n$, $\cL$)]
With
$\ell_{\rt_n,r_n}\big(\ol{u}(0),\ol{u}(0)\big)$
defined by~\eqref{e:def-l} for $V=\rmB$, $k=0$, $\zeta=0$, $v=\ol{u}(0)$,
we write
\begin{equation}
\frac{\cL^{r_{n}}_{n,0}(u,\ol{u}(0))}{\ol{u}(0)^-}
=\frac{\cL^{r_{n}}_{n,0}(u,\ol{u}(0))
-\ell_{\rt_n,r_{n}+\lfloor\log\eps\rfloor}\big(\ol{u}(0),\ol{u}(0)\big)}{\ol{u}(0)^-}
+\frac{\ell_{\rt_n,r_{n}+\lfloor\log\eps\rfloor}\big(\ol{u}(0),\ol{u}(0)\big)}{\ol{u}(0)^-}\,,
\end{equation}
Using also the choice of $r_T$, we obtain that the first term on the right-hand side is at most $C_\eps(1+\ol{u}(0))^{\eps-1}$ by Lemma~\ref{l:LUV}, thus the assertion on $\cL_n$ follows from Proposition~\ref{p:1.5}.

To show the assertion on $\cL$, let $\eps'>0$.
Using the assertion on $\cL_n$, we find $M<\infty$ with the following property:
for each $U\in\frU^\eta_\eps$, $\wh{u}(0)\in\bbH(U^\eta)$ with $\wh{u}(0)^->M$, we find $n\geq c_{\eps}$, $u_n\in\bbR^{\partial U_n}$ such that on the right-hand side of
\begin{equation}
\label{e:p-14L13}
\frac{\cL(\wh u)}{\wh u(0)^-}=
\frac{\cL(\wh u)-\cL_n(u_n)}{\wh u(0)^-}
+\frac{\cL_n(u_n)}{\ol{u_n}(0)^-}\frac{\ol{u_n}(0)^-}{\wh u(0)^-}\,,
\end{equation}
the first summand is bounded by $\eps'$ by Proposition~\ref{p:LR-infty}, and both factors in the second summand are in $(1-\eps',1+\eps')$. This shows the assertion on $\cL$.
\end{proof}

\begin{proof}[Proof of Proposition~\ref{p:1-5} (for $\cL_n$, $\cL$)]
To show the assertion on $\cL_n$, we first consider $u\in\bbR^{\partial U_n}$ and $v\in\bbR^{\partial\rmB^-_0}$ that satisfy~\eqref{e:uv-as} and $\ol{v}(\infty)^-\geq n^{\eps}$ with $\zeta=0$.
Then a comparison between Proposition~\ref{c:1.1} and Lemma~\ref{l:DGFF-lb} (where we set $V=\rmB$ and $k=0$) shows that 
\begin{equation}
\label{e:p:1-5:weak}
\cL_{n}(u)>c_{\eps}\big(1+\ol{u}(0)^-\big)
\end{equation}
for sufficiently large $n$, depending only on $\eps$.
As the left-hand side in~\eqref{e:p:1-5:weak} does not depend on $v$, the assertion follows for all $u\in\bbR^{U_n}$ with $\osc\,\ol{u}_\eta<\eps^{-1}$ and $\ol{u}(0)\in (-n^{1-\eps},\eps^{-1})$. The definition of $\cL_n(u)$ in Proposition~\ref{p:1-2} shows that $\cL_n(u)$ can only decrease when we replace $u$ by $u-\ol{u}(0)$  and remove the restriction that $\ol{u}(0)<-n^{1-\eps}$. Hence, $\cL_n(u)>c$ for all $U, n,u$ as in the assertion for some constant $c=c_{\eps}>0$.

To show the assertion on $\cL$, we find for every $U\in\frU^\eta_\eps$, $\wh u \in\bbH(U^\eta)$ by Proposition~\ref{p:LR-infty} some $n\geq c_{\eps}$, $u_n\in\bbR^{\partial U_n}$ with
$\big|\cL(\wh u)-\cL_n(u_n)\big|< c/2$. Using the assertion on $\cL_n$, we then obtain $\cL(\wh u)>c/2$.
\end{proof}
\begin{proof}[Proof of Proposition~\ref{p:1-6} (for $\cL_n$, $\cL$)]

We recall the definition $\cL_n(u)=\cL^{r_{n}}_{n,\eta,U,0,\rmB}(u,\ol{u}(0))$ from~\eqref{e:cLmr}. We write
\begin{multline}
\label{e:p-1-6-1}
\big|\cL^{r_{n}}_{n,\eta,U,0,\rmB}(u,\ol{u}(0))-\cL^{r}_{n,\eta,U,0,\rmB}(u,\ol{u}(0))\big|
\leq
\big|\cL^{r_{n}}_{n,\eta,U,0,\rmB}(u,\ol{u}(0))-\ell_{\rt_n,r_{n}+\lfloor\log\eps\rfloor}(\ol{u}(0),\ol{u}(0))\big|\\
+\big|\ell_{\rt_n,r_{n}+\lfloor\log\eps\rfloor}(\ol{u}(0),\ol{u}(0))
-\ell_{\rt_n,r+\lfloor\log\eps\rfloor}(\ol{u}(0),\ol{u}(0))\big|+
\big|\ell_{\rt_n,r+\lfloor\log\eps\rfloor}(\ol{u}(0),\ol{u}(0))-
\cL^{r}_{n,\eta,U,0,\rmB}(u,\ol{u}(0))\big|\,,
\end{multline}
where we define
$\ell_{\rt_n,r_n}\big(\ol{u}(0),\ol{u}(0)\big)$
by~\eqref{e:def-l} for $V=\rmB$, $k=0$, $\zeta=0$, $v=\ol{u}(0)$.
By Lemma~\ref{l:LUV}, Proposition~\ref{t:LR-as}, and by choice of $r_{n}$,
all terms on the right-hand side of~\eqref{e:p-1-6-1} are
$o_\eps(1)\big(1+\ol{u}(0)^-\big)$ as $n\to\infty$ followed by $r\to\infty$.
Let $\eps'\in(0,1/2)$. Using Proposition~\ref{p:1-5} and the above, we obtain
\begin{equation}
\label{e:p-1-6-p1-5X}
(1-\eps')\cL^{r}_{n,\eta,U,0,\rmB}(u,\ol{u}(0))
\leq \cL^{r_n}_{n,\eta,U,0,\rmB}(u,\ol{u}(0))
\leq (1+\eps')\cL^{r}_{n,\eta,U,0,\rmB}(u,\ol{u}(0))
\end{equation}
for all sufficiently large $n$, $r$ (and all $U$, $\eta$, $u$ as in the assertion).
In particular, using Propositions~\ref{p:1-4} and~\ref{p:1-5}, we obtain a constant $\wt c=c_{r,\eps}>0$ such that
\begin{equation}
\wt c (1+\ol{u}(0)^-)
\leq \cL^r_{n,\eta,U,0,\rmB}(u,\ol{u}(0))
\leq \wt c^{-1}(1+\ol{u}(0)^-)
\end{equation}
and analogously also
\begin{equation}
\wt c (1+\ol{u'}(0)^-)
\leq \cL^r_{n,\eta',U',0,\rmB}(u',\ol{u'}(0))
\leq \wt c^{-1}(1+\ol{u'}(0)^-)
\end{equation}
for sufficiently large $n$ and all $U,u,\eta$ as in the assertion.
Hence, by Lemma~\ref{l:1-6-m}, there exists $c=c_{r,\eps}\in(0,1)$ such that
\begin{equation}
\label{e:p:1-6-1-6-m}
(1-\eps'Y)\cL^{r}_{n,\eta',U',0,\rmB}(u',\ol{u'}(0))
\leq \cL^{r}_{n,\eta,U,0,\rmB}(u,\ol{u}(0))
\leq (1+\eps'Y)\cL^{r}_{n,\eta',U',0,\rmB}(u',\ol{u'}(0))
\end{equation}
where
\begin{equation}
Y:=1+\wt c^{-1} +\wt c^{-2}\frac{1+\ol{u}(0)^-}{1+\ol{u'}(0)^-}\,,
\end{equation}
for all $U,U',u,u',\eta,\eta'$ as in the assertion with
\begin{equation}
\label{e:cond-p-1-6-c}
\dH(U,U')+|\eta'-\eta|+\|\ol{u}-\ol{u'}\|_{U^\eta_n\cap U'^{\eta'}_n}
+n^{-1}\leq c\,,
\end{equation}
for which we note that $Y$ is bounded by a constant (that depends on $\eps,r$).
From~\eqref{e:p:1-6-1-6-m}, \eqref{e:p-1-6-p1-5X}, and the argument leading to~\eqref{e:p-1-6-p1-5X} albeit now for $\cL^{r_n}_{n,\eta',U',0,\rmB}(u',\ol{u'}(0))$, we obtain
\begin{equation}
\label{e:pf-1-6-concl-Ln}
(1-\eps')^2(1-\eps'Y)\cL^{r_n}_{n,\eta',U',0,\rmB}(u',\ol{u'}(0))
\leq \cL^{r_n}_{n,\eta,U,0,\rmB}(u,\ol{u}(0))
\leq (1+\eps')^2(1+\eps'Y)\cL^{r_n}_{n,\eta',U',0,\rmB}(u',\ol{u'}(0))
\end{equation}
whenever $n$ is sufficiently large and~\eqref{e:cond-p-1-6-c} holds. As the choice of $r$ depended only on $\eps$ and $\eps'$, assertion~\eqref{e:p:1-6-L} follows.

To show also the assertion on $\cL_{\eta,U}$ and $\cL_{\eta',U'}$, we find, for given $\wh u,\wh u', U, U'$ as in the assertion, $n$, $u_n\in\bbR^{\partial U_n}$, $u'_n\in\bbR^{\partial U'_n}$ such that
\begin{equation}
(1-\eps')\cL_{\eta,U}(\wh u)
\leq \cL^{r_n}_{n,\eta,U,0,\rmB}(u,\ol{u}(0))
\leq (1+\eps')\cL_{\eta,U}(\wh u')
\end{equation}
and the analogous statement with $\eta',U',\wh u', u'_n$ hold by Proposition~\ref{p:LR-infty}. From the second inequality in~\eqref{e:pf-1-6-concl-Ln} with $u_n, u'_n$ in place of $u,u'$, we then obtain
\begin{multline}
(1-\eps')\cL_{\eta,U}(\wh u)
\leq \cL^{r_n}_{n,\eta,U,0,\rmB}(u_n,\ol{u_n}(0))
\leq (1+\eps')^2(1+\eps'Y)\cL^{r_n}_{n,\eta',U',0,\rmB}(u'_n,\ol{u'_n}(0))\\
\leq (1+\eps')^3(1+\eps'Y)\cL_{\eta',U'}(\wh u')
\end{multline}
and the analogous bound in the other direction for which we use the first inequality in~\eqref{e:pf-1-6-concl-Ln}. This implies the assertion on $\cL_{\eta,U}$, $\cL_{\eta',U'}$.
\end{proof}

\subsection{Proofs of auxiliary results}
\label{s:pf-aux}
First we show that the distribution of the maximum of the DGFF on a subset of the domain is absolutely continuous with respect to Lebesgue measure. We need  this property to prove Lemmas~\ref{l:cont-V} and~\ref{l:1-6-m}.
\begin{lem}
\label{l:h-cont}
Let $r\geq 0$, $\eps\in(0,1)$, $\eta\in[\eps,\eps^{-1}]$, $\zeta\in [0,\eps^{-1}]$. There exists $C=C_{r,\eps}<\infty$ such that
\begin{equation}
\label{e:h-cont}
\bbP_{V,k,v}^{U,n,u}\Big(\max_{U^\eta_n\cap \rmB^-_{n-r}} h \in [-\lambda,\lambda]\Big)\leq C \lambda(1+\lambda) \rme^{\alpha(\lambda+\ol{u}(0))}\big(1+\ol{u}(0)^-\big)
\end{equation}
for all $\lambda\geq 0$, $U\in\frU^\eta_\eps$, $V\in\frV_\eps$, $n,k\geq 0$ with $\partial\rmB_{n-r}\subset U^{\eta}_n$, $\partial\rmB^-_{n-r-1}\subset V^{-,\zeta\vee\eps}_k$, and all $u\in\bbR^{\partial U_n}$, $v\in\bbR^{\partial V^-_k}$ satisfying~\eqref{e:uv-diff}.
\end{lem}
\begin{proof}
Let $\frg:U_n\cap V^-_k\to\bbR$ be the harmonic function with boundary values $u$ on $\partial U_n$ and $v$ on $\partial V^-_k$. Then, by definition of the DGFF, the superposition principle for the harmonic extension, and by a union bound, we have
\begin{multline}
\label{e:h-cont-p-UB}
\bbP_{V,k,v}^{U,n,u}\Big(\max_{x\in U^\eta_n\cap\rmB_{n-r}} h  \in [-\lambda,\lambda]\Big)\\
\leq\sum_{x\in U^\eta_n\cap\rmB^-_{n-r}}
\bbP_{V,k,0}^{U,n,0}\Big(h(x) +\frg \in [-\lambda,\lambda], \max_{y\in U^\eta_n\cap\rmB^-_{n-r}}h(y)+\frg(y)\leq h(x)+\frg(x) \Big)\,.
\end{multline}
Using the Gibbs-Markov property and monotonicity, we obtain
\begin{multline}
\label{e:h-cont-p-GM}
\bbP_{V,k,0}^{U,n,0}\Big(h(x)+\frg  \in [-\lambda,\lambda],
\max_{y\in U^\eta_n\cap\rmB^-_{n-r}} h(y)+\frg(y)\leq h(x)+\frg(x) \Big)\\
\leq \int_{w\in\bbR^{\partial\rmB^-_{n-r-1}}}
\bbP_{V,k,0}^{U,n,0}\Big(
h_{\partial \rmB^-_{n-r-1}}+m_{n-r}\in\rmd w\Big)
\bbP^{U,n,0}_{\rmB,n-r-1,w}\Big(h(x)+\frg(x)  \in [-\lambda,\lambda]\Big)\\
\times\bbP^{U,n,0}_{\rmB,n-r-1,w}\Big(\max_{y\in U^\eta_n\cap\rmB^-_{n-r}} h(y)+\frg(y) \leq \lambda\,\Big|\, h(x) +\frg(x) = -\lambda\Big)\,.
\end{multline}
Using Assumption~\eqref{e:uv-diff} and Lemmas~\ref{l:ruin} and~\ref{l:UV-U}, we obtain
\begin{equation}
\label{e:h-cont-p-u}
\max_{x\in U^\eta_n\cap\rmB^-_{n-r}}\big|
\frg(x)-\ol{u}(0)\big|\leq C_{r,\eps}\,.
\end{equation}
By considering the domains $U\cap \rme^{-r-1}\rmB-\rme^{-n}x$ and $\tfrac{1}{10}\rmB$ in place of $U$ and $V$, we may apply Theorem~\ref{t:2.3} with $n,0$ in place of $n,k$ to the last conditional probability in~\eqref{e:h-cont-p-GM} so as to bound it from above by a constant times
$(1+\lambda)(1+\ol{u}(0)^-)(1+\ol{w}(\infty)^-+\osc_{\rmB^-_{n-r}}\ol{w})/n$.

As we show in~\eqref{e:l:2.4:pVar}, we have
\begin{equation}
gn -C_{r,\eps}\leq \Var^{U,n}_{\rmB,n-r-1} h(x)
=G_{U_n\cap\rmB^-_{n-r-1}}(x,x)\leq gn + C_{r,\eps}
\end{equation}
for all $x\in U^\eta_n\cap\rmB^-_{n-r}$.
Let $\mu(x)=m_n+\bbE_{\rmB,n-r-1,w}^{U,n,0} h(x) +\frg$.
Then, as $h$ under $\bbP_{\rmB,n-r-1,w}^{U,n,0}$ is Gaussian,
\begin{equation}
\bbP_{\rmB,n-r-1,w}^{U,n,0}
\Big( h(x) +\frg \in[-\lambda,\lambda] \Big)
\leq C_{r,\eps}n^{-1/2}\lambda 
\sup_{s\in[-\lambda,\lambda]}
\rme^{-\frac{(s-\mu(x)+m_n)^2}{2gn + C_{r,\eps}}}\,.
\end{equation}
Moreover, using that $m_n=2\sqrt{g}n-\tfrac34\sqrt{g}\log^+ n$, $\alpha=2/\sqrt{g}$ and  binomial expansion (as in~(5.13) of~\cite{Crit}) yields
\begin{equation}
\rme^{-\frac{(s-\mu(x)+m_n)^2}{2gn + C_{r,\eps}}}\leq
C_{r,\eps}n^{3/2}
\rme^{-2n-\alpha(s-\mu(x))}\,.
\end{equation}
Using also~\eqref{e:h-cont-p-u} and that
\begin{equation}
\bbE_{\rmB,n-r-1,w}^{U,n,0} h(x)
\leq -m_n +|\ol{w}(\infty)| + \osc_{\rmB^-_{n-r}}\ol{w}+C_{r,\eps}
\end{equation}
by Lemmas~\ref{l:UV-U} and~\ref{l:ruin}, we obtain
\begin{equation}
\bbP_{\rmB,n-r-1,w}^{U,n,0}
\Big( h(x)+\frg \in[-\lambda,\lambda] \Big)
\leq C_{r,\eps} n \lambda
\rme^{-2n+\alpha\big(\lambda+\ol{u}(0) +|\ol{w}(\infty)| + \osc_{\rmB^-_{n-r}}\ol{w}\big)}\,.
\end{equation}
Hence, the right-hand side of~\eqref{e:h-cont-p-GM} is bounded from above by
\begin{multline}
\label{e:h-cont-p-GM-2}
C_{r,\eps}\int_{t\in\bbR}
\bbP_{V,k,0}^{U,n,0}\Big(
\he{h}{\partial \rmB^-_{n-r-1}}(x)+m_{n-r}\in\rmd t\Big)\\
\times\sum_{j=1}^\infty\bbP_{V,k,0}^{U,n,0}\Big(\osc_{\rmB^-_{n-r}}\he{h}{\partial \rmB^-_{n-r-1}}\geq j-1
\,\Big|\,\he{h}{\partial \rmB^-_{n-r-1}}(x)+m_{n-r}= t\Big)\\
\times\lambda(1+\lambda)
\rme^{-2n+\alpha\big(\lambda+\ol{u}(0) +|t| + j\big)}(1+\ol{u}(0)^-)(1+t^-+j)\,.
\end{multline}
By Proposition~\ref{p:Cov-bd},
\begin{equation}
\big|\Var_{V,k,v}^{U,n,u}h(x)-gn\big|\leq C_{r,\eps}\,.
\end{equation}
We write $\bar\mu(x)=m_{n-r}+\bbE_{V,k,0}^{U,n,0}\he{h}{\partial \rmB^-_{n-r-1}}(x)$.
As $\he{h}{\partial \rmB^-_{n-r-1}}(x)$ is Gaussian, we have
\begin{equation}
\label{e:h-cont-p-GM-2-r}
\bbP_{V,k,0}^{U,n,0}\Big(
\he{h}{\partial \rmB^-_{n-r-1}}(x)+m_{n-r}\in\rmd t\Big)/\rmd t
\leq C_{r,\eps} \rme^{-c_{r,\eps}(t -\bar\mu(x))^2}\,,
\end{equation}
where, by Lemma~\ref{l:blml} and Proposition~\ref{l:E},
we have
$|\bar\mu(x)|\leq C_{r,\eps}$.	
Furthermore, by Proposition~\ref{p:2.6},
\begin{equation}
\label{e:h-cont-p-GM-2-j}
\bbP_{V,k,0}^{U,n,0}\Big(\osc_{\rmB^-_{n-r}}\he{h}{\partial \rmB^-_{n-r-1}}\geq j-1
\,\Big|\,\he{h}{\partial \rmB^-_{n-r-1}}(x)+m_{n-r}= t\Big)
\leq C_{r,\eps}\rme^{2\alpha(-j+|t|)}\,.
\end{equation}
Plugging in~\eqref{e:h-cont-p-GM-2-r} and~\eqref{e:h-cont-p-GM-2-j} into~\eqref{e:h-cont-p-GM-2}, and using~\eqref{e:h-cont-p-UB} shows that
\begin{equation}
\bbP_{V,k,v}^{U,n,u}\Big(\max_{x\in U^\eta_n\cap\rmB_{n-r}} h  \in [-\lambda,\lambda]\Big)\\
\leq C_{r,\eps}\big|\{x\in U^\eta_n\cap\rmB^-_{n-r}\}\big|
\lambda(1+\lambda)\rme^{-2n+\alpha(\lambda+\ol{u}(0))}\big(1+\ol{u}(0)^-\big)\,,
\end{equation}
which implies the assertion.
\end{proof}

The next lemma bounds the probability that $h$ is non-positive outside of an intermediately scaled disk. This estimate will be used in the proof of Lemma~\ref{l:LUV} below.
\begin{lem}
\label{l:half}
Let $\eps\in(0,1)$ and $\eta,\zeta\in[0,\eps^{-1}]$.
There exists $C=C_{\eps}<\infty$ such that
\begin{equation}
\label{e:halfP}
\bbP_{k,v}^{n,u}\Big( h_{U^\eta_n\cap \rmB^-_{n-r} } \leq 0 \Big)
\leq C\big(1+\max\ol{u}_\eta^-\big)^2 r^{-1/2}\,,
\end{equation}
for all $U\in\frU^\eta_\eps$, $V\in\frV_\eps$, $n\geq 0$,
$r\in(0,(n-k)^\eps]$ such that $\partial\rmB_{n-r}\subset U^{\eta\vee\eps}_n$ and $\partial\rmB^-_{n-r}\subset V^{-,\zeta\vee\eps}_k$,
and all $u\in\bbR^{\partial U_n}$, $v\in\bbR^{\partial V^{-}_k}$
satisfying~\eqref{e:uv-diff}.
\end{lem}
\begin{proof}
As $h$ under $\bbP_{k,v}^{n,u}$ is distributed as $h+\bbE_{k,v}^{n,u} h - \bbE_{k,0}^{n,0} h$ under $\bbP_{k,0}^{n,0}$, and as
\begin{equation}
\big|\bbE_{k,v}^{n,u}\he{h}{\partial \rmB_{n-r}}(0)
-\bbE_{k,0}^{n,0}\he{h}{\partial \rmB_{n-r}}(0)\big|
\leq C_\eps+\max_{U^\eta_n}\ol{u}^-
\end{equation}
by Proposition~\ref{l:E} and~\eqref{e:uv-diff}, we have
\begin{equation}
\label{e:halfP0}
\bbP_{k,v}^{n,u}\Big( h_{U^\eta_n\cap \rmB^-_{n-r} } \leq 0 \Big)\leq
\bbP_{k,0}^{n,0}\Big( h_{U^\eta_n\cap \rmB^-_{n-r} } -C_\eps-\max_{U^\eta_n}\ol{u}^- \leq 0 \Big)\,.
\end{equation}
The right-hand side of~\eqref{e:halfP0} is bounded by
\begin{multline}
\label{e:half-P-p}
\int_{t\in\bbR}\bbP_{k,0}^{n,0}\Big(\he{h}{\partial\rmB_{n-r}}(0)+m_{n-r} \in\rmd r\Big)
\sum_{a=1}^\infty
\bbP_{k,0}^{n,0}\Big(\osc_{\rmB^{\pm,\eps}_{n-r}}\he{h}{\partial\rmB_{n-r}}\geq a \,\Big|\, \he{h}{\partial\rmB_{n-r}}(0)+m_{n-r} = t\Big)\\
\times\bbP_{\rmB,n-r, t -a}^{U,n,0}\big(h_{U^\eta_n\cap \rmB^-_{n-r}}
-C_\eps-\max_{U^\eta_n}\ol{u}^-\leq 0\big)\,,
\end{multline}
where we also used that the probability in the second line is increasing in $a$. The conditional probability in~\eqref{e:half-P-p} is bounded by $\rme^{-c_\eps a+C_\eps|t|/r}$ by Proposition~\ref{p:2.6}, and
Theorem~\ref{t:2.3} yields
\begin{equation}
\bbP_{\rmB,n-r, t -a}^{U,n,0}\big(h_{U^\eta_n\cap \rmB^-_{n-r}}
-C_\eps-\max_{U^\eta_n}\ol{u}^-\leq 0\big)
\leq C_{\eps}\frac{(1+\max_{U^\eta_n}\ol{u}^-)(1+\max_{U^\eta_n}\ol{u}^-+ t^- +a)}{r}\,.
\end{equation}
Furthermore,
$\he{h}{\partial\rmB_{n-r}}(0)$ is a Gaussian with
$\Var_k^n\he{h}{\partial\rmB_{n-r}}(0)\leq gr + C_\eps$ by Proposition~\ref{p:Cov-bd} and
$|\bbE_{k,0}^{n,0}\he{h}{\partial\rmB_{n-r}}(0)+m_{n-r}|\leq C_\eps(1+\log r)$
by Lemma~\ref{l:blml} and Proposition~\ref{l:E}, hence
\begin{equation}
\bbP_{k,0}^{n,0}\Big(\he{h}{\partial\rmB_{n-r}}(0)+m_{n-r}\in\rmd t\Big)
\leq C_\eps r^{-1/2}\rme^{-c_\eps t^2r^{-1}}\rmd t
\end{equation}
Plugging in these bounds into~\eqref{e:half-P-p}, changing variables such that $tr^{-1/2}$ becomes $t$, and using~\eqref{e:halfP0} yields~\eqref{e:halfP}.
\end{proof}

We are now ready to prove the remaining statements from Section~\ref{s:comp-ballot}.
\begin{proof}[Proof of Lemma~\ref{l:LUV}]
We go back to the notation that was introduced in Sections~\ref{s:concdec} and~\ref{s:def-DRW}.
In particular, we keep assuming the coupling of $h^{U_n\cap V^-_k}$, $h$ and $(\varphi_p,h_p)_{p=1}^{\rt}$ and~\eqref{e:hphih} hold $\bbP_{k,v}^{n,u}$-a.\,s.
Set $\wt r=r+\lfloor \log\eps\rfloor$. Then $\rmB^-_{n-r}=\Delta'_{\wt r}$ for $\wt r=1,\ldots,\rt -1$ by~\eqref{e:Delta'}, and by Theorem~\ref{t:drw-i} also
\begin{equation}
\Big\{ \max_{i=1}^{\wt r} \big( \cS'_i + \beta_i + \cD_i \big) \leq 0 \Big\}
= \Big\{ h_{U^\eta_n \cap \rmB^-_{n-r}} \leq 0 \Big\}\,.
\end{equation}
Using also Definitions~\eqref{e:def-l} and~\eqref{e:def-cLm}, we obtain for any $z\in\partial\rmB_{n-2r}$ that
\begin{multline}
\label{e:LUV-p1}
\Big| \cL^{r}_{n,k}(u,v) -
\ell_{\rt,\wt r}\big(\ol{u}(0),\ol{v}(\infty)\big) \Big|
\leq \bbE_{k,v}^{n,u}\Big(\big|\cS'_{\wt r} - \cY_{\wt r}\big|\Big)
+\bbE_{k,v}^{n,u}\Big(\big|\cY_{\wt r}-\varphi_{0,\wt r}(z)\big|;\,h_{U_n^\eta \cap \rmB^-_{n-r}} \leq 0\Big)\\
+\bbE_{k,v}^{n,u}\Big(\big|\varphi_{0,\wt r}(z) - \he{h^{U_n\cap V^-_k}}{\partial \rmB^-_{n-r}}(0)\big|\Big)
+ \big|\he{\gamma}{\partial \rmB^-_{n-r}}(0) \big|
\bbP_{k,v}^{n,u}\Big(h_{U_n^\eta \cap \rmB^-_{n-r}} \leq 0 \Big)\,,
\end{multline}
where
\begin{equation}
\gamma(x) =   m_{n-\wt r} - \beta_{\wt r} + \he{(-m_n \Ind_{\partial U_n}+ u - m_k \Ind_{\partial V^-_k} + v)}{\partial U_n\cup \partial V^-_k}(x)
\end{equation}
is defined as in~\eqref{e:corr-bd} and we have
$\big|\he{\gamma}{\partial \rmB^-_{n-r}}(0)\big|\leq C_{\eps,\delta}(1+r^\delta)$
by~\eqref{e:bd-gamma} for an arbitrarily small constant $\delta>0$.
The first term on the right-hand side of~\eqref{e:LUV-p1} is bounded from above by
$\bbE_{k,v}^{n,u}\Big( \big(\cS'_{\wt r} - \cY_{\wt r}\big)^2 \Big)^{1/2}$ by the Jensen inequality.
This expression is further bounded by 
a constant times $(n-k-{\wt r})^{-1/2}$ by Lemma~\ref{l:dec-coupl} and as $\cS'_{\wt r}$, $\cY_{\wt r}$ are centered Gaussians.
The probability in the second line of~\eqref{e:LUV-p1} is bounded by its $\eps/4$-th power, which in turn is bounded by a constant times $\big(1+\max_{U^\eta_n}\ol{u}^-\big)^{\eps/2} {\wt r}^{-\eps/8}$ by Lemma~\ref{l:half}.

The third term on the right-hand side of~\eqref{e:LUV-p1} is, by the Jensen inequality, bounded from above by the square root of
\begin{multline}
\label{e:p-l-LUV-hphi}
\bbE_{k,v}^{n,u}\Big(\big(\varphi_{0,\wt r}(z) - \he{h^{U_n\cap V^-_k}}{\partial \rmB^-_{n-r}}(0)\big)^2\Big)=
\sum_{w,w\in\partial\Delta_{\wt r}}
\big|\Pi_{\Delta_{\wt r}\cap V^-_k}(z,w)-\Pi_{\Delta_{\wt r}}(0,w)\big|\\
\times\big|\Pi_{\Delta_{\wt r}\cap V^-_k}(z,w')-\Pi_{\Delta_{\wt r}}(0,w')\big|
\bbE\Big(h^{U_n\cap V^-_k}(w)h^{U_n\cap V^-_k}(w')\Big)\,,
\end{multline}
for the last equality we used that $\varphi_{0,{\wt r}}(z)=\he{h^{U_n\cap V^-_k}}{\partial \rmB^-_{n-r}\cup\partial V^-_k}(z)$, \eqref{e:phi-rw}, and the analogous representation of the harmonic extension $\he{h^{U_n\cap V^-_k}}{\partial \rmB^-_{n-r}}(0)$ in terms of simple random walk.
By Lemma~\ref{l:Poisson} with $j=r$ and $p=2r$,
\begin{equation}
\big|\Pi_{\Delta_{\wt r}\cap V^-_k}(z,w)-\Pi_{\Delta_{\wt r}}(0,w)\big|
\leq C_\eps \Big(\frac{r^2}{n-k-r}+r^{-2}\Big)\Pi_{\Delta_{\wt r}}(0,w)
\end{equation}
in~\eqref{e:p-l-LUV-hphi}. Hence,~\eqref{e:p-l-LUV-hphi} is bounded by
\begin{equation}
C_\eps \Big(\frac{r}{n-k-r}+r^{-1}\Big)^2
\Var\Big(\he{h^{U_n\cap V^-_k}}{\partial B_{n-r}}(0)\Big)\,,
\end{equation}
where the variance is bounded by $C_\eps r$ by Proposition~\ref{p:Cov-bd}.

The second term on the right-hand side of~\eqref{e:LUV-p1} is bounded by
\begin{equation}
\bbE_{k,v}^{n,u}\Big(\big(\cY_{\wt r}-\varphi_{0,\wt r}(z)\big)^2\Big)^{1/2}\bbP_{k,v}^{n,u}\Big(h_{U_n^\eta \cap \rmB^-_{n-r}} \leq 0 \Big)^{1/2}
\end{equation}
by the Cauchy-Schwarz inequality. By the above, it suffices to show that the second moment in this expression is bounded by some $C_\eps<\infty$.
By~\eqref{e:R}, \eqref{e:Xm}, \eqref{e:phi-rw}, and the definition of the binding field $\varphi_j$ from Lemma~\ref{l:GM}, we can set
\begin{equation}
\label{e:p-l-LUV-Y}
\cY_{\wt r}=\sum_{j=1}^{\wt r}\sum_{y\in\partial \Delta_j}
\frac{s_{{\wt r},\rt}}{s_{j,\rt}}
\Pi_{\Delta_j}(0,y)
\sum_{w\in\partial \Delta_{j-1}}\Pi_{\Delta_{j-1}\cap V^-_k}(y,w)
h^{\Delta_{j-1}\cap V^-_k}(w)\,,
\end{equation}
where the DGFFs $h^{\Delta_{j-1}\cap V^-_k}$, $j=1,\ldots {\wt r}$ are independent. Likewise, using the definition of $\varphi_{0,{\wt r}}$ and again~\eqref{e:phi-rw}, we have
\begin{equation}
\label{e:p-l-LUV-phi}
\varphi_{0,{\wt r}}(z)=\varphi_0(z)+
\sum_{j=1}^{\wt r}\sum_{y\in\partial \Delta_j}
\Pi_{\Delta_j\cap V^-_k}(z,y)
\sum_{w\in\partial \Delta_{j-1}}\Pi_{\Delta_{j-1}\cap V^-_k}(y,w)
h^{\Delta_{j-1}\cap V^-_k}(w)
\end{equation}
with the same realizations of $h^{\Delta_{j-1}\cap V^-_k}$ as in~\eqref{e:p-l-LUV-Y}, which are independent of $\varphi_0$. As in~\eqref{e:dec-tail-Pois} and~\eqref{e:p-dectail-lPoisson} with $p=2r$, we have
\begin{equation}
\label{e:p-l-LUV-diff}
\Big|\Pi_{\Delta_j\cap V^-_k}(z,y)
-\frac{s_{{\wt r},\rt}}{s_{j,\rt}}
\Pi_{\Delta_j}(0,y)
\Big|\leq C_\eps\frac{\log r}{r}\Pi_{\Delta_j}(0,y)\,.
\end{equation}
By subtracting \eqref{e:p-l-LUV-Y} from \eqref{e:p-l-LUV-phi} and then using~\eqref{e:p-l-LUV-diff}, we obtain
\begin{equation}
\bbE_{k,v}^{n,u}\Big(\big(\cY_{\wt r}-\varphi_{0,\wt r}(z)\big)^2\Big)
\leq \Var_k^n\varphi_0(z) + C_\eps\Big(\frac{\log r}{r}\Big)^2\sum_{j=1}^{\wt r}
\Var\Big(\he{h^{\Delta_{j-1}\cap V^-_k}}{\partial \Delta_j}(0)\Big)\,,
\end{equation}
which is bounded by a constant as $\Var_k^n\varphi_0(z)$ and the other variances on the right-hand side are bounded by a constant by Lemma~\ref{l:varXm} and Proposition~\ref{p:Cov-bd}, respectively. This yields the assertion.
\end{proof}

\begin{proof}[Proof of Lemma~\ref{l:delta0}]
We can assume that $W\subset W'$: if this is not the case, we compare both $\cL^r_{n,W,U,k,V}(u,v)$ and $\cL^r_{n,W',U,k,V}(u,v)$ to $\cL^r_{n,W\cap W',U,k,V}(u,v)$, using that $W,W'\supset W\cap W'$.

By the definition of $\cL^r_{n,W,U,k,V}(u,v)$ in~\eqref{e:def-Lmuv}, monotonicity and linearity of the expectation therein, and by FKG, we have
\begin{multline}
0\leq \cL^r_{n,W,U,k,V}(u,v)-\cL^r_{n,W',U,k,V}(u,v)\\
=\bbE_{V,k,v}^{U, n,u} \Big( \ol{h}_{\rmB^-_{n-r}}(0)^- + m_{n-r} ;\;
h_{W_n \cap \rmB^-_{n-r}} \leq 0 \,,
\max_{x\in W'_n\setminus W_n} h(x) > 0 \Big)\\
\leq \cL^r_{n,W,U,k,V}(u)
\bbP_{V,k,v}^{U,n,u}\Big(\max_{x\in W'_n\setminus W_n} h(x) > 0\Big)
\end{multline}
and we bound the probability on the right-hand side using Lemma~\ref{p:viol}. This yields the assertion.
\end{proof}

A key for the proof of Lemma~\ref{l:cont-V}, and for the proof of the continuity statement in Lemma~\ref{l:1-6-m}, is the following lemma, for which we define
\begin{equation}
\label{e:LmUV}
\cL^r_{n,k}(u,v,t)
=\bbE_{k,v}^{n,u}
\Big( \big(\he{h}{\rmB^-_{n-r}}(0) + t + m_{n-r}\big)^-;\;
h_{U^\eta_n\cap \rmB^-_{n-r}} + t\leq 0
\Big)
\end{equation}
for $t\in\bbR$ and all other parameters as before.
\begin{lem}
\label{l:LUVt}
Let $r\geq 0$, $\eps\in(0,1)$, $\eta\in[\eps,\eps^{-1}]$, $\zeta\in[0,\eps^{-1}]$. There exists $C=C_{r,\eps}<\infty$ such that for all
$U\in\frU^\eta_\eps$, $V\in\frV_\eps$, $0\leq k<n$ with $\partial\rmB_{n-r}\subset U^{\eta}_n$ and $\partial\rmB^-_{n-r}\subset V^{-,\zeta\vee\eps}_k$, all $u\in\bbR^{\partial U_n}$, $v\in\bbR^{\partial V^-_k}$ satisfying~\eqref{e:uv-diff},
and all $t\geq 0$, we have
\begin{equation}
\label{e:LUVt-pos}
\cL^r_{n,k}(u,v,t)
\leq \cL^r_{n,k}(u,v)\\
\leq	\cL^r_{n,k}(u,v,t)
+ C t^{1/2}\rme^{\alpha t}
\end{equation}
and
\begin{equation}
\label{e:LUVt-neg}
\cL^r_{n,k}(u,v,-t)- C t^{1/2}\rme^{\alpha t}
\leq \cL^r_{n,k}(u,v)\\
\leq	\cL^r_{n,k}(u,v,-t)\,.
\end{equation}
\end{lem}
\begin{proof}
The first inequality in~\eqref{e:LUVt-pos} follows by monotonicity in $t$ of the expectation in~\eqref{e:LmUV}.
For the second inequality in~\eqref{e:LUVt-pos}, we note that
\begin{multline}
\bbE_{k,v}^{n,u}
\Big( \big(\he{h}{\rmB^-_{n-r}}(0) + t + m_{n-r}\big)^-;\;
h_{U^\eta_n\cap \rmB^-_{n-r}} + t\leq 0
\Big)\\
\geq
\bbE_{k,v}^{n,u}
\Big( \big(\he{h}{\rmB^-_{n-r}}(0) + m_{n-r}\big)^-;\;
h_{U^\eta_n\cap \rmB^-_{n-r}}\leq 0
\Big)\\
-\bbE_{k,v}^{n,u}\Big( \big(\he{h}{\rmB^-_{n-r}}(0) + m_{n-r}\big)^-;\;
\max_{U^\eta_n\cap \rmB^-_{n-r}} h\in(-t,0]
\Big)
-t\,.
\end{multline}
Therefore, by Cauchy-Schwarz,
\begin{multline}
\label{e:LUVt-p2}
\cL^r_{n,k}(u,v) - \cL^r_{n,k}(u,v,t)\\
\leq t
+ \Big(\bbE_{V,k,v}^{U,n,u}\Big( \big(\he{h}{\partial \rmB^-_{n-r}}(0) + m_{n-r}\big)^2 \Big)\Big)^{1/2}
\bbP_{V,k,v}^{U,n,u}\Big(\max_{U^\eta_n\cap \rmB^-_{n-r}} h\in(-t,0]\Big)^{1/2}\,.
\end{multline}
The probability on the right-hand side is bounded by $C_{r,\eps}\rme^{\alpha\ol{u}(0)}(1+\ol{u}(0)^-)t\rme^{2\alpha t}$ by Lemma~\ref{l:h-cont}.
The expectation on the right-hand side of~\eqref{e:LUVt-p2} is bounded by
$C_{r,\eps}(1+\ol{u}(0)^2)$ by Propositions~\ref{l:E}, \ref{p:Cov-bd}, Lemma~\ref{l:blml} and Assumption~\eqref{e:uv-diff}.
Hence, \eqref{e:LUVt-p2} is bounded by $t+C_{r,\eps}t^{1/2}\rme^{\alpha t}$ which yields assertion~\eqref{e:LUVt-pos} when we consider large and small $t$ separately. The proof of~\eqref{e:LUVt-neg} is analogous.
\end{proof}

\begin{proof}[Proof of Lemma~\ref{l:cont-V}]
We show~\eqref{e:cont-V-assertion} with $\cL^r_{n,\eta,U,0,\eps\rmB}(u,\ol{u}(0))$ in place of $\cL^r_{n,\eta,U,0,\rmB}(u,\ol{u}(0))$. Then the assertion of the lemma follows as
\begin{multline}
\big|\cL^r_{n,\eta,U,0,\rmB}(u,\ol{u}(0))-\cL^r_{n,\eta,U,k,V}(u,v) \big|\\
\leq
\big|\cL^r_{n,\eta,U,0,\rmB}(u,\ol{u}(0))-\cL^r_{n,\eta,U,0,\eps\rmB}(u,\ol{u}(0))\big|
+\big|\cL^r_{n,\eta,U,0,\eps\rmB}(u,\ol{u}(0))
-\cL^r_{n,\eta,U,k,V}(u,v) \big|\,.
\end{multline}

By Definition~\eqref{e:def-UV} of $\frV_\eps$, we have
$(\eps\rmB^-)_0\supset V^-_k$.
To account for the boundary values, let $\frg$ denote the harmonic function on $U_n\cap V^-_k$ with boundary values $-m_n+u$ on $\partial U_n$ and $-m_k+v$ on $\partial V^-_k$, and let $\wt \frg$ denote the harmonic function on $U_n\cap (\eps\rmB^-)_0$ with boundary values $-m_n+u$ on $\partial U_n$ and $\ol{u}(0)$ on $\partial (\eps\rmB^-)_0$.
Then, by the Gibbs-Markov property,
$(h-\wt\frg)_{U_n\cap V^-_k}$ under $\bbP_{\eps\rmB,0,\ol{u}(0)}^{U,n,u}$ is distributed as the sum of the independent fields $h-\frg$ under $\bbP_{V,k,v}^{U,n,u}$ and
$\varphi^{U_n\cap (\eps\rmB^-)_0, U_n\cap V^-_k}$, where the latter is the usual binding field.
Using also Definitions~\eqref{e:def-cLm} and~\eqref{e:LmUV}, in particular the monotonicity in $t$, we obtain
\begin{equation}
\label{e:cont-v-p-int-lb}
\cL^r_{n,\eta,U,0,\eps\rmB}(u,\ol{u}(0))\geq
\int\bbP\Big(
\max_{U^\eta_n\cap \rmB^-_{n-r}}
\Big|\varphi^{U_n\cap (\eps\rmB^-)_0, U_n\cap V^-_{k}}+\wt\frg-\frg\Big|\in\rmd t\Big)
\cL^r_{n,\eta,U,k,V}(u,v,t)
\end{equation}
and
\begin{equation}
\label{e:cont-v-p-int-ub}
\cL^r_{n,\eta,U,0,\eps\rmB}(u,\ol{u}(0))\leq
\int\bbP\Big(
\max_{U^\eta_n\cap \rmB^-_{n-r}}
\Big|\varphi^{U_n\cap(\eps\rmB^-)_0, U_n\cap V^-_k}+\wt\frg-\frg\Big|\in\rmd t\Big)
\cL^r_{n,\eta,U,k,V}(u,v,-t)\,.
\end{equation}
In light of Assumption~\eqref{e:uv-diff}, Lemma~\ref{l:gV} below gives
\begin{equation}
\label{e:cont-v-p-g}
\max_{U^\eta_n\cap \rmB^-_{n-r}}|\wt \frg - \frg|\leq
r(n-k)^{-\eps}\,.
\end{equation}
Applying Lemma~\ref{l:LUVt} in~\eqref{e:cont-v-p-int-lb} and~\eqref{e:cont-v-p-int-ub}, we obtain
\begin{multline}
\label{e:pf-l-cont-V}	
\big|\cL^r_{n,\eta,U,0,\eps\rmB}(u,\ol{u}(0))-\cL^r_{n,\eta,U,k,V}(u,v)\big|
\leq C_{r,\eps}
\Big(\bbE\Big(
\max_{U^\eta_n\cap \rmB^-_{n-r}}
\big| \varphi^{U_n\cap (\eps\rmB^-)_0, U_n\cap V^-_k}
+ \wt \frg - \frg \big| \Big)\Big)^{1/2}\\
+\sum_{t=1}^\infty \bbP\Big(
\max_{U^\eta_n\cap \rmB^-_{n-r}}
\Big|\varphi^{U_n\cap(\eps\rmB^-)_0, U_n\cap V^-_k}+\wt\frg-\frg\Big|> t\Big)C_{r,\eps}\rme^{C_{r,\eps}t}\,,
\end{multline}
where we also used Jensen's inequality in the expectation on the right-hand side.
From this, \eqref{e:cont-v-p-g} and Lemma~\ref{l:VV-fluct} (with $(\eps\rmB)^-_0$ in place of $\rmB^-_0$), we obtain~\eqref{e:cont-V-assertion} with $\eps\rmB$ in place of $\rmB$, as required.
\end{proof}
\begin{lem}
\label{l:gV}
Let $\eps\in(0,1)$. Then,
\begin{equation}
\label{e:bdg-harm}
\max_{U_n\cap \rmB^-_{n-r}}|\frg-\wt\frg|\leq (n-k)^{-\eps}r
\end{equation}
for all $\eta,\zeta\in[0,\eps^{-1}]$, $U\in\frU^\eta_\eps$, $V\in\frV_\eps$, $0\leq k\leq n$,
$u\in\bbR^{\partial U_n}$, $v\in\bbR^{\partial V^-_k}$ satisfying~\eqref{e:uv-diff},
all $r\geq 0$ with $\rmB^-_{n-r}\subset V^{-,\zeta}_k$, and for $\frg$, $\wt\frg$ defined as in the proof of Lemma~\ref{l:cont-V}.
\end{lem}
\begin{proof}
Let $\frg^0$ be defined as $\frg$ although with $\ol{u}(0)1_{\partial V^-_k}$ in place of $v$. Then, by harmonicity,
$\frg=\frg^0+\heb{v-\ol{u}(0)1_{\partial V^-_k}}{\partial U_n\cup\partial V^-_k}$.
By Lemmas~\ref{l:osc-far} and~\ref{l:ruin} and Assumption~\eqref{e:uv-diff},
\begin{multline}
\max_{U_n\cap\rmB^-_{n-r}} \big|\heb{v-\ol{u}(0)1_{\partial V^-_k}}{\partial U_n\cup\partial V^-_k}\big|
\leq C_\eps \max_{x\in U_n\cap\rmB^-_{n-r}} P_x\big( \tau^{V^-_k} < \tau^{U_n} \big)
\big[ (n-k)^{1-\eps} + \eps^{-1} \big]\\
\leq C_\eps (n-k)^{-\eps}r\,.
\end{multline}
It therefore suffices to prove the lemma with $\frg^0$ in place of $\frg$. 

As $\frg^0 - \wt\frg$ is harmonic on $U_n\cap V^-_k$, and as $\wt\frg$ is harmonic on $U_n\cap (\eps\rmB^-)_0$, we have
\begin{multline}
\label{e:proof-gV-Markov}
\max_{U_n\cap \rmB^-_{n-r}}|\frg^0-\wt\frg|
\leq \max_{x\in U_n\cap \rmB^-_{n-r}}\sum_{y\in\partial V^-_k}
P_x\big(S_{\tau^{U_n\cap V^-_k}}=y\big)\\
\times\bigg|-m_k+\ol{u}(0)
-\sum_{z\in \partial U_n}
P_y\big(S_{\tau^{U_n\cap (\eps\rmB^-)_0}}=z\big)(-m_n+u(z))
-P_y\big(\tau^{(\eps\rmB^-)_0}\leq\tau^{U_n}\big)\ol{u}(0)\bigg|\\
\leq \max_{x\in U_n\cap \rmB^-_{n-r}}
P_x\big(\tau^{V^-_k}<\tau^{U_n}\big)
\max_{y\in\partial V^-_k}\Big|
P_y\big(
\tau^{U_n}<\tau^{(\eps\rmB^-)_0}\big)\ol{u}(0)-\he{u}{\partial U_n\cup \partial(\eps\rmB^-)_0}(y)\\
+P_y\big(
\tau^{U_n}<\tau^{(\eps\rmB^-)_0}\big)
(m_n-m_k)
+P_y\big(
\tau^{(\eps\rmB^-)_0}\leq\tau^{U_n}\big)
(-m_k)\Big|\,,
\end{multline}
where we use the representations of $\frg^0$ and $\wt\frg$ in terms of two simple random walks which we couple to be equal until the first hitting time of $\partial V^-_k$ at which we apply the strong Markov property.
By Lemma~\ref{l:UV-U},
$\big|\he{u}{\partial U_n\cup \partial(\eps\rmB^-)_0}(y)-P_y\big(
\tau^{U_n}<\tau^{(\eps\rmB^-)_0}\big)\ol{u}(0)\big|\leq \osc\,\ol{u}_\eta$.
Hence, using also Lemma~\ref{l:ruin}, we can further bound the expression on the right-hand side of~\eqref{e:proof-gV-Markov} by
\begin{equation}
\label{e:gV-p5}
\frac{m+C_\eps}{n-k}\Big[
\eps^{-1}+\Big|\frac{k}{n} m_n-m_k\Big|\Big]\,,
\end{equation}
where we also used that $m_n$ is of order $n$.
As $\tfrac{k}{n}m_n=\wh m_{n,k,0}$ by~\eqref{e:bl}, Lemma~\ref{l:blml} shows that the absolute difference in~\eqref{e:gV-p5} is bounded by a constant times $\log(n-k)$, which yields the assertion.
\end{proof}

\begin{proof}[Proof of Lemma~\ref{l:1-6-m}]
Let $\eps'>0$. First we assume that $\eta\wedge\eta'>\eps'\vee (2\dH(U,U'))$.
Let $\frg$ be harmonic on $U_n\cap \rmB^-_0$ with boundary values $-m_n+u$ on $\partial U_n$ and $\ol{u}(0)$ on $\partial \rmB^-_0$. Let $\wt\frg$ be harmonic on $U'_n\cap \rmB^-_0$ with boundary values $-m_n+u'$ on $\partial U'_n$ and $\ol{u'}(0)$ on $\partial\rmB^-_0$.
As we are assuming that $\eta\wedge\eta'>2\dH(U,U')$, we have $U^{\eta/2}\subset U'$ and hence $U^\eta\subset (U\cap U')^{\eta/2}$, and analogously also $U'^{\eta'}\subset (U\cap U')^{\eta/2}$.
By the Gibbs-Markov property,
$(h-\frg)_{(U\cap U')_n\cap \rmB^-_0}$ under $\bbP_{\rmB,0,\ol{u}(0)}^{U,n,u}$ is distributed as the sum of the independent fields $h^{(U\cap U')_n\cap\rmB^-_0}$ and the binding field
$\varphi^{U_n\cap \rmB^-_0, (U\cap U')_n\cap \rmB^-_0}$. Analogously,
$(h-\wt\frg)_{(U\cap U')_n\cap \rmB^-_0}$ under $\bbP_{\rmB,0,\ol{u'}(0)}^{U',n,u'}$ is distributed as the sum of the independent fields $h^{(U\cap U')_n\cap \rmB^-_0}$ and the binding field
$\varphi^{U'_n\cap \rmB^-_0, (U\cap U')_n\cap \rmB^-_0}$.
Using these representations in~\eqref{e:def-cLm} and~\eqref{e:LmUV}, we obtain that
\begin{multline}
\label{e:1-6-p-int}
\cL^r_{n,\eta',U',0,\rmB}(u',\ol{u'}(0))
\leq\int\cL^r_{n,\eta,U,0,\rmB}(u,0,-t)\\
\times\bbP\Big(
\max_{U'^{\eta'}_n\cap \rmB^-_{n-r}}\Big\{
\Big|\varphi^{U_n\cap \rmB^-_0, (U\cap U')_n\cap\rmB^-_0 }\Big|
+\Big|\varphi^{U'_n\cap \rmB^-_0, (U\cap U')_n\cap\rmB^-_0 }\Big|
+\big|\wt\frg-\frg\big|\Big\}\in\rmd t\Big)
\,,
\end{multline}
where the two binding fields in the integral can be defined as independent.
We also have the analogous lower bound so that we obtain for sufficiently large $n$ from Lemma~\ref{l:LUVt} that
\begin{equation}
\label{e:cont-V-p20}
\big|\cL^r_{n,\eta',U',0,\rmB}(u',\ol{u'}(0))-
\cL^r_{n,\eta,U,0,\rmB}(u,\ol{u}(0))\big|\leq
\bbE\psi(M)\leq
C_{r,\eps,\eps'}\bbE M + \sum_{t=1}^\infty\bbP\big(M>t\big)\psi(t+1)\,,
\end{equation}
where we set $\psi(t)=C_{r,\eps,\eps'}t^{1/2}\rme^{\alpha t}$ and
\begin{multline}
M:=\max_{U^\eta_n\cap U'^{\eta'}_n \cap \rmB^-_{n-r}}
\Big|\varphi^{U_n\cap \rmB^-_0, (U\cap U')_n\cap\rmB^-_0 }\Big|
+\max_{U^\eta_n\cap U'^{\eta'}_n \cap \rmB^-_{n-r}}\Big|\varphi^{U'_n\cap \rmB^-_0, (U\cap U')_n\cap\rmB^-_0 }\Big|\\
+\max_{U^\eta_n\cap U'^{\eta'}_n \cap \rmB^-_{n-r}}\big|\wt\frg-\frg\big|\,,
\end{multline}
the maximum can be taken over $U^\eta_n\cap U'^{\eta'}_n \cap \rmB^-_{n-r}$ instead of $U'^{\eta'}_n \cap \rmB^-_{n-r}$ by symmetry of the left-hand side of~\eqref{e:cont-V-p20}.
By Lemmas~\ref{l:UU-fluct} and~\ref{l:gU}, the right-hand side of~\eqref{e:cont-V-p20} is smaller than $\eps'$ whenever $\dH(U,U')$, $|\eta-\eta'|$, $\big\|\ol{u}-\ol{u'}\big\|_{\bbL_\infty(U^\eta_n)}$ and $n^{-1}$ are sufficiently small (only depending on $r,\eps,\eps'$).

It remains to consider general $\eta,\eta'\geq 0$. To this aim, we write
\begin{multline}
\label{e:l-1-6-p-tr}
\big|\cL^r_{n,\eta',U',0,\rmB}(u',\ol{u'}(0))-
\cL^r_{n,\eta,U,0,\rmB}(u,\ol{u}(0))\big|
\leq \big|\cL^r_{n,\eta',U',0,\rmB}(u',\ol{u'}(0))-
\cL^r_{n,\eta'\vee\eps',U',0,\rmB}(u',\ol{u'}(0))\big|\\
+\big|\cL^r_{n,\eta'\vee\eps',U',0,\rmB}(u',\ol{u'}(0))-
\cL^r_{n,\eta\vee\eps',U,0,\rmB}(u,\ol{u}(0))\big|
+\big|\cL^r_{n,\eta\vee\eps',U,0,\rmB}(u,\ol{u}(0))-
\cL^r_{n,\eta,U,0,\rmB}(u,\ol{u}(0))\big|\,.
\end{multline}
By the first part of the proof, the second term on the right-hand side is bounded by $\eps'$ for $\dH(U,U')$, $|\eta-\eta'|$, $\big\|\ol{u}-\ol{u'}\big\|_{\bbL_\infty(U^\eta_n)}$, $n^{-1}$ sufficiently small. By Lemma~\ref{l:delta0}, the first and third terms in~\eqref{e:l-1-6-p-tr} are bounded by $C_\eps\eps'\cL^r_{n,\eta',U',0,\rmB}(u',\ol{u'}(0))$ and
$C_\eps\eps'\cL^r_{n,\eta,U,0,\rmB}(u,\ol{u}(0))$, respectively.
This implies the assertion as $\eps'$ was arbitrary.
\end{proof}

\begin{lem}
\label{l:gU}
Let $r\geq 0$, $\eps,\eps'>0$, $\eta,\eta'>\eps'$, and let $\frg$ and $\wt\frg$ be defined as in the above proof of Lemma~\ref{l:1-6-m}. Then
\begin{equation}
\max_{U^\eta_n\cap U'^{\eta'}_n\cap \rmB^-_{n-r}}|\wt\frg-\frg|
\leq \big\|\ol{u}-\ol{u'}\big\|_{\bbL_\infty(U^\eta_n\cap U'^{\eta'}_n)}+o_{\eps,\eps',r}(1)\,,
\end{equation}
where $o_{\eps,\eps',r}(1)$ tends to zero as $n\to\infty$ and $\dH(U,U')\to 0$.
\end{lem}
\begin{proof}
For $x\in U^\eta\cap U'^{\eta'}\cap \rme^{-r}\rmB^-$,
we write
\begin{multline}
\label{e:l:gU-p-dec}
\frg(\lfloor \rme^n x\rfloor )-\wt\frg(\lfloor \rme^n x\rfloor)
=m_n \Big( P_{\lfloor \rme^n x\rfloor}\big(\tau^{\partial (U \cap U')_n}\leq \tau^{\partial\rmB^-_0}\big)
-P_{\lfloor \rme^n x\rfloor}\big(\tau^{\partial U_n}\leq \tau^{\partial\rmB^-_0}\big) \Big)\\
+m_n \Big( P_{\lfloor \rme^n x\rfloor}\big(\tau^{\partial U'_n}\leq \tau^{\partial\rmB^-_0}\big)
-P_{\lfloor \rme^n x\rfloor}\big(\tau^{\partial (U \cap U')_n}\leq \tau^{\partial\rmB^-_0}\big) \Big)\\
+\Big(\he{u}{\partial U_n \cup \rmB^-_0}(\lfloor \rme^n x\rfloor)-\ol{u}(\lfloor \rme^n x\rfloor)\Big)
+\Big(\ol{u}(\lfloor \rme^n x\rfloor)-\ol{u'}(\lfloor \rme^n x\rfloor)\Big)
+\Big(\he{u'}{\partial U'_n \cup \rmB^-_0}(\lfloor \rme^n x\rfloor)-\ol{u'}(\lfloor \rme^n x\rfloor)\Big)\,.
\end{multline}
The first and the third difference in the last line are bounded by $C_\eps n^{-1}$ by Lemmas~\ref{l:UV-U}, \ref{l:ruin}\ref{i:ruin-C} and the assumptions.
For $\delta\in(\dH(U,U'),\eps'/2)$, we obtain by monotonicity, Lemma~\ref{l:ruin}\ref{i:ruin-U}, and as $U^\delta\subset U\cap U'$ by definition of the Hausdorff distance, that
\begin{multline}
0\leq m_n \Big(
P_{\lfloor \rme^n x\rfloor}\big(\tau^{\partial (U \cap U')_n}\leq \tau^{\partial\rmB^-_0}\big)
-P_{\lfloor \rme^n x\rfloor}\big(\tau^{\partial U_n}\leq \tau^{\partial\rmB^-_0}\big)\Big)\\
\leq m_n \Big(
P_{\lfloor \rme^n x\rfloor}\big(\tau^{\partial U_n}> \tau^{\partial\rmB^-_0}\big)
-P_{\lfloor \rme^n x\rfloor}\big(\tau^{\partial U^\delta_n}> \tau^{\partial\rmB^-_0}\big)\Big)\\
\leq 2\sqrt{g}\Big(\int_{\partial U}\Pi_{U}(x,\rmd z)\log|z|-
\int_{\partial U^\delta}\Pi_{U^\delta}(x,\rmd z)\log|z|\Big)+o_{\eps,\eps',r}(1)\,,
\end{multline}
for $x\in U^{\eps'}\cap\rme^{-r}\rmB$, where the difference in the last line tends to zero as $\delta\downarrow 0$ with rate depending only on $\eps,\eps'$ as in Lemma~A.2 of~\cite{BL-Conf}. The difference in the second line of~\eqref{e:l:gU-p-dec} is bounded in the same way. This yields the assertion.
\end{proof}

\begin{proof}[Proof of Lemma~\ref{l:SDconv}]
As in~\eqref{e:h-S} and by the Gibbs-Markov property, we have
\begin{equation}
\label{e:SDconv-p20}
\cS'_p+\beta_p+\cD_p=\max_{y\in A_{p}}\big\{
\he{h}{\partial \rA_{p}}(y)+h_p(y)\big\}
=\max_{y\in A_{p}} h(y)
\end{equation}
under $\bbP_{k,v}^{n,u}$ with the coupling from~\eqref{e:hphih} which we always assume. We define
$\wt A_1=U^{\eta}\cap\rme^{-1+\lfloor\log\eps\rfloor}\rmB^{-}$ and
$\wt A_p=\rme^{-p+\lfloor\log\eps\rfloor+1}\rmB\cap\rme^{-p+\lfloor\log\eps\rfloor}\rmB^{-}$ for $p\geq 2$. Let $\eps'>0$. Then,
$\rme^{n}\wt A^{\eps'}_p\subset \rA_p$ for $p\geq 1$ and $n$ sufficiently large, where $A'_p$ is defined in Section~\ref{s:concdec}.

Let $\frg_{n,k}$ be the harmonic function on $U_n\cap V^-_k$ with $\frg_{n,k}=-m_n+u$ on $\partial U_n$, $\frg_{n,k}=-m_k+v$ on $\partial V^-_k$.
We consider the Gaussian process $(\cX_p,\phi_p)_{p=1}^q$ under $\bbP_{k,v}^{n,u}$, where $\phi_p(y):=\he{h}{\partial \rA_{p}}(\lfloor\rme^n y\rfloor)+m_{n-p}$ for $y\in \wt A^{\eps'}_p$, and $\cX_p$ is as in~\eqref{e:Xm}.
Using~\eqref{e:hphih} and applying the Gibbs-Markov property as in~\eqref{e:SDconv-p20}, we obtain
\begin{equation}
\label{e:GM-pSDconv}
\he{h}{\partial \rA_{p}}=\varphi_{0,p}
+\frg_{n,k}(\lfloor\rme^n y\rfloor)
\end{equation}
under $\bbP_{k,v}^{n,u}$, as
$\bbE_{k,v}^{n,u}\he{h}{\partial \rA_{p}}(\lfloor\rme^n y\rfloor)
=\bbE_{k,v}^{n,u} h(\lfloor\rme^n y\rfloor)=\frg_{n,k}(\lfloor\rme^n y\rfloor)$.
As $\varphi_{0,p}=\sum_{j=0}^p\varphi_j$ and as the $\varphi_p$, $p=0,\ldots,\rt$ are independent, Lemmas~\ref{l:varXm} and~\ref{l:bdg-ann} below imply the convergence of the covariances of the process
$(\cX_p,\phi_p)_{p=1}^q$ under $\bbP_{k,v}^{n,u}$ as $n-k\to\infty$.
Moreover, $\frg_{n,k}(\lfloor\rme^n y\rfloor)-m_{n-p}$ converges as $n-k\to\infty$ uniformly in $y\in\wt A_p^{\eps'}$ by the assumptions on $u$, $v$, Lemma~\ref{l:harm-hole} below, continuity of $\Pi_U(\cdot,\rmd z)$, and as $m_n - m_{n-p} \to 2\sqrt{g}p$. As moreover $\cX_p$ is centered, it follows that the finite-dimensional distributions of the process $(\cX_p,\phi_p)_{p=1}^q$ converge.
The process $(\cX_p,\phi_p)_{p=1}^q$ also converges in distribution when the state space is endowed with the supremum norm as the argument from the proof of Lemma~4.4 and Section~6.5 of~\cite{Bi-LN} passes through in our setting.

For $p=1,\ldots,q$, the quantities $s_{1,p}$ and $s_{1,\rt}/\rt$ converge by Lemma~\ref{l:varXm}. Hence, $\beta_p=\wh u_\infty(0)+o_{\eps,q}(1)$ for all $p=1,\ldots, q$. Moreover, $(\cW_t)_{t\leq s_{1,q}}$ from~\eqref{e:defW} can be expressed as a continuous function (with respect to the supremum metric) of $s_{1,1},\ldots, s_{1,q}$, $(\cX_1,\ldots, \cX_q)$ and of the independent Brownian bridges $(\cB^{(1)},\ldots,\cB^{(q)})$.
As also $(s_{1,\rt}-t)/(s_{1,\rt}-\tau)\to 1$ uniformly in $\tau,t\in[0,s_{1,q}]$ as $n-k\to\infty$, it now follows from~\eqref{e:deftS} and~\eqref{e:defS} that $(\cS'_p+\beta_p,\phi_p)_{p=1}^q$ converges in distribution to a limit process which we denote by $(\cS^{(\infty)},\phi^{(\infty)})$ under $\bbP$.

Let $c_n>0$ with $c_n=o(\rme^{n})$, $c_n\to\infty$ as $n\to\infty$,
and let
\begin{equation}
\label{e:def-etap}
\xi_p=
\sum_{(x,t)\in\wt A_p^{\eps'}\times\bbR}
\delta_{(x,t)}\Ind_{
\big\{
x\in \rme^{-n}\bbZ^2,\,
h_p(\rme^n x)=\max_{y\in\bbZ^2:|y-\rme^n x|\leq c_n}h_p(y)=m_{n-p}+t\big\}}
\end{equation}
on $\wt A_p^{\eps'}\times(\bbR\cup\{\infty\})$.
By Theorem~2.1 of~\cite{BL-Full} and Theorem~2.1 of~\cite{BL-Conf}, $\xi_p$ converges vaguely in distribution as $n-k\to\infty$ to a Cox process with intensity $Z_p(\rmd x)\otimes \rme^{-\alpha t}\rmd t$ where $\alpha=2/\sqrt{g}$ and $Z_p$ is a random measure with a.\,s.\ positive and finite mass and no atoms.

Let $\kappa_p,\lambda_p\in\bbR$ for $p=1,\ldots,q$.
Then, by~\eqref{e:SDconv-p20}, \eqref{e:GM-pSDconv}, \eqref{e:def-etap}, the definition of $\phi_p$, and~\eqref{e:proof-LR-lim-eta} below,
\begin{multline}
\label{e:proof-LR-lim}
\bbP_{k,v}^{n,u}\big(\cS'_1+\beta_1> \kappa_1,\ldots,\cS'_q+\beta_q>\kappa_q,\cS'_1+\beta_1+\cD_1>\lambda_1,\ldots,\cS'_q+\beta_q+\cD_q>\lambda_q\big)\\
\leq\int\bbP_{k,v}^{n,u}\big(\cS'_p+\beta_p\in \rmd t_p\,,  \phi_p\in\rmd w_p,\,p= 1,\ldots, q \big)\prod_{p=1}^q
\Ind_{\{t_p>\kappa_p\}}\\
\times
\bbP_{k,v}^{n,u}\big(\xi_p\big\{(x,t)\in \wt A_p^{\eps'}\times\bbR:\: t+w_p(x)
+\eps'>\lambda_p\big\}>0\big)\\
+\sum_{p=1}^q\bbP_{k,v}^{n,u}(E_{p,n})
+\sum_{p=1}^q\bbP_{k,v}^{n,u}(F_{p,n})
+\bbP_{k,v}^{n,u}\Big(\bigcup_{p=1}^q\Big\{
\max_{ \rA_p\setminus\rme^n\wt A^{\eps'}_p} h>\lambda_p\wedge\lambda_{(p+1)\wedge q}\Big\}\Big)\,,
\end{multline}
where
\begin{equation}
E_{p,n}:=\Big\{\sup_{x,y\in \wt A^{\eps'}_p, |y-x|\leq 2\rme^{-n}c_n }
\big|\phi_p(x)-\phi_p(y)\big|\geq \eps'\Big\}\,,
\end{equation}
and $F_{p,n}$ is defined as the event that there exists $y\in U_n\cap V^-_k$ with
$c_n /2 < |y-\arg\max_{\lfloor \rme^n \wt A^{\eps'}_p\rfloor} h| < 2 c_n$ and
$h(y)+ \eps'\geq \max_{\lfloor \rme^n \wt A^{\eps'}_p\rfloor} h$.
By Lemma~\ref{l:1N}, we have $\lim_{n-k\to\infty}\bbP_{k,v}^{n,u}(F_{p,n})=0$.
Moreover, as $\bbE\phi_p$ is uniformly continuous and by the argument from Section~6.5 of~\cite{Bi-LN}, $\phi^{(\infty)}	$ is a.\,s.\ uniformly continuous on $\wt A^{\eps'}_p$, and it follows that $\lim_{n-k\to\infty}\bbP_{k,v}^{n,u}(E_{p,n})=0$.
On $E_{p,n}^\rmc\cap F_{p,n}^\rmc$, there exists in a $c_n/2$-neighborhood of $\wt x:=\arg\max_{\lfloor \rme^n \wt A_p^{\eps'} \rfloor} h$ a local maximum of $h_p$ that corresponds to a point of $\xi_p$.
Indeed, for all $y\in \lfloor \rme^n \wt A^{\eps'}_p \rfloor$ with $c_n/2<|y-\wt x|<2c_n$, we then have
\begin{equation}
h_p(y)=h(y)- \he{h}{\partial \rA_p}(y)
<h(\wt x)-\eps' - \he{h}{\partial \rA_p}(\wt x)+\eps'=h_p(\wt x)\,.
\end{equation}
Hence, still on $E_{p,n}^\rmc\cap F_{p,n}^\rmc$ and if $\cS'_p+\beta_p+\cD_p=h(\wt x)> \lambda_p$, there exists $\wt y\in\lfloor \rme^n \wt A^{\eps'}_p\rfloor$ and $|\wt x - \wt y|\leq c_n/2$ such that
\begin{equation}
\label{e:proof-LR-lim-eta}
h_p(\wt y)\geq h_p(\wt x)=h(\wt x)-t_p-w_p(\rme^{-n}\wt x)
>\lambda_p -t_p-w_p(\rme^{-n}\wt x)
\geq \lambda_p-t_p-w_p(\rme^{-n}\wt y) -\eps'\,,
\end{equation}
which implies $\xi_p\big(\rme^{-n}\wt y,(\lambda_p-t_p-w_p(\rme^{-n}\wt y)-\eps',\infty) \big)>0$ and shows~\eqref{e:proof-LR-lim}.

The integral on the right-hand side of~\eqref{e:proof-LR-lim}
is bounded from above by
\begin{multline}
\label{e:SDconv-eta}
\int\bbP(\cS^{(\infty)}_p\in \rmd t_p\,,\phi^{(\infty)}_p\in\rmd w_p,\,p= 1,\ldots, q \big)
\prod_{p=1}^q
\Ind_{\{t_p+\eps'>\kappa_p\}}\\
\times\bbP_{k,v}^{n,u}\big(\xi_p\big\{(x,t)\in \wt A^{\eps'}_p\times\bbR:\: t+w_p(x)+2\eps'>\lambda_p\big\}>0\big)\\
+\sum_{p=1}^q\bbP\Big(\big| \cS'_p + \beta_p - \cS^{(\infty)}_p \big| \vee
\sup_{x\in \wt A^{\eps'}_p}\big|\phi_p(x)-\phi^{(\infty)}_p(x)\big|>\eps'\Big)
\,,
\end{multline}
where we assume w.\,l.\,o.\,g.\ a coupling under $\bbP$ in which $(\cS'_p+\beta_p,\phi_p)_{p=1}^{q}$ converges to $(\cS^{(\infty)}_p,\phi^{(\infty)}_p)_{p=1}^q$ in probability.
The integral in~\eqref{e:SDconv-eta} converges as $n-k\to\infty$ followed by $\eps'\to 0$ by dominated convergence and as the second component of the intensity measure of the limit of $\xi_p$ is absolutely continuous with respect to Lebesgue measure. By the above, the other summands on the right-hand sides of~\eqref{e:proof-LR-lim} and~\eqref{e:SDconv-eta} vanish as $n-k\to\infty$ followed by $\eps'\to 0$:
for the last term on the right-hand side of~\eqref{e:proof-LR-lim} we use that
\begin{equation}
\bbP_{k,v}^{n,u}\Big(\bigcup_{p=1}^q\Big\{
\max_{ \rA_p\setminus \rme^n\wt A^{\eps'}_p} h>\lambda_p\wedge \lambda_{(p+1)\wedge a}\Big\}\Big)\leq C_\eps q\eps'
\end{equation}
by Lemma~\ref{p:viol}.
As a lower bound analogous to~\eqref{e:proof-LR-lim} follows by the same argument (in which we now do not require Lemmas~\ref{l:1N} and~\ref{p:viol}), this yields the assertion on convergence.
The assertion on absolute continuity follows as
$(\cS^{(\infty)}_1,\ldots,\cS^{(\infty)}_q)$ is a non-degenerate Gaussian by Lemmas~\ref{l:dep} and~\ref{l:varXm},
and as the limiting measures $\xi^{(\infty)}_p$ of $\xi_p$, $p=1,\ldots,q$ are independent and have intensity $\rme^{-\alpha t}\rmd t$ in their height components.
\end{proof}
It remains to show the following two lemmas.
For $W\in\frD$, we denote the continuum Green kernel on $W$ by
\begin{equation}
G_W(x,y)=g\int\Pi_W(x,\rmd z)\log|y-z| - g\log|x-y|\,,
\end{equation}
where $\Pi_W$ denotes the continuum Poisson kernel, i.\,e.\ $\Pi_W(x,\cdot)$ is the exit distribution from the domain $W$ of Brownian motion started at $x$.
\begin{lem}
\label{l:bdg-ann}
Let $\eps\in(0,1)$, $\eta,\zeta \in[0,\eps^{-1}]$, $q\in\bbN$. For $U\in\frU^\eta_\eps$, $V\in\frV_\eps$, $p=1,\ldots,q$, $\wt A_p$ defined as in the proof of Lemma~\ref{l:SDconv}, and $x,y\in \wt A^\eps_q$, we have, as $n-k\to\infty$, that
\begin{equation}
\label{e:bdg-ann-ff}
\Cov_k^n\Big(\varphi_p(\lfloor \rme^n x\rfloor),
\varphi_p(\lfloor \rme^n y\rfloor)\Big)
=G_{\wt \Delta_{p-1}}(x,y)-G_{\wt A_p\cup \wt \Delta_p}(x,y)+o_{\eps,q}(1)
\end{equation}
and
\begin{equation}
\label{e:bdg-ann-fa}
\Cov_k^n\Big(\varphi_p(\lfloor\rme^n x\rfloor),\cX_p\Big)
=g\int_{z,z'\in\partial \wt \Delta_p}\Pi_{\wt A_p \cup\wt \Delta_p}(x,\rmd z)
\Pi_{\wt \Delta_p}(0, \rmd z')G_{\wt \Delta_{p-1}}(z,z')
+o_{\eps,q}(1)\,,
\end{equation}
where $\wt \Delta_0 =U^\eta$, $\wt \Delta_p = \rme^{-p+\lfloor\log\eps\rfloor}\rmB$.
\end{lem}
\begin{proof}
By the Gibbs-Markov property and the definition of $\varphi_p$, the left-hand side in~\eqref{e:bdg-ann-ff} equals 
\begin{multline}
\label{e:bdg-ann-pGM}
G_{\Delta_{p-1}}(\lfloor \rme^n x\rfloor, \lfloor \rme^n y\rfloor)
-\big( G_{\Delta_{p-1}} - G_{\Delta_{p-1}\cap V^-_k} \big)(\lfloor \rme^n x\rfloor, \lfloor \rme^n y\rfloor)
-G_{A_p\cup \Delta_{p}}(\lfloor \rme^n x\rfloor, \lfloor \rme^n y\rfloor)\\
+\big( G_{A_p\cup \Delta_{p}} - G_{A_p\cup \Delta_{p}\cap V^-_k} \big)(\lfloor \rme^n x\rfloor, \lfloor \rme^n y\rfloor)\,.
\end{multline}
The differences in brackets are by the Gibbs-Markov property equal to the left-hand side in~\eqref{e:Cov-bdg-V} (for the first one with $\Delta_{p-1}$ in place of $U_n$, and the second one is nonzero only for $p<q$ when we take $\Delta_p$ in place of $U_n$), hence it is bounded by $q^2(n-k-q)^{-1}$ by Lemma~\ref{l:Cov-bdg}. The remaining expression is equal to
\begin{equation}
g\int_{\partial \wt \Delta_{p-1}}\Pi_{\wt \Delta_{p-1}}(x,\rmd z)\log|z-y|
-g\int_{\partial \wt A_{p}}
\Pi_{\wt A_{p}\cup\wt \Delta_p}(x,\rmd z)\log|z-y|
+ o_{\eps,q}(1)
\end{equation}
by~\eqref{e:444}, \eqref{e:462} and Lemma~\ref{l:BL-Pois}\,.

To show~\eqref{e:bdg-ann-fa}, we use the representation~\eqref{e:phi-rw} and Definition~\eqref{e:Xm} in
\begin{equation}
\label{e:bdg-ann-fa-p1}
\Cov_k^n\Big(\varphi_p(\lfloor\rme^n x\rfloor),\cX_p\Big)
=\sum_{z\in\partial A_p, z'\in\partial \Delta_p}\Pi_{A_{p}\cup\Delta_p\cap V^-_k}(\lfloor \rme^n x\rfloor ,z)
\Pi_{\Delta_{p}}(0,z')G_{\Delta_{p-1}\cap V^-_k}(z,z')\,.
\end{equation}
By Lemma~\ref{l:Cov-bdg} and the Gibbs-Markov property, we have
$G_{\Delta_{p-1}\cap V^-_k}(z,z')
=G_{\Delta_{p-1}}(z,z')+O_{\eps,q}(1/(n-k))$ which we plug into~\eqref{e:bdg-ann-fa-p1}. For $p<q$ we also use that
\begin{equation}
\Pi_{\Delta_p\cap(\delta\rmB)_n}(\cdot,z)\leq \Pi_{\Delta_p\cap V^-_k}(\cdot,z)\leq \Pi_{\Delta_p}(\cdot,z)
\end{equation}
for $z\in\Delta_p$. Then we use~\ref{l:BL-Pois} and finally let $\delta\to 0$. This yields the assertion.
\end{proof}
\begin{lem}
\label{l:harm-hole}
Let $\eps\in(0,1)$, $\eta,\zeta \in[0,\eps^{-1}]$.
For all $U\in\frU^\eta_\eps$, $V\in\frV_\eps$, $0\leq k<n$, $\wh u_\infty\in\bbH_\infty(U^\eta)$, $u\in\bbR^{\partial U_n}$, $v\in\bbR^{\partial V^-_k}$ satisfying~\eqref{e:uv-as},
let $\frg_{n,k}=-m_n+u$ on $\partial U_n$, $\frg_{n,k}=-m_k+v$ on $\partial V^-_k$, and let $\frg_{n,k}$ be harmonic on $U_n\cap V^-_k$.
Then, as $n-k\to\infty$,
\begin{multline}
\max_{x\in (U^\eta\cap\eps\rmB)_n}\Big|\frg_{n,k}(x)+m_n-\wh u_\infty(\rme^{-n}x) -2\sqrt{g}\log|\rme^{-n}x|+2\sqrt{g}\int_{\partial U}\Pi_{U}(\rme^{-n}x,\rmd z)\log|z|\Big|\\
\leq \max_{x\in (U^\eta\cap\eps\rmB)_n}\big|\ol{u}(x)-\wh u_\infty(\rme^{-n}x)\big|
+o_\eps(1)\,.
\end{multline}
\end{lem}
\begin{proof}
We write
\begin{equation}
\label{e:p-harm-hole-g}
\frg_{n,k}(x)=
\heb{-m_n 1_{\partial U_n} - m_k 1_{\partial V^-_k}}{\partial U_n\cup \partial V^-_k}(x)
+\he{u }{\partial U_n\cup \partial V^-_k}(x)
+ \he{v}{\partial U_n\cup \partial V^-_k}(x)\,.
\end{equation}
The first term on the right-hand side of~\eqref{e:p-harm-hole-g} equals
\begin{equation}
-m_n+(m_n -m_k) P_x\big(\tau^{V^-_k}<\tau^{U_n}\big)
=-m_n+2\sqrt{g}\big(\int_{\partial U}
\Pi_U(\rme^{-n}x,\rmd z)\log|z|-\log|\rme^{-n}x|\Big)+o_\eps(1)
\end{equation}
where the equality holds by Lemma~\ref{l:ruin}.
By Lemma~\ref{l:osc-far}, the third term on the right-hand side of~\eqref{e:p-harm-hole-g} is bounded by $P_x(\tau^{V^-_k}<\tau^{U_n})|\ol{v}(\infty)|+ C_\eps\rme^{k-n}\osc\,\ol{v}_\zeta$, which in turn is bounded by $C_\eps(n-k)^{-\eps}$ by Lemma~\ref{l:ruin}\ref{i:ruin-C} and the assumptions.
For the second term on the right-hand side of~\eqref{e:p-harm-hole-g}, we estimate
\begin{equation}
\big|\he{u }{\partial U_n\cup \partial V^-_k}(x)-\wh u_\infty(\rme^{-n}x)\big|
\leq \big|\he{u }{\partial U_n\cup \partial V^-_k}(x)-\ol{u}(x)\big|
+\big|\ol{u}(x)-\wh u_\infty(\rme^{-n}x)\big|\,,
\end{equation}
where the first term on the right-hand side is bounded by
$(|\ol{u}(0)|+\osc\,\ol{u}_\eta) P_x(\tau^{V^-_k}\leq\tau^{U_n})$
by Lemma~\ref{l:UV-U}, which in turn is bounded by $C_\eps(n-k)^{-\eps}$ by~\eqref{e:uv-as} and Lemma~\ref{l:ruin}\ref{i:ruin-C}. This shows the assertion.
\end{proof}

\subsection{The functional $\cR_k$}
\label{s:R-funct}
Analogously to~\eqref{e:def-cLm}, we define the functional
\begin{equation}
\label{e:def-cRm}
\cR^r_{k,n}(v,u):=\cR_{k,\zeta,V,n,U}^r(v,u)
:=\bbE_{V,k,v}^{U,n,u}\big(
(\he{h}{\partial\rmB_{k+r}} (\infty)+ m_{k+r})^-;
h_{V^{-,\zeta}_k\cap \rmB_{k+r}} \leq 0\big)\,.
\end{equation}
In accordance with Proposition~\ref{p:1-2}, let
\begin{equation}
\label{e:cRmr}
\cR_k(v):=\cR_{k,\zeta,V}(v):=\lim_{n\to\infty}\cR^{r_{n-k}}_{k,\zeta,V,n,\rmB}(v,0)
\end{equation}
where existence of the limit is shown in Lemma~\ref{l:Rk} below.
We use the setting from Section~\ref{s:w-outw} and define
\begin{equation}
\label{e:def-l-R}
\ell_{\rt_{n-k},r}\big(\ol{v}(\infty),\ol{u}(0)\big)=
\bbE_{V,k,v}^{U,n,u}\Big(\big(\cS^{\rm o}_r\big)^-;\, \max_{i=1}^r \big(\cS^{\rm o}_i + \cD^{\rm o}_i\big)\leq 0\Big)
\end{equation}
in accordance with~\eqref{e:5.4} and using the correspondence given by Theorem~\ref{t:drw-o}.
\begin{lem}
\label{l:Rk}
Let $\eps>0$, $\zeta,\eta\in[0,\eps^{-1}]$, $V\in\frV_\eps$,
$n,k\geq 0$, $v\in\bbR^{\partial V^-_k}$ with $\osc_\zeta\,\ol{v}\leq \eps^{-1}$ and
$-(n-k)^{1-\eps} \leq \ol{v}(\infty) \leq \eps^{-1}$.
Then the limit $\cR_k(v)$ in~\eqref{e:cRmr} exists and satisfies
\begin{equation}
\label{e:Rk}
\cR_k(v)=\big(1+o_\eps(1)\big)\cR^{r_{n-k}}_{k,\zeta,V,n,\rmB}(v,0)
+o_\eps(1) \big(1+\ol{v}(\infty)^-\big)
\end{equation}
as $n-k\to\infty$, and
\begin{equation}
\label{e:Rkl}
\cR_k(v)=\lim_{n\to\infty}\ell_{\rt_{n-k}, r_{n-k}-\lceil \log(\eps^{-1}+\zeta)\rceil}\big(\ol{v}(\infty),\ol{v}(\infty)\big)\,,
\end{equation}
where $\ell_{\rt_{n-k}, r}\big(\ol{v}(\infty),\ol{v}(\infty)\big)$ is
defined as in~\eqref{e:def-l-R} with $U=\rmB$, $u=\ol{v}(\infty)$.
\end{lem}
In the proofs in this subsection, we use analogs of the auxiliary lemmas from Section~\ref{s:comp-ballot} which are proved along the lines of Section~\ref{s:pf-aux}, using the outward concentric decomposition:
\begin{lem}
\label{l:LUV-R}
Let $\eps\in(0,1)$.
Then,
\begin{equation}
\label{e:LUV-add-R}
\lim_{\substack{r\to\infty\\r\in\bbN}}
\lim_{n-k\to\infty}
\big(1+\ol{v}(\infty)^-\big)^{-\eps}
\Big|\ell_{\rt_{n,k},r-\lceil\log(\eps^{-1}+\zeta)\rceil}\big(\ol{v}(\infty),\ol{u}(0)\big)
-\cR^{r}_{k,n}(v,u)\Big|=0
\end{equation}
uniformly in $\eta,\zeta \in[0,\eps^{-1}]$, $U\in \frU^\eta_\eps$, $V\in\frV_\eps$, $u\in\bbR^{\partial U_n}$, $v\in\bbR^{\partial V^-_k}$ satisfying~\eqref{e:uv-diff}.
\end{lem}
\begin{proof}
This is proved analogously to Lemma~\ref{l:LUV}.
\end{proof}
With
\begin{equation}
\cR^r_{k,W,V,n,U}:=\bbE_{V,k,v}^{U,n,u}\Big(\big(\he{h}{\partial\rmB^-_{k+r}}(\infty)+m_{k+r}\big)^-;\;
h_{W_k\cap \rmB_{k+r}}\leq 0\Big)\,,
\end{equation}
(so that in particular $\cR^r_{k,W,V,n,U}=\cR^r_{k,\zeta,V,n,U}$ for $W=V^{-,\zeta}$), we obtain the following analog of Lemma~\ref{l:delta0}:
\begin{lem}
\label{l:delta0-o}
Let $\eps\in(0,1)$ and $\eta,\zeta\in[0,\eps^{-1}]$.
There exists $C=C_\eps<\infty$ such that
\begin{equation}
\bigg|\frac{\cR^r_{k,W,V,n,U}(v,u)}{\cR^r_{k,W',V,n,U}(v,u)}-1\bigg|
\leq C\Leb(W \triangle W')
\end{equation}
for all $U\in\frU^\eta_\eps$, $V\in\frV_\eps$, Borel measurable $W,W'$ with $\eps^{-2}\rmB^- \subset W,W'\subset V^{-,\zeta}$, $n,k,r\geq 0$ with
$\partial \rmB_{k+r}\subset U^{\eta}_n\cap (\eps^{-2}\rmB^-)_k$,
and $u\in\bbR^{\partial U_n}$, $v\in\bbR^{\partial V^-_k}$ that satisfy
$\max\{\ol{u}(0),\osc\,\ol{u}_\eta, \ol{u'}(0), \osc\,\ol{u}'_{\eta'}\} < \epsilon^{-1}$.
\end{lem}
\begin{proof}
This is proved analogously to Lemma~\ref{l:delta0}.
\end{proof}
\begin{lem}
\label{l:cont-V-R}
Let $r\geq 0$ and $\eps\in(0,\tfrac{1}{10})$. Then there exists $C=C_{r,\eps}<\infty$ such that for all $\zeta\in[\eps,\eps^{-1}]$, $\eta\in[0,\eps^{-1}]$, $U\in\frU^\eta_\eps$, $V\in\frV_\eps$, $0<k<n$ with $\partial \rmB_{k+r+1}\subset U^{\eta\vee\eps}_n$ and $\partial \rmB^-_{k+r}\subset V^{-,\zeta}_k$,
and all $u\in\bbR^{\partial U_n}$, $v\in\bbR^{\partial V^-_k}$ satisfying~\eqref{e:uv-diff}, 
we have
\begin{equation}
\label{e:cont-V-assertion-R}
\big|\cR^r_{k,\zeta,V,n,\rmB}(v,\ol{v}(\infty))-\cR^r_{k,\zeta,V,n,U}(v,u) \big|
\leq C(n-k)^{-\eps/ 4}\,.
\end{equation}
\end{lem}
\begin{proof}
This is proved analogously to Lemma~\ref{l:cont-V}, we use Lemma~\ref{l:LUVt-R} below.
\end{proof}
\begin{lem}
\label{l:SDconv-o}
Let $\eps\in(0,1)$ and $\eta,\zeta \in[0,\eps^{-1}]$, $q\in\bbN$.
For $n,k\geq 0$, assume that $u\in\bbR^{\partial U_n}$, $v\in\bbR^{\partial V^-_k}$ satisfy~\eqref{e:uv-as} and that
\begin{equation}
\label{e:ass-SDconv-o}
\lim_{k\to\infty}
\max_{x\in V^{-,\zeta}_k}\big| \ol{v}(x) - \wh v_\infty(\rme^{-k}x)\big|
=0
\end{equation}
for some $\wh v_\infty\in \bbH_\infty(V^{-,\zeta})$.
Then $(\cS^{\rm o}_p,\cS^{\rm o}_p+\cD^{\rm o}_p)_{p=1}^q$ converges in distribution as $n-k\to\infty$ to some limit $(\cS^{{\rm o},(\infty)}_p,\cS^{{\rm o},(\infty)}_p+\cD^{{\rm o},(\infty)}_p)_{p=1}^q$ whose distribution is absolutely continuous with respect to Lebesgue measure and depends only on $\wh v_\infty$ and $\zeta$.
\end{lem}
\begin{proof}
We first consider the case that $n-k\to\infty$ while $k\geq 0$ stays fixed.
In this case Assumption~\eqref{e:ass-SDconv-o} is not needed and the assertion follows as $\bbP_{V,k,v}^{\rmB,n,0}(h_{\Delta^{\rm o}_q}\in\cdot)$ converges weakly to
$\bbP(h_{\Delta^{\rm o}_q}\in\cdot\mid h_{\partial V^-_k}=v)$, and as $\big(\cS^{\rm o}_j, \cD^{\rm o}_j\big)_{j=1}^q$ is a continuous functional of $h_{\Delta^{\rm o}_q}$ by the definitions in Section~\ref{s:w-outw}. The case when also $k\to\infty$ is treated analogously to Lemma~\ref{l:SDconv}.
\end{proof}
Also, the following lemma is proved analogously to Lemma~\ref{l:DGFF-lb}:
\begin{lem}
\label{l:DGFF-lb-o}
Let $\eps\in(0,1)$ and $\eta,\zeta\in[0,\eps^{-1}]$. Then there exists $c=c_\eps>0$ such that
\begin{equation}
\label{e:DGFF-lb-o}
\bbP_{k,v}^{n,u}\big(h_{U^\eta_n\cap V^{-,\zeta}_k}\leq 0\big)
\geq c\frac{(1+\ol{v}(\infty)^-)\ol{u}(0)^-}{n-k}
\end{equation}
for all $U\in\frU^\eta_\eps$, $V\in\frV_\eps$, $0\leq k< n$ with $T_{n-k}\geq 1$, $u\in\bbR^{\partial U_n}$, $v\in\bbR^{\partial V^-_k}$ satisfying~\eqref{e:uv-as} and $\ol{u}(0)^- \geq (n-k)^\eps$.
\end{lem}
We will also refer to the following lemma:
\begin{lem}
\label{l:LUVt-R}
Let $\eps\in(0,1)$, $r\geq 0$, $\zeta\in[\eps,\eps^{-1}]$, $\eta\in[0,\eps^{-1}]$. There exists $C=C_{r,\eps}<\infty$ such that for all
$V\in\frV_\eps$, $U\in\frU^\eta_\eps$, $0\leq k<n$ with $\partial\rmB_{k+r}\subset V^{-,\zeta}_k$ and $\partial\rmB^-_{k+r}\subset U^{\eta\vee\eps}_k$, all $u\in\bbR^{\partial U_n}$, $v\in\bbR^{\partial V^-_k}$ satisfying~\eqref{e:uv-diff},
and all $t\in\bbR$, we have
\begin{equation}
\label{e:LUVt-R}
\big|\cR^r_{k,n}(v,u,t)-
\cR^r_{k,n}(v,u)\big|\leq
C |t|^{1/2}\rme^{\alpha |t|}\,.
\end{equation}
\end{lem}
\begin{proof}
This is proved analogously to Lemma~\ref{l:LUVt}.
\end{proof}
We are now ready to give the proofs of Lemma~\ref{l:Rk}, Proposition~\ref{c:1.2}, and the remaining parts for Propositions~\ref{p:LR-infty} --~\ref{p:1-6}.
\begin{proof}[Proof of Lemma~\ref{l:Rk}]
We obtain existence of the limit on the right-hand side of~\eqref{e:Rkl}  from Proposition~\ref{t:ta} which we apply as follows:
for any sequence
$(n^{(i)})_{i=1}^\infty$ with $n^{(i)}\to\infty$, we set
$a^{(i)}=\ol{v}(\infty)$,
$b^{(i)}=\ol{v}(\infty)$,
$\rt^{(i)}=\rt_{n^{(i)}-k}$,
$r^{(i)}=
r_{n^{(i)}-k}-\lceil \log(\eps^{-1}+\zeta)\rceil$,
$\rt^{(\infty)}=r_{\infty}=\infty$, $a^{(\infty)}=b^{(\infty)}=\ol{v}(\infty)$,
and in place of $\big(\cS^{(i)}_j\big)$, $\big(\cD^{(i)}_j\big)$, we use
$\big(\cS^{\rm o}_j\big)$, $\big(\cD^{\rm o}_j\big)$
as defined in Section~\ref{s:w-outw} for $u=\ol{v}(\infty)$, $U=\rmB$ and $\eta=0$.
Lemma~\ref{l:SDconv-o} yields that $(\cS^{(i)}_j, \cD^{(i)}_j)$ has a limit in distribution $\big(\cS^{(\infty)}_j, \cD^{(\infty)}_j\big)$, that Assumption~\eqref{e:5.7} and the assumption on stochastic absolute continuity are satisfied, and that $\cR_k$ depends only on $V$, $k$, $v$ and $\zeta$. Assumptions~\ref{i.a1} --~\ref{i.a3} are verified by Theorem~\ref{t:drw-o}.

Next we show the equality in~\eqref{e:Rkl}. By Lemma~\ref{l:LUV-R} and the choice of $(r_n)$, we have
\begin{equation}
\label{e:p-Rk-LUV}
\lim_{n\to\infty} \Big(\cR^{r_{n-k}}_{k,n}(v,\ol{v}(\infty))-
\ell_{\rt_{n-k}, r_{n-k}-\lceil \log(\eps^{-1}+\zeta)\rceil}\big(\ol{v}(\infty),\ol{v}(\infty)\big)\Big)
=0\,.
\end{equation}
Moreover, Lemma~\ref{l:cont-V-R} shows that
\begin{equation}
\label{e:p-Rk-cont-V}
\lim_{n\to\infty}\Big(\cR^{r_{n-k}}_{k,n}(v,\ol{v}(\infty))-
\cR^{r_{n-k}}_{k,n}(v,0)\Big)
=0
\end{equation}
if $\zeta>0$, and Lemma~\ref{l:delta0-o} allows to extend~\eqref{e:p-Rk-cont-V} to the case $\zeta=0$. Combining~\eqref{e:p-Rk-LUV} and \eqref{e:p-Rk-cont-V} then yields that the limits 
in~\eqref{e:cRmr} and~\eqref{e:Rkl} coincide.

To show assertion~\eqref{e:Rk}, we let $\eps'\in(0,\eps)$, and we first show the Cauchy property for $\cR^{r_{n-k}}_{k,n}(v,0)$ as $n-k\to\infty$, for $\zeta\in[\eps',\eps^{-1}]$ and fixed $k\geq 0$. We note that $h$ under $\bbP_{V,k,v}^{\rmB,n,0}$ is distributed as
$h^{U_n\cap V^-_k}+\frg_n$, where $\frg_n$ is the harmonic function on $\rmB_n\cap V^-_k$ with boundary values $-m_n$ on $\partial \rmB_n$ and $-m_k+v$ on $\partial V^-_k$.
For $n'\geq n$, by the Gibbs-Markov property, $h-\frg_{n'}$ under $\bbP_{V,k,v}^{\rmB,n',0}$ is distributed as the sum of independent fields $h-\frg_{n} + \varphi^{\rmB_{n'}\cap V^-_k,\rmB_n\cap V^-_k}$ under $\bbP_{V,k,v}^{\rmB,n,0}$.
We set
\begin{equation}
\varphi^*=\max_{V^{-,\zeta}_k\cap \rmB_{k+r}}
\big|\varphi^{\rmB_{n'}\cap V^-_k,\rmB_n\cap V^-_k}\big|\,,\qquad
\frg^*=\max_{V^{-,\zeta}_k\cap \rmB_{k+r}}
\big| \frg_{n'} - \frg_n \big|\,.
\end{equation}
Then, with
\begin{equation}
\cR^r_{n,k}(v,0,t):=
\bbE_{V,k,v}^{\rmB,n,0}\big(
(\he{h}{\partial\rmB_{k+r}} (\infty)+ m_{k+r}+t)^-;
h_{V^{-,\zeta}_k\cap \rmB_{k+r}}+t \leq 0\big)
\end{equation}
for $t\in\bbR$, we obtain analogously to~\eqref{e:cont-v-p-int-lb} and~\eqref{e:cont-v-p-int-ub} that
\begin{equation}
\label{e:Rk-t-lb}
\cR^r_{k,n'}(v,0)
\geq \int\bbP\big(\varphi^* + \frg^* \in\rmd t\big)\cR^r_{k,n}(v,0,t)
\end{equation}
and
\begin{equation}
\label{e:Rk-t-ub}
\cR^r_{k,n'}(v,0)
\leq \int\bbP\big(\varphi^* + \frg^* \in\rmd t\big)\cR^r_{k,n}(v,0,-t)\,.
\end{equation}
By Lemma~\ref{l:LUVt-R} and the choice of $(r_n)$, we have
\begin{equation}
\big|\cR^r_{k,n}(v,0,t)-\cR^r_{k,n}(v,0)\big|
\leq C_{\eps'} |t|^{1/2}\rme^{\alpha |t|}
\end{equation}
for $t\in\bbR$, uniformly in $r\leq r_{n-k}$. Lemma~\ref{l:VV-fluct} and the choice of $(r_n)$ give
\begin{equation}
\bbE\varphi^*\leq C_\eps (n-k)^{-1/4} \,,\qquad
\bbP\big(\varphi^*>s\big)\leq C_\eps\rme^{-c_\eps (n-k) s^2}
\end{equation}
for $s\geq 0$, uniformly in $r\leq r_{n-k}$.
Furthermore, it can be shown analogously to Lemma~\ref{l:gV} that
\begin{equation}
\frg^*\leq C_\eps r(n-k)^{-\eps}\,.
\end{equation}
Plugging these estimates into~\eqref{e:Rk-t-lb} and~\eqref{e:Rk-t-ub}, we obtain
\begin{equation}
\big|\cR^r_{k,n'}(v,0)-\cR^r_{k,n}(v,0)\big|
\leq C_\eps\big(\bbE\varphi^*+\frg^*\big)^{1/2}
+\sum_{s=1}^\infty\bbP\big(\varphi^*>s-\frg^*\big)C_{\eps'}\rme^{\alpha s}
=o_\eps(1)
\end{equation}
as $n-k\to\infty$, uniformly in $r\leq r_{n-k}$.
In particular,
\begin{equation}
\cR^{r_{n-k}}_{k,n'}(v,0)-\cR^{r_{n-k}}_{k,n}(v,0)
=o_{\eps'}(1)
\end{equation}
as $n-k\to\infty$, where we recall that $n'\geq n$. By Lemmas~\ref{l:LUV-R}, \ref{l:cont-V-R} and the choice of $(r_n)$,
\begin{equation}
\label{e:p-Rk-lR}
\cR^{r_{n-k}}_{k,n'}(v,0)-
\ell_{\rt_{n'-k}, r_{n-k}-\lceil \log(\eps^{-1}+\zeta)\rceil}\big(\ol{v}(\infty),\ol{v}(\infty)\big)
=o_{\eps'}(1)\big(1+\ol{v}(\infty)^-\big)
\end{equation}
as $n-k\to\infty$.
By Proposition~\ref{t:LR-as},
\begin{equation}
\ell_{\rt_{n'-k}, r_{n'-k}-\lceil \log(\eps^{-1}+\zeta)\rceil}\big(\ol{v}(\infty),\ol{v}(\infty)\big)
-\ell_{\rt_{n'-k}, r_{n-k}-\lceil \log(\eps^{-1}+\zeta)\rceil}\big(\ol{v}(\infty),\ol{v}(\infty)\big)
=
o_\eps(1) \big(1+\ol{v}(\infty)^-\big)
\end{equation}
as $n-k\to\infty$. Using again~\eqref{e:p-Rk-lR} shows that
\begin{equation}
\cR^{r_{n'-k}}_{k,n'}(v,0)-\cR^{r_{n-k}}_{k,n}(v,0)=o_{\eps'}(1) \big(1+\ol{v}(\infty)^-\big)\,,
\end{equation}
as $n-k\to\infty$, as desired. This shows
\begin{equation}
\label{e:Rk-eps}
\cR_k(v)=\cR^{r_{n-k}}_{k,\zeta,V,n,\rmB}(v,0)
+o_{\eps'}(1) \big(1+\ol{v}(\infty)^-\big)
\end{equation}
and hence~\eqref{e:Rk} for $\zeta\in[\eps',\eps^{-1}]$.
By Lemma~\ref{l:delta0-o}, we now choose $\eps'$ sufficiently small, depending only on $\eps$, such that
\begin{equation}
(1-\eps')\cR^{r_{n-k}}_{k,\zeta,V,n,\rmB}(v,0)
\leq \cR^{r_{n-k}}_{k,\eps',V,n,\rmB}(v,0)
\leq (1+\eps')\cR^{r_{n-k}}_{k,\zeta,V,n,\rmB}(v,0)
\end{equation}
for all $\zeta\in[0,\eps']$. From this and~\eqref{e:Rk-eps}, we then obtain~\eqref{e:Rk} for all $\zeta\in[0,\eps^{-1}]$.
\end{proof}

\begin{proof}[Proof of Proposition~\ref{c:1.2}]
For $\zeta\in[\eps,\eps^{-1}]$, it suffices by Lemmas~\ref{l:Rk} and~\ref{l:DGFF-lb-o} to prove that
\begin{equation}
\label{e:p-c:1.2}
\bbP_{k,v}^{n,u}\Big( h_{U^\eta_n \cap V^{-,\zeta}_k} \leq 0 \Big)
=\big(2+o_{\eps}(1)\big)
\frac{\cR^{r_{n-k}}_{k,\zeta,V,n,\rmB} (v,0) \ol{u}(0)^-
}{g(n-k)}
\end{equation}
The asymptotics~\eqref{e:p-c:1.2} and the extension of~\eqref{e:p-c:1.2} to $\zeta\in[0,\eps]$ follow analogously to the proof of Proposition~\ref{c:1.1}, where we also use Lemma~\ref{l:DGFF-lb-o}.
\end{proof}

\begin{proof}[Proof of Proposition~\ref{p:LR-infty} (for $\cR_k$)]
With $\ell_{\rt_{n-k},r_{n-k}}\big(\ol{v_k}(\infty),\ol{v_k}(\infty)\big)$
defined by~\eqref{e:def-l-R} for $v=v_k$, $U=\rmB$, $u=\ol{v_k}(\infty)$,
we have
\begin{equation}
\lim_{k\to\infty}\cR_k(v_k)
=\lim_{k\to\infty}\lim_{n-k\to\infty}
\ell_{\rt_{n-k},r_{n-k}}\big(\ol{v_k}(\infty),\ol{v_k}(\infty)\big)
=\lim_{n\to\infty}\ell_{\rt_{n-k_n},r_{n-k_n}}\big(\ol{v_{k_n}}(\infty),\ol{v_{k_n}}(\infty)\big)
\end{equation}
by~\eqref{e:Rkl} if the limit on the right-hand side exists for all $(k_n)_{n\geq 0}$ that tend to infinity sufficiently slowly as $n\to\infty$.
%The existence of the limit on the right-hand side follows from Proposition~\ref{t:ta} and Lemma~\ref{l:SDconv-o}.
To show existence of the limit on the right-hand side, we apply Proposition~\ref{t:ta} as follows:
for any sequence
$(n^{(i)},k^{(i)})_{i=1}^\infty$ with $n^{(i)}-k^{(i)}\to\infty$, we set
$a^{(i)}=\ol{v_{k^{(i)}}}(\infty)$,
$b^{(i)}=\ol{v_{k^{(i)}}}(\infty)$,
$\rt^{(i)}=\rt_{n^{(i)}-k^{(i)}}$,
$r^{(i)}=
r_{n^{(i)}-k^{(i)}}-\lceil \log(\eps^{-1}+\zeta)\rceil$,
$\rt^{(\infty)}=r_{\infty}=\infty$, $a^{(\infty)}=b^{(\infty)}=\wh{v}(\infty)$,
and in place of $\big(\cS^{(i)}_j\big)$, $\big(\cD^{(i)}_j\big)$, we use
$\big(\cS^{\rm o}_j\big)$, $\big(\cD^{\rm o}_j\big)$
as defined in Section~\ref{s:w-outw} for $k=k^{(i)}$, $v=v_{k^{(i)}}$, $u=\ol{v_{k^{(i)}}}(\infty)$, $U=\rmB$ and $\eta=0$.
Lemma~\ref{l:SDconv-o} yields that $(\cS^{(i)}_j, \cD^{(i)}_j)$ has a limit in distribution $\big(\cS^{(\infty)}_j, \cD^{(\infty)}_j\big)$, that Assumption~\eqref{e:5.7} and the assumption on stochastic absolute continuity are satisfied, and that $\cR_k$ depends only on $V$, $k$, $v$ and $\zeta$. Assumptions~\ref{i.a1} --~\ref{i.a3} are verified by Theorem~\ref{t:drw-o}.
\end{proof}

\begin{proof}[Proof of Proposition~\ref{p:1-4} (for $\cR_k$, $\cR$)]
The assertion on $\cR_k$ follows from~\eqref{e:Rkl} and Proposition~\ref{p:1.5} with $\rt=\infty$. The assertion on $\cR$ follows by Proposition~\ref{p:LR-infty} as in~\eqref{e:p-14L13}.
\end{proof}

\begin{proof}[Proof of Proposition~\ref{p:1-5} (for $\cR_k$, $\cR$)]
The assertion on~$\cR_k$ follows analogously to $\cL_n$ from Lemma~\ref{l:DGFF-lb-o} and Proposition~\ref{c:1.2}.
The assertion on $\cR$ follows from Proposition~\ref{p:LR-infty} as before for $\cL$.
\end{proof}

\begin{proof}[Proof of Proposition~\ref{p:1-6} (for $\cR_k$, $\cR$)]
For any $\eps'>0$, we find, by Lemma~\ref{l:Rk} and Proposition~\ref{p:1-5}, $n\geq k$ such that
\begin{equation}
(1-\eps')\cR^{r_{n-k}}_{k,\zeta,V,n,\rmB}(v,0)
\leq \cR_k(v) \leq
(1+\eps')\cR^{r_{n-k}}_{k,\zeta,V,n,\rmB}(v,0)
\end{equation}
which we then use for all $V,k,v,\zeta$ as in the assertion in place of~\eqref{e:p-1-6-p1-5X} and proceed as in the proof for $\cL_k$, using the analog of Lemma~\ref{l:1-6-m} to show the assertion on $\cR_k$. The assertion on $\cR$ follows analogously to $\cL$ using Proposition~\ref{p:LR-infty}.
\end{proof}

\vspace{0.5cm}

\begin{appendices}
\begin{samepage}
{\Large \bf Appendices}\nopagebreak[4]
\section{General facts}
\label{s:app-gen}
\end{samepage}
In this section we collect several results from the literature that we use frequently or in a modified way. Recall from Section~\ref{s:bno} that $(S_i)$ under $P_x$ denotes simple random walk on $\bbZ^2$ started in $x$, and $\tau^A$ the exit time from $A\subset\bbZ^2$.

\subsection{Extreme value theory for the DGFF}

We need an upper bound for the tail of the DGFF maximum. Recall that for $A\subsetneq \bbZ^2$, we denote by $h^A$ under $\bbP$ a DGFF on $A$ with covariances given by~\eqref{e:1.1a} and zero mean.
For $B\subset A$, we write $\varphi^{A,B}$ for the binding field $\bbE(h^A\mid {h^A}_{A\setminus B})$.
\begin{lem}
\label{l:DGFF-ut}
Let $\eps\in(0,1)$. There exist constants $C=C_\eps<\infty$ and $c=c_\eps>0$ such that for all $U\in\frU_\eps^0$, $V\in\frV_\eps$, $k,n\geq 0$ and $t>0$,
\begin{equation}
\label{e:A1}
\bbP\big(\max h^{U_n} > m_n+t\big)\leq C\rme^{-c t}
\end{equation}
and
\begin{equation}
\label{e:A1-o}
\bbP\big(\max_{(2\eps^{-1}\rmB)_k} h^{V^-_k} > m_n+t\big)\leq C\rme^{-c t} \,.
\end{equation}
\end{lem}
\begin{proof}
Assertion~\eqref{e:A1} follows e.\,g.\ from Lemma~8.3 of~\cite{Bi-LN} and the Gibbs-Markov decomposition (Lemma~\ref{l:GM}). For assertion~\eqref{e:A1-o}, we decompose
$h^{V^-_k}$ as $h^{V^-_k\cap (4\eps^{-1}\rmB)_k}+\varphi^{V^-_k,V^-_k\cap (4\eps^{-1}\rmB)_k}$ by Gibbs-Markov and estimate
\begin{multline}
\bbP\big(\max_{(2\eps^{-1}\rmB)_k} h^{V^-_k} > m_n+t\big)\\
\leq \bbP\big(\max_{(2\eps^{-1}\rmB)_k} h^{V^-_k\cap (4\eps^{-1}\rmB)_k} > m_n+t/2\big)
+\bbP\big(\max_{(2\eps^{-1}\rmB)_k} \varphi^{V^-_k,V^-_k\cap (4\eps^{-1}\rmB)_k} > t/2\big)\,.
\end{multline}
The second summand on the right-hand side of the last display is bounded from above by $C_\eps\rme^{-c_\eps t^2}$ by Lemma~\ref{l:VV-fluct} (where we let $\wt n\to\infty$). For the first summand, we Gibbs-Markov in
\begin{equation}
\label{e:DGFF-ut-GMp}
\tfrac12 \bbP\big(\max_{(2\eps^{-1}\rmB)_k} h^{V^-_k\cap (4\eps^{-1}\rmB)_k} > m_n+t/2\big)\leq
\bbP\big(\max_{(2\eps^{-1}\rmB)_k} h^{(4\eps^{-1}\rmB)_k} > m_n+t/2\big)\,,
\end{equation}
where the factor $\tfrac12$ stands for the probability that at $\arg \max_{(2\eps^{-1}\rmB)_k} h^{V^-_k\cap (4\eps^{-1}\rmB)_k}$, the binding field
$\varphi^{(4\eps^{-1}\rmB)_k,(4\eps^{-1}\rmB)_k\cap V^-_k}$ is positive. Then we bound the right-hand side of~\eqref{e:DGFF-ut-GMp} by~\eqref{e:A1}.
\end{proof}

We also need an upper bound for the left tail. To this aim, we build on the proof of Theorem~1.1 of Ding~\cite{Di13} where (in particular) double exponential bounds for the left tail of the global maximum of the DGFF are shown for moderately small values. In the following lemma, we consider the maximum in a subdomain, and we state the bound for the full left tail by using also the Borell-TIS bound.
\begin{lem}
\label{l:Ding}
Let $\eps\in(0,1)$.
There exists $C=C_\eps<\infty$ such that for all domains $U,V\in\frD_\eps$,
$x\in U^\eps\cap\eps^{2}\rmB$, 
$0\leq k\leq n$
with $\rmB(x,\eps/2)_n\subset V^{\eps,-}_k$,
and for all $t>0$,
we have
\begin{equation}
\label{e:Ding}
\bbP\Big(\max_{\rmB(x,\eps/2)_n} h^{U_n} -m_{n}\leq -t\Big)\leq C\rme^{-t^{2-\eps}}\,,
\end{equation}
\begin{equation}
\label{e:DingV}
\bbP\Big(\max_{\rmB(x,\eps/2)_n} h^{U_n\cap V^-_k} -m_{n}\leq -t\Big)
\leq C\rme^{-t^{2-\eps}}\,,
\end{equation}
and for all
$y\in V^{-,\eps}\cap\eps^{-2}\rmB$ with
$\rmB(y,\eps/2)_k\subset U^\eps_n$, we have
\begin{equation}
\label{e:DingV-o}
\bbP\Big(\max_{\rmB(y,\eps/2)_k} h^{U_n\cap V^-_k} -m_{k}\leq -t\Big)
\leq C\rme^{-t^{2-\eps}}\,.
\end{equation}
\end{lem}
\begin{proof}
We first consider a square domain $E\subset \bbR^2$ of side length $\eps/2$.
Let $E'\subset E$ be a square domain of side length $\eps/4$ with the same center point as $E$.
We show that
\begin{equation}
\label{e:Ding-sq}
\bbP\Big(\max_{E'_n} h^{E_n} -m_n \leq -t\Big)\leq C_\eps\rme^{-t^{2-\eps}}
\end{equation}
for all $n\geq 0$ and $t>0$.

Let $\lambda=t^\eps$. First we consider the case that $\lambda\leq n^{2/3}$ (or equivalently $t\leq n^{\tfrac{2}{3\eps}}$).
Then, as in Section~2.4 of Ding~\cite{Di13}, there exist constants $c_1,c_2,c_3>0$ and disjoint balls
$\mathcal C_i\subset E'_n$ of radius $r$ for $i=1,\ldots,m$,
where $m=\lfloor (2\ell)^{-1}\rme^n\rfloor$, $r =\exp\{n-c_1\lambda+c_2\}$ and
$\ell=\exp\{n-c_1\lambda/3+c_2/3\}$, with the following properties:

For $i=1,\ldots,m$, let $B_i$ be the box of side-length $r/8$ that is centered in the same point as $\mathcal C_i$, let
$\mathcal C=\bigcup_{i=1}^m\mathcal C_i$ and
\begin{equation}
\chi:={\arg\max}_{z\in\bigcup_{i=1}^m B_i}h^{\mathcal C}(z)\,.
\end{equation}
Then we have
\begin{equation}
\label{e:Ding-outer}
\bbP(h^\mathcal C(\chi)\leq m_n-\lambda)\leq
\exp\left(-c_3 \rme^{c_1 \lambda/3}\right)
\end{equation}
which can be seen as a consequence of~(14) in~\cite{Di13}.

Let $t_n>0$ be such that
\begin{equation}
\label{e:Dingtn}
\rme^{-t_n^2}= \exp\Big(-c_3\rme^{c_1 n^{2/3}/3}\Big)\,.
\end{equation}

For $t\in\big[n^{\tfrac{2}{3\eps}}, t_n\big]$, we define $\mathcal C_i$, $B_i$, and $m$ as for $\lambda=n^{2/3}$, and we use monotonicity in $t$ to bound the probability on the left hand side of~\eqref{e:Ding-outer} from above by
$\exp\big(-c_3\rme^{c_1 n^{2/3}/3}\big)$.
Thus the left hand side of~\eqref{e:Ding-outer}
is bounded from above by $C_\eps\exp(- t^2)$ for all $t\in[0, t_n]$.

For all $x\in\bigcup_{i=1}^m B_i$, the Gibbs-Markov property (Lemma~\ref{l:GM}) yields
\begin{equation}
\label{e:Ding-varC}
\Var\varphi^{E_n,\mathcal C}(x)
=\Var h^{E_n}(x)-\Var h^{\mathcal C}(x)\,.
\end{equation}
Analogously e.\,g.\ to Proposition~\ref{p:Cov-bd}, $\Var h^{E_n}(x)\leq gn+ C_\eps$ and
$\Var h^{\mathcal C}(x)\geq gn-gc_1\lambda - C_\eps$.
Inserting these bounds into~\eqref{e:Ding-varC}, we obtain from the Gibbs-Markov property that for $t\in[0, t_n]$, the left-hand side of~\eqref{e:Ding-sq} is bounded from above by
\begin{multline}
\bbP\left(h^{\mathcal C} -m_n \leq- t^{\eps}\right)
+\bbP\left(\varphi^{E_n,\mathcal C}(\chi)\leq - t/2\right)
\leq C_\eps\exp(- t^{2})+\exp\bigg(-\frac{( t/2)^2}{2gc_1 t^{\eps}+C_\eps}\bigg)\\
\leq C_\eps\exp\left(- t^{2-2\eps}\right)\,.
\end{multline}

For $t\geq t_n$, we bound the left-hand side in~\eqref{e:Ding} from above by the Borell-TIS bound
$2e^{-c_\eps t^2/n}$ (which is also mentioned below Theorem~1.1 in~\cite{Di13}).
By choice of $t_n$ in~\eqref{e:Dingtn}, we have $n\leq t_n^{3\eps/2}$, hence the left-hand side in~\eqref{e:Ding} is bounded from above by $C_\eps\rme^{-t^{2-3\eps}}$ for all $t>0$.

To derive~\eqref{e:Ding}, we assume that $x$ is the center point of $E$. Then, by inclusion of events and the Gibbs-Markov property,
\begin{multline}
\label{e:p-Ding-U-E}
\bbP\Big(\max_{\rmB(x,\eps/2)_n} h^{U_n} -m_{n}\leq -t\Big)
\leq \bbP\Big(\max_{E'_n} h^{U_n} -m_{n}\leq -t\Big)\\
\leq \bbP\Big(\max_{E'_n} h^{E_n} -m_{n}\leq -t/2\Big)
+\bbP\Big(\min_{E'_n}\varphi^{U_n,E_n}\leq -t/2\Big)\,.
\end{multline}
The second summand on the right-hand side has a uniformly Gaussian tail, which follows from e.\,g.\ Lemma~4.4 of~\cite{Bi-LN}. This shows assertion~\eqref{e:Ding}. For~\eqref{e:DingV}, we argue analogously using the binding field $\varphi^{U_n\cap V^-_k, E_n}$ for which the same tail bound holds as the intrinsic metric and the variance that are used in the Fernique and Borell-TIS estimates in Lemma~4.4 of~\cite{Bi-LN} can only be smaller than for $\varphi^{U_n,E_n}$. For~\eqref{e:DingV-o}, we assume that $y$ is the center point of $E$ and argue again as in~\eqref{e:p-Ding-U-E}, where the tail bound for $\varphi^{U_n\cap V^-_k, E_n}$ now follows from Lemma~\ref{l:VV-fluct} (where we can replace $\rmB$ with $U\in\frD_\eps$).
\end{proof}

We also need the following statement on the geometry of local extrema which is a variant of Theorem~9.2 of~\cite{Bi-LN}.
\begin{lem}
\label{l:1N}
Let $\eps\in(0,1)$, $\eta,\zeta\geq 0$, and let $(c_n)_{n\geq 0}$ satisfy $c_n=o(\rme^n)$ and $c_n\to\infty$ as $n\to\infty$.
Let $U\in\frU^\eta_\eps$, $A$ open with $\eps\rmB\subset A\subset U$, $V\in\frV_\eps$, $0\leq k<n$,
$u\in\bbR^{\partial U_n}$, $v\in\bbR^{\partial V^-_k}$ satisfy~\eqref{e:uv-as} and~\eqref{e:SDconv-u} for some $\wh u_\infty\in\bbH_\infty(U^\eta)$.
Then we have $\lim_{n-k\to\infty}\bbP_{k,v}^{n,u}(F_\eps)=0$, where $F_\eps$ denotes the event that there exists $y\in U_n\cap V^-_k$ with
$c_n /2 < |y-\arg\max_{\lfloor \rme^n A\rfloor} h| < 2 c_n$ and
$h(y)+ \eps\geq \max_{\lfloor \rme^n A\rfloor} h$.
\end{lem}
\begin{proof}
As $\bbE_{k,v}^{n,u} h
= \heb{-m_n\Ind_{\partial U_n}+u}{\partial U_n\cup\partial V^-_k} + \heb{-m_k\Ind_{\partial V^-_k}+v}{\partial U_n\cup\partial V^-_k}$, Lemma~\ref{l:harm-hole} yields the convergence of $x\mapsto m_n+\bbE_{k,v}^{n,u} h(\lfloor\rme^n x\rfloor)$ to some $\wh u_\infty$ in $\bbL_\infty(U^\eta\cap\eps\rmB)$.
By definition of the DGFF, $h$ under $\bbP_{k,v}^{n,u}$ is distributed as $h^{U_n\cap V_k}+\bbE_{k,v}^{n,u} h$.
By the Gibbs-Markov property, $h^{U_n}$
is distributed as the sum of the fields $\varphi^{U_n,U_n\cap V^-_k}$ and $h^{U_n\cap V^-_k}$. Let $\wt F$ denote the event that there exist $x,x'\in \lfloor \rme^n A \rfloor$ with $c_n/2<|x-y|<2c_n$ and $h^{U_n}(x)\wedge h^{U_n}(y)\geq m_n -t$. From Theorem~9.2 of~\cite{Bi-LN}, we have
$\lim_{n-k\to\infty}\bbP(\wt F)=0$.
Then, by a union bound,
\begin{multline}
\label{e:1N-p9}
\bbP_{k,v}^{n,u}(F_\eps)\leq \bbP(\wt F)+\bbP(\max_{\lfloor \rme^n A \rfloor} h^{U_n}-m_n\leq -t+3\eps)
+\bbP(\max_{\lfloor \rme^n A\rfloor}|\varphi^{U_n,U_n\cap V^-_k}|>\eps)\\
+\Ind_{\sup_{y,y'\in \lfloor\rme^n A\rfloor, |x-y|\leq 2c_n}
|\bbE_{k,v}^{n,u}h(x)-\bbE_{k,v}^{n,u}h(y)|>\eps}\,.
\end{multline}
Let $\eps'>0$. By Lemma~\ref{l:Ding}, we find $t>0$ such that
$\bbP(\max_{\lfloor \rme^n A \rfloor} h^{U_n}-m_n\leq -t+3\eps)\leq \eps'$
for sufficiently large $n-k$.
Analogously to Lemma~\ref{l:VV-fluct}, we have
$\bbP(\max_{\lfloor \rme^n A\rfloor}|\varphi^{U_n,U_n\cap V^-_k}|>\eps/2)<\eps'$ for sufficiently large $n-k$.
The indicator variable on the right-hand side of~\eqref{e:1N-p9} is equal to zero for $n-k$ sufficiently large as $\wh u_\infty$ is harmonic and thus uniformly continuous on $A$. Hence, the right-hand side in~\eqref{e:1N-p9} is bounded by $3\eps'$ for $n-k$ sufficiently large which shows the assertion as $\eps'>0$ was arbitrary.
\end{proof}

The next lemma bounds the probability that the DGFF $h$ exceeds zero on a test set near the boundary of its domain. Here we do not have entropic repulsion as in Proposition~\ref{p:stitch-lb} and therefore the probability only becomes small when the test set is small or the boundary values are low. This is a version of Lemma~B.12 of~\cite{BL-Full} and Lemma~3.8 of~\cite{BDZ16} in our setting under $\bbP_{k,v}^{n,u}$.
\begin{lem}
\label{p:viol}
Let $\eps\in(0,1)$ and $\eta,\zeta\in[0,\eps^{-1}]$. There exists $C=C_\eps<\infty$ such that for all $U\in\frU^\eta_\eps$, $V\in\frV_\eps$, $0\leq k<n$, $u\in\bbR^{\partial U_n}$, $v\in\bbR^{\partial V^-_k}$ with $\osc\,\ol{u}_\eta\leq(n-k)^{1-\eps}$, $\osc\,\ol{v}_\zeta\leq (n-k)^{1-\eps}$, we have
\begin{equation}
\label{e:viol-U}
\bbP_{k,v}^{n,u}\Big(\max_{x\in W_n} h(x) > 0\Big)
\leq C \Leb(W)
(1+|M|)\rme^{\alpha M}
\end{equation}
and
\begin{equation}
\label{e:viol-V}
\bbP_{k,v}^{n,u}\Big(\max_{x\in W'_k} h(x) > 0\Big)
\leq C \Leb(W')
(1+|M|)\rme^{\alpha M}
\end{equation}
for all measurable $W,W'\in\bbR^2$ with $W_n\subset U^\eta_n\cap(\eps^2\rmB^-)_n\cap V^{-,\zeta}_k$ and $W'_k\subset V^{-,\zeta}_k\cap(\eps^{-2}\rmB)_k\cap U^\eta_n$,  where we write
$M= (\ol{u}_\eta(0) + \osc\,\ol{u}_\eta+\osc\,\ol{v}_\zeta)\vee (\ol{v}(\infty)+ \osc\,\ol{v}_\zeta+\osc\,\ol{u}_\eta)$.
\end{lem}
\begin{proof}
By Lemma~\ref{l:UV-U}, we have that
\begin{equation}
\max_{W_n}\heb{(-m_n+u)1_{\partial U_n} + (-m_k+v)1_{\partial V^-_k}}{\partial U_n\cup\partial V^-_k}\leq
-m_n + M + C_\eps\,.
\end{equation}
Then the left-hand side of~\eqref{e:viol-U} is bounded from above by the left-hand side of
\begin{equation}
\bbP\Big(\max_{W_n} h^{U_n\cap V^-_k}> m_n - M -C_\eps \Big)
\leq C_\eps\Leb(W)(1+|M|)\rme^{-\alpha M}\,,
\end{equation}
where the last inequality follows from Lemma~B.12 of~\cite{BL-Full}.

To show~\eqref{e:viol-V}, we assume w.\,l.\,o.\,g.\ that $2\eps^{-2}\rmB\subset U^\eta$ (otherwise we can argue as for~\eqref{e:viol-U}).
Let $\frg$ be the harmonic function on $U_n\cap V^-_k$ with boundary values $u$ on $\partial U_n$ and $v$ on $\partial V^-_k$. Again from Lemma~\ref{l:UV-U}, we obtain $\frg\leq M + C_\eps$ on $U^\eta_n\cap V^{-, \zeta}_k$. By the Gibbs-Markov property and monotonicity in the boundary values,
\begin{multline}
\label{e:p-p:viol-GM}
\bbP_{k,v}^{n,u}\Big(\max_{x\in W'_n} h(x) > 0\Big)
=\bbP_{k,0}^{n,0}\Big(\max_{x\in W'_n} h(x) +\frg(x)> 0\Big)\\
=\int \bbP_{k,0}^{n,0}\Big(h_{\partial (2\eps^{-2}\rmB)_k} +m_k \in\rmd w\Big)
\bbP_{V,k,0}^{2\eps^{-2}\rmB,k,w}\Big(\max_{x\in W'_n} h(x) + M > 0\Big)\\
\leq \int \bbP_{k,0}^{n,0}\Big(\he{h}{\partial (2\eps^{-2}\rmB)_k}(0)+m_k \in\rmd r\Big)
\sum_{a=1}^\infty
\bbP_{k,0}^{n,0}\Big(\osc_{(\eps^{-2}B)_k} h \geq a
\,\Big|\,\he{h}{\partial (2\eps^{-2}\rmB)_k}(0) +m_k\in\rmd r\Big)\\
\times
\bbP_{V,k,|r|+a}^{2\eps^{-2}\rmB,k,|r|+a}\Big(\max_{x\in W'_n} h(x) + M > 0\Big)\,.
\end{multline}
Using Propositions~\ref{p:Cov-bd} and~\ref{p:2.6} for the first two probabilities in the last expression and arguing as for~\eqref{e:viol-U} for the third probability, we further bound~\eqref{e:p-p:viol-GM} by
\begin{equation}
C_\eps\int\rme^{-c_\eps r^2}\sum_{a=1}^\infty \rme^{-(\alpha + 1) a + (\alpha +1 ) |r|}
\Leb(W')(1+a+|r|)(1+|M|)\rme^{\alpha (a+|r|)}\rme^{\alpha M}
\end{equation}
which implies~\eqref{e:viol-V}.
\end{proof}

\subsection{Discrete harmonic analysis}
First we need an estimate for the probability that simple random walk $S$ on $U_n\cap V^-_k$ started in a point $x$ at scale $l$ reaches $U_n$ before $V^-_k$. This is essentially the ruin probability for simple random walk on $\{k,\ldots,n\}$. The following lemma is a variant of Proposition~6.4.1 of~\cite{LaLi}, we consider more general sets $U$, $V$ instead of balls at the cost of a larger error. We recall that $\Pi_U$ denotes the continuum Poisson kernel, i.\,e.\ $\Pi_U(x,\cdot)$ is the exit distribution from the domain $U$ of Brownian motion started at $x$.
\begin{lem}
\label{l:ruin}
Let $\eps\in(0,1)$. Then there exists $C=C_\eps<\infty$ such that for all $U\in\frU_\eps^0$, $V\in\frV_\eps$, $0\leq k<n$,
the following holds:
\begin{enumerate}[label=(\roman{*}), ref=(\roman{*})]
\item\label{i:ruin-C}
For all  
$x\in U_n\cap V^{-}_k$, we have
\begin{equation}
\Big|P_x\big(S_{\tau^{U_n\cap V^-_k}}\in\partial U_n\big)-\frac{\log |x|-k}{n-k}\Big|
\leq C(n-k)^{-1}\,.
\end{equation}
\item\label{i:ruin-U}
For all $x\in U^\eps \cap \eps \rmB^-$, we have
\begin{equation}
P_{\lfloor\rme^n x\rfloor}\big(S_{\tau^{U_n\cap V^-_k}}\in\partial V^-_k\big)
=\frac{\int_{\partial U}\Pi_{U}(x,\rmd z) \log|z| - \log |x| +o_\eps(1)}{n-k}
\end{equation}
as $n-k\to\infty$.
\item\label{i:ruin-V}
For all $x\in V^{-,\eps} \cap \eps^{-1} \rmB$, we have
\begin{equation}
P_{\lfloor\rme^n x\rfloor}\big(S_{\tau^{U_n\cap V^-_k}}\in\partial U_n\big)
=\frac{\log|x|-\int_{\partial V^-}\Pi_{V^-}(x,\rmd z)\log|z| + o_\eps(1)}{n-k}
\end{equation}
whenever $n-k\to\infty$ followed by $k\to\infty$.
\end{enumerate}
\end{lem}
\begin{proof}
As in the proof of Proposition~6.4.1 of~\cite{LaLi}, we use that
$\Big(\fra\big(S_{i\wedge\tau^{U_n\cap V^-_k}} \big)\Big)_{i\geq 0}$ is a bounded martingale under $P_x$ for any $x\in U_n\cap V^-_k$.
To show assertion~\ref{i:ruin-C}, let $q:=P_x\big(\tau^{V^-_k}<\tau^{U_n}\big)$.
By optional sampling, using that $\tau^{U_n}$ is stochastically bonded by a geometric random variable, we have
\begin{equation}
\label{e:ruin-opts}
\fra(x)= E_x \fra\big(S_{\tau^{U_n\cap V^-_k}}\big)
= (1-q)\, E_x \big(\fra \big(S_{\tau^{U_n}}\big)\,\big|\, \tau^{U_n} \leq \tau^{V^-_k}\big)
+ q\, E_x \big(\fra \big(S_{\tau^{V^-_k}}\big)\,\big|\, \tau^{V^-_k} < \tau^{U_n}\big)\,.
\end{equation}
For assertion~\ref{i:ruin-C}, it suffices to estimate by~\eqref{e:444} the conditional expectations on the right-hand side of the last display by $n+O_\eps(1)$ and $k+O_\eps(1)$, respectively, and the left-hand side by $\log|x|+O_\eps(1)$, and to solve for $1-q$.

To show assertion~\ref{i:ruin-U}, we set
$q:=P_{\lfloor\rme^n x\rfloor}\big(\tau^{V^-_k}<\tau^{U_n}\big)$,
and obtain analogously to~\eqref{e:ruin-opts} that
\begin{equation}
\label{e:p-ruin-U-opt}
gn + g\log|x|=(1-q) gn
+g E_{\lfloor \rme^n x\rfloor}
\log\big|\rme^{-n}S_{\tau^{U_n}}\big|
+ qg(k+O(1))
+o_\eps(1)
\end{equation}
for $x\in U^\eps\cap\eps\rmB^-$, where we now also used~\eqref{e:444}, that
\begin{equation}
E_{\lfloor \rme^n x\rfloor}\fra\big(S_{\tau^{U_n\cap V^-_k}}\big)
=(1-q)gn+(1-q)g E_{\lfloor \rme^n x\rfloor}
\Big(\log\big|\rme^{-n}S_{\tau^{U_n}}\big|\,\Big|\, \tau^{U_n}\leq \tau^{V^-_k} \Big) + o_\eps(1)
+q(gk+O(1))
\end{equation}
by~\eqref{e:444}, and that
\begin{equation}
E_{\lfloor \rme^n x\rfloor}
\log\big|\rme^{-n}S_{\tau^{U_n}}\big|
-qC_\eps
\leq E_{\lfloor \rme^n x\rfloor}
\Big(\log\big|\rme^{-n}S_{\tau^{U_n}}\big|\,\Big|\, \tau^{U_n}\leq \tau^{V^-_k} \Big)
\leq (1-q)^{-1}E_{\lfloor \rme^n x\rfloor}
\log\big|\rme^{-n}S_{\tau^{U_n}}\big|\,,
\end{equation}
where $q=o_\eps(1)$ by~\ref{i:ruin-C}, and
$\big|\rme^{-n}S_{\tau^{U_n}}\big|\in[\eps,\eps^{-1}]$ by definition of $\frU^0_\eps$. By Lemma~\ref{l:BL-Pois},
\begin{equation}
\label{e:ruin-p-Pois}
E_{\lfloor \rme^n x\rfloor}
\log\big|\rme^{-n}S_{\tau^{U_n}}\big|
=\int_{\partial U}\Pi_{U}(0,\rmd z)\log|z|+o_\eps(1)
\end{equation}
and~\ref{i:ruin-U} follows by plugging~\eqref{e:ruin-p-Pois} into~\eqref{e:p-ruin-U-opt} and solving for $q$.

Assertion~\ref{i:ruin-V} follows analogously.
\end{proof}

In the next lemma, we work in the setting of Section~\ref{s:concdec}.
We compare the discrete Poisson kernel on the annulus-like set $\Delta_{j}\cap V^-_k$ and the ball-like set $\Delta_j$.
For the latter, we multiply the discrete Poisson kernel with the probability that the random walk that underlies the discrete Poisson kernel first reaches a part of the boundary of the annulus-like set that is also contained in the boundary of the ball-like set.
\begin{lem}
\label{l:Poisson}
Let $\eps\in(0,1)$ and $\eta,\zeta\in[0,\eps^{-1}]$. There exists a constant $C=C_\eps<\infty$ such that
\begin{equation}
\Big|\Pi_{\Delta_j\cap V^-_k}(x,z)-\frac{\rt+1-p}{\rt+1-j}\Pi_{\Delta_j}(0,z)\Big|
\leq C\Pi_{\Delta_j}(0,z)\Big(\frac{1}{\rt+1-j}+(p-j)\rme^{-(p-j)}\Big)
\end{equation}
for all $U\in\frU_\eps^0$, $V\in\frV_\eps$, $0\leq k<n$, and all $j=1,\ldots,\rt -2$, $p=j+2,\ldots,\rt$, $x\in A_{p}$, and $z\in\partial \Delta_j$.
\end{lem}
\begin{proof}
We apply Lemma~\ref{l:ruin} and Proposition~6.4.5 of~\cite{LaLi} ($\mathcal C_n$, $A$, and $\mathcal C_{2m}$ there correspond to our $\Delta_j$, $\Delta_{j}\cap V^-_k$, and $\Delta_{p-1}$, respectively).
\end{proof}
We now give a modification  of the previous lemma in the setting of the outward concentric decomposition from Section~\ref{s:w-outw}.
\begin{lem}
\label{l:Poisson-out}
Let $\eps\in(0,1)$ and $\zeta,\eta\in[0,\eps^{-1}]$. There exists a constant $C=C_\eps<\infty$ such that
\begin{equation}
\Big|\Pi_{\Delta^{\rm o}_{j}\cap U_n}(x,z)-\frac{\rt+1-p}{\rt+1-j}\Pi_{\Delta^{\rm o}_j}(\infty,z)\Big|
\leq C\Pi_{\Delta^{\rm o}_j}(\infty,z)\Big(\frac{1}{\rt+1-j}+(p-j)\rme^{-(p-j)}\Big)
\end{equation}
for all $U\in\frU_\eps^0$, $V\in\frV_\eps$, $0\leq k<n$, and all $j=1,\ldots,\rt -2$, $p=j+2,\ldots,\rt$, $x\in A^{\rm o}_{p}$, and $z\in\partial \Delta^{\rm o}_j$.
\end{lem}
\label{s:app-harm}
\begin{proof}
By Lemma~\ref{l:ruin},
\begin{equation}
\label{e:pPois-o-lruin}
\Big|\Pi_{\Delta^{\rm o}_j\cap U_n}(x,\partial \Delta^{\rm o}_j)
-\frac{\rt+1-p}{\rt+1-j}\Big|\leq
\frac{C_\eps}{\rt+1-j}
\end{equation}
for $x\in A^{\rm o}_{p}$. Moreover, by Lemma~\ref{l:Poisson-diff-o} below, we have
\begin{equation}
\Pi_{\Delta^{\rm o}_j}(x,z)=\Pi_{\Delta^{\rm o}_j}(\infty,z)\big[1+O(\rme^{-(p-j)})\big]
\end{equation}
for $x\in A^{\rm o}_{p}$. Using this in place of~(6.23) in~\cite{LaLi}, we then proceed as in the proof of Proposition~6.4.5 of~\cite{LaLi} to show that
\begin{equation}
\label{e:645o}
P_x\big(S_{\tau^{\Delta^{\rm o}_{j}\cap U_n}}=z\,\big|\, S_{\tau^{\Delta^{\rm o}_{j}\cap U_n}} \in\partial \Delta^{\rm o}_j \big)
=\Pi_{\Delta^{\rm o}_j}(\infty,z)\big[1+O\big((p-j)\rme^{-(p-j)}\big)\big]\,.
\end{equation}
Combined with~\eqref{e:pPois-o-lruin} this yields the assertion.
\end{proof}

\begin{lem}
\label{l:Poisson-diff-o}
Let $\eps>0$. There exists a constant $C=C_\eps<\infty$ such that for all $k\geq 0$, $l\geq k+\eps$, $x\in\rmB^-_l$, $z\in\partial\rmB^-_k$, we have
\begin{equation}
\label{e:harm-diff-o}
\big|\Pi_{\rmB^-_k}(x,z)-\Pi_{\rmB^-_k}(\infty,z)\big|\leq C\rme^{-(l-k)}\Pi_{\rmB^-_k}(\infty,z)\,.
\end{equation}
\end{lem}
\begin{proof}
By optional stopping, it suffices to estimate the absolute difference
$|\Pi_{\rmB^-_k}(x,z)-\Pi_{\rmB^-_k}(y,z)|$ for $x,y\in \partial \rmB^-_l$ (and $z \in \partial \rmB^-_k$).
By Lemma~6.3.6 of~\cite{LaLi},
\begin{equation}
\label{e:LaLi6.3.6}
\Pi_{\rmB^-_k}(x,z)=\sum_{w\in\rmB^-_{k+\eps/2}}
G_{\rmB^-_k}(x,w)P_w\big(S_{\mathring\tau^{\rmB^-_k\setminus\rmB^-_{k+\eps/2}}}=z\big)\,,
\end{equation}
where $\mathring\tau^A=\inf\{i\geq 1:\: S_i\notin A\}$ for $A\subset\bbZ^2$.
By~\eqref{e:444} and~\eqref{e:462}, for $x,y,w$ as before,
\begin{multline}
\label{e:6362}
G_{\rmB^-_k}(x,w)-G_{\rmB^-_k}(y,w)
=g\Big\{\log|y-w|-
\log|x-w| \\
+\sum_{w'\in\partial \rmB^-_k}
\Pi_{\rmB^-_k}(w,w')
\Big(\log|w'-x|-\log|w'-y|\Big)
\Big\}+O(\rme^{-2l})\,.
\end{multline}
For our $x,y\in\partial \rmB^-_l$ and $w\in \partial ( \rmB^-_k\setminus \rmB^-_{k+\eps/2} )$, we obtain, as $\partial \rmB^-_l$ is spherically symmetric up to lattice effects that are negligible in comparison to the range of possible $w$, that
\begin{equation}
\label{e:6363}
\log|y-w|-
\log|x-w| \leq
\log\frac{\rme^{l}+\rme^{k}
}{\rme^{l}- \rme^{k}}
\leq C_\eps\rme^{k-l}\,.
\end{equation}
From~\eqref{e:LaLi6.3.6}, \eqref{e:6362} and~\eqref{e:6363}, we obtain
\begin{equation}
\label{e:6364}
\Big|\Pi_{\rmB^-_k}(y,z)-\Pi_{\rmB^-_k}(x,z)\Big|
\leq C_\eps\rme^{k-l}
\sum_{w\in \rmB^-_{k+\eps/2}}P_w\big(S_{\mathring\tau^{\rmB^-_k\setminus\rmB^-_{k+\eps/2}}}=z\big)\,.
\end{equation}
Summing over the undirected paths that start in
$w\in\rmB^-_{k+\eps/2}$ and thereafter are in $\rmB^-_k\setminus\rmB^-_{k+\eps/2}$ until they hit $z\in\partial\rmB^-_k$, we obtain
\begin{equation}
P_w\big(S_{\mathring\tau^{\rmB^-_{k}\setminus\rmB^-_{k+\eps/2}}}=z\big)
=P_z\big(S_{\mathring\tau^{\rmB^-_{k}\setminus\rmB^-_{k+\eps/2}}}=w\big)
\end{equation}
from the symmetry of the transition kernel of simple random walk.
Hence,~\eqref{e:6364} is further bounded from above by the left-hand side of
\begin{equation}
\label{e:p-P-out-hit}
C_\eps\rme^{k-l}P_z\big(S_{\mathring\tau^{\rmB^-_k\setminus\rmB^-_{k+\eps/2}}}\in \partial( \rmB^-_{k+\eps/2})^\rmc\big)
\leq C_\eps\rme^{-(l-k)-k}
\leq C_\eps\rme^{-(l-k)}\Pi_{\rmB^-_k}(y,z)\,.
\end{equation}
For the second inequality in~\eqref{e:p-P-out-hit}, we used a straightforward modification of Lemma~6.3.7 of~\cite{LaLi}. As in Lemma~6.3.4 of~\cite{LaLi}, optional stopping, dominated convergence and~\eqref{e:444} yield
\begin{multline}
gk + O\big(\rme^{-k}\big)
=\fra(z)=E_z\fra\big(S_{\mathring\tau^{\rmB^-_k\setminus\rmB^-_{k+\eps/2}}}\big)\\
=P_z\big(S_{\mathring\tau^{\rmB^-_{k}\setminus\rmB^-_{k+\eps/2}}}\in \partial \rmB^-_{k+\eps/2}\big)
g(k+\eps/2)
+\Big(1-P_z\big(S_{\mathring\tau^{\rmB^-_k\setminus\rmB^-_{k+\eps/2}}}\in \partial \rmB^-_{k+\eps/2}\big)\Big)gk
\end{multline}
where the error term $O\big(\rme^{-k}\big)$ comes from the discretization and lattice effects in the approximation of the potential kernel. Solving for $P_z\big(S_{\mathring\tau^{\rmB^-_{k}\setminus\rmB^-_{k+\eps/2}}}\in \partial \rmB^-_{k+\eps/2}\big)$ then yields the first inequality in~\eqref{e:p-P-out-hit}.
\end{proof}
We also use the following strengthening of Lemma~A.1 of~\cite{BL-Conf}:
\begin{lem}
\label{l:BL-Pois}
Let $\eps\in(0,1)$.
For any bounded continuous $f:\bbR^2\to\bbR$,
\begin{equation}
\label{e:BL-Pois}
\sum_{z\in\partial D^\pm_n}\Pi_{D^\pm_n}(\lfloor \rme^n x\rfloor,z)f(\rme^{-n}z)
\underset{n\to\infty}\longrightarrow \int_{\partial D}\Pi_{D^\pm}(x,\rmd z)f(z)
\end{equation}
uniformly in $D\in\mathfrak D_\eps$ and $x\in D^{\pm,\eps}$.
\end{lem}
\begin{proof}
For $x\in D^\eps$, this can be gleaned from the proof of Lemma~A.1 of~\cite{BL-Conf} (which uses Skorohod coupling with a Brownian path).

To show the assertion for $x\in D^{-,\eps}$, we fix $\eps'>0$ and assume w.\,l.\,o.\,g.\ that $\|f\|_\infty\leq 1$.
Letting $g(w)=\int_{\partial D}\Pi_{D^-}(w,\rmd z)f(z)$ for $w\in\eps^{-1}\rmB^-$, we find by e.\,g.\ \cite[Theorem~3.46]{MP10} and a standard approximation argument some $p>-\log\eps$ such that
\begin{equation}
\bigg|\int_{\partial \eps^{-1}\rmB}\Big[\Pi_{\eps^{-1}\rmB^-}(x,\rmd z)-\Pi_{\eps^{-1}\rmB^-}(\infty,\rmd z)\Big]g(z)\bigg|<\eps'
\end{equation}
for all $x\in\rme^p\rmB^-$. By definition of $g$ and the strong Markov property of the Brownian motion underlying the Poisson kernel, it follows that
\begin{equation}
\label{e:BL-Pois-infty}
\bigg|\int_{\partial D}\Big[\Pi_{D^-}(x,\rmd z)-\Pi_{D^-}(\infty,\rmd z)\Big]f(z)\bigg|<\eps'
\end{equation}
for all $x\in\rme^p\rmB^-$.

As before, we glean from the proof of Lemma~A.1 of~\cite{BL-Conf} that~\eqref{e:BL-Pois} holds uniformly in $D\in\frD_\eps$ and $x\in D^{-,\eps}\cap\rme^{p+1}\rmB$ (to apply the Skorohod coupling, we use that the hitting time of $D^-$ for planar Brownian motion is tight in its starting point $x\in D^{-,\eps}\cap\rme^{p+1}\rmB$, which can be seen by replacing $D$ with its subdomain $\eps \rmB$ so that tightness follows by monotonicity and spherical symmetry from the finiteness of the hitting time for a single starting point $x\in\partial \rme^{p+1}\rmB$).

To include also $x\in\rme^p\rmB^-$, it suffices by~\eqref{e:BL-Pois-infty} to note that for fixed $x'\in \rmB^-_{n+p}\cap \rmB_{n+p+1}$, the difference
$\sum_{z\in\partial D^\pm_n}\Pi_{D^\pm_n}(x,z)f(\rme^{-n}z)-\sum_{z\in\partial D^\pm_n}\Pi_{D^\pm_n}(x',z)f(\rme^{-n}z) $
is bounded harmonic in $x\in\rmB^-_{n+p}$ and hence, by the maximum principle for harmonic functions, attains it absolute maximum on $\rmB_{n+p+1}\cap\rmB^-_{n+p}$, which is, by~\eqref{e:BL-Pois-infty} and the assertion for $x\in\rme^{p+1}\rmB$, bounded by $2\eps'$ for sufficiently large $n$, depending only on $\eps$ and $f$.
\end{proof}

\section{Estimates for the harmonic extensions of the DGFF}
\label{s:tools}

\subsection{Harmonic extensions at intermediate scales}
\label{s:he}

In many applications, the ballot estimates are used on a subset of the domain of definition of the field, with either the inner or outer part of the boundary conditions being on an ``intermediate scale''. To be more precise, if $U_n \cap V_k^-$ is the original domain on which $h$ is considered, with $U,V \in \frD$ and $n,k \geq 0$, then one would like to apply the ballot theorems on $U_n \cap W_l^-$ or $W_l \cap V_k^-$, where $W \in \frD$ and $l \geq 0$ is some ``intermediate'' scale satisfying $k + O(1) \leq l \leq n + O(1)$. 

Thanks to the Gibbs-Markov property and the disjointness of $U_n \cap W_l^-$ and $W_l \cap V_k^-$, conditional on the values of $h$ on $\partial W_l^\pm$, the restrictions of $h$ to the latter sub-domains form independent fields, with each distributed like a DGFF with boundary conditions given by the original ones on $\partial U_n$ and  $\partial V_k^-$ together with the values on $\partial W_l^\pm$ we condition $h$ to take. We may then use the ballot estimates in both sub-domains, but for these to be useful, we must have good control over the harmonic extension of the values of $h$ on $\partial W_l^\pm$, as many of the conditions, bounds and asymptotics in the previous subsection are expressed in terms of the latter. 

The purpose of this subsection is therefore to study $\he{h}{\partial W^\pm_l}$ -- the unique bounded harmonic extension of the values of $h$ on $\partial W^\pm_l$ to the entire plane, 
where $h$ is the DGFF on $U_n \cap V_k^-$ with boundary conditions given as before by $u$ on $\partial U_n$ and $v$ on $\partial V_k^-$. As the harmonic extension is a linear function of the field, it is also Gaussian and hence its law is completely determined by its mean and covariance, which we denote by
\begin{equation}
\label{e:mu}
\mu(x ; u,v) := \mu_{n,l,k}(x ; u,v) = 
\bbE \Big(\he{h}{\partial W^\pm_l}(x) \,\Big|\,
h_{\partial U_{n}} = -m_n + u ,\,
h_{\partial V^-_{k}} = -m_k + v
\Big)
\end{equation}
and
\begin{equation}
\sigma(x,y) = \sigma_{n,l,k}(x,y) := \Cov \Big(\he{h}{\partial W^\pm_l}(x),\, \he{h}{\partial W^\pm_l}(y) \,\Big|\,
h_{\partial U_{n}} = 0,\,
h_{\partial V^-_{k}} = 0 \Big) \,.
\end{equation}

For what follows, we let $\wh{m}_l = \wh{m}_{n,l,k}$ be the linear interpolation at scale $l$ between $m_k$ and $m_n$, that is
\begin{equation}
\label{e:bl}
\wh{m}_l = \wh{m}_{n,l,k} := \frac{(l-k) m_n + (n- l) m_k}{n-k}
\quad ; \qquad 0 \leq k < l < n \,.
\end{equation}
Setting $\wedge_{n,l}=l\wedge (n-l)$ and $\wedge_{n,l,k}=(l-k)\wedge(n-l)$, the following is well known and shows that the difference between $\wh{m}_{n,l,k}$ and $m_l$ is of logarithmic order.
\begin{lem}
\label{l:blml}
For all $0\leq k< l<n<\infty$,
\begin{equation}
-2\sqrt{g}\leq \wh{m}_l - m_l\leq \tfrac32 \sqrt{g}\log^+ \wedge_{n,l,k} +\tfrac{3}{2}\sqrt{g}\,.
\end{equation}
\end{lem}
The proofs for this and the next subsection are given in Section~\ref{s:p-compl}.
The following three propositions provide bounds on the mean, covariances and (conditional) oscillations of $\he{h}{\partial W^\pm_l}$.
\begin{prop}
\label{l:E}
Fix $\eps\in(0,1)$, $\eta,\zeta \in[0,\eps^{-1}]$ and let $U,V\in\frD_\eps$ such that $U$ and $V^-$ are connected.
Let also $k,l,n \geq 0$ with $k < n$ and $W \in \frD_\epsilon$ such that $\partial W_l \subset U^{\eta}_{n}$ and $\partial W^-_l \subset V^{-,\zeta}_{k}$. Then, for all $u \in \bbR^{\partial U_n}$, $v \in \bbR^{\partial V^-_k}$ and $x \in \bbZ^2\cup\{\infty\}$,
\begin{multline}
%\label{e:l:2.4:E}
\bigg| \mu(x ;\: u,v) - \bigg(-\wh{m}_l +
\frac{( l-k)\ol{u}(0)  +  (n- l)\ol{v}(\infty)}{n-k} \bigg) \bigg| \\
\leq 2\osc\,\ol{u}_\eta + 2\osc\,\ol{v}_\zeta +  C\frac{|\ol{u}(0)|+|\ol{v}(\infty)|}{n-k} +C \,.
\end{multline}
Above $C = C_{\epsilon} < \infty$.
\end{prop}
\begin{prop}
\label{p:Cov-bd}
Let $\eps\in(0,1)$. There exists $C=C_\eps<\infty$ such that for all $U,V\in\frD_\eps$ for which $U$ and $V^-$ are connected, all $0\leq k < n$ and
$W \in \frD_\epsilon$ with $\partial W_l \subset U^{\eps}_{n}$ and $\partial W^-_l \subset V^{-,\eps}_{k}$, and for all $x,y \in W^{\pm,\eps}_{l}\cup\{\infty\}$,
we have
\begin{equation}
\label{e:l:2.4:io}
\bigg| \sigma(x,y) - g \frac{(n- l) ( l-k)}{n-k} \bigg| \leq C\,.
\end{equation}
%Above $C = C_{\epsilon, \eta} < \infty$.
\end{prop}
The next proposition bounds the oscillation of $\he{h}{\partial W^\pm_l}$ in the bulk.
\begin{prop}
\label{p:2.6}
Let $\eps\in(0,1)$, $\eta,\zeta\in[0,\eps^{-1}]$. There exist $C=C_\eps<\infty$ and $c=c_\eps>0$
such that for all $U,V\in\frD_\eps$ for which $U$ and $V^-$ are connected, all $0\leq k<l<n$ and
$W \in \frD_\epsilon$ with $\partial W_l \subset U^{\eta\vee\eps}_{n}$ and $\partial W^-_l \subset V^{-,\zeta\vee\eps}_{k}$,
and for all $x\in W^{\pm,\eps}_l\cup\{\infty\}$, $w\in\bbR$, $t\geq 0$,
we have
\begin{multline}
\label{e:p:2.6:2}
\bbP \Big(\osc_{W^{\pm,\eps}_{l}} \he{h}{\partial W^\pm_l} > t \,\Big|\, 
h_{\partial U_n} = -m_n + u ,\,
h_{\partial V^-_k} = -m_k + v ,\,
\he{h}{\partial W^\pm_l}(x) = -m_l + w\Big) \\
 \leq C \exp \Big(-c \big((t-C \mu)^+\big)^2\Big)\,,
\end{multline}
where 
\begin{equation}
\mu = \mu_{n,l,k, \eta}(u,v,w) := \frac{|w|+ \,\osc\,\ol{u}_\eta + 
	\osc\,\ol{v}_\zeta}{(n-l)\wedge(l-k)} + \frac{|\ol{u}(0)|}{n- l} +  \frac{|\ol{v}(\infty)|}{ l-k} \,.
\end{equation}
%Above $C = C_{\epsilon, \eta} < \infty$ and $c = c_{\epsilon, \eta} > 0$.
\end{prop}
We note the following one-sided variant of Proposition~\ref{p:2.6} which will be of use in~\cite{Crit}.
\begin{prop}
\label{p:unc-bd}
Let $\eps\in(0,1)$, $\eta\in[0,\eps^{-1}]$. There exist $C=C_{\eps}<\infty$ and
$c=c_{\eps}>0$ such that for all connected $U\in\frD_\eps$, all $W\in\frD_\eps$, and all
$0 \leq   l < n <\infty$ with $\partial W_l \subset U_n^{\eta\vee\eps}$ and all $t > 0$, $t\in\bbR$, $x\in W^{\pm,\eps}_l\cup\{\infty\}$,
and all boundary conditions $u\in\bbR^{\partial U_n}$, we have
\begin{multline}
\bbP \Big(\osc_{W^{\pm,\eps}_l}\he{h}{\partial W^{\pm}_l} > t \,\Big|\, h_{\partial U_n}=-m_n+u,\,\he{h}{\partial W^{\pm}_l}(x) = -m_l + w\Big) \\
\leq C \exp \Big(-c \big((t-C\mu)^+\big)^2\Big)\,,
\end{multline}
where
\begin{equation}
\mu = \mu(u, w) := \frac{|w|+ |\ol{u}(0)|+\osc\,\ol{u}_\eta}{n- l}\,.
\end{equation}
\end{prop}

For the proof of Proposition~\ref{p:LR-infty}, we also need an asymptotic estimate for the covariance.
\begin{lem}
\label{l:Cov-as}
Let $\eps\in(0,1)$ and $j\in\bbN$. Uniformly in $U,V,W\in\frD_\eps$ with $W \subset \rme^j U^\eps$ and for which $U$ and $V^-$ are connected, and in $x,y\in W^{\pm,\eps}\cup\{\infty\}$, we have
\begin{multline}
\label{e:Cov-as-n}
\lim_{n-k\to\infty} \sigma_{n,n-j,k}\big(\lfloor \rme^{n-j} x\rfloor , \lfloor \rme^{n-j}y\rfloor)\\
=g \int_{w,w'\in\partial W}
\Pi_{\bbR^2\setminus\partial W}(x,\rmd w)
\Pi_{\bbR^2\setminus\partial W}(y,\rmd w')
\Big(\int_{z\in\partial (\rme^j U)}\Pi_{\rme^j U}(w,\rmd z)
\log|z-w'| -\log |w-w'|\Big)\,.
\end{multline}
Uniformly in $U,V,W\in\frD_\eps$ with $\rme^j W^- \subset V^{-,\eps}$ and for which $U$, $V^-$ are connected, and in $x,y\in W^{\pm,\eps}\cup\{\infty\}$,
we have
\begin{multline}
\label{e:Cov-as-k}
\lim_{\substack{n-k\to\infty\\\text{with\ } k\to\infty}} \sigma_{n,k+j,k}(\lfloor \rme^{k+j} x\rfloor , \lfloor \rme^{k+j}y \rfloor )\\
=g \int_{w,w'\in\partial W}
\Pi_{\bbR^2\setminus\partial W}(x,\rmd w)
\Pi_{\bbR^2\setminus\partial W}(y,\rmd w')
\Big(\int_{z\in\partial (\rme^{-j} V)}\Pi_{\rme^{-j} V^-}(w,\rmd z)
\log|z-w'| -\log |w-w'|\Big)\,.
\end{multline}
\end{lem}

We will also need the following one-sided estimates for the variance.
\begin{lem}
\label{l:Var-inw-1}
Let $\eps\in(0,1)$. There exists $C=C_{\eps}<\infty$
such that for all $U,V,W\in\frD_\eps$, $0\leq k\leq l\leq n$
with $\partial W_l\subset U^\eps_n$ and $\partial W^-_l\subset V^{-,\eps}_k$, and for all $x,y\in W^{\pm,\eps}_l\cup\{\infty\}$, we have
\begin{equation}
\label{e:Var-inw}
\Big|\Cov\Big(\he{h}{\partial W^\pm_l}(x),\he{h}{\partial W^\pm_l}(y)\,\Big|\, h_{\partial U_n}=0\Big)
-g(n-l)\Big|\leq C\,,
\end{equation}
\begin{equation}
\label{e:Var-outw}
\Big|\Cov\Big(\he{h}{\partial W^\pm_l}(x),\he{h}{\partial W^\pm_l}(y)\,\Big|\, h_{\partial V^-_k}=0\Big)
-g(l-k)\Big|\leq C\,.
\end{equation}
\end{lem}
We also bound the intrinsic metric of the centered DGFF.
\begin{lem}
\label{l:metr}
Let $\eps\in(0,1)$. Then, there exists $C = C_\eps < \infty$ such that
for all $U,V,W\in\frD_\eps$ and $0\leq k\leq l\leq n$
with $\partial W_l \subset U_n^\epsilon$ and $\partial W_l^- \subset V_k^{-,\epsilon}$, all $x,y$
that are in the same connected component of $W^{\pm,\eps}_l$, we have
\begin{equation}
\label{e:l:2.4:metr}
\bbE\Big[\Big(\he{h^{U_n\cap V^-_k}}{\partial_l W}(x) - \he{h^{U_n\cap V^-_k}}{\partial_l W}(y) \Big)^2\Big] 
\leq C \frac{|x-y|^2}{\rme^{2 l}}\,.
\end{equation}
\end{lem}

\subsection{Harmonic extensions as binding fields}
\label{s:compl-bdg}
This subsection contains estimates on the harmonic extension of the DGFF on $U_n\cap V^-_k$ to a subdomain of $U_n\cap V^-_k$. A harmonic extension on such a domain occurs as the {\em binding field} in the Gibbs-Markov decomposition, namely $\varphi$ in following lemma:
\begin{lem}[Gibbs-Markov property]
\label{l:GM}
Let $h$ be a DGFF on a nonempty set $A\subsetneq\bbZ^2$, let $B\subset A$ be finite, and $\varphi=\bbE(h\mid h_{A\setminus B})$. Then $h-\varphi$ is a DGFF on $B$ with zero boundary conditions, and $\varphi$ is independent of $h-\varphi$ and satisfies $\varphi(x)=\he{h}{\overline{A}\setminus B}$ a.\,s.
\end{lem}
\begin{proof}
See e.\,g.\ Lemma~3.1 in~\cite{Bi-LN}.
\end{proof}

The estimates from this subsection will be used in the proofs in Sections~\ref{s:weak} and~\ref{s:proof-DGFF}.

The first of these estimates shows that the binding field vanishes in the bulk when the Hausdorff distance of the complement of the domains goes to zero. It contains a discrete version of Lemma~3.7 in~\cite{BL-Conf}.
\begin{lem}
\label{l:UU-fluct}
Let $\eps\in(0,1)$, $r\geq 0$.
There exist $C=C_{\eps,r}<\infty$ and $\rho=\rho_{\eps,r}:\bbR_+\to\bbR_+$ with $\rho(s)\downarrow 0$ as $s\downarrow 0$ such that
\begin{equation}
\label{e:UU-fluct-i}
\bbE\Big(\max_{(U\cap \wt U)^\eps_n\cap \rmB^{-}_{n-r}}
\big|\he{h^{U_n\cap \rmB^-_{0}}}{\partial(U\cap \wt U)_n\cup \partial \rmB^-_0}\big|\Big)\leq \rho\big(\dH(U,\wt U)\big)+o_{\eps,r}(1)
\end{equation}
\begin{equation}
\label{e:UU-fluct-i-P}
\bbP\Big(\max_{(U\cap \wt U)^\eps_n\cap \rmB^{-}_{n-r}}
\big|\he{h^{U_n\cap \rmB^-_0}}{\partial(U\cap \wt U)_n\cup \partial \rmB^-_0}\big|>t\Big)\leq C\exp\Big\{-\frac{t^2}{\rho\big(\dH(U,\wt U)\big)+o_{\eps,r}(1)}\Big\}
\end{equation}
and
\begin{equation}
\label{e:UU-fluct-o}
\bbE\Big(\max_{\rmB_{k+r}\cap(V^-\cap \wt V^-)^\eps_k}
\big|\he{h^{\rmB_n\cap V^-_k}}{\partial \rmB_n\cup \partial (V^-\cap \wt V^-)_k}\big|\Big)\leq \rho\big(\dH(V^-,\wt V^-)\big) + o_{\eps,r}(1)
\end{equation}
\begin{equation}
\label{e:UU-fluct-o-P}
\bbP\Big(\sup_{\rmB_{k+r}\cap(V^-\cap \wt V^-)^\eps_k}
\big|\he{h^{\rmB_n\cap V^-_k}}{\partial \rmB_n\cup \partial (V^-\cap \wt V^-)_k}\big|>t\Big)
\leq C\exp\Big\{-\frac{t^2}{\rho\big(\dH(V^-,\wt V^-)\big)+o_{\eps,r}(1)}\Big\}
\end{equation}
where $o_{\eps,r}(1)$ tends to zero as $n\to\infty$ in~\eqref{e:UU-fluct-i}, \eqref{e:UU-fluct-i-P}, and as $n-k\to\infty$ followed by $k\to\infty$ in \eqref{e:UU-fluct-o} and~\eqref{e:UU-fluct-o-P}, and the statements hold for all $t>0$, $U,\wt U,V,\wt V\in\frD_\eps$ such that $U$, $\wt U$, $V^-$ and $\wt V^-$ are connected, $0\leq k < n <\infty$.
\end{lem}

The next lemma shows that the binding field also becomes small in a region that is many scales away from where the domains differ.
\begin{lem}
\label{l:VV-fluct}
Let $\eps\in(0,1)$, $r\geq 0$.
There exist $C=C_{r,\eps}<\infty$, $c=c_{r,\eps}>0$ such that
\begin{equation}
\label{e:VV-fluct-i}
\bbE\Big(\max_{U_n\cap \rmB^-_{n-r}}
\big|\he{h^{U_n\cap \rmB^-_0}}{\partial U_n\cup \partial V^-_{k}}\big|\Big)
\leq C \frac{\sqrt{\wedge_{n,k,0}}}{n-k}
\,,
\end{equation}
\begin{equation}
\label{e:VV-fluct-i-P}
\bbP\Big(\max_{U_n\cap \rmB^-_{n-r}}
\big|\he{h^{U_n\cap \rmB^-_0}}{\partial U_n\cup \partial V^-_{k}}\big|>t\Big)
\leq C \rme^{-c(n-k)t^2}\,,
\end{equation}
and
\begin{equation}
\label{e:VV-fluct-o}
\bbE\Big(\max_{V^-_k\cap \rmB_{k+r}}
\big|\he{h^{\rmB_{\wt n}\cap V^-_k}}{ U_{n}\cup \partial V^-_k}\big|\Big)
\leq C\frac{\sqrt{\wedge_{\wt n,n,k}}}{n - k}\,,
\end{equation}
\begin{equation}
\label{e:VV-fluct-o-P}
\bbP\Big(\max_{V^-_k\cap \rmB_{k+r}}
\big|\he{h^{\rmB_{\wt n}\cap V^-_k}}{ U_{n}\cup \partial V^-_k}\big|>t\Big)
\leq C \rme^{-c(n-k)t^2}
\end{equation}
for all $t>0$, $0\leq k < n \leq \wt n <\infty$, and
$U,V\in\frD_\eps$ such that $U$ and $V^-$ are connected.
\end{lem}

We also note the following estimate for the covariance.
\begin{lem}
\label{l:Cov-bdg}
Let $\eps\in(0,1)$. Then, there exists $C = C_\eps < \infty$ such that
\begin{equation}
\label{e:Cov-bdg-V}
\Cov\Big(\he{h^{U_n}}{\partial U_n\cup \partial V^-_k}(x), \he{h^{U_n}}{\partial U_n\cup \partial V^-_k}(y)\Big)
\leq C \frac{\big(n- \log(|x|\wedge|y|)\big)^2}{n-k}
\end{equation}
and
\begin{equation}
\label{e:Cov-bdg-U}
\Cov\Big(\he{h^{V^-_k}}{\partial U_n\cup \partial V^-_k}(x), \he{h^{V^-_k}}{\partial U_n\cup \partial V^-_k}(y)\Big) 
\leq C \frac{\big(\log(|x|\vee|y|)-k\big)^2}{n-k}
\end{equation}
for all $U,V\in\frD_\eps$ such that $U$ and $V^-$ are connected, 
$0 \leq k  < n < \infty$,
and all $x,y\in U_n^\eps\cap V^{-,\eps}_k$.
\end{lem}

\subsection{Proofs}
\label{s:p-compl}
We now prove the statements from Sections~\ref{s:he} and~\ref{s:compl-bdg}.

\subsubsection{Mean, covariance, and intrinsic metric}

\begin{proof}[Proof of Lemma~\ref{l:blml}]
We write
$e_l=\big(\tfrac{3}{4}\sqrt{g}\big)^{-1} (\wh m_l-m_l)$.
By definition of $\wh m_l$ and $m_l$,
\begin{equation}
e_l =\log^+ l - \frac{ l-k}{n-k} \log^+ n
-\frac{n- l}{n-k}\log^+ k\,\qquad
\frac{\rmd}{\rmd l} e_l=
\frac{\Ind_{\{l > 1\}}}{ l} -\frac{\log^+ n-\log^+ k}{n-k}\,,
\end{equation}
for all $0\leq k<l<n$, and $l\neq 1$ for the derivative. Note that $e_k=e_n=0$ and $-(k\vee 1)^{-1}\leq \tfrac{\rmd}{\rmd l} e_l\leq l^{-1}\Ind_{\{l\geq 1\}} $. Hence, if $l\in[n-2,n]$ or $l\in[k,k+2]$, then $|e_l|\leq 2$.

If $l\in[k+1,n-2]$, then $e_l\leq 2+\log( l-k)$ where we also used subadditivity of the logarithm. For $n\geq 1\vee(2k)$, we have
$\tfrac{\rmd}{\rmd l}e_l\geq -2n^{-1} \log n$. Hence, if $k\leq n/2$ and $l\in[k+1,n-2]$, we also have $e_l\leq 2+(n- l)2n^{-1}\log n\leq 2+2\log(n- l)$, again by subadditivity of the logarithm for the last inequality.

Furthermore, $e_l\leq 2$ if $k\geq\max\{ 1,n/2\}$.
By concavity of the logarithm, $e_l\geq 0$ if $l\geq 1$.
\end{proof}

In the following lemma, we compare harmonic averages of the boundary conditions.
According to our notational conventions, e.\,g.\ $\he{u}{\partial U_n\cup \partial V^-_k}$ is obtained by harmonically extending the boundary conditions $u$ on $\partial U_n$ and $0$ on $\partial V_k\setminus \partial U_n$.

\begin{lem}
\label{l:UV-U}
Let $\eps\in(0,1)$.
Then
\begin{equation}
\label{e:UV-U-u}
\Big|
P_x\big(\tau^{U_n}\leq \tau^{V^-_k}\big)\ol{u}(x)
-\he{u}{\partial U_n\cup \partial V^-_k}(x)\Big|
\leq
P_x\big(\tau^{V^-_k}<\tau^{U_n}\big)\,\osc_{U'}\ol{u}
\end{equation}
and
\begin{equation}
\label{e:UV-U-v}
\Big|P_x\big(\tau^{V^-_k}<\tau^{U_n}\big)\ol{v}(x)
-\he{v}{\partial U_n\cup \partial V^-_k}(x)\Big|
\leq
P_x\big(\tau^{U_n}\leq \tau^{V^-_k}\big)\,\osc_{V'}\ol{v}
\end{equation}
for all $U,V\in\frD_\eps$, $0\leq k<n<\infty$, $u\in\bbR^{\partial U_n}$, $v\in\bbR^{\partial V^-_k}$, $x\in U_n\cap V^-_k$, $U'\supset\{x\}\cup\partial V^-_k$ and $V'\supset\{x\}\cup\partial U_n$.
\end{lem}

\begin{proof}
We show~\eqref{e:UV-U-u}, the proof of~\eqref{e:UV-U-v} is analogous. Let $0\leq k\leq n$.
Decomposing at the exit time from $V^-_k$ yields 
\begin{multline}
E_x u(S_{\tau^{U_n}})
=E_x\big(u(S_{\tau^{U_n\cap V^-_k}});
S_{\tau^{U_n\cap V^-_k}}\in \partial U_n\big)\\
+P_x\big(\tau^{V^-_k}<\tau^{U_n}\big)
\sum_{y\in \partial V^-_k}
P_x\big(S_{\tau^{V^-_k}}= y\,\big|\, \tau^{V^-_k}<\tau^{U_n} \big)
E_y u(S_{\tau^{U_n}})\,.
\end{multline}
Inserting $E_y u(S_{\tau^{U_n}})=\ol{u}(x)+\big(\ol{u}(y)-\ol{u}(x)\big)$ and rearranging terms, we obtain
\begin{equation}
\big|E_x\big(u(S_{\tau^{U_n\cap V^-_k}});
S_{\tau^{U_n\cap V^-_k}}\in \partial U_n\big)
-P_x\big(\tau^{U_n}\leq \tau^{V^-_k}\big)\ol{u}(x)\big|
\leq P_x\big(\tau^{V^-_k}<\tau^{U_n}\big)
\max_{y\in \partial V^-_k}|\ol{u}(x)-\ol{u}(y)|\,.
\end{equation}
The expectation in the last display equals $\he{u}{\partial U_n\cup \partial V^-_k}$.
\end{proof}
To further bound the oscillation terms in Lemma~\ref{l:UV-U}, we have the following lemma:
\begin{lem}
\label{l:osc-far}
Let $\eps\in(0,1)$. There exist $C=C_\eps<\infty$ and $\theta=\theta_\eps<\eps/3$ such that the following holds:
For all $\eta,\zeta\in[0,\eps^{-1}]$, $U\in\frU^\eta_\eps$, $V\in\frV_\eps$, $0\leq k<n<\infty$, $u\in\bbR^{\partial U_n}$, $v\in\bbR^{\partial V^-_k}$, $x\in U^\eta_n\cap V^{-,\zeta}_k$, we have
\begin{equation}
\label{e:osc-far-u-raw}
\osc_{\{x\}\cup\partial V^-_k}\ol{u}\leq C \rme^{-n}|x|\osc\,\ol{u}_\eta
\end{equation}
and
\begin{equation}
\label{e:osc-far-v-raw}
\osc_{\{x\}\cup\partial U_n}\ol{v}\leq C \frac{\rme^{k}}{|x|}\osc\,\ol{v}_\zeta\,.
\end{equation}
If $x\in(\theta\rmB)_n\cap(\theta^{-1}\rmB^-)_k$ and $n-k>C$, then we also have
\begin{equation}
\label{e:osc-far-u}
\osc_{\{x\}\cup\partial V^-_k}\ol{u}\leq
P_x\big(\tau^{U_n}<\tau^{V^-_k}\big)\osc\,\ol{u}_\eta
\end{equation}
and
\begin{equation}
\label{e:osc-far-v}
\osc_{\{x\}\cup\partial U_n}\ol{v}\leq P_x\big(\tau^{V^-_k}<\tau^{U_n}\big)\osc\,\ol{v}_\zeta\,.
\end{equation}
\end{lem}
\begin{proof}
We first show assertions~\eqref{e:osc-far-u-raw} and~\eqref{e:osc-far-u}.
If $x\in U^\eta_n\setminus (\eps\rmB)_n$, then we have
$\rme^{-n}|x|\geq \eps$, and there is nothing to show in~\eqref{e:osc-far-u-raw} as $\{x\}\cup\partial V^-_k\subset U^\eta_n$.
We can thus assume that $x\in (\eps\rmB)_n\cap V^{-,\zeta}_k$.
From the definition of the oscillation and the strong Markov property of the random walk underlying the definition of the Poisson kernel, we obtain
\begin{multline}
\label{e:osc-far-p-Markov}
\osc_{\{x\}\cup \partial V^-_k}\ol{u}\leq
2\max_{y\in V^-_k}\bigg|\sum_{w\in \partial U_n}
\Big[\Pi_{U_n}(y,w)-\Pi_{U_n}(x,w)\Big]\big[u(w)-\ol{u}(x)\big]\bigg|\\
\leq 2\max_{y\in V^-_k}\bigg|\sum_{w'\in \partial (\eps\rmB)_n}
\Big[\Pi_{(\eps\rmB)_n}(y,w')-\Pi_{(\eps\rmB)_n}(x,w')\Big]
\sum_{w\in\partial U_n}\Pi_{U_n}(w',w)\big[u(w)-\ol{u}(x)\big]\bigg|\\
\leq 2\max_{y\in V^-_k}\sum_{w'\in \partial (\eps\rmB)_n}
\Big|\Pi_{(\eps\rmB)_n}(y,w')-\Pi_{(\eps\rmB)_n}(x,w')\Big|
\big|\ol{u}(w')-\ol{u}(x)\big|\,.
\end{multline}
Applying Theorem~6.3.8 of~\cite{LaLi} along a shortest path $(x_i)_{i=0}^l$ from $x_0:=x$ to $x_l:=y$ in $U^\eta_n\cap V^{-,\zeta}_k$, we obtain that
\begin{equation}
\label{e:diff-Pois-p-i}
\Big|\Pi_{(\eps\rmB)_n}(y,w')-\Pi_{(\eps\rmB)_n}(x,w')\Big|
\leq C_\eps\rme^{-n}\sum_{i=1}^l\Pi_{(\eps\rmB)_n}(x_i,w')\,,
\end{equation}
hence we can bound the right-hand side of~\eqref{e:osc-far-p-Markov} from above by
\begin{equation}
C_\eps l\rme^{-n}\max_{w'\in U^\eta_n}\big|\ol{u}(w')-\ol{u}(x)\big|\leq C_\eps l\rme^{-n} \osc\,\ol{u}_\eta
\end{equation}
which shows~\eqref{e:osc-far-u-raw} as $l$ is the length of a shortest path from $x\in V^-_k$ to $y\in\partial V^-_k$. Moreover, 
\begin{equation}
C_\eps\rme^{-n}l\leq C_\eps\frac{|x|}{\rme^n}
= C_\eps \frac{|x|\rme^{-k}}{\rme^{n-k}}
\leq \tfrac12 \frac{\log|x|-k}{n-k}
\leq P_x\big(\tau^{U_n}<\tau^{V^-_k}\big)
\end{equation}
for $x\in(\theta\rmB)_n\cap(\theta^{-1}\rmB^-)_k$ and $n-k>C$, where we choose
$\theta$ sufficiently small and $C$ sufficiently large such that last and second last inequalities hold by a straightforward computation and Lemma~\ref{l:ruin}\ref{i:ruin-C}, respectively.
This yields assertion~\eqref{e:osc-far-u}.

To show assertion~\eqref{e:osc-far-v-raw}, we argue analogously but use Lemma~\ref{l:Poisson-diff-o} in place of~\eqref{e:diff-Pois-p-i} in
\begin{equation}
\Big|\Pi_{\eps^{-1}\rmB^-_k}(y,w)-\Pi_{\eps^{-1}\rmB^-_k}(x,w')\Big|
\leq C_\eps\frac{\rme^{k}}{|x|}\Pi_{\eps^{-1}\rmB^-_k}(\infty,w)\,.
\end{equation}
For~\eqref{e:osc-far-v}, we then also have
\begin{equation}
C_\eps\frac{\rme^k}{|x|}\leq
\tfrac12 \frac{\log \frac{\rme^n}{|x|}}{\log \rme^{n-k}}
\leq P_x\big(\tau^{V^-_k}<\tau^{U_n}\big)\,,
\end{equation}
again by Lemma~\ref{l:ruin}\ref{i:ruin-C} and for $x\in(\theta\rmB)_n\cap(\theta^{-1}\rmB^-)_k$ for $\theta$ sufficiently small and $n-k$ sufficiently large.
\end{proof}
\begin{proof}[Proof of Proposition~\ref{l:E}]
By linearity of the expectation, we have
\begin{equation}
\label{e:E-p1}
\mu(x;\;u,v)=\sum_{y\in\partial W^\pm_l}\Pi_{\bbZ^2\setminus W^\pm_l}(x,y)\bbE_{k,v}^{n,u} h(y)\,,
\end{equation}
where
\begin{equation}
\label{e:E}
\bbE_{k,v}^{n,u} h(y) = 
-m_n + (m_n-m_k)P_w\big(\tau^{V^-_k}<\tau^{U_n}\big)+
\he{u}{\partial U_n\cup \partial V^-_k}(y) + \he{v}{\partial U_n\cup \partial V^-_k}(y)\,.
\end{equation}
By~\eqref{e:E}, as $|\log|y|-l|\leq C_\eps$ for $y\in \partial W^\pm_l$, and by Lemmas~\ref{l:UV-U} and~\ref{l:ruin}\ref{i:ruin-C},
\begin{equation}
\label{e:E0}
\bigg|\bbE_{k,v}^{n,u} h(y) - \frac{ l-k  }{n-k}\ol{u}(0) - \frac{n-l}{n-k}\ol{v}(\infty) -\wh{m}_l\bigg|
\leq 2\osc\,\ol{u}_\eta+2\osc\,\ol{v}_\zeta+C_\eps+C_\eps\frac{|\ol{u}(0)|+|\ol{v}(\infty)|}{n-k}
\end{equation}
and the assertion follows from~\eqref{e:E-p1}.
\end{proof}

\begin{proof}[Proof of Proposition~\ref{p:Cov-bd}]
Representing the harmonic extension in terms of the Poisson kernel and inserting the covariance of the DGFF, we obtain
\begin{equation}
\label{e:l:2.4:pVar}
\Cov_{k}^{n}\big(\he{h}{\partial W^{\pm}_l}(x),\he{h}{\partial W^{\pm}_l}(y)\big)
=\sum_{w,w'\in \partial W^{\pm}_l}\Pi_{\bbZ^2\setminus\partial W^{\pm}_l}(x,w)\Pi_{\bbZ^2\setminus\partial W^{\pm}_l}(y,w')G_{U_n\cap V^-_k}(w,w')
\end{equation}
for $x,y\in\bbZ^2\cup\{\infty\}$. Here, by~\eqref{e:462},
\begin{multline}
\label{e:GUV}
G_{U_n\cap V^-_k}(w,w')=\sum_{z\in\partial U_n}\Pi_{U_n\cap V^-_k}(w,z)
\big(\fra(z-w')-l\big)+\sum_{z\in \partial V^-_k}\Pi_{U_n\cap V^-_k}(w,z)\big(\fra(z-w')-l\big)\\
-\big(\fra(w-w')-l\big)
\end{multline}
for $w,w'\in\partial W^{\pm}_l$.
By Lemma~\ref{l:ruin} and \eqref{e:444},
the first sum on the right-hand side of~\eqref{e:GUV} equals
\begin{equation}
g\frac{ l-k+O_\eps(1)}{n-k}\big(\log\big(\rme^{n}+O_\eps(\rme^l)\big)-l\big)\,.
\end{equation}
Using also the analogous estimate for the second sum in~\eqref{e:GUV}, we obtain
\begin{equation}
\label{e:l:2.4:GVar}
G_{U_n\cap V^-_k}(w,w')
=g\frac{( l-k)(n- l)}{n-k}-\big(\fra(w-w')-l\big)
+O_\eps(1)
\end{equation}
for $w,w'\in\partial W^{\pm}_l$.

We now plug~\eqref{e:l:2.4:GVar} into~\eqref{e:GUV}. Then it remains to show that the second term on the right-hand side of~\eqref{e:l:2.4:GVar}, under summed over $w,w'$ from~\eqref{e:GUV}, remains bounded. To this aim, we let $\delta>0$, and decompose this sum into two parts, first we consider $w,w'$ which are further than $\delta\rme^l$ apart:
uniformly in $W\in\frD_\eps$ and $x',y'\in W^{\pm,\eps}\cup\{\infty\}$, we have
\begin{multline}
\label{e:logdiv-Cov}
\lim_{l-k\to\infty}\sum_{\substack{w,w'\in \partial W^{\pm}_l:\\
|w-w'|\geq\delta\rme^l}}
\Pi_{\bbZ^2\setminus \partial W^{\pm}_l}(\lfloor\rme^l x'\rfloor,w)
\Pi_{\bbZ^2\setminus \partial W^{\pm}_l}(\lfloor\rme^l y'\rfloor,w')
\big(\fra(w-w')-l\big)\\
=\int_{\substack{w,w'\in\partial W:\\|w-w'|\geq\delta}}\Pi_{\bbR^2\setminus \partial W}(x',\rmd w)\Pi_{\bbR^2\setminus\partial W}(y',\rmd w')\log|w-w'|
\end{multline}
by~\eqref{e:444} and by Lemma~\ref{l:BL-Pois}. Hence, if $l-k$ is larger than a constant, the sum on the left-hand side of~\eqref{e:logdiv-Cov} stays bounded, and this is also the case when $l$ is smaller than a constant by~\eqref{e:444} and the definition of $\frD_\eps$.
The sum over $w,w'\in\partial W^{\pm}_l$ with $|w-w'|<\delta\rme^{l}$ that is  complementary to~\eqref{e:logdiv-Cov} can be made arbitrarily small uniformly also in $l-k$ by choosing $\delta>0$ small enough by Lemma~\ref{l:logdiv} below.
\end{proof}
\begin{proof}[Proof of Lemma~\ref{l:Var-inw-1}]
Assertion~\eqref{e:Var-inw} follows by the technique from the proof of Proposition~\ref{p:Cov-bd}. For~\eqref{e:Var-outw}, we obtain from the representation~\eqref{e:l:2.4:pVar} and monotone convergence that the covariance in the assertion is the limit as $n\to\infty$ of the covariance in Proposition~\ref{p:Cov-bd}.
\end{proof}

\begin{proof}[Proof of Lemma~\ref{l:metr}]
First we show assertion~\eqref{e:l:2.4:metr} when $x,y\in W^\eps_l$ and $x$ is adjacent to $y$. Then we have
\begin{multline}
\label{e:l:2.4:p-metr}
\bbE\Big[\Big(\he{h^{U_n\cap V^-_k}}{\partial W^{\pm}_l}(x) - \he{h^{U_n\cap V^-_k}}{\partial W^{\pm}_l}(y)\Big)^2\Big]
=\sum_{w\in \partial W_l}\big(\Pi_{W_l}(x, w)-\Pi_{W_l}(y, w)\big)\\
\times\sum_{w'\in \partial W_l}\big(\Pi_{W_l}(x,w')-\Pi_{W_l}(y, w')\big)
E\Big(h^{U_n\cap V^-_k}(w)h^{U_n\cap V^-_k}(w')\Big)\,.
\end{multline}
We set $p=\max\{j:\: W_l \subset \Delta_{j+1}\}$ with $\Delta_j$ as in Section~\ref{s:concdec}, here we assume w.\,l.\,o.\,g.\ that $n$ is sufficiently large for such $j$ to exists, if this is not the case, we replace $U$ with $\eps^{-1}\rmB$ and increase $n$, which can only increase the left-hand side of~\eqref{e:l:2.4:metr} by the Gibbs-Markov property.
By applying the Gibbs-Markov property successively at $\partial \Delta_1,\ldots,\partial \Delta_{p-1}$ as in the proof of Proposition~\ref{p:concdec-in}, and inserting the covariance kernel $G_{\Delta_{j-1}\cap V^-_k}$ of a DGFF on $\Delta_{j-1}\cap V^-_k$, we obtain
\begin{multline}
\label{e:p:metr-cov-p}
\bbE\Big(h^{U_n\cap V^-_k}(w)h^{U_n\cap V^-_k}(w')\Big)
= \sum_{j=0}^{p}\bbE\Big(\varphi^{\Delta_{j-1}\cap V^-_k,\Delta_j\cap V^-_k}(w)\varphi^{\Delta_{j-1}\cap V^-_k,\Delta_j\cap V^-_k}(w')\Big)\\
= \sum_{j=0}^{p}\sum_{z,z'\in\partial \Delta_j}
\Pi_{\Delta_j\cap  V^-_k}(w,z)
\Pi_{\Delta_j\cap V^-_k}(w',z')
G_{\Delta_{j-1}\cap V^-_k}(z,z')
\end{multline}
for $w,w'\in\partial W_l$. The expression on the right-hand side can only increase when we replace the discrete Poisson kernels $\Pi_{\Delta_j\cap V^-_k}$ by $\Pi_{\Delta_j}$ as $G_{\Delta_{j-1}\cap V^-_k}(z,z')$ is nonnegative and vanishes for $z$ or $z'$ in $\partial V^-_k$.
Moreover, we have $\sum_{w\in\partial W_l}\Pi_{W_l}(x,w)\Pi_{\Delta_j}(w,z)=\Pi_{\Delta_j}(x,z)$ for our $x\in W^\eps_l$, which follows from the strong Markov property of the random walk underlying the definition of the discrete Poisson kernel. Thus, we obtain from~\eqref{e:l:2.4:p-metr} and~\eqref{e:p:metr-cov-p} that
\begin{multline}
\label{e:p:metr:pdifP}
\bbE\Big[\Big(\he{h^{U_n\cap V^-_k}}{\partial W^{\pm}_l}(x) - \he{h^{U_n\cap V^-_k}}{\partial W^{\pm}_l}(y)\Big)^2\Big]
\leq \sum_{j=0}^{p}\sum_{z,z'\in\partial \Delta_j}
\big|\Pi_{\Delta_j}(x,z)-
\Pi_{\Delta_j}(y,z)\big|\\
\times \big|\Pi_{\Delta_j}(x,z')-
\Pi_{\Delta_j}(y,z')\big|
G_{\Delta_{j-1}\cap V^-_k}(z,z')\,.
\end{multline}
By Theorem~6.3.8 of~\cite{LaLi}, the absolute differences of the Poisson kernels in the previous expression are bounded by
$C_\eps \rme^{-(n-j)} \Pi_{\Delta_j}(x,z)$ and
$C_\eps \rme^{-(n-j)} \Pi_{\Delta_j}(x,z')$,
respectively. Hence, the right-hand side of~\eqref{e:p:metr:pdifP} is bounded by a constant times
\begin{equation}
\sum_{j=0}^p\rme^{-(n-j)}\Var\Big(\he{h^{\Delta_{j-1}\cap V^-_k}}{\Delta_j}(x)\Big)
\leq C_\eps\rme^{-2(n-p)}\leq C_\eps\rme^{-2l}\,,
\end{equation}
where we used that the variances in the last display are bounded by a constant by Proposition~\ref{p:Cov-bd}. This shows~\eqref{e:l:2.4:metr} for adjacent $x,y\in W^\eps_l$.

For adjacent $x,y\in W^{-,\eps}_l$ we argue in the same way, also using the inward concentric decomposition, but we set $p:=\max\{j:\: x,y\in\Delta_{j+2}\}$ and again assume w.\,l.\,o.\,g.\ that such $j$ exist.

For $x,y$ as in the assertion, there exists by Lemma~\ref{l:pathdistance} (with $\eps/4$ in place of $\eps$) a path $x=x_0, x_1,\ldots,x_r=y$
of length $r\leq C_\eps|x-y|$ within a connected component of $W^{\pm,\eps/2}$, hence
\begin{multline}
\bbE\Big[\Big(\he{h^{U_n\cap V^-_k}}{\partial W^{\pm}_l}(x) - \he{h^{U_n\cap V^-_k}}{\partial W^{\pm}_l}(y)\Big)^2\Big]^{1/2}\\
\leq\sum_{i=1}^r \bbE\Big[\Big(\he{h^{U_n\cap V^-_k}}{\partial W^{\pm}_l}(x_i) - \he{h^{U_n\cap V^-_k}}{\partial W^{\pm}_l}(x_{i-1})\Big)^2\Big]^{1/2}
\leq C\eps r\rme^{-l}
\end{multline}
by the assertion for adjacent vertices which we already proved (now with $\eps/2$ in place of $\eps$), and by the triangle inequality for the intrinsic metric.
\end{proof}

\begin{lem}
\label{l:pathdistance}
Let $\eps\in(0,1)$.
There exists $C=C_\eps<\infty$ such that for all $W\in\frD_\eps$, $l\geq 0$, any connected component $\Gamma$ of $W^{\pm,\eps}_l$, and all $x,y\in\{z\in\Gamma:\:\rmd(x,\Gamma^\rmc)\geq \eps\rme^l\}$, there exists a path from $x$ to $y$ in $\Gamma$ of length at most
$C|x-y|$.
\end{lem}
\begin{proof}
We define $\Gamma'=\{z\in\Gamma:\:\rmd(x,\Gamma^\rmc)\geq \eps\rme^l\}\cap (3\eps^{-1}\rmB)_l$. The intersection with $(3\eps^{-1}\rmB)_l$ ensures that $\rme^{-l}\Gamma'$ has bounded diameter. 
We first bound the maximal length of a shortest path in $\Gamma$ between two vertices of $\Gamma'$. To this aim, we observe that the $\tfrac{\eps}{2}\rme^l$-balls around points in $\{z\in\Gamma:\:\rmd(x,\Gamma^\rmc)\geq \tfrac{\eps}{2}\rme^l\}\cap (3\eps^{-1}\rmB)_l\cap(\lceil\tfrac{\eps}{3}\rme^{l}\rceil\bbZ^2)$ form a finite covering $\mathcal C$ of $\Gamma'$ whose cardinality is bounded by a constant that depends only on $\eps$, and that $\cup\mathcal C\subset \Gamma$. Hence, between any two vertices of $\Gamma'$, there exists a path that crosses each ball in $\mathcal C$ only once and whose length is bounded by a constant times $\rme^l$.

If a shortest path in $\bbZ^2$ from $x$ to $y$ as in the assertion lies in $\Gamma$, then the assertion holds.

If no shortest path in $\bbZ^2$ between $x,y\in\Gamma'$ lies in $\Gamma$, then we have $|x-y|\geq \eps\rme^l$ as $\eps\rme^l$-ball around $x$ is contained in $\Gamma$, and the assertion follows as a shortest path has length at most a constant times $\rme^l$ by the above.

If $x,y\in(\eps^{-1}\rmB^-)_l$, then there always exists a path in $(\eps^{-1}\rmB^-)_l$ of length $3|x-y|$ between $x$ and $y$, and $(\eps^{-1}\rmB^-)_l$ is contained in a connected component of $W^-_l$ by definition of $\frD_\eps$.

If $x\notin(3\eps^{-1}\rmB)_l$ and $y\in\Gamma'\setminus (\eps^{-1}\rmB^-)_l$, then a shortest path from $x$ to $y$ must go through some $z\in\partial (2\eps^{-1}\rmB)_l$. The length of the subpath from $z$ to $y$ is bounded by a constant times $\rme^l$ by the above, and the length of the subpath from $x$ to $z$ is bounded by $3|x-z|$. As moreover $|x-y|\geq 2\eps^{-1}\rme^l$ in this case, and $|x-y|+4\eps^{-1}\rme^l\geq |x-z|$, the assertion follows also in this case.
\end{proof}

\begin{proof}[Proof of Lemma~\ref{l:Cov-as}]
The proof is similar to Proposition~\ref{p:Cov-bd}. We discuss assertion~\eqref{e:Cov-as-n}, the proof of~\eqref{e:Cov-as-k} is analogous.
First we represent the covariance under consideration as in~\eqref{e:l:2.4:pVar} and~\eqref{e:GUV}.
For $\delta>0$ and $l:=n-j$, we again obtain by Lemma~\ref{l:BL-Pois} and~\eqref{e:444} that
\begin{multline}
\lim_{n\to\infty}\sum_{\substack{w,w'\in\partial W^{\pm}_l:\\
|w-w'|\geq \delta\rme^l}}
\Pi_{\bbZ^2\setminus\partial W}(x,w)
\Pi_{\bbZ^2\setminus\partial W}(y,w')
\sum_{z\in\partial U_n}\Pi_{U_n}(w,z)\big(\fra(z-w')-l\big)\\
=g\int_{\substack{w,w'\in\partial W:\\|w-w'|>\delta}}\Pi_{\bbR^2\setminus\partial W}(x,\rmd w)
\Pi_{\bbR^2\setminus\partial W}(y,\rmd w')
\int_{z\in\partial(\rme^{j}U)}\Pi_{\rme^j U}(w,\rmd z)
\log|z-w'|
\end{multline}
uniformly as in the assertion.
The difference of the left-hand side in the previous display and of the first term on the right-hand side of~\eqref{e:GUV}, as well as the second term there, vanish as $n-k\to\infty$ followed by $\delta\to 0$ by~\eqref{e:444} and Lemmas~\ref{l:ruin}, \ref{l:logdiv}. The third term on the right-hand side of~\eqref{e:GUV} is treated analogously.
\end{proof}

It remains to prove the following lemma. We recall the constant $c_0$ from~\eqref{e:444}.
\begin{lem}
\label{l:logdiv}
Let $\eps\in(0,1)$. Then
\begin{equation}
\label{e:logdiv}
\lim_{\delta\to 0}\varlimsup_{l\to\infty}
\sum_{\substack{w,w'\in\partial W^{\pm}_l:\\|w-w'|<\delta\rme^l}}
\Pi_{\bbZ^2\setminus \partial W^{\pm}_l}(\lfloor x\rme^l\rfloor,w)\Pi_{\bbZ^2\setminus \partial W^{\pm}_l}(\lfloor y\rme^l\rfloor,w')(\fra(w-w')-l)=0
\end{equation}
uniformly in $W\in\frD_\eps$, $x,y\in W^{\pm,\eps}\cup\{\infty\}$.
\end{lem}
\begin{proof}
By plugging in the estimate~\eqref{e:444} for the potential kernel, we bound the sum in~\eqref{e:logdiv} by
\begin{equation}
\sum_{j=0}^{j^*}
\Big(|\log\delta|  + j + 1 + c_0 + C_\eps\delta^{-c/2}\rme^{-c(l-j-1)/2}\Big)\pi_j+l\pi_{j^*}\,,
\end{equation}
where
\begin{gather}
\pi_j = \sum_{\substack{w,w'\in\partial W^{\pm}_l:\\ \delta\rme^{l-j-1}\leq |w-w'|<\delta\rme^{l-j}}}
\Pi_{\bbZ^2\setminus \partial W^{\pm}_l}(\lfloor x\rme^l\rfloor,w)\Pi_{\bbZ^2\setminus \partial W^{\pm}_l}(\lfloor y\rme^l\rfloor,w')\quad\text{for } j<j^*\,,\\
\pi_{j^*}= \sum_{\substack{w,w'\in\partial W^{\pm}_l:\\  |w-w'|<\delta\rme^{l-j^*}}}
\Pi_{\bbZ^2\setminus \partial W^{\pm}_l}(\lfloor x\rme^l\rfloor,w)\Pi_{\bbZ^2\setminus \partial W^{\pm}_l}(\lfloor y\rme^l\rfloor,w')\,,
\end{gather}
and $j^*\in\bbN$, $c\in(0,4)$ to be chosen later such that $|j^*-\log\delta-l|$ is bounded by a constant and $c$ is a constant, both constants depending only on $\eps$. (By our convention, the symbols $c_\eps$ and $C_\eps$ denote constants that depend only on $\eps$ and may vary from use to use.)

The assertion of the lemma is implied by the claim that
$\pi_j\leq C_\eps\rme^{-c_\eps(j-\log \delta)}$ for $j\leq j^*$.
To prove the claim, we first observe that
\begin{equation}
\pi_j\leq \max_{\substack{z\in W^{\pm,\eps}_l\\ w\in\partial W^{\pm}_l}}
P_z\Big( \rmd\Big(S_{\tau^{\partial W^{\pm}_l}},w\Big) \leq \delta \rme^{l-j} \Big)\,.
\end{equation}
Consider, for $j\leq j^*$ and sufficiently large $l$, the annuli
$(w+B(\delta\rme^{l-j+i}))\setminus (w+B(\delta\rme^{l-j+i-1}))$ with $i=0,\ldots,\lfloor \log\eps-\log 2-\log\delta+j\rfloor$, where $B(r)\subset\bbZ^2$ denotes the Euclidean ball of radius $r$ around $0$.
In such an annulus, every closed path that runs around the hole has a nonempty intersection with $\partial W^{\pm}_l$ as every connected component of $\partial W$ has Euclidean diameter at least $\eps$ and by the discretization~\eqref{e:discr}. There exists a numerical constant $r$ such that when the radius $\delta\rme^{l-j+i}$ is larger than $r$, then by Skorohod coupling with Brownian motion, the probability that simple random walk started at the outer component of the boundary of such an annulus runs around the hole before hitting the inner component of the boundary is bounded from below by a numerical constant $p>0$. We choose $j^*=\lfloor l+\log\delta-\log r\rfloor$.
As the starting point $z$ of $S$ has distance at least $\eps\rme^l-1$ from $w$, it follows for $j\leq j^*$ that the probability that $S$ visits $w+B(\delta\rme^{l-j})$ before $\partial W^{\pm}_l$ is bounded from above by $(1-p)^{\lfloor \log\eps -\log 2-\log\delta + j\rfloor}$.
\end{proof}

\subsubsection{Fluctuation}
\begin{proof}[Proof of Proposition~\ref{p:2.6}]
The oscillation is the absolute maximum of the difference field. To center this Gaussian field, we bound its conditional expectation by
\begin{multline}
\label{e:cE-osc}
\Big|\bbE_{k,v}^{n,u}\Big(\he{h}{\partial W^{\pm}_l}(y)-\he{h}{\partial W^{\pm}_l}(z)
\,\Big|\, \he{h}{\partial W^{\pm}_l}(x)=-m_l+w\Big)\Big|\\
\leq\frac{\Big|\Cov_{k}^{n}\Big(\he{h}{\partial W^{\pm}_l}(x), \he{h}{\partial W^{\pm}_l}(y) - \he{h}{\partial W^{\pm}_l}(z)\Big) \Big|}{\Var_{k}^{n}\Big(\he{h}{\partial W^{\pm}_l}(x)\Big)}
\times\bigg|w-m_l+\wh m_l-\frac{\ol{u}(0)( l-k)+\ol{v}(\infty)(n- l)}{n-k}\\
+ 2\,\osc_{\ol{W_l}}\ol{u} + 
	2\,\osc_{W_l^-}\ol{v} + C_\eps(1+|\ol{u}(0)|+|\ol{v}(\infty)|)(n-k)^{-1}\bigg|
\end{multline}
for $y,z\in W^{\pm,\eps}_l$, this bound follows from the usual expression for Gaussian conditional expectations and Proposition~\ref{l:E}.

First we consider the oscillation in each connected component $\Gamma$ of $W^\eps_l$ separately. By Proposition~\ref{p:Cov-bd}, the covariance in the numerator on the right-hand side of~\eqref{e:cE-osc} is bounded from above by a constant, and the variance in the numerator is bounded from below by a constant times $(n-l)\wedge(l-k)$. Using also Lemma~\ref{l:blml}, we further bound \eqref{e:cE-osc} by a constant times $\mu$ for $y,z\in\Gamma$.
Now we apply Fernique majorization for the oscillation (Theorem~6.14 of~\cite{Bi-LN} with $R=\infty$) to the centered Gaussian field
$\he{h}{\partial W^{\pm}_l}(y)-\bbE_{k,v}^{n,u}\big(\he{h}{\partial W^{\pm}_l}(y)\,\big|\, \he{h}{\partial W^{\pm}_l}(x)=-m_l+w \big)$, $y\in\Gamma$
with the uniform probability measure on $\Gamma$ as the majorizing measure, so as to see that $\bbE_{k,v}^{n,u}\big(\osc_\Gamma \he{h}{\partial W^{\pm}_l} \,\big|\, \he{h}{\partial W^{\pm}_l}(x)=-m_l+w\big)$ is bounded by a constant. To this aim, we recall from theory of the multivariate Gaussian distribution that the intrinsic metric (which can be seen as a variance) under the conditioning on $\he{h}{\partial W^{\pm}_l}(x)$ is bounded by the unconditional intrinsic metric, which in turn is bounded by Lemma~\ref{l:metr}. To bound the conditional tail probability (i.\,e.\ to show \eqref{e:p:2.6:2} but still with $\Gamma$ instead of $W^{\pm,\eps}_l$), we use the Borell-TIS inequality, now using Lemma~\ref{l:metr} as a bound for the variance of the difference field (and for the conditional variance which can only be smaller).

For the connected component $W^{-,\eps}_l$, we apply Fernique majorization and the Borell-TIS inequality as above but for $y,z\in W^{-,\eps}_l$ with $|y|,|z|\leq \rme^{l+2}\diam W$, and note that
$\he{h}{\partial W^{\pm}_l}(y)-\he{h}{\partial W^{\pm}_l}(z)$ assumes its maximum (and its minumum) over $(W^{-,\eps}_l)^2\cup\{\infty\}$ at some $(y,z)$ with $|y|,|z|\leq\rme^{ l+2}\diam W$, which follows from the maximum principle for bounded harmonic functions, applied to $y$ and $z$ separately.

It remains to control the oscillations between the connected components of $W^{\pm,\eps}_l\cup\{\infty\}$. The number of these connected components is bounded by one plus the number of connected components of $W^{\pm,\eps}_l\cap (\eps^{-1}\rmB)_l$, which in turn is bounded by $16\eps^{-4}$ as $|(\eps^{-1}\rmB)_l|\leq 4\eps^{-2}\rme^{2l}$, and as $W^\pm_l\cap (\eps^{-1}\rmB)_l$ contains balls of radius $\eps\rme^l$ that are centered in each of the connected components of $W^{\pm,\eps}_l\cap (\eps^{-1}\rmB)_l$ and do not intersect.
Hence, it suffices to show that
$\bbP_{k,v}^{n,u}(\he{h}{\partial W^{\pm}_l}(y)-\he{h}{\partial W^{\pm}_l}(z)>t+\mu\mid\he{h}{\partial W^{\pm}_l}(x)=-m_l+w\big)$ has uniformly Gaussian tails for $y,z\in W^{\pm,\eps}_l\cup\{\infty\}$. To this aim, it suffices by~\eqref{e:cE-osc} to bound
\begin{multline}
\label{e:osc-var}
\Var_{k}^{n}\Big(\he{h}{\partial W^{\pm}_l}(y)-\he{h}{\partial W^{\pm}_l}(z)\,\Big|\,\he{h}{\partial W^{\pm}_l}(x)=-m_l+w\Big)\\
\leq \Var_{k}^{n}\Big(\he{h}{\partial W^{\pm}_l}(y)\Big)
-2\Cov_{k}^{n}\Big(\he{h}{\partial W^{\pm}_l}(y),\he{h}{\partial W^{\pm}_l}(z)\Big)
+\Var_{k}^{n}\Big(\he{h}{\partial W^{\pm}_l}(z)\Big)\,,
\end{multline}
which in turn is bounded by a constant Proposition~\ref{p:Cov-bd}.
\end{proof}

\begin{proof}[Proof of Proposition~\ref{p:unc-bd}]
This is proved analogously to Proposition~\ref{p:2.6}. Let $\bbE^{n,u}$, $\Var^n$ and $\Cov^n$ denote the expectation, variance and covariance associated with $\bbP^{n,u}$.
For $y,z\in W^{\pm,\eps}_l$, we consider the Gaussian conditional expectation
\begin{multline}
\label{e:unc-bd-condE}
\Big|\bbE^{n,u}\Big(\he{h}{\partial W^{\pm}_l}(y)-\he{h}{\partial W^{\pm}_l}(z)
\,\Big|\, \he{h}{\partial W^{\pm}_l}(x)=-m_l+w\Big)\Big|
\leq\frac{\Big|\Cov^{n}\Big(\he{h}{\partial W^{\pm}_l}(x), \he{h}{\partial W^{\pm}_l}(y) - \he{h}{\partial W^{\pm}_l}(z)\Big) \Big|}{\Var^{n}\Big(\he{h}{\partial W^{\pm}_l}(x)\Big)}\\
\times\Big|w-m_l-\bbE^{n,u}\he{h}{\partial W^{\pm}_l}(x)\Big|\,,
\end{multline}
where we have
\begin{equation}
\bbE^{n,u}\he{h}{\partial W^{\pm}_l}(x)
=-m_n+\ol{u}(x)\,.
\end{equation}
Lemma~\ref{l:Var-inw-1} bounds the covariance in the numerator in~\eqref{e:unc-bd-condE} by a constant and the variance in the denominator by a constant times $n-l$. We use~\eqref{e:unc-bd-condE} to center the difference field that appears in the oscillation. In each connected component $\Gamma$ of $W_l\cap (2\eps^{-1}\rmB)_l$, we then apply the Fernique and Borell-TIS inequalities as before, using the estimate
\begin{equation}
\bbE\Big[\Big(\he{h^{U_n}}{\partial_l W}(y') - \he{h^{U_n}}{\partial_l W}(z') \Big)^2\Big] 
\leq C_\eps \frac{|y'-z'|^2}{\rme^{2 l}}
\end{equation}
for $y',z'\in\Gamma$,
which is a straightforward analog of Lemma~\ref{l:metr}.
The oscillation between different connected components is estimated analogously to~\eqref{e:osc-var} using Lemma~\ref{l:Var-inw-1}.
\end{proof}

\subsubsection{Proofs of binding field estimates}

\begin{proof}[Proof of Lemma~\ref{l:UU-fluct}]
We show assertions~\eqref{e:UU-fluct-i} and~\eqref{e:UU-fluct-i-P}, the proofs of the other assertions are analogous.
We write	$\varphi=\he{h^{U_n\cap \rmB^-_{0}}}{\partial(U\cap \wt U)_n\cup \partial \rmB^-_0}$. Let $\eta\in(\dH(U,U'),\eps/2)$, where we assume w.\,l.\,o.\,g.\ that $\dH(U,\wt U)<\eps/2$ and $n\geq r-\log\eps$. Then we have $U\cap\wt U\subset U^\eta$ and $B^-_{n-r}\subset B^{-,\eps}_0$.
For $x\in(U\cap \wt U)^\eps_n\cap \rmB^{-}_{n-r}$, we bound by the Gibbs-Markov property (Lemma~\ref{l:GM})
\begin{multline}
\Var \varphi(x)\leq \Var\varphi^{U_n\cap\rmB^-_0,U^\eta_n\cap\rmB^-_0}(x)
=G_{U_n\cap\rmB^-_0}(x,x)-G_{U^\eta_n\cap\rmB^-_0}(x,x)\\
\leq g\sum_{z\in\partial U_n}\Pi_{U_n}(x,z)\log|\rme^{-n}z|
-g\sum_{z\in\partial U^\eta_n}\Pi_{U^\eta_n}(x,z)\log|\rme^{-n}z|
+\max_{z\in\rmB^-_0}\fra(z) P_x\big(\tau^{\rmB^-_0}<\tau^{U^\eta}\big)+o_\eps(1)\\
=g\int_{\partial U}\Pi_U(\rme^{-n} x,\rmd z)\log|z|
-g\int_{\partial U^\eta}\Pi_{U^\eta}(\rme^{-n} x,\rmd z)\log|z|
+o_{\eps,r}(1)
\,,
\end{multline}
where we used~\eqref{e:444}, \eqref{e:462} and Lemmas~\ref{l:ruin}\ref{i:ruin-C}, \ref{l:BL-Pois}. The difference of the integrals in the last line can be bounded by $\rho(\eta)$ by Lemma~A.2 of~\cite{BL-Conf}.

For $x,y\in(U\cap \wt U)^\eps_n\cap \rmB^{-}_{n-r}$, we also bound the (squared) intrinsic metric
\begin{multline}
\bbE\Big(\big(\varphi(x)-\varphi(y)\big)^2\Big)
\leq \bbE\Big(\big(\varphi^{U_n\cap\rmB^-_0,U^\eta_n\cap\rmB^-_0}(x)-\varphi^{U_n\cap\rmB^-_0,U^\eta_n\cap\rmB^-_0}(y)\big)^2\Big)\\
=\sum_{z,z'\in\partial U^\eta_n}\big|\Pi_{U^\eta_n\cap\rmB^-_0}(x,z)-\Pi_{U^\eta_n\cap\rmB^-_0}(y,z)\big|
\big|\Pi_{U^\eta_n\cap\rmB^-_0}(x,z')-\Pi_{U^\eta_n\cap\rmB^-_0}(y,z')\big|\\
\times\bbE\Big(\varphi^{U_n\cap\rmB^-_0,U^\eta_n\cap\rmB^-_0}(z)\varphi^{U_n\cap\rmB^-_0,U^\eta_n\cap\rmB^-_0}(z')\Big)
\end{multline}
where we again used the Gibbs-Markov property and harmonicity of $\varphi^{U_n\cap\rmB^-_0,U^\eta_n\cap\rmB^-_0}$.
Analogously to the proof of Lemma~\ref{l:metr}, the absolute differences in the last display are bounded by $C_{\eps,r}|x-y|\rme^{-n}\Pi_{U_n}(x,z)$ and $C_{\eps,r}|x-y|\rme^{-2n}\Pi_{U^\eta_n}(x,z')$, respectively. Hence,
\begin{equation}
\label{e:UU-fluct-intr}
\bbE\Big(\big(\varphi(x)-\varphi(y)\big)^2\Big)
\leq C_{\eps,m}\frac{|x-y|^2}{\rme^{2n}}\Var\varphi^{U_n\cap\rmB^-_0,U^\eta_n\cap\rmB^-_0}(x)\leq
C_{\eps,r}\frac{|x-y|^2}{\rme^{2n}}\rho(\eta)\,,
\end{equation}
where we used the bound for the variance from before in the last step. Using~\eqref{e:UU-fluct-intr}, we can now apply the Fernique inequality as in the proof of Proposition~\ref{p:2.6} which yields assertion~\eqref{e:UU-fluct-i}.
From this, the bound for the variance, and the Borell-TIS inequality, we then obtain~\eqref{e:UU-fluct-i-P}.
\end{proof}

\begin{proof}[Proof of Lemma~\ref{l:VV-fluct}]
We show~\eqref{e:VV-fluct-i} and~\eqref{e:VV-fluct-i-P},
the proof of~\eqref{e:VV-fluct-o} and~\eqref{e:VV-fluct-o-P} is analogous.
W.\,l.\,o.\,g.\ we assume that $B^-_{n-r}\subset V^{-,\eps}_k$.
For $x,y\in U^\eps_n\cap\rmB^-_{n-r}$, we estimate the (squared) intrinsic metric
\begin{multline}
\label{e:p-metr-bdg-harm}
\bbE\Big[\Big(\he{h^{U_n\cap \rmB^-_0}}{\partial U_n\cup \partial V^-_{k}}(x) - \he{h^{U_n\cap \rmB^-_0}}{\partial U_n\cap \partial V^-_{k}}(y)\Big)^2\Big]
=\sum_{z,z'\in\partial V^-_{k}}\big[
\Pi_{U_n\cap V^-_{k}}(x,z) - \Pi_{U_n\cap V^-_{k}}(y,z)\big]\\
\times\big[
\Pi_{U_n\cap V^-_{k}}(x,z') - \Pi_{U_n\cap V^-_{k}}(y,z')\big]
\bbE\Big[ h^{U_n\cap \rmB^-_0}(z)h^{U_n\cap \rmB^-_0}(z')\Big]\,.
\end{multline}
As in the proof of Lemma~\ref{l:metr}, there exists a path from $x$ to $y$ in $U^\eps_n\cap \rmB^-_{n-r}$
whose length is bounded by a constant times $|x-y|$.
Applying Theorem~6.3.8 of~\cite{LaLi} repeatedly along that path gives
\begin{equation}
\big|\Pi_{U_n\cap V^-_{k}}(x,z) - \Pi_{U_n\cap V^-_{k}}(y,z)\big|
\leq C_{r,\eps}\,\frac{|x-y|}{\rme^n}
\Pi_{U_n\cap V^-_{k}}(x,z)\,.
\end{equation}
Furthermore, defining $\Delta^{\rm o}_1$ from $U$, $V$, $n$, $k$ and $\eta=\zeta=0$ as in Section~\ref{s:w-outw} and using the Markov property of
the random walk underlying the definition of the discrete Poisson kernel, we estimate
\begin{multline}
\Pi_{U_n\cap V^-_{k}}(x,z)
= \sum_{w\in\partial \Delta^{\rm o}_1}\Pi_{U_n\cap \Delta^{\rm o}_1}(x,w)
\Pi_{U_n\cap V^-_{k}}(w,z)\\
\leq
\frac{C_{r,\eps}}{n- k}
\sum_{w\in\partial \Delta^{\rm o}_1}\Pi_{\Delta^{\rm o}_1}(\infty,w)
\Pi_{V^-_{k} }(w,z)
=\frac{C_{r,\eps}}{n-k}\Pi_{V^-_{k}}(\infty,z)\,,
\end{multline}
for the inequality, we used Lemma~\ref{l:Poisson-out} for
$\Pi_{U_n\cap \Delta^{\rm o}_1}$, and in $\Pi_{U_n\cap \wt V^-_{\wt k}}$, we used that $z\in\partial V^-_{k}$ and left out the absorbing boundary at $\partial U_n$.
Plugging these estimates into~\eqref{e:p-metr-bdg-harm} yields
\begin{multline}
\bbE\Big[\Big(\he{h^{U_n\cap \rmB^-_0}}{\partial U_n\cup \partial V^-_{k}}(x) - \he{h^{U_n\cap \rmB^-_0}}{\partial U_n\cap \partial V^-_{k}}(y)\Big)^2\Big]\\
\leq C_\eps\,
|x-y|^2\rme^{-2n}
(n-k)^{-2}
\sum_{z,z'\in\partial V^-_{k}}
\Pi_{V^-_{k}}(\infty,z)
\Pi_{V^-_{k}}(\infty,z')
G_{U_n\cap \rmB^-_0}(z,z')\\
\leq C_\eps\,
|x-y|^2\rme^{-2n}
(n-k)^{-2}
\wedge_{n,k,0}\,,
\end{multline}
where we used Proposition~\ref{p:Cov-bd} in the last inequality.
We now apply the Fernique inequality as in the proof of Proposition~\ref{p:2.6} which yields~\eqref{e:VV-fluct-i} when we replace in the maximum there $U_n$ with $U^\eps_n$, which does not decrease the maximum by harmonicity of the binding field.
Assertion~\eqref{e:VV-fluct-i-P} follows from the Borell-TIS inequality by using~\eqref{e:VV-fluct-i} and Lemma~\ref{l:Cov-bdg}.
\end{proof}

\begin{proof}[Proof of Lemma~\ref{l:Cov-bdg}]
We show~\eqref{e:Cov-bdg-V}, the proof of~\eqref{e:Cov-bdg-U} is analogous.
By harmonicity, we write the covariance in~\eqref{e:Cov-bdg-V} as
\begin{equation}
\label{e:varXm-1}
\sum_{z,z'\in \partial V^-_k}\Pi_{U_n\cap V^-_k}(x,z)
\Pi_{U_n\cap V^-_k}(y,z')
\Cov\big( h^{U_n}(z),h^{U_n}(z')\big)\,.
\end{equation}
We define $\rt =\lfloor n - k\rfloor +\lfloor\log\eps\rfloor -\lceil\log(\eps^{-1}+\zeta)\rceil$ and $\Delta^{\rm o}_i$ in terms of $U$, $V$, $n$, $k$, and $\eta=\zeta=\eps$ as in Section~\ref{s:concdec}.
W.\,l.\,o.\,g.\ we assume that $x,y\in \Delta^{\rm o}_1$.
For $z\in\partial V^-_k$,
\begin{equation}
\label{e:varXm-2}
\Pi_{U_n\cap V^-_k}(x,z)\leq\sum_{w\in\partial \Delta^{\rm o}_1}\Pi_{\Delta^{\rm o}_{1}\cap U_n}(x,w)
\Pi_{\Delta^{\rm o}_{-1}}(w,z)\,,
\end{equation}
where we used the strong Markov property of the random walk that underlies the definition of the discrete Poisson kernel, and we removed an absorbing boundary. By Lemma~\ref{l:Poisson-out},
\begin{equation}
\Pi_{\Delta^{\rm o}_{1}\cap U_n}(x,w)\leq C_\eps\frac{\rt+1-p}{\rt}\Pi_{\Delta^{\rm o}_1}(\infty,w)
\end{equation}
for all $w\in\partial \Delta^{\rm o}_1$ when $x\in A^{\rm o}_{p}$, corresponding to $k+ p - 1 -\log\eps \leq \log |x|\leq k+ p - \log\eps$.
Inserting this back into~\eqref{e:varXm-2}, we further bound~\eqref{e:varXm-1} by
\begin{equation}
C_\eps\frac{(\rt +1-p)^2}{\rt^2}
\Var\, \he{h^{U_n}}{\partial V^-_k}(\infty)\,.
\end{equation}
The variance in the last display is bounded by $g\rt +C_\eps$ by Lemma~\ref{l:Var-inw-1}.
\end{proof}

\section{Ballot theorems for decorated random walks}
\label{s:3}

\subsection{Statements}

The first result concerning such walks is a uniform ballot-type upper bound.
\begin{thm}
%[Two sided
[Upper bound]
\label{t:3.2}
Fix $\delta \in (0,1/3)$. There exists $C = C_\delta < \infty$ such that
\begin{equation}
\label{e:drw-ub}
\bfP \Big(\max_{k =\lceil\delta^{-1}\rceil}^{\lfloor \rt - \delta^{-1}\rfloor} (\cS_k + \cD_k) \leq 0 \Big)\\
\leq C \frac{\big(a^- + 1 \big)
\big(b^- + 1 \big)
}{\rt} \,,
\end{equation}
for all $\rt \in \bbN$, $a \in \bbR$, $b < -\rt^\delta$ and 
all $(\cS_k)_{k=0}^\rt$, $(\cD_k)_{k=1}^\rt$ satisfying \ref{i.a1}, \ref{i.a2} and \ref{i.a3}. The inequality holds without any restriction on $b$, if we replace $(b^-+1)$ by $(b^- + \rt^\delta)$ 
on the r.h.s.
\end{thm}

An analogous lower bound is given by the following theorem.
\begin{thm}[Lower bound]
\label{t:lb}
Fix $\delta \in (0,1/3)$. There exist $C = C_\delta > 0$ and $a_0 = a_{0,\delta} > -\infty$, such that
\begin{equation}
\label{e:5.8}
\bfP \Big(\max_{k=1}^\rt (\cS_k + \cD_k) \leq 0 \Big) 
\geq C \frac{(a^-+1) b^-}{\rt} \,,
\end{equation}
for all $\rt \in \bbN$, $a \leq a_0$, $b < -\rt^{\delta}$, $(a^-+1)b^- \leq \rt^{1-\delta}$ and 
all $(\cS_k)_{k=0}^\rt$, $(\cD_k)_{k=1}^\rt$ satisfying Assumptions~\ref{i.a1},~\ref{i.a2} and~\ref{i.a3}.
\end{thm}

We now turn to asymptotics. For $1 \leq \rr \leq \rt \leq \infty$ and $\rr < \infty$, set
\begin{equation}
\label{e:5.4}
\ell_{\rt,\rr}(a,b) := \bfE \Big(\cS_\rr^- ;\: \max_{k=1}^\rr (\cS_k + \cD_k) \leq 0  \Big) 
\end{equation}
and observe that this quantity does not depend on $b$ when $\rt =\infty$.
Then,
\begin{thm}
%Two sided
[Asymptotics]
\label{t:3.4}
Fix $\delta \in (0,1/3)$ and let $e_{\rt,\rr}(a,b)$ be defined via
\begin{equation}
\label{e:3.5}
\bfP \Big(\max_{k =1}^\rt (\cS_k + \cD_k) \leq 0 \Big) 
= 2 \frac{\ell_{\rt,\rr}(a,b) b^-}{s_\rt} + e_{\rt,\rr}(a,b) \,.
\end{equation}
Then, 
\begin{equation}
\label{e:3.6}
\frac{\rt}{(a^-+1)b^-} e_{\rt,\rr}(a,b) \lto 0 \text{ as }\rr \to \infty \,,
\end{equation}
uniformly in $\rr^{4/\delta} \leq \rt \leq \infty$, 
$a \leq \delta^{-1}$, $b < -\rt^{\delta}$, $(a^-+1)b^- \leq \rt^{1-\delta}$ and 
all $(\cS_k)_{k=0}^\rt$, $(\cD_k)_{k=1}^\rt$ satisfying Assumptions~\ref{i.a1},~\ref{i.a2} and~\ref{i.a3}. 
\end{thm}

\begin{comment}
We remark that Theorems~\ref{t:lb} and~\ref{t:3.4} together show that
\begin{equation}
\label{e:ell-LB}
\ell_{\rt,r}(a,b)\geq c_\delta(a^- + 1)
\end{equation}
for sufficiently large $\rr$ (depending only on $\delta$), and hence
\begin{equation}
\label{e:3.5'}
\bfP \Big(\max_{k =1}^\rt (\cS_k + \cD_k) \leq 0 \Big) = 
2 \frac{\ell_{\rt,\rr}(a,b) b^-}{s_\rt} \big(1+o(1)\big) \,,
\end{equation}
where $o(1) \to 0$ as $\rr \to \infty$. Both statements hold uniformly as in Theorem~\ref{t:3.4} with the additional restriction that $a < a_0$, with $a_0$ given by Theorem~\ref{t:lb}.
\end{comment}

Next we state some properties of $\ell_{\rt,\rr}(a,b)$. The first property shows that $\ell_{\rt,\rr}(a,b)$ is Cauchy in $\rr$ and consequently that one can extend the definition of $\ell_{\rt,\rr}(a,b)$ to the case when $\rr=\infty$.
\begin{prop}
\label{t:LR-as}
Fix $\delta \in (0,1/3)$. Then,
\begin{equation}
\label{e:20.39'}
\lim_{\rr,\rr' \to \infty} \lim_{\rt' \to \infty} \sup_{\rt \in [\rt', \infty]}
\Bigg| \frac{\ell_{\rt,\rr}(a,b) - \ell_{\rt, \rr'}(a,b )}{a^-+1} \Bigg| = 0 \,,
\end{equation}
The limits above hold uniformly in $a \leq \delta^{-1}$, $b < \delta^{-1}$,  $|b-a|/(a^-+1) \leq \rt^{1-\delta}$ and all $(\cS_k)_{k=0}^\rt$, $(\cD_k)_{k=1}^\rt$ satisfying  Assumptions~\ref{i.a1},~\ref{i.a2} and~\ref{i.a3}. 
In particular, the limit
\begin{equation}
\ell_{\infty,\infty}(a,b) := \lim_{\rr \to \infty} \ell_{\infty, \rr}(a,b) \,,
\end{equation}
holds uniformly in all $a \leq \delta^{-1}$, $b < \delta^{-1}$ and all $(\cS_k)_{k=0}^\infty$, $(\cD_k)_{k=1}^\infty$ satisfying Assumptions~\ref{i.a1},~\ref{i.a2} and~\ref{i.a3}.
\end{prop}
We again remark that by definition $\ell_{\infty,\infty}(a,b)$ does not depend on $b$, but we leave the dependence on this parameter in the notation to allow a uniform treatment of all cases. Next we treat the asymptotics of $\ell_{\rt,\rr}(a,b)$ in $a^-$ for all $\rt$ and $\rr$.
\begin{prop}
\label{p:1.5}
Fix $\delta \in (0,1/3)$. Then for all $1 \leq \rr \leq \infty$,
\begin{equation}
\lim_{a \to -\infty} \frac{\ell_{\rt,\rr}(a,b)}{a^-} = 1 \,,
\end{equation}
uniformly in $\rr \leq \rt \leq \infty$, $b < \delta^{-1}$, $|b-a|\rr/\rt \leq (a^-+1)^{1-\delta}$ and all $(\cS_k)_{k=0}^\rt$, $(\cD_k)_{k=1}^\rt$ satisfying Assumptions~\ref{i.a1},~\ref{i.a2} and~\ref{i.a3}. Furthermore, there exists $(\rt_{(\rr)})_{\rr=1}^\infty \in \bbN^\infty$ such that the above convergence is also uniform in $\rr \geq 1$, provided that $\rt_{(\rr)} \leq \rt \leq \infty$.
\end{prop}

We now turn to the question of continuity of $\ell_{\rt,\rr}(a,b)$ in the quantities through which it is defined. The first statement is a simple consequence of the Cauchy-Schwarz Inequality.
\begin{lem}
\label{l:f-cont}
Fix $\delta \in (0,1/3)$. Then, for some $C = C_\delta < \infty$,
\begin{multline}
\Big| \bfE \Big(\cS_\rr^- ;\: \max_{k=1}^\rr (\cS_k + \cD_k) \leq \lambda  \Big) 
- 
\bfE \Big(\cS_\rr^- ;\: \max_{k=1}^\rr (\cS_k + \cD_k) \leq -\lambda  \Big) \Big| \\
\leq C (a^-+1) \rr \, \Big(\max_{k=1}^\rr \bfP \big(\cS_k + \cD_k \in (-\lambda, \lambda]\big)\Big)^{1/2} \,,
\end{multline}
for all $1 \leq \rr \leq \rt \leq \infty$, $b < \delta^{-1}$, $|b-a|\rr/\rt \leq \delta^{-1} (a^-+1)^{1-\delta}$ and all $(\cS_k)_{k=0}^\rt$, $(\cD_k)_{k=1}^\rt$ satisfying  Assumptions~\ref{i.a1},~\ref{i.a2} and~\ref{i.a3}.
\end{lem}

In the second continuity statement we treat the setting of a triangular array. We therefore assume that for each $i \in \bbN \cup \{\infty\}$, there are defined $a^{(i)}, b^{(i)} \in \bbR$, $1 \leq \rr^{(i)} \leq \rt^{(i)} \leq \infty$ and random variables $\big(\cS^{(i)}_k\big)_{k=0}^{\rt^{(i)}}$, $\big(\cD^{(i)}_k\big)_{k =1}^{\rt^{(i)}}$ taking the roles of $a$, $b$, $\rr$, $\rt$ and 
$(\cS_k)_{k=0}^\rt$, $(\cD_k)_{k =1}^{\rt}$ from before. We also denote by $\bfP^{(i)}$, $\bfE^{(i)}$ the underlying probability measure and expectation and let $\ell^{(i)}_{\rr^{(i)}}(a^{(i)}) = \ell^{(i)}_{\rt^{(i)}, \rr^{(i)}}(a^{(i)}, b^{(i)})$ be defined as in~\eqref{e:5.4} only with respect to $\big(\cS^{(i)}_k\big)_{k=0}^{\rt^{(i)}}$, $\big(\cD^{(i)}_k\big)_{k=1}^{\rt^{(i)}}$.
Given $\delta>0$, let $(\rr_\rt)_{\rt=0}^\infty$ be a sequence tending to infinity sufficiently slowly such that~\eqref{e:20.39'} still holds when we replace $\rr'$ there with $\rr_{\rt'}$.
\begin{prop}[Triangular array]
\label{t:ta}
Fix $\delta \in (0,1/3)$ and suppose that for all $1 \leq i \leq \infty$ 
$\big(\cS^{(i)}_k\big)_{k=0}^{\rt^{(i)}}$ and $\big(\cD^{(i)}_k\big)_{k =1}^{\rt^{(i)}}$ satisfy Assumptions~\ref{i.a1},~\ref{i.a2} and~\ref{i.a3} with $a^{(i)}$, $b^{(i)}$ and $\rt^{(i)}$. Suppose also that in the sense of weak convergence of finite dimensional distributions we have
\begin{multline}
\label{e:5.7}
\Big(\big(\cS^{(i)}_k\big)_{k=0}^{\rt^{(i)}}, \big(\cD^{(i)}_k\big)_{k =1}^{\rt^{(i)}},\, \rt^{(i)},  r^{(i)}, a^{(i)}, b^{(i)} \Big) \\
\underset {i \to \infty}{\longrightarrow}
\Big(\big(\cS^{(\infty)}_k\big)_{k=0}^{\rt^{(\infty)}}, \big(\cD^{(\infty)}_k\big)_{k=1}^{\rt^{(\infty)}},\, \rt^{(\infty)}, r^{(\infty)}, a^{(\infty)}\,, b^{(\infty)} \Big) \,,
\end{multline}
and that $0$ is a stochastic continuity point of the law of $\cS^{(\infty)}_k + \cD^{(\infty)}_k$ 
under $\bfP^{(\infty)}$, for all $k=1, \dots, r^{(\infty)}$. Lastly assume also that if $r^{(\infty)} = \infty$ then
%$\rt^{(i)} > \rt_{r^{(i)}}$ for some predefined $(\rt_r)_{r \geq 1}$
$\rr^{(i)}\leq\rr_{\rt^{(i)}}$
and that 
$a^{(i)} < \delta^{-1}$, $b^{(i)} < \delta^{-1}$, $|b^{(i)}-a^{(i)}|/(a^{(i)-}+1) \leq (\rt^{(i)})^{1-\delta}$ for all $i$ large enough. Then,
\begin{equation}
\label{e:5.10}
\lim_{i \to \infty} \ell^{(i)}_{r^{(i)}}\big(a^{(i)}\big) = \ell^{(\infty)}_{r^{(\infty)}} \big(a^{(\infty)}\big) \,.
\end{equation}
\end{prop}

\subsection{Proofs}
\label{s:proof-drw}

A key step in the proofs is a reduction to the case where there are no decorations. This will come at the cost of adding/subtracting a deterministic curve from the barrier. We shall then rely on the following barrier probability estimates, which involve only the process $(\cS_k)_{k=0}^\rt$. While barrier estimates for random walks are well known, the statements in the proposition below are tailored to our needs and therefore take a rather nonstandard form. In the end of this section, we therefore include a proof, in which it shown how to convert these statements into a form which was already handled in the literature before.
\begin{prop}
\label{p:4.1}
Fix $\delta \in (0,1/3)$. There exist non-increasing functions $\bar{f}, \ubar{f}: \bbR \to (0,\infty)$ satisfying
$\bar{f}(a) \sim \ubar{f}(a) \sim a^-$ as $a \to -\infty$ such that for all $\rt < \infty$, $a,b, a', b' \in \bbR$ with $|a'-b'| < \rt^{1/2-\delta}$ and
$(\cS_k)_{k=0}^\rt$ satisfying Assumption~\ref{i.a1},
\begin{equation}
\bfP \Big(\max_{k=0}^\rt \big(\cS_k - \delta^{-1} \wedge_{\rt,k}^{1/2-\delta} - a'1_{[0,\rt/2)}(k) - b'1_{[\rt/2, \rt]}(k) \big) \leq 0 \Big) \leq \big(2+o(1)\big) \frac{\bar{f}(a-a') \bar{f}(b-b')}{s_\rt} \,.
\end{equation}
If in addition $\big((a-a')^-+1)\big)\big((b-b')^-+1\big) \leq \rt^{1-\delta}$ then also,
\begin{equation}
\bfP \Big(\max_{k=0}^\rt \big(\cS_k + \delta^{-1} \wedge_{\rt,k}^{1/2-\delta} - a'1_{[0,\rt/2)}(k) - b'1_{[\rt/2, \rt]}(k) \big) \leq 0 \Big) \geq \big(2-o(1)\big) \frac{\ubar{f}(a-a') \ubar{f}(b-b')}{s_\rt} \,.
\end{equation}
Both $o(1)$ terms above depend only on $\delta$ and tend to $0$ as $\rt \to \infty$.
\end{prop}

To carry through the reduction to the case where there are no decorations, we need to control the growth of the latter as well as the growth of the steps of the walk. Setting $\xi_0 \equiv 0$ and $\xi_k := \cS_k - \cS_{k-1}$, we then introduce the ``control variable''
\begin{equation}
\label{e:2.4}
\clR := \min \Big\{ \rr \geq 1  :\: 
\max_{k=1}^\rt \Big(\big(|\cD_k| - \delta^{-1} \wedge_{\rt,k}^{1/2-\delta}\big) \vee |\xi_k| - (k \vee \rr)^{\delta/2}\Big) \leq 0
\Big\} \,.
\end{equation}

The first lemma shows that $\clR$ cannot be large, even conditional on the ballot event. 
\begin{lem}
\label{l:4.2}
Fix $\delta \in (0,1/3)$. There exists $C = C_\delta < \infty$ such that 
for all $1 \leq \rr \leq \rt < \infty$ and $a \leq \delta^{-1}$, $b \leq -\rt^{\delta}$.
\begin{equation}
\label{e:2.5}
\bfP \Big(\max_{k =\lceil \delta^{-1}\rceil}^{\lfloor \rt -\delta^{-1}\rfloor} (\cS_k + \cD_k) \leq 0 ,\, R \geq \rr \Big)
\leq  C \frac{(a^- + 1) (b^- +1)}{\rt} \rme^{-\rr^{\delta^2/5}} \,.
\end{equation} 
\end{lem}
\begin{proof}
Without loss of generality we can assume that $\rt$ is large enough. For any $1 \leq \rr < \rt/2$, setting $\rr' := \lceil \rr \vee \delta^{-1} \rceil$, the event 
\begin{equation}
\label{e:4.6}
\Big\{ \max_{k =\lceil \delta^{-1}\rceil}^{\lfloor \rt -\delta^{-1}\rfloor} \big(\cS_k + \cD_k\big) \leq 0 ,\, R = \rr \Big\} \cap \Big\{\cS_\rt = b \Big\}
\end{equation}
is included in the intersection of 
\begin{equation}
\label{e:4.8}
\Big\{\exists k \in \{1,\ldots,\rr-1\}  :\: 
	\big(|\cD_k| - \delta^{-1} \wedge_{\rt,k}^{1/2-\delta}\big) \vee |\xi_k|
		 > (\rr-1)^{\delta/2} \Big\} 
\quad ,\qquad 
\big\{ \cS_{\rr'} \geq a - {\rr'}^2 \big\}
\end{equation}	
and
\begin{equation}
\label{e:4.9}
\Big\{ \max_{k = \rr'}^\rt \big(\cS_k - 2\delta^{-1} k^{1/2-\delta}1_{[\rr',\rt/2)}(k)  - (\delta^{-1} (\rt -k)^{1/2-\delta} + \rt^{\delta/2})1_{[\rt/2, \rt]}(k) \big) \leq 0 
\Big\}.
\end{equation}
Above we have used that~\eqref{e:4.6} implies that
$\cS_{k} + \cD_{k} \leq b + \delta^{-2} + (1+\delta^{-1}) \rt^{\delta/2} \leq 0$ for all $k =\lceil \rt -\delta^{-1}\rceil,\ldots ,\rt$,
whenever $\rt$ is large enough, in light of the restrictions on $b$.

By Assumptions~\ref{i.a1} and~\ref{i.a3} and the union bound, the probability of the left event in~\eqref{e:4.8} is bounded above by $C\rme^{-\rr^{\delta^2/3}}$. On the other hand, we observe that
$2\delta^{-1} k^{1/2-\delta}1_{[\rr',\rt/2)}(k) + (\delta^{-1} (\rt -k)^{1/2-\delta} + \rt^{\delta/2})1_{[\rt/2, \rt]}(k)$ is at most 
\begin{multline}
2 \delta^{-1} \big( \rr'^{1/2-\delta} + (k-\rr')^{1/2-\delta}\big)1_{[1,(\rt -\rr')/2)}(k-\rr') \\ + 
	\big( 2 \delta^{-1} \big((\rt -\rr')-(k-\rr')\big)^{1/2-\delta} + \rt^{\delta/2} \big)1_{[(\rt -\rr')/2, \rt -\rr']}(k-\rr') \,.
\end{multline}
Setting $\rt' := \rt -\rr'$ and using stochastic monotonicity, conditional on the second event in~\eqref{e:4.8}, the $\bfP$-probability of~\eqref{e:4.9} is at most the $\bfP_{0,a-\rr'^2}^{\rt',b}$ probability of
\begin{equation}
\Big\{ \max_{k = 0}^{\rt'} \big(\cS'_k - 2\delta^{-1} \wedge_{\rt',k}^{1/2-\delta} - 
2 \delta^{-1} \rr'^{1/2-\delta} 1_{[0,\rt'/2)}(k) - \rt^{\delta/2} 1_{[\rt'/2, \rt']}(k) \big) \leq 0 \Big\} \,,
\end{equation}
where the law of $(\cS'_k)_{k=0}^{\rt'}$ under $\bfP_{0,a-\rr'^2}^{\rt',b}$ is that of 
$(\cS_{k+\rr'})_{k=0}^{\rt'}$ under $\bfP(\cdot\,|\, \cS_{\rr'} = a-\rr'^2)$. Thanks to Assumption~\ref{i.a2} and Proposition~\ref{p:4.1} the last probability is at most
\begin{equation}
C \frac{(a - \rr'^2 - \rr'^{1/2-\delta})^- (b - \rt^{\delta/2})^-}{(\rt -\rr')}
\leq C \frac{(a^- + 1) (b^- +1)}{\rt} (\rr^2+1) \,.
\end{equation}
whenever $b^- \geq \rt^{\delta}$ and $a \leq \delta^{-1}$.

Invoking the union bound and the product rule and combining the above estimates, we get that the probability of~\eqref{e:4.6} is bounded above by 
\begin{equation}
C \frac{(a^- + 1) (b^- +1)}{\rt} \rme^{-\rr^{\delta^2/4}} \,.
\end{equation}
The same bound applies also when $\rr \geq \rt/2$, in which case we only take the first event in~\eqref{e:4.8} as the one including~\eqref{e:4.6}. Summing the right hand side in the last display from $\rr$ to $\infty$ we obtain~\eqref{e:2.5}. 
\end{proof}
As immediate consequences, we get proofs for Theorem~\ref{t:3.2} and Theorem~\ref{t:lb}.
\begin{proof}[Proof of Theorem~\ref{t:3.2}]
Using $\rr=1$ in Lemma~\ref{l:4.2} give the desired bound when $b \leq -\rt^\delta$. If $b$ is not that low, we let $\cS'_k := \cS_k - (b + \rt^{\delta})k/\rt$ for all $k=0, \dots, \rt$. Then $(\cS'_k)_{k=0}^\rt$, $(\cD_k)_{k=1}^\rt$ satisfies Assumptions~\ref{i.a1},~\ref{i.a2} and~\ref{i.a3} with the same parameters, except that $b$ is now $-\rt^{-\delta}$. Therefore the same argument as in the first part of the proof gives $C(a^-+1)(\rt^\delta+1)/\rt$ as an upper bound on the left hand side in~\eqref{e:drw-ub} only with $\cS'_k$ in place of $\cS_k$. Since $\cS_k \geq \cS'_k$ for all $k$, the original left hand side is even smaller. It remains to observe that the upper bound $C'(a^-+1)(b^-+\rt^\delta)/\rt$ applies in both cases.
\end{proof}

\begin{proof}[Proof of Theorem~\ref{t:lb}]
For any $1 \leq \rr \leq \rt$, 
the desired probability is bounded below by
\begin{multline}
\label{e:9.11}
\bfP \Big(\max_{k=1}^\rt \big(\cS_k +  \delta^{-1} \wedge_{\rt,k}^{1/2-\delta} + (k \vee \rr)^{\delta/2}\Big) \leq 0 \,, \clR \leq \rr \Big) \\
\geq 
\bfP \Big(\max_{k=1}^\rt \big(\cS_k + 2\delta^{-1} \wedge_{\rt,k}^{1/2-\delta} + 
\rr^{\delta/2} 1_{[1,\rt/2)}(k) + \rt^{\delta/2} 1_{[\rt/2, \rt]}(k) \big) \leq 0\Big) \\
- \bfP \Big(\max_{k=1}^\rt \cS_k \leq 0 ,\, \clR > \rr \Big) \,.
\end{multline}
Then from Proposition~\ref{p:4.1}, for any $\rr$ the first term on the right hand side is at least $C (a^--\rr^{\delta/2})(b^--\rt^{\delta/2})/\rt \geq C_0 a^-b^-/\rt$ for some $C_0 = C_{0,\delta} > 0$ whenever $a^- > 2\rr^{\delta/2}$ and under the restrictions in the theorem. 

At the same time, by assuming $\cD_k \equiv 0$ for all $k$, it follows from Lemma~\ref{l:4.2} that the second term on the right hand side of~\eqref{e:9.11} can be made at most $(C_0/2) a^-b^-/\rt$ by choosing $\rr$ large enough. Combined this gives the desired lower bound for all $a < a_0$ for some $a_0 = a_{0,\delta} > -\infty$.
\end{proof}

Our next task is to derive asymptotics, but to this end, we shall first need several preliminary results. The first one shows that under the ballot event the random walk is repulsed.
In all lemmas in the remaining part of this section, we assume~\ref{i.a1}, \ref{i.a2} and \ref{i.a3}. 
\begin{lem}
\label{l:4.3}
Fix $\delta \in (0,1/3)$. There exists $C = C_\delta < \infty$ such that for all $0 \leq \rr \leq \rt/2$, $a \leq \delta^{-1}$ and $b \leq -\rt^\delta$, 
\begin{equation}
\bfP \Big(\max_{k =1}^\rt (\cS_k + \cD_k) \leq 0 ,\, \cS_\rr > -\rr^{1/2-\delta/2} \Big)
\leq C \frac{(a^- + 1) (b^- +1)}{\rt} \rr^{-\delta/2} \,.
\end{equation}

\end{lem}
\begin{proof}
For $1 \leq \rr' \leq \rr \leq \rt/2$, the event in the statement of the lemma intersected with $\{\clR \leq \rr'
\}$ is included in
\begin{multline}
\label{e:4.14}
\Big\{\max_{k =1}^\rr \big(\cS_k - 2\delta^{-1}  (k \vee \rr')^{1/2-\delta}\big) \leq 0\big) \Big\}
\cap
\{ \cS_\rr \geq -\rr^{1/2-\delta/2} \}
\\
\cap \Big\{ \max_{k=\rr}^\rt \big(\cS_k - 2\delta^{-1}  k^{1/2-\delta}1_{[\rr,\rt/2)}(k)  - (\delta^{-1} (\rt -k)^{1/2-\delta} + \rt^{\delta/2})1_{[\rt/2, \rt]}(k) \big) \leq 0 
\Big\}.
\end{multline}
As in Lemma~\ref{l:4.2}, we bound $2\delta^{-1}  k^{1/2-\delta}1_{[\rr,\rt/2)}(k) + (\delta^{-1} (\rt -k)^{1/2-\delta} + \rt^{\delta/2})1_{[\rt/2, \rt]}(k)$ by
\begin{equation}
2 \delta^{-1} \wedge_{\rt-\rr, k-\rr}^{1/2-\delta} + 2\delta^{-1} \rr^{1/2-\delta} 1_{[0,(\rt -\rr)/2)}(k-\rr) 
+ \rt^{\delta/2} 1_{[(\rt -\rr)/2, \rt -\rr]}(k-\rr) \,.
\end{equation}
and $(k \vee \rr')^{1/2-\delta}$ by $\wedge_{\rr,k}^{1/2-\delta} + \rr'^{1/2-\delta} 1_{[0,\rr/2)}(k) + \rr^{1/2-\delta}1_{[\rr/2,\rr]}(k)$.
We then use Assumption~\ref{i.a1} and Proposition~\ref{p:4.1} to bound the probability of~\eqref{e:4.14} by
\begin{equation}
C \int_{-\rr^{1/2-\delta/2}}^{2 \delta^{-1} \rr^{1/2-\delta}}
 \frac{(a-2\delta^{-1} \rr'^{1/2-\delta})^- \big((w-2 \delta^{-1} \rr^{1/2-\delta})^-\big)^2 (b-\rt^{\delta/2})^-}{s_\rr(s_\rt-s_\rr)}
\bfP(\cS_\rr \in \rmd w)
\end{equation}
When $a \leq \delta^{-1}$, $b \leq -\rt^{\delta}$, the last fraction can be bounded by $C (a^-+1)(b^-+1) \rr' \rr^{-\delta} \rt^{-1}$ in the domain of integration and therefore this bound applies to the entire integral as well. Notice that we have used that $s_k \in (k\delta, k\delta^{-1})$.
 Choosing $\rr' = \rr^{\delta/2}$, invoking Lemma~\ref{l:4.2} for the case when $\{\clR > \rr'\}$ and 
finally using the union bound then completes the proof.
\end{proof}

The next three lemmas provide needed bounds on the mean of the walk at step $m$. 
\begin{lem}
\label{l:4.4}
Fix $\delta \in (0,1/2)$. There exists $C = C_\delta < \infty$ such that for all $0 \leq \rr \leq \rt$, $a \leq \delta^{-1}$ and $b \leq -\rt^{\delta}$ satisfying $(a^-+1)b^- \leq \rt^{1-\delta}$, 
\begin{equation}
\label{e:4.18a}
\bfE \Big( \cS_\rr^- ;\; \max_{k =1}^\rr (\cS_k + \cD_k) \leq 0 ,\, 
	\cS_\rr \notin \big[a-\rr^2, -\rr^{1/2-\delta/2}\big] \Big) 
\leq C (a^-+1) \rr^{-\delta/2} \,.
\end{equation}
\end{lem}
\begin{proof}
We bound the expectation restricted to the events $\{\cS_\rr > -\rr^{1/2-\delta/2}\}$ and $\{\cS_\rr < a-\rr^2\}$ separately. In the first range, the expectation is bounded above by 
\begin{equation}
\label{e:4.19a}
\rr^{1/2-\delta/2} \times 
\bfP \Big(\max_{k =1}^\rr (\cS_k + \cD_k) \leq 0 ,\, \cS_\rr > -\rr^{1/2-\delta/2} \Big) \,.
\end{equation}
By Assumption~\ref{i.a3} and the union bound, the probability that there exists $k \in (0,\rr]$ such that $-\cD_k \geq  \delta^{-1} k^{1/2-\delta} + \rr^{\delta/4}$ is at most $\rme^{-\rr^{\delta^2}/8}$. On the complement the probability of the event in~\eqref{e:4.19a} is at most
\begin{equation}
\label{e:4.21a}
\int_{-\rr^{1/2-\delta/2}}^{2\delta^{-1} \rr^{1/2-\delta}}
\bfP_{0,a}^{\rr,w} \Big(\max_{k=1}^\rr (\cS_k - \delta^{-1}k^{1/2-\delta} - \rr^{\delta/4}) \leq 0\Big)
\bfP \big(\cS_\rr \in \rmd w\big) \,.
\end{equation}

Upper bounding $k^{1/2-\delta}$  by $\wedge_{\rr,k}^{1/2-\delta} + \rr^{1/2-\delta} 1_{[\rr/2, \rr]}(k)$,
we may use Proposition~\ref{p:4.1} to bound the first probability in the integral by
\begin{equation}
C \frac{\big((a-\rr^{\delta/4})^- + 1\big)\big((w-2\delta^{-1} \rr^{1/2-\delta})^-+1\big)}{s_\rr}
\leq C \frac{(a^-+1) \rr^{\delta/4} \rr^{1/2-\delta/2}}{\rr} \leq C (a^-+1) \rr^{-1/2-\delta/4} \,.
\end{equation}
Since this is also an upper bound on the entire integral, plugging this in~\eqref{e:4.19a} gives
$C (a^-+1) \rr^{-3\delta/4}$ as an upper bound for the expectation in the first range.

For the restriction to the event $\{\cS_\rr < a-\rr^2\}$, we bound the expectation in~\eqref{e:4.18a} by
$\bfE(\cS_\rr^- ;\; \cS_\rr \leq a-\rr^2\big)$ and then use Cauchy-Schwarz. This gives the bound
\begin{equation}
\label{e:20.44}
\Big(\big(\bbE_{0,a}^{\rt,b} \cS_\rr^2\big) \bbP_{0,a}^{\rt,b} \big(\cS_\rr \leq a-\rr^2\big)\Big)^{1/2}
\end{equation}
Under the restrictions on $a$,$b$ and $\rt$ we have, 
$|\bfE (\cS_\rr - a)| = |b-a|\rr/\rt \leq 2\rr$ and $\bfVar\, \cS_\rr \leq C \rr$. Therefore, the second moment in~\eqref{e:20.44} is at most $C ((a^-+ 1)^2 + \rr^2)$ while the probability there is at most $C \rme^{-C' \rr^3}$. Combined,~\eqref{e:20.44} is at most $C (a^-+1) \rme^{-\rr^2}$, while the entire expectation in~\eqref{e:4.18a} obeys the desired bound.
 \end{proof}

\begin{lem}
\label{l:4.5a}
Fix $\delta \in (0,1/2)$. There exists $C = C_\delta$ such that for all $0 \leq \rr \leq \rt/2$, $a \leq \delta^{-1}$ and $b \in \bbR$ such that $|b-a|\rr/\rt \leq \delta^{-1} (a^-+1)^{1-\delta}$, 
\begin{equation}
\label{e:2.24}
\bfE \Big( \cS_\rr^- ;\; \max_{k =1}^\rr \big(\cS_k - \delta^{-1} k^{1/2-\delta}\big) \leq 0 \Big)
\leq C (a^- + 1) 
\end{equation}
\end{lem}
\begin{proof}
Thanks the conditions in the lemma and Assumption~\ref{i.a1}, 
\begin{equation}
\label{e:20.41}
\Big|\bfE (\cS_\rr - a)\Big| = \frac{s_\rr}{s_\rt}\big| b-a \big|
\leq C  (a^-+1)^{1-\delta} \quad, 
\qquad
\bfVar\, \cS_\rr = \frac{s_\rr (s_\rt - s_\rr)}{s_\rt} \leq C' \rr \,.
\end{equation}
Now without the restricting event, the expectation in~\eqref{e:2.24} is at most
\begin{equation}
\label{e:2.26}
\bfE \Big(\big(\cS_\rr - \bfE \cS_\rr \big)^-\Big) + \Big(\bfE \cS_\rr\Big)^-
\leq \Big(\bfVar\, \cS_\rr\Big)^{1/2} + \Big(a^- + \big(\bfE (\cS_\rr - a)\big)^-\Big)
\leq C \big(\rr^{1/2} +a^- + 1) \,.
\end{equation}
Therefore, if $a^- \geq \sqrt{\rr}$ this shows the desired claim immediately. 

In the converse case, we write the expectation in~\eqref{e:2.24} as
\begin{equation}
\int_{-\infty}^0 w^- \bfP_{0,a}^{\rr,w} \Big(\max_{k =1}^\rr \big(\cS_k - \delta^{-1} k^{1/2-\delta}\big) \leq 0 \Big) \bfP \big(\cS_\rr \in \rmd w\big) \, .
\end{equation}
The first probability in the integrand only increases if we replace $k^{1/2-\delta}$ by $\wedge_{\rr,k}^{1/2-\delta} + \rr^{1/2-\delta}1_{[\rr/2,\rr]}(k)$. But then we can use the upper bound in Proposition~\ref{p:4.1} and Assumption~\ref{i.a1} to bound this probability by $C (a^-+1)(w^- + \rr^{1/2-\delta})/\rr$. Using this in the last integral gives the bound
\begin{equation}
C \frac{a^-+1}{\rr} \Big(\bfE \cS_\rr^2 + \rr^{1/2-\delta} \bfE \cS_\rr^- \Big)  
\leq C (a^-+1)\frac{\bfVar\, \cS_\rr + \big(a + \big|\bfE (\cS_\rr - a)\big|\big)^2 + \rr^{1/2-\delta} \bfE \cS_\rr^-}{\rr} \,.
\end{equation}
Using the bounds in~\eqref{e:20.41} and~\eqref{e:2.26} proves~\eqref{e:2.24} in the case that $a^- < \sqrt{\rr}$ as well.
\end{proof}

Next, we strengthen the previous lemma by adding also the decorations.
\begin{lem}
\label{l:4.5}
Fix $\delta \in (0,1/3)$. There exists $C = C_\delta$ such that for all $0 \leq \rr \leq \rt^{1/3}$, $a \leq \delta^{-1}$ and $b \in \bbR$ such that $|b-a|\rr/\rt \leq \delta^{-1} (a^-+1)^{1-\delta}$, 
\begin{equation}
\label{e:2.29}
\bfE \Big( \cS_\rr^- ;\; \max_{k =1}^\rr (\cS_k + \cD_k) \leq 0 \Big)
\leq C (a^- + 1) 
\end{equation}
\end{lem}

\begin{proof}
The proof is similar to that of Lemma~\ref{l:4.2} and therefore allow ourselves to be brief.
Analogously to $\clR$, we define a control variable:
\begin{equation}
\label{e:2.4a}
L := \min \Big\{ \ell \geq 0  :\: 
\max_{k =1}^\rr \Big(\big(|\cD_k| - \delta^{-1} k^{1/2-\delta}\big) \vee |\xi_k| - (k \vee \ell)^{\delta/2}\Big) \leq 0
\Big\} \,,
\end{equation}
and observe that by the union bound, it is clearly enough to show that for all $\ell \geq 0$,
\begin{equation}
\label{e:2.31}
\bfE \Big( \cS_\rr^- ;\; \max_{k =1}^\rr (\cS_k + \cD_k) \leq 0 \,, L = \ell \Big) 
< C(a^- + 1) \rme^{-\ell^{\delta^2/4}} \,.
\end{equation}

When $\ell < \rr$, the last expectation is at most
\begin{multline}
C \rme^{-\ell^{\delta^2/3}} \bfE_{0,a-\ell^2}^{\rt',b}
\Big( \cS^{'-}_{\rr'} ;\; \max_{k =1 }^{\rr'} \Big(\cS'_k - 2\delta^{-1} \big(\ell^{1/2-\delta} + k^{1/2-\delta}\big)\Big) \leq 0 \Big) \\
\leq 
C \rme^{-\ell^{\delta^2/3}} \bfE_{0,a'}^{\rt',b'}
\Big( \cS^{''-}_{\rr'} ;\; \max_{k =1}^{\rr'} \big(\cS''_k - 2\delta^{-1} k^{1/2-\delta} \big) \leq 0 \Big) \,,
\end{multline}
where $\rt' := \rt - \ell$, $\rr' := \rr - \ell$, $a' := a-\ell^2 - 2\delta^{-1} \ell^{1/2-\delta}$, $b' := b - 2\delta^{-1} \ell^{1/2-\delta}$ and $(\cS'_k)_{k=0}^{\rt'}$ under
$\bfP_{0,a-\ell^2}^{\rt',b}$ and $(\cS''_k)_{k=0}^{\rt'}$ under $\bfP_{0,a'}^{\rt',b'}$ have the same laws as $(\cS_{\ell+k}-S_\ell)_{k=0}^{\rt'}$ and $(\cS_{\ell+k}-S_\ell - 2\delta^{-1} \ell^{1/2-\delta})_{k=0}^{\rt'}$
under $\bfP(-|S_\ell = a-\ell^2)$ respectively.

Thanks to Lemma~\ref{l:4.5a}, the last expectation is at most $C(a'^-+1) \leq C(a^-+1)(1+\ell^2)$ since
\begin{equation}
|b'-a'|\rr'/\rt' \leq C (|b-a| + \ell^2) \rr/\rt) \leq C (a^-+1)^{1-\delta} \leq C' (a'^-+1)^{1-\delta} \,.
\end{equation}
This shows that~\eqref{e:2.31} indeed holds. When $\ell \geq \rr$, we remove the ballot event from the expectation in~\eqref{e:2.31} and then use the definition of $L$ to bound it by 
\begin{equation}
C (a^- + \ell^2) \bfP \big(L = \ell) \leq C'(a^-+\ell^2) \rme^{-\ell^{\delta^2/3}} \leq 
C''(a^- + 1) \rme^{-\ell^{\delta^2/4}}  \,.
\end{equation}
\end{proof}

\begin{proof}[Proof of Theorem~\ref{t:3.4}]
Let $1 \leq \rr \leq \rt^{\delta/4}$. We will show the desired statement by upper and lower bounding the left hand side of~\eqref{e:3.5} separately. For the upper bound, we can appeal to Lemma~\ref{l:4.2} and Lemma~\ref{l:4.3} to consider instead of the left hand side of~\eqref{e:3.5} 
the probability of the event
\begin{equation}
\label{e:4.18}
\Big\{\max_{k =1}^\rt (\cS_k + \cD_k) \leq 0 ,\, \clR \leq \rr ,\, \cS_\rr \leq -\rr^{1/2 - \delta/2} \Big\} \,.
\end{equation}
which is further included in
\begin{multline}
\label{e:4.19}
\Big\{\max_{k =1}^\rr (\cS_k + \cD_k) \leq 0 \Big\} \cap \big\{a - \rr^2 \leq \cS_\rr \leq -\rr^{1/2 - \delta/2} \big\}
 \\
\cap \Big\{ \max_{k =\rr}^\rt \big(\cS_k - 2 \delta^{-1} k^{1/2-\delta}1_{[\rr,\rt/2)}(k)  - (\delta^{-1}  (\rt -k)^{1/2-\delta} + \rt^{\delta/2})1_{[\rt/2, \rt]}(k) \big) \leq 0 
\Big\} \,.
\end{multline}
As before, we bound $2\delta^{-1} k^{1/2-\delta}1_{[\rr,\rt/2)}(k) + (\delta^{-1} (\rt -k)^{1/2-\delta} + \rt^{\delta/2})1_{[\rt/2, \rt]}(k)$ by
\begin{multline}
\label{e:4.20a}
\big(2\delta^{-1}  \rr^{1/2-\delta} + 2 \delta^{-1}  (k-\rr)^{1/2-\delta}\big)1_{[0,(\rt -\rr)/2)}(k-\rr) \\ + 
	\big( 2\delta^{-1}  \big((\rt -\rr)-(k-\rr)\big)^{1/2-\delta} + \rt^{\delta/2} \big)1_{[(\rt -\rr)/2, \rt -\rr]}(k-\rr) \,,
\end{multline}
and then use Assumption~\ref{i.a2} and Proposition~\ref{p:4.1} to upper bound the probability of~\eqref{e:4.19} by 
\begin{equation}
\label{e:4.20}
(2+o(1))\int_{a-\rr^2}^{-\rr^{1/2-\delta/2}}
\bfP_{0,a}^{\rr,w} \Big(\max_{k =1}^\rr (\cS_k + \cD_k) \leq 0\Big)
 \frac{(w-2\delta^{-1}  \rr^{1/2-\delta})^- (b-\rt^{\delta/2})^-}{(s_\rt-s_\rr)} 
\bfP \big(\cS_\rr \in \rmd w\big)\,,
\end{equation}
where the $o(1)$ depends only on $\rr$ and goes to $0$ when $\rr \to \infty$. To get the $(2+o(1))$ multiplying the integral, we have also used that both factors in the numerator above tend to $\infty$ as $\rr \to \infty$ uniformly in the range of integration.

In the range of integration, the fraction in the integrand is at most $(1+o(1)) w^-b^-/s_\rt$, with the $o(1)$ as before. Integrating over all $w \in \bbR$ then gives
$(1+o(1)) 2 \ell_{\rt,\rr}(a,b)(a) b^-/s_\rt$ as an upper bound. 
Lemma~\ref{l:4.5} can then be used to turn the multiplicative error into an additive one satisfying~\eqref{e:3.6} uniformly as desired.

Turning to the lower bound, we now consider the event
\begin{multline}
\label{e:4.24}
\Big\{\max_{k =1}^\rr (\cS_k + \cD_k) \leq 0 \Big\} \\
\cap \Big\{ \max_{k =\rr}^\rt \big(\cS_k + 2 \delta^{-1} k^{1/2-\delta}1_{[\rr,\rt/2)}(k) + (\delta^{-1} (\rt -k)^{1/2-\delta} + \rt^{\delta/2})1_{[\rt/2, \rt]}(k) \big) \leq 0 
\Big\} \,.
\end{multline}
By Lemma~\ref{l:4.2} with $(\cD_k)_{k=1}^\rt$ replaced by $(\cD_k 1_{(0,\rr]})_{k=1}^\rt$, we may further intersect the above event with $\{\clR \leq \rr\}$ at the cost of decreasing its probability by at most a quantity satisfying the same as $e_{\rt,\rr}(a,b)$ in~\eqref{e:3.6}. Since the new event is now a subset of the event in the statement of theorem, for the sake of the lower bound, it suffices to bound the probability of~\eqref{e:4.24} from below.

Using Assumption~\ref{i.a2}, Proposition~\ref{p:4.1} and the bound in~\eqref{e:4.20a}, we proceed as in the upper bound to bound this probability from below by
\begin{equation}
(2-o(1))\int_{a-\rr^2}^{-\rr^{1/2-\delta/2}}
\bfP_{0,a}^{\rr,w} \Big(\max_{k =1}^\rr (\cS_k + \cD_k) \leq 0\Big)
 \frac{(w+2\delta^{-1} \rr^{1/2-\delta})^- (b+\rt^{\delta/2})^-}{(s_\rt-s_\rr)} 
\bfP \big(\cS_\rr \in \rmd w\big)\,,
\end{equation}
where the $o(1)$ depends only on $\rr$ and tends to $0$ as $\rr \to \infty$. The lower bound in Proposition~\ref{p:4.1} is in force (with $\delta/3$), since the numerator above is at most $C(a^-+1)b^-\rr^2 \leq \rt^{1-\delta/3}$ for $\rr$ large enough, thanks to the assumptions in the theorem. 

Proceeding as in the upper bound, the last display is at least
\begin{equation}
(2-o(1)) \frac{b^-}{s_\rt} \bfE \Big( \cS_\rr^- ;\; \max_{k =1}^\rr (\cS_k + \cD_k) \leq 0
,\, \cS_\rr \in [a-\rr^2, -\rr^{1/2-\delta/2}] \Big) \,.
\end{equation}
Invoking Lemma~\ref{l:4.4} and then Lemma~\ref{l:4.5} we remove the additional restriction on $\cS_\rr$ and turn the multiplicative error into an additive one satisfying~\eqref{e:3.6} uniformly as desired.
\end{proof}

Next we turn to the proofs of Propositions~\ref{t:LR-as} and~\ref{p:1.5}. The following lemma is key.
\begin{lem}
\label{l:20.6}
Fix $\delta \in (0,1/3)$. For all $\rr \geq 1$,
\begin{equation}
\label{e:20.37}
\lim_{a \to -\infty} \frac{\ell_{\rt,\rr}(a,b)}{a^-} = 1
\end{equation}
uniformly in $\rt \in [\rr,\infty]$, $b \in \bbR$, $|b-a|\rr/\rt \leq (a^-+1)^{1-\delta}$ and all $(\cS_k)_{k=0}^\rt$, $(\cD_k)_{k=1}^\rt$ satisfying  Assumptions~\ref{i.a1},~\ref{i.a2} and~\ref{i.a3}.
Furthermore, 
\begin{equation}
\label{e:20.39}
\lim_{\rr,\rr' \to \infty} \lim_{\rt' \to \infty} \sup_{\rt \in [\rt', \infty]}
\Bigg| \frac{\ell_{\rt,\rr}(a,b) - \ell_{\rt, \rr'}(a,b )}{a^-+1} \Bigg| = 0 \,,
\end{equation}
uniformly in $a \leq \delta^{-1}$, $b < \delta^{-1}$, $|b-a|/(a^-+1) \leq \rt^{1-\delta}$ and all $(\cS_k)_{k=0}^\rt$, $(\cD_k)_{k=1}^\rt$ satisfying  Assumptions~\ref{i.a1},~\ref{i.a2} and~\ref{i.a3}. 
\end{lem}

\begin{proof}
Starting with the first statement, for fixed $\rr \geq 1$, if $|b-a|\rr/\rt \leq (a^-+1)^{1-\delta}$ and $\rt \in [\rr,\infty]$ then as in~\eqref{e:20.41} for all $k \leq \rr$,
\begin{equation}
\label{e:20.42}
\Big|\bfE (\cS_k - a)\Big| \leq C (a^-+1)^{1-\delta} \quad, 
\qquad
\bfVar\,\cS_k \leq C' \rr \,.
\end{equation}
If in addition $a^- > \rr^{1/2 + \delta}$ then also,
\begin{equation}
\bfP \big(\cS_k + \cD_k  > 0\big) 
\leq \bfP \big(\cD_k > a^-/2 \big)
+ \bfP \big(\cS_k - a > a^-/2 \big) \leq C \rme^{-(a^-)^{\delta}}  \,.
\end{equation}
Using then the Cauchy-Schwarz inequality and the Union Bound gives for $\rr$ large enough, 
\begin{equation}
\label{e:2.47h}
\begin{split}
\bfE \big(|\cS_\rr| ;\: \max_{k =1}^\rr (\cS_k + \cD_k) > 0  \big) 
& \leq  
\Big(\bfE S^2_\rr \sum_{k=1}^\rr \bfP \big(\cS_k + \cD_k  > 0\big) \Big)^{1/2} \\
& \leq C \Big(\big(\rr + a^2 \big) \rr \rme^{-(a^-)^\delta}\Big)^{1/2} \leq \rme^{-(a^-)^{\delta/3}} \,.
\end{split}
\end{equation}
We define $\tilde\ell_{\rt,\rr}(a,b)=\bfE(-\cS_\rr;\: \max_{k =1}^\rr (\cS_k + \cD_k) \leq 0 \big)$, 
then the left hand side in~\eqref{e:2.47h} bounds $\big|\bfE (-\cS_\rr) - \tilde\ell_{\rt,\rr}(a,b)\big|$. Using also the first inequality in~\eqref{e:20.42} with $k=\rr$, we get
\begin{equation}
\label{e:tildeell-a}
\Bigg|\frac{\tilde\ell_{\rt,\rr}(a,b)}{a^-} - 1\Bigg| \leq C (a^-)^{-\delta}  \,.
\end{equation}
Moreover,
\begin{multline}
\big|\tilde\ell_{\rt,\rr}(a,b) - \ell_{\rt,\rr}(a,b)\big|
=\bfE\big( \cS_\rr^+;\: \max_{k =1}^\rr (\cS_k + \cD_k) \leq 0 \big)
\leq \Big(\bfE \big( \big( \cS_\rr -  a\big)^2 \big)\Big)^{1/2} + a^+\\
\leq C(a^- + 1)^{1-\delta} + \sqrt{C'\rr} + a^+=o(a^-)\,.
\end{multline}
Since also the right hand side of~\eqref{e:tildeell-a} goes to $0$ as $a \to -\infty$, this shows the first part of the lemma.

Turning to the proof of~\eqref{e:20.39}, we first treat the case $\rt < \infty$. 
If we assume that $|b-a|/(a^-+1) \leq \rt^{1-\delta}$ and $\rt > \rr^{2/\delta}$ then as in~\eqref{e:20.42} we get $|\bfE(S_k-a)\big| \leq C (a^-+1)\rt^{-\delta/2}$ for all $k \leq \rr$. Repeating the argument before, we get
\begin{equation}
\Bigg|\frac{\ell_{\rt,\rr}(a,b)}{a^-} - 1\Bigg| \leq C  r^{-\delta/2}  \,,
\end{equation}
for all $a^- > \rr^{1/2 + \delta}$. This shows that~\eqref{e:20.39} hold uniformly whenever
$|b-a|/(a^-+1) \leq \rt^{1-\delta}$, $a^- > (\rr \vee \rr')^{1/2+\delta}$ and for any choice of sequences 
$(\cS_k)_{k=0}^\rt$, $(\cD_k)_{k=1}^\rt$ satisfying  Assumptions~\ref{i.a1},~\ref{i.a2} and~\ref{i.a3}. 

Now let $\delta' \in (0,\delta)$. If we assume that $a \leq \delta'^{-1}$, $b < -\rt^{\delta'}$ and $(a^-+1)b^- \leq \rt^{1-\delta'}$, then it follows from Theorem~\ref{t:3.4}, applied once with $\rr$ and once with $\rr'$ (and with all other arguments the same), that the left hand side in~\eqref{e:20.39} is at most
\begin{equation}
\lim_{\rr,\rr' \to \infty} \varlimsup_{\rt \to \infty}
C
\Bigg(\frac{\rt |e_{\rt,\rr}(a,b)|}{(a^-+1) b^-} + \frac{\rt |e_{\rt,\rr'}(a,b)|}{(a^-+1) b^-} \Bigg) = 0 \,.
\end{equation}
This shows that~\eqref{e:20.39} holds uniformly in $a \leq \delta'^{-1}$, $b < -\rt^{\delta'}$, $(a^-+1)b^- \leq \rt^{1-\delta'}$ and 
all $(\cS_k)_{k=0}^\rt$, $(\cD_k)_{k=1}^\rt$ satisfying Assumptions~\ref{i.a1},~\ref{i.a2} and~\ref{i.a3}. 

We now we claim that the union of the last two domains of uniformity includes the one in the statement of the lemma with the additional restriction that $b < -\rt^{\delta}$. Indeed, if $a \leq \delta^{-1}$, $b < -\rt^{\delta}$, $|b-a|/(a^-+1) \leq \rt^{1-\delta}$ but $(a^-+1)b^- > \rt^{1-\delta'}$, then
$(a^-+1)^3(1+\rt^{1-\delta}) > \rt^{1-\delta'}$, which implies that $a^- > \rt^{(\delta-\delta')/3}$ and hence for $\rt$ large enough, depending only on $\rr,\rr'$, also that $a^- > (\rr \vee \rr')^{1/2+\delta}$. 

Next to replace the condition $b < -\rt^{\delta}$ with $b < \delta^{-1}$, we consider the processes $(\cS'_k)_{k=0}^\rt$ and $(\cD'_k)_{k=1}^\rt$, defined by setting $\cS'_k := \cS_k - 2k\rt^{\delta-1}$ and $\cD'_k := \cD_k + 2k\rt^{\delta-1}1_{\{k \leq \rr\}}$. Then the new processes satisfy the same assumptions as before with $a' := a$, $b' := b - 2\rt^{\delta}$ provided $\rt$ is large enough. Moreover, if $b  < \delta^{-1}$, then for $\rt$ large enough also $b' < -\rt^{\delta}$ and $|b'-a'|/(a'^-+1) \leq C \rt^{1-\delta}$. Therefore~\eqref{e:20.39} holds uniformly as in the conditions of the lemma, if we replace $\ell_{\rt,\rr}(a,b)$ by $\ell'_{\rt,\rr}(a',b')$, which we defined in terms of $\cS'_k$ and $\cD'_k$ in place of $\cS_k$ and $\cD_k$ in the obvious way. But then,
$|\ell_{\rt,\rr}(a,b) - \ell'_{\rt,\rr}(a',b')| \leq 2\rr\rt^{\delta-1}$ which goes to $0$ in the stated limits uniformly in $a$ and $b$.

Finally, to treat the case $\rt =\infty$, we again define new processes $(\cS'_k)_{k=0}^{\rt'}$ and $(\cD'_k)_{k=1}^{\rt'}$ for some finite $\rt'$ to be chosen depending on $\rr$. To do this, we let $(\cW_s :\: s \in [0,s_{\rt'}])$ be a standard Brownian Motion coupled with $(\cS_k)_{k=0}^{\infty}$ such that $\cS_k = \cW_{s_k}+a$ for all $k=0, \dots,\rt'$. We then define
\begin{equation}
\cW'_s:=(s_{\rt'}-s) \int_0^s\frac{\rmd \cW_t}{s_{\rt'}-t} \,,\quad s\in[0,s_{\rt'}]  \,,
\end{equation}
and finally set $\cS'_k := a + \cW'_{s_k}$ for $k=0, \dots, \rt'$ and $\cD'_k := \cD_k + (\cS'_k - \cS_k)1_{\{k\leq \rr\}}$ for all $k=1, \dots, \rt'$.

Now it is a standard fact that $\cW'_s$ is a Brownian Bridge on $[0,s_\rt']$, so that $\cS'_k$ is a random walk as in Assumption~\ref{i.a1} with $a' = a,b' = a,\rt'$ in place of $a,b,\rt$. At the same time,
\begin{equation}
\cS_k - \cS'_k = \cW_{s_k} - \cW'_{s_k} = \int_0^{s_k}\frac{s_k-t}{s_{\rt'}-t} \rmd \cW_t \,,
\end{equation}
is a centered Gaussian with variance bounded by $s_k^3/(s_{\rt'}-s_k) \leq C \rr^3/(\rt'-\rr)$ for all $k \leq \rr$. It follows from the Union Bound that for any fixed $\rr$, if $\rt'$ is large enough, then Assumption~\ref{i.a3} will hold for $(\cS'_k, \cD'_k)_{k=0}^{\rt'}$. Since the remaining Assumption~\ref{i.a2} holds as well and since $a',b',\rt'$ fall into the domain of uniformity, we have~\eqref{e:20.39}, again with $\ell'_{\rt',\rr}(a',b')$ in place of $\ell_{\rt,\rr}(a,b)$. Lastly we bound the difference between the two quantities by $(\bfE |\cS_\rr - \cS'_\rr|^2)^{1/2}$, which tends to $0$ as $\rt' \to \infty$ followed by $\rr \to \infty$, uniformly in $a$.
\end{proof}

\begin{proof}[Proof of Proposition~\ref{t:LR-as}]
The first part of the proposition is precisely~\eqref{e:20.39}. The second part is an immediate consequence.
\end{proof}

\begin{proof}[Proof of Proposition~\ref{p:1.5}]
The first part of the proposition with $\rr < \infty$ follows immediately from the first part of Lemma~\ref{l:20.6}. To get the desired uniformity and the case $\rr=\infty$, we use~\eqref{e:20.39} from the same lemma. This implies the existence of $\rr_\epsilon > 0$ for all $\eps\in(0,1)$ and $\rt_{(\rr)} \geq \rr$ for all $\rr > 0$ such that 
\begin{equation}
\Bigg|\frac{\ell_{\rt,\rr}(a,b) - \ell_{\rt,\rr_\epsilon}(a,b)}{a^-+1}\Bigg| \leq \epsilon \,,
\end{equation}
for all $\eps\in(0,1)$, $\rr \in [\rr_\epsilon,\infty)$, $\rt \in [\rt_{(\rr)}, \infty]$ and $a,b$ satisfying the conditions in the proposition.

Thanks to~\eqref{e:20.37}, for any $\epsilon > 0$, we may also find $a_\epsilon < 0$ such that 
$|\ell_{\rt,\rr}(a,b)/a^- - 1| \leq \epsilon$, for all $a < a_\epsilon$, $\rr \leq \rr_\epsilon$, $\rt \in [\rr, \infty]$ and $b$ satisfying the conditions in the proposition. It follows from the Triangle Inequality that $|\ell_{\rt,\rr}(a,b)/a^- - 1| \leq 3\epsilon$ for all $a < a_\epsilon$, $\rr \in [1, \infty)$, $\rt \in [\rt_{(\rr)}, \infty]$ and all $b$ in its allowed range. Taking $\rr \to \infty$ then shows the same bound for $\ell_{\infty, \infty}(a,b)$.
\end{proof}

We now turn to the proofs of the continuity statements. They are not difficult, given what we have proved so far.
\begin{proof}[Proof of Lemma~\ref{l:f-cont}]
Writing the difference in the expectations as $\bfE \big( \cS_\rr^- ;\: \max_{k =1}^\rr (\cS_k + \cD_k) \in (-\lambda, \lambda]\big)$ the latter is upper bounded thanks to the Cauchy-Schwarz Inequality and the Union Bound by
\begin{equation}
\Big(\bfE \cS_\rr^2 \sum_{k=1}^\rr
\bfP \big(\cS_k + \cD_k \in (-\lambda, \lambda]\big) \Big)^{1/2} \,.
\end{equation}
As in the proof of Lemma~\ref{l:20.6}, the second moment is bounded above by
$C(\rr + a^2) \leq C' \rr (a^- + 1)^2$. The sum, on the other hand, is bounded by $\rr$ times the supremum of all terms.
\end{proof}

\begin{proof}[Proof of Proposition~\ref{t:ta}]
If $\rr^{(\infty)} < \infty$, then we may assume without loss of generality that $\rr^{(i)} \equiv \rr$ for all $i \geq 1$ and some $\rr < \infty$. Then thanks to the assumption of stochastic continuity, we know that $\cS^{(i)-}_{\rr} 1_{\{\max_{k =1}^\rr (\cS^{(i)}_k + \cD^{(i)}_k) \leq 0\}}$ is $\bfP^{(i)}$-almost surely continuous in $\big(\cS^{(i)}_k\big)_{k=0}^{\rr}$,  $\big(\cD^{(i)}_k\big)_{k=1}^{\rr}$ when $i=\infty$ and bounded by $\cS^{(i)-}_{\rr}$ for all $i \leq \infty$. The latter converges in mean as $i \to \infty$ to $\cS^{(\infty)-}_{\rr}$, under the assumption of~\eqref{e:5.7} as the underlying random variables are Gaussian. \eqref{e:5.10} therefore follows from of~\eqref{e:5.7} in light of the Dominated Convergence Theorem.

When $\rr^{(\infty)} = \infty$, we use the third part of Lemma~\ref{l:20.6} to approximate 
$\ell_{\rr^{(i)}}^{(i)}(a^{(i)})$ by $\ell_{\rr}^{(i)}(a^{(i)})$ up to an arbitrarily small error for all $i \leq \infty$ large enough, by choosing $\rr$ large enough and using the assumptions on $a^{(i)}$, $b^{(i)}$ and $\rt^{(i)}$. The result then follows by the first part of the proposition together with a standard three $\epsilon$-s argument.
\end{proof}

Finally, we give,
\begin{proof}[Proof of Proposition~\ref{p:4.1}]
The case when $a'=b'=0$ is folklore and can be shown, for instance, via a straightforward  modification of the proofs of Proposition~1.1 and Proposition~1.2 in the supplementary material to~\cite{cortines2019structure} (also available as a standalone manuscript in~\cite{cortines2019decorated}). To reduce the general case to that when $a'=b'=0$, one can simply tilt both probabilities in the statement of the proposition by subtracting $a' + (b'-a')s_k/s_\rt$ from the value of $\cS_k$. From the assumptions in the proposition and since $\sigma_k \in (\delta, \delta^{-1})$, we have
\begin{equation}
\frac{|b'-a'|s_k}{s_\rt} \leq \delta^{-2} \rt^{-1/2-\delta} k \leq \delta^{-2} k^{1/2-\delta}
\end{equation}
for all $k \leq \rt/2$. A similar bound holds for $|b'-a'|(s_\rt - s_{\rt-k})/s_\rt$. 

Consequently, the first and second probabilities in the statement of the proposition are bounded above and below respectively by
\begin{equation}
\bfP_{0,a-a'}^{\rt,b-b'} \Big(\max_{k =0}^\rt \big(\cS_k \mp \delta^{-1} \wedge_{\rt,k}^{1/2-\delta} 
\mp \delta^{-2} \wedge_{\rt,k}^{1/2-\delta}\big) \leq 0 \Big) \,.
\end{equation}
Applying now the result in the first case with $(\delta^{-1} + \eps^{-2})^{-1}$ in place of $\delta$, the desired statement follows. 
\end{proof}

\end{appendices}

\section*{Acknowledgments}
\noindent
This project has been supported in part by ISF grant No.~1382/17 and BSF award 2018330. The first author has also been supported in part at the Technion by a Zeff Fellowship and a Postdoctoral Fellowship of the Minerva Foundation.

\bibliography{DGFF}
\bibliographystyle{plain}
\end{document}